\documentclass[10pt,a4paper,reqno]{amsart}
\usepackage[english]{babel}
\usepackage{latexsym}
\usepackage{amssymb,amsbsy,amsmath,amsfonts,amssymb,amscd}
\usepackage{amsmath, amsthm}
\usepackage[utf8]{inputenc}
\usepackage{epsfig, graphicx}
\usepackage{color}
\usepackage{fullpage}
\usepackage{graphics}
\usepackage{wrapfig}
\usepackage{enumitem}
\usepackage{mathrsfs}
\usepackage{environ}
\usepackage{lipsum}
\usepackage{enumitem} 
\usepackage{tcolorbox}
\usepackage{dsfont}
\usepackage[colorlinks=true,linkcolor=wine-stain,citecolor=blue]{hyperref}
\usepackage{bm}
\usepackage{mathtools,accents}
\usepackage{calc}
\usepackage{etoolbox}
\usepackage[foot]{amsaddr}
\usepackage{breakcites}
\usepackage{bbm}
\usepackage{bm}

\allowdisplaybreaks
\definecolor{wine-stain}{rgb}{0.7,0,0}

\newcommand{\pr}[1]{\left(#1 \right)}
\newcommand{\LRVert}[1]{\left\Vert #1 \right\Vert}

\makeatletter
\patchcmd{\@maketitle}
  {\ifx\@empty\@dedicatory}
  {\ifx\@empty\@date \else {\vskip3ex \centering\footnotesize\@date\par\vskip1ex}\fi
   \ifx\@empty\@dedicatory}
  {}{}
\patchcmd{\@adminfootnotes}
  {\ifx\@empty\@date\else \@footnotetext{\@setdate}\fi}
  {}{}{}
\makeatother

\def\X{\mathrm{X}}

\def\V{\mathrm{V}}

\def\W{\mathrm{W}}
\def\J{\mathrm{J}}
\newcommand{\dd}{\mathrm{d}}

\def\pa{\partial}

\def\R{\mathbb R}
\def\Z{\mathbb Z}
\def\N{\mathbb N}

\def\T{\mathbb T}

\def\det{\mathrm{det}}

\def\Ld{\mathrm{L}}
\def\H{\mathrm{H}}
\def\M{\mathrm{M}}
\def\B{\mathrm{B}}
\def\E{\mathrm{E}}
\def\W{\mathrm{W}}

\newcommand{\na}{\nabla}

\newcommand{\reg}{\textnormal{reg}}

\newcommand{\eps}{\varepsilon}

\newtcolorbox{dev}{arc=0pt}

\counterwithin{compteur}{section}

\definecolor{thmcolor}{rgb}{0.8,0.14,0.2}
\definecolor{defcolor}{rgb}{0.0,0.50,0.0}
\definecolor{excolor}{rgb}{0.50,0.0,0.990}
\definecolor{applicolor}{rgb}{0.50,0.0,0.990}

\newtheorem{thm}{Theorem}[section]

\newtheorem{coro}[thm]{Corollary}
\newtheorem{propo}[thm]{Proposition}
\newtheorem{lem}[thm]{Lemma}
\newtheorem{nota}[thm]{Notation}
\newtheorem{defi}[thm]{Definition}

\theoremstyle{definition}
\newtheorem{rem}[thm]{Remark}

\numberwithin{equation}{section}

\makeatletter

\renewcommand{\tocsection}[3]{%
  \indentlabel{\@ifnotempty{#2}{\bfseries\ignorespaces#1 #2\quad}}\bfseries#3}
\renewcommand{\tocsubsection}[3]{%
  \indentlabel{\@ifnotempty{#2}{\ignorespaces#1 #2\quad}}#3}
\renewcommand{\tocsubsubsection}[3]{%
  \indentlabel{\@ifnotempty{#2}{\ignorespaces#1 #2\quad}}#3}

\newcommand\@dotsep{4.5}
\def\@tocline#1#2#3#4#5#6#7{\relax
  \ifnum #1>\c@tocdepth 
  \else
    \par \addpenalty\@secpenalty\addvspace{#2}%
    \begingroup \hyphenpenalty\@M
    \@ifempty{#4}{%
      \@tempdima\csname r@tocindent\number#1\endcsname\relax
    }{%
      \@tempdima#4\relax
    }%
    \parindent\z@ \leftskip#3\relax \advance\leftskip\@tempdima\relax
    \rightskip\@pnumwidth plus1em \parfillskip-\@pnumwidth
    #5\leavevmode\hskip-\@tempdima{#6}\nobreak
    \leaders\hbox{$\m@th\mkern \@dotsep mu\hbox{.}\mkern \@dotsep mu$}\hfill
    \nobreak
    \hbox to\@pnumwidth{\@tocpagenum{\ifnum#1=1\bfseries\fi#7}}\par
    \nobreak
    \endgroup
  \fi}
\AtBeginDocument{%
\expandafter\renewcommand\csname r@tocindent0\endcsname{0pt}
}
\def\l@subsection{\@tocline{2}{0pt}{2.5pc}{5pc}{}}
\def\l@subsubsection{\@tocline{3}{0pt}{4pc}{5pc}{}}
\makeatother

\usepackage{cancel}

\addtocontents{toc}{\protect\setcounter{tocdepth}{2}}

\newcommand{\nocontentsline}[3]{}
\newcommand{\tocless}[2]{\bgroup\let\addcontentsline=\nocontentsline#1{#2}\egroup}

\title{\Large On well-posedness for thick spray equations}
\author{\large Lucas Ertzbischoff$^\sharp$}%
\address{$^\sharp$Centre de Math\'ematiques Laurent Schwartz (UMR 7640), Ecole polytechnique, Institut Polytechnique de Paris, 91128 Palaiseau Cedex, France (\href{mailto:lucas.ertzbischoff@polytechnique.edu}{lucas.ertzbischoff@polytechnique.edu})}
\author{\large Daniel Han-Kwan$^\flat$}
\address{$^\flat$CNRS, Laboratoire de Mathématiques Jean Leray (UMR 6629), Université de Nantes, 44322 Nantes Cedex 03, France (\href{mailto:daniel.han-kwan@univ-nantes.fr}{daniel.han-kwan@univ-nantes.fr})}

\let\origmaketitle\maketitle
\def\maketitle{
  \begingroup
  \let\MakeUppercase\relax 
  \origmaketitle
  \endgroup
}

\begin{document}


\begin{abstract}
In this paper, we prove the local in time well-posedness of thick spray equations in Sobolev spaces, for initial data satisfying a Penrose-type stability condition. 
This system is a coupling between particles described by a kinetic equation and a surrounding fluid governed by compressible Navier-Stokes equations. In the thick spray regime, the volume fraction of the dispersed phase is not negligible compared to that of the fluid. 

We identify a suitable stability condition bearing on the initial conditions that provides estimates without loss, 
ensuring that the system is well-posed. It coincides with a Penrose condition appearing in earlier works on singular Vlasov equations. We also rely on crucial new estimates for averaging operators. Our approach allows to treat many variants of the model, such as collisions in the kinetic equation, non-barotropic fluid or density-dependent drag force.

\end{abstract}

\ \vskip -1cm  \hrule \vskip 1cm \vspace{-8pt}
 \maketitle 
{ \textwidth=4cm \hrule}

\renewcommand{\contentsname}{Contents}

\vspace{6pt}

{ \hypersetup{linktoc=page, linkcolor=wine-stain}
\tableofcontents
}

\section{Introduction and main results}
\subsection{System and notations}
In this paper, we are interested in the following coupling between fluid and particles:
\begin{equation}\label{eq:TSgenBaro}
\tag{TS}
\left\{
      \begin{aligned}
\partial_t f +v \cdot \nabla_x f +{\rm div}_v \big[ f (u-v)-f \nabla_x p(\varrho) \big]&=0,  \\
\partial_t (\alpha \varrho ) + \mathrm{div}_x (\alpha \varrho u)&=0, \\
\partial_t (\alpha \varrho u) + \mathrm{div}_x(\alpha \varrho u \otimes u)+\alpha\nabla_x p -\mathrm{div}_x( \tau)&=j_f-\rho_f u. \\
\end{aligned}
    \right.
\end{equation}

This system describes the evolution of a cloud of particles (e.g. droplets or dust specks) in an underlying compressible fluid (e.g. a gas). Such a suspension is commonly referred to as a \textit{spray} \cite{will}. More generally, the system \eqref{eq:TSgenBaro} belongs to the broad family of ``multiphase flows'' equations \cite{D}.

In this work, we study \eqref{eq:TSgenBaro} in the phase space $\in \T^d \times \R^d$, with $d \in \N {\setminus} \lbrace 0 \rbrace$. The first equation of \eqref{eq:TSgenBaro} is a kinetic equation of Vlasov-type on the particle distribution function $f(t,x,v) \in \R^+$ in the phase space (position-velocity), set for $t>0$ and $(x,v) \in \T^d \times \R^d$.  The other equations of \eqref{eq:TSgenBaro} are set for $t>0$ and $x \in \T^d$ and are barotropic compressible Navier-Stokes equations on the fluid density $\varrho(t,x) \in \R^+$ and fluid velocity $u(t,x) \in \R^d$. Here, the function $\alpha(t,x) \in [0,1]$ is the volume fraction of the fluid.

This type of models can be used to capture various natural phenomena and has widespread applications. Examples include, among others, 
combustion phenomena in engines \cite{Kiva2, Laurent-PHD},
evaporation of droplets \cite{massot2001modelisation, de2009modeles},
aerosols for medical purposes \cite{boudin2020fluid,BGLM},
marine and volcanic aerosols and their impact on the atmosphere \cite{veron2020eulerian, mather2003tropospheric}, 
as well as aerosols in the atmosphere of gas giants or exoplanets \cite{west1986clouds, gao2021aerosols}.

Here, the system \eqref{eq:TSgenBaro} describes the so-called \textit{thick sprays} and has been introduced and derived by O'Rourke in \cite{oro}. 
Such coupling is appropriate for modeling two-phase mixtures where particles are small but occupy a non-negligible volume fraction of the whole suspension. The thick (or dense) spray regime is tipycally found in regions where droplets are injected in a carrying gas (see \cite{Duko,oro, liu2002science}). The system \eqref{eq:TSgenBaro} has also been recognized as a set of equations linked to multilfuid systems, which are thoroughly described in \cite{IshiiHibiki}. 
Further details can be found in the overview provided in Section \ref{Section:Overview}.


Let us detail and close the previous equations (with a normalization of all coefficients), explaining the meaning of the different terms which are involved. 
\begin{itemize}
\item The kinetic moments (of order $0$ and $1$) of the distribution function $f$ are defined as
\begin{align*}
\rho_f(t,x)&:= \int_{\R^d} f(t,x,v) \, \mathrm{d}v,    \ \ \mathrm{(Local \ density)}, \\
j_f(t,x)&:=\int_{\R^d} f(t,x,v) v \, \mathrm{d}v,  \ \ \mathrm{(Local \ current)}.
\end{align*}
\item The volume fraction $\alpha=\alpha(t,x) \in \R^+$ of the fluid is given by
 \begin{align}
 \alpha(t,x):=1-\int_{\R^d} f(t,x,v) \, \mathrm{d}v=1-\rho_f(t,x).
 \end{align}
In the \textit{thick spray} regime, this quantity is not assumed to be close to $1$ (this concerns suspensions for which typically, $\alpha$ is around the value $0.9$),  so that the volume fraction for the cloud of particles is not negligible compared to that of the fluid. It is in sharp contrast with the \textit{thin spray} regime, that we shall recall below (see \eqref{eq:VNSthin} in Section \ref{Section:Overview}), where $\alpha$ is somehow directly set to $1$ and thus does not explicitly appear in the system. Here, this quantity induces an extra coupling between both phases. We refer to \cite{oro, Duko, D} for comparisons between the thin and thick regimes.
\item The pressure $p : \R^+ \rightarrow \R^+ $ is a given $\mathscr{C}(\R^+ ) \cap \mathscr{C}^{\infty}( \R^+ {\setminus} \lbrace 0 \rbrace)$ function such that, in the barotropic regime, the pressure term is a function $p=p(\varrho)$ depending only on the fluid density. One usually assumes that $p(0)=0$ and $p'(\rho)>0$. For reasons that will appear later, we shall also assume that $p$ is such that 
\begin{align}\label{Assumption-pressure}
\rho \mapsto \rho p'(\rho) \, \, \text{is nondecreasing on} \ \ \R^+.
\end{align}
A common example is for instance $p(\varrho)=\varrho^{\gamma}$ for some $\gamma>1$. 

\item The viscous stress tensor $\tau=\tau[u]$ is given by 
\begin{align*}
\tau=2 \mu \mathrm{D}(u)+\lambda \mathrm{div}_x  u \,  \mathrm{I}_d,
\end{align*}
for some constants $\nu >0$ and $\lambda \in \R$, and where $\mathrm{D}(u)$ stands for the deformation tensor defined as
$
\mathrm{D}(u)=(\nabla_x u+(\nabla_x u)^\intercal)/2
$. In this paper (see the possible generalization in Section \ref{Section:GeneralizeIntro}), we choose the constants $\mu$ and $\lambda$ so that
\begin{align*}
\mathrm{div}_x( \tau[u])=\Delta_x u + \nabla_x \mathrm{div}_x u,
\end{align*}
but we emphasize that this choice has been made solely for simplicity, and that no special algebraic property arises in this case.
\item The force $\Gamma=\Gamma(t,x,v) \in \R^d$ acting on the cloud of particles is defined by
\begin{align*}
\Gamma(t,x,v)=u(t,x)-v-\nabla_x [p(\varrho)](t,x).
\end{align*}
The first term $u(t,x)-v$, referred to as the drag force exerted by the fluid on the particles, is common to all fluid-kinetic couplings. Here, it is linear in the relative velocity between fluid and particles, and creates friction. The retroaction of this term in the momentum equation for the fluid is known as the \textit{Brinkman force} and can be expressed as
\begin{align*}
-\int_{\R^d} (u-v)f \, \mathrm{d}v= j_f-\rho_f u.
\end{align*} 
The second term in $\Gamma$, which is a pressure gradient from the fluid, is a specific feature of thick spray models (see also \cite{Duko, oro}) where the particles occupy a significant volume fraction of the two-phase mixture.
The presence of this term is consistent with the fact that the system \eqref{eq:TSgenBaro} (or some of its variants) is formally linked to bifluid equations \cite{DesMathiaud}, where there is a common pressure gradient to both phases (see the overview below in Section \ref{Section:Overview}). Note that the feedback of this term in the source term of the fluid momentum equation is
\begin{align*}
-\int_{\R^d} f(-\nabla_x [p(\varrho)]) \, \mathrm{d}v= \nabla_x [p(\varrho)]\rho_f.
\end{align*}
Since $\alpha=1-\rho_f$, it explains the term $\alpha \nabla_x [p(\varrho)]$ in the left-hand side of the equation on $u$ in \eqref{eq:TSgenBaro}.

\end{itemize}
The unknowns of the problem are thus
\begin{align*}
f=f(t,x,v) \in  \R^+, \ \ \varrho=\varrho(t,x) \in \R^+, \ \ u=u(t,x) \in \R^d,
\end{align*}
and they are coupled through the drag term $u-v$, the pressure gradient $\nabla_x p(\varrho)$ and the volume fraction of the fluid $\alpha=1-\rho_f$. We finally prescribe initial conditions
\begin{align*}
f^{\mathrm{in}}=f^{\mathrm{in}}(x,v) \in  \R^+, \ \ \varrho^{\mathrm{in}}=\varrho^{\mathrm{in}}(x) \in \R^+, \ \ u^{\mathrm{in}}=u^{\mathrm{in}}(x) \in \R^d.
\end{align*}
We normalize the torus $\T^d := \R^d /(2\pi \Z)^d$ endowed with the normalized Lebesgue measure, so that $\mathrm{Leb}(\T^d)=1$.

\medskip

In this work, we investigate the \textit{local well-posedness} of the thick spray equations \eqref{eq:TSgenBaro}. The main difficulty comes from the fact that rough energy estimates on the transport and kinetic part of \eqref{eq:TSgenBaro} seem to result in a \textit{loss of several derivatives}. As a matter of fact, suppose that we have a smooth solution $(f, \varrho,u)$ with compact support to the system \eqref{eq:TSgenBaro}. Since it is a coupling between a parabolic type equation on $u$ and two transport equations on $f$ and $\varrho$, the following observations can be made:
\begin{itemize}
\item[$-$] a standard energy estimate for transport equations shows that a control of $k$ derivatives of $\varrho$ seems to require the control of $k+1$ derivatives of $f$. This is due to the coupling with the volume fraction $\alpha$ in the mass conservation equation;
\item[$-$] a standard energy estimate for transport(-kinetic) equations shows that a control of $k$ derivatives of $f$ seems to require the control of $k+1$ derivatives of $\varrho$. This comes from the pressure gradient in the force field of the Vlasov equation.
\end{itemize}
As a consequence, we obtain a control of $k$ derivatives of $f$ by $k+2$ derivatives of $f$ (and similarly for $\varrho$) for any $k \in \N$, and so on. Rough estimates of transport type therefore seem to involve a formal loss of $2$ derivatives on $f$ or $\varrho$, which \textbf{makes the coupling singular}. Hence, standard techniques cannot be applied to obtain a solution to the system. In \cite{BarD}, Baranger and Desvillettes conjectured anyway that the system is well-posed in Sobolev spaces.

In this article, we partly confirm this conjecture by showing that \eqref{eq:TSgenBaro} is locally well-posed in Sobolev spaces, when the initial data satisfy a \textit{stability condition} of Penrose type. 
However, when this stability condition is violated, the system is actually ill-posed in the sense of Hadamard, see \cite{BEHK-ill}, which means, loosely speaking, that it indeed displays losses of derivatives in this case.



\medskip

The rest of the introduction is structured as follows. In Section \ref{Section:MainResult}, we introduce the Penrose stability condition for thick sprays and state our main result of local well-posedness for \eqref{eq:TSgenBaro}. In Section \ref{Section:GeneralizeIntro}, we describe several generalizations of the thick spray equations \eqref{eq:TSgenBaro}, taking into account possible density-dependent drag or collisions in the kinetic equation, as well as non-barotropic Navier-Stokes equations. These variants will be treated in Sections \ref{Section:AppendixInternalEnergy}--\ref{Section:AppendixInelasticBoltzmann}--\ref{Section:AppendixDragTerm}. Section \ref{Section:Overview} is a general overview on fluid-kinetic systems, as well as on singular Vlasov equations. Finally, Section \ref{Section:Strat} provides a detailed outline of our method of proof.

\subsection{Assumptions and main result}\label{Section:MainResult}
For $\alpha=(\alpha_1, \cdots, \alpha_d) \in \N^d$, we write 
$\vert \alpha \vert= \sum_{i=1}^d \alpha_i$ and $ \partial_x^{\alpha}= \partial_{x_1}^{\alpha_1} \cdots \partial_{x_d}^{\alpha_d}$.
In this article, we will denote by $\H^k$ (or $\H^k(\T^d)$) the standard $\Ld^2$ Sobolev spaces for functions depending on $x \in \T^d$. When it is necessary, we will denote $\H^k_{x,v}$ the same spaces for functions depending on $(x,v) \in \T^d \times \R^d$. To ease readibility, we shall sometimes abbreviate $\Ld^2(0,T;\Ld^2(\T^d))$ and $\Ld^2(0,T;\H^k(\T^d))$ as $\Ld^2_T \Ld^2$ and $\Ld^2_T \H^k$.  We will also use {\it weighted} $\Ld^2_{x,v}-$Sobolev spaces. They are defined as follows:
\begin{defi}
For $k \in \N$ and $r \geq 0$ and $\mathrm{f}:\T^d \times \R^d \rightarrow \R^+$, we define the weighted (in velocity) Sobolev norms as
\begin{align*}
\Vert \mathrm{f}  \Vert_{\mathcal{H}_r^k}:=\left(\sum_{\vert \alpha \vert + \vert \beta \vert \leq k} \int_{\T^d} \int_{\R^d} \langle v \rangle^{2r} \vert \partial_x^{\alpha} \partial_v^{\beta} \mathrm{f}(x,v) \vert^2 \, \mathrm{d}x \, \mathrm{d}v  \right)^{\frac{1}{2}},
\end{align*}
where
$\langle v \rangle=(1+\vert v \vert^2)^{\frac{1}{2}}$.
\end{defi}

We now introduce the following Penrose function, which will be used to state the relevant corresponding Penrose stability condition.
\begin{defi}
For any distribution function $\mathrm{f}(x,v)$ and density $\rho(x)$, we define the Penrose function
\begin{align}\label{def:PenroseSymbol}
\mathscr{P}_{\mathrm{f}, \rho}(x,\gamma, \tau, k):=\frac{p'(\rho(x))\rho(x)}{1-\rho_{\mathrm{f}}(x)}\int_0^{+\infty} e^{-(\gamma+i \tau) s} \frac{i k}{1+ \vert k \vert^2} \cdot \left( \mathcal{F}_v \nabla_v \mathrm{f} \right)(x,ks)  \, \mathrm{d}s,
\end{align}
for $(x,\gamma, \tau,k)\in \T^d \times (0,+\infty) \times \R  \in \R^d {\setminus} \lbrace 0 \rbrace$. 
\end{defi}

\begin{defi}
We say that the couple $(\mathrm{f}(x,v), \rho(x))$ satisfies the $c$-\textbf{Penrose stability condition} (for the thick spray equations \eqref{eq:TSgenBaro}) if there exists $c>0$ such that
\begin{align}\label{cond:Penrose}
\tag{\textbf{P}}
\forall x \in \T^d, \ \ \underset{(\gamma, \tau,k)\in (0,+\infty) \times \R \times \R^d {\setminus} \lbrace 0 \rbrace}{\inf} \, \vert 1- \mathscr{P}_{\mathrm{f}, \rho}(x,\gamma, \tau, k) \vert>c.
\end{align}
When needed, we shall denote this condition by $\eqref{cond:Penrose}_c$.
\end{defi}

As we shall explained later on, such a condition stems from the study of Vlasov equations in plasma physics. 
In the context of the Vlasov-Benney equation,  a similar condition was the key to obtain local well-posedness in Sobolev spaces (and for its derivation in the quasineutral limit), see \cite{HKR}.
We refer to Section \ref{Section:Strat} for more details.

\begin{rem}[Sufficient conditions ensuring \eqref{cond:Penrose}]\label{Rmk:Penrose}
Let us describe a few classes of profiles $(\mathrm{f}(x,v), \rho(x))$ which satisfy the Penrose stability condition \eqref{cond:Penrose}.
Several criteria on the shape of $\mathrm{f}(x,\cdot)$ for every $x \in \T^d$ indeed provide sufficient condition for \eqref{cond:Penrose} to hold (see for instance \cite{MV}). We shall assume that $\mathrm{f}(x,\cdot)$ is at least integrable and  that the nonnegative prefactor in front of the integral in $\mathscr{P}_{\mathrm{f}, \rho}$ is bounded from above on $\T^d$.
\begin{itemize}
\item First, any sufficiently small smooth profile $\mathrm{f}$ satisfies \eqref{cond:Penrose}.
\item When $d=1$, the one bump profiles in velocity, that is to say profiles such that for all $x \in \T^d$, the function $v \mapsto \mathrm{f}(x, v)$ is increasing then decreasing, satisfy the Penrose condition. 
\item In any dimension $d \geq 1$, any profile such that for all $x \in \T^d$, $f(x ,\cdot)$ is a  radial non-increasing function in velocity, is Penrose stable. In particular, it includes the case of (local smooth) Maxwellians in velocity. If $d \geq 3$,  any profile such that for all $x \in \T^d$, $f(x ,\cdot)$ is a radial and strictly positive function in velocity, satisfies the condition.
\item More generally, the following criterion on the marginals of $\mathrm{f}$ has been devised in \cite{MV} and ensures the Penrose stability: for all $x \in \T^d$, 
\begin{align*}
\forall \omega_0 \in \R, \ \ \forall k \in \R^d {\setminus} \lbrace 0 \rbrace, \ \ \frac{1}{1+ \vert k \vert^2}\frac{p'(\rho(x))\rho(x)}{1-\rho_{\mathrm{f}}(x)}\mathrm{p.v.}   \int_{\R}  \partial_y \mathrm{f}_{\frac{k}{\vert k \vert}}(x,y)\frac{1}{y+ \omega_0} \, \mathrm{d}y<1,
\end{align*}
where $p.v.$ stands for the principal value on $\R$ and where  
\begin{align*}
\forall r \in \R, \ \ \mathrm{f}_{\frac{k}{\vert k \vert}}(r):=\int_{{\frac{k}{\vert k \vert}}^{\perp}} \mathrm{f}\left(r \frac{k}{\vert k \vert}+w \right) \, \mathrm{d}w.
\end{align*}
\item Finally, any sufficiently small smooth perturbation of a Penrose stable profile will still satisfy \eqref{cond:Penrose}. In particular, a slightly perturbed one-bump profile  $f(x,\cdot)$ remains Penrose stable.
\end{itemize}
\end{rem}

Our \textbf{main result} reads as follows.

\begin{thm}\label{THM:main}
There exist $m_0>0$ and $r_0>0$, depending only on the dimension, such that the following holds for all $m \geq m_0$ and $r \geq r_0$. Let 
\begin{align*}
f^{\mathrm{in}} \in  \mathcal{H}^{m}_{r}, \ \ \varrho^{\mathrm{in}} \in \mathrm{H}^{m+1}, \ \ u^{\mathrm{in}} \in \mathrm{H}^{m},
\end{align*}
such that $(f^{\mathrm{in}}, \varrho^{\mathrm{in}})$ satisfies the $c$-Penrose stability condition $\eqref{cond:Penrose}_c$ for some $c>0$ and 
\begin{align*}
0 &\leq f^{\mathrm{in}}, \ \ \ \rho_{f^{\mathrm{in}}}<\Theta<1, \ \ \ 0<\mu \leq \varrho^{\mathrm{in}}, \ \ \ 0<\underline{\theta} \leq (1-\rho_{f^{\mathrm{in}}})\varrho^{\mathrm{in}} \leq \overline{\theta},
\end{align*}
for some constants $\Theta, \mu, \underline{\theta}, \overline{\theta}$. 
Then there exist $T>0$ and a solution $(f, \varrho,u)$ to \eqref{eq:TSgenBaro} with initial condition $(f^{\mathrm{in}}, \varrho^{\mathrm{in}}, u^{\mathrm{in}})$ such that
\begin{align*}
f \in  \mathscr{C}\left([0,T];\mathcal{H}^{m-1}_{r} \right), \ \ \varrho \in \Ld^2(0,T;\mathrm{H}^m), \ \ u \in \mathscr{C}\left([0,T]; \mathrm{H}^{m} \right)\cap \Ld^2(0,T;\mathrm{H}^{m+1}),
\end{align*}
and with $(f(t), \varrho(t))$ satisfying the $c/2$-Penrose stability condition $\eqref{cond:Penrose}_{c/2}$ for all $t \in [0,T]$. In addition, this solution is unique in this regularity class.
\end{thm}

In short, our result yields the local well-posedness for the thick spray equations, in the class of Penrose stable initial data. 
On the other hand, as already mentioned earlier, outside of this class, the system is  ill-posed in the sense of Hadamard. {We refer to \cite{BEHK-ill}, in the spirit of \cite{HKN, baradat2020nonlinear} for singular Vlasov equations (see also the discussion in Section \ref{Section:Overview}).}

\begin{rem}
The uniqueness part in the previous statement must be understood as follows: if $(f_1, \varrho_1,u_1)$ and $(f_2, \varrho_2,u_2)$ are two solutions  to \eqref{eq:TSgenBaro} on $[0,T]$ with the previous regularity and with the same initial condition $(f^{\mathrm{in}}, \varrho^{\mathrm{in}}, u^{\mathrm{in}})$ (satisfying \eqref{cond:Penrose}), then $(f_1, \varrho_1,u_1)=(f_2, \varrho_2,u_2)$ if $t \mapsto(f_1(t), \varrho_1(t))$ satisfies the Penrose stability condition \eqref{cond:Penrose} on $[0,T]$.

\end{rem}
\begin{rem}
Let us point out the shift of one derivative in the regularity between $f$ ans $\varrho$, which is reminiscent of the formal loss of derivatives that was evoked earlier.
\end{rem}

\begin{rem}
In \cite{BDD}, the authors consider the linearization of \eqref{eq:TSgenBaro} around radially non-increasing  and homogeneous profile (for the kinetic part). This can be seen as a particular case of the Penrose stability condition \eqref{cond:Penrose}.
They obtain a stability estimate in $\Ld^2$ around the solution generated by such particular data.
However, it is not sufficient to provide a well-posedness result for the full non-linear equations. As a matter of fact, since the equations are quasilinear, one would need to prove the analogue of such stability estimates for all functions in a neighborhood of the aforementioned solution.
\end{rem}

\begin{rem}\label{Rem:Penrose}
\begin{enumerate}
\item It appears to be more natural (see in particular the  proof of uniqueness in Section~\ref{Subsection:Uniqueness}) to consider the optimal Penrose function
\begin{align*}
\mathcal{P}_{\mathrm{f}, \rho}(x,\gamma, \tau,k):=\frac{p'(\rho(x))\rho(x)}{1-\rho_{f}(x)}\int_0^{+\infty} e^{-(\gamma+i \tau) s} i k \cdot \left( \mathcal{F}_v \nabla_v \mathrm{f} \right)(x,ks)  \, \mathrm{d}s, 
\end{align*}
instead of $\mathscr{P}_{\mathrm{f}, \rho}$, as well as the related stability condition
\begin{align}\label{condPenroseOptiINTRO}
\tag{\textbf{Opt-P}}
\forall x \in \T^d, \ \ \underset{(\gamma, \tau,k)\in (0,+\infty) \times \R \times \R^d {\setminus} \lbrace 0 \rbrace}{\inf} \, \vert 1- \mathcal{P}_{\mathrm{f}, \rho}(x,\gamma, \tau, k) \vert>c.
\end{align}
We refer to \eqref{condPenroseOptiINTRO} as the \textit{optimal Penrose stability condition}. The condition \eqref{condPenroseOptiINTRO} is weaker than \eqref{cond:Penrose}. Indeed, using a homogeneity argument combined with a continuity argument, it is possible to prove that if \eqref{cond:Penrose} holds for some $c>0$ for $(\mathrm{f}, \rho)$, then \eqref{condPenroseOptiINTRO} holds as well\footnote{We refer to the uniqueness part of the proof in Section \ref{Subsection:Uniqueness} for more details.}. 

However, our strategy in this article will be based on a regularization of the force field in the Vlasov equation, and requires the stability condition \eqref{cond:Penrose}. It is likely that, using the techniques of \cite{CHKR}, one could assume \eqref{condPenroseOptiINTRO} instead of \eqref{cond:Penrose} on $(f^{\mathrm{in}}, \varrho^{\mathrm{in}})$ to prove the well-posedness of thick spray equations. However, this would require substantial work. 
\item The factor $\frac{1}{1+ \vert k \vert^2}$ in the Penrose function $\mathscr{P}_{\mathrm{f}, \rho}$ could appear as arbitrary. It is actually related to the explicit regularization procedure of the force field that will be made clearer later on.
By a homogeneity argument, it is however possible to prove that the condition \eqref{cond:Penrose} is equivalent to
\begin{align}\label{Penrose:ReformuleLAMBDA}
\forall x \in \T^d, \ \ \underset{(\gamma, \tau,k, \lambda)\in S^+ \times (0,1]  }{\inf} \, \left\vert 1- \lambda\mathcal{P}_{\mathrm{f}, \rho}(x,\gamma, \tau,k)  \right\vert>c,
\end{align}
where 
\begin{align*}
S^+:= \left\lbrace (\gamma, \tau, k) \in (0,+\infty) \times \R \times \R^d {\setminus} \lbrace 0 \rbrace \mid \gamma^2+ \tau^2+ k^2=1 \right\rbrace.
\end{align*}
This implies in particular that the factor $\frac{1}{1+ \vert k \vert^2}$ could by instance be replaced by  $\frac{1}{(1+ \vert k \vert^2)^\alpha}$ for any $\alpha>0$ in~\eqref{def:PenroseSymbol}.

We refer to Remark \ref{rem:RegProcedure} in Section \ref{Subsection:EstimateREG} for the use of such reformulation of \eqref{cond:Penrose} in the regularization procedure.
\end{enumerate}
\end{rem}

\begin{rem}
In the Euler case for the fluid, that is for the same system on $(f,\varrho,u)$ but {\bf without} the term $-\Delta_x u - \nabla_x \mathrm{div}_x u$ in the equation for $u$, the question of well-posedness remains open. 
\end{rem}

\subsection{Generalization to several variants}\label{Section:GeneralizeIntro}
We are also interested in more complex versions of the system \eqref{eq:TSgenBaro} that take into account more physical effects. In this section, several such variants are presented. The main strategy of proof performed in this article for \eqref{eq:TSgenBaro} will be robust enough to handle such models: we will show how to obtain their local well-posedness in Sections \ref{Section:AppendixInternalEnergy}--\ref{Section:AppendixInelasticBoltzmann}--\ref{Section:AppendixDragTerm}.

\subsubsection{Non-barotropic Navier-Stokes equations}
If one wants to get rid of the barotropic-type assumption on the fluid, one has to consider its internal energy $\mathfrak{e}=\mathfrak{e}(t,x) \in \R^+$ as an additional unknown. This leads to the following system of equations
\begin{equation}\label{eq:TSgenFULL2}
\left\{
      \begin{aligned}
\partial_t f +v \cdot \nabla_x f +{\rm div}_v \Big[ f (u-v)-f \nabla_x p(\varrho, \mathfrak{e}) \Big]&=0,  \\
\partial_t (\alpha \varrho ) + \mathrm{div}_x (\alpha \varrho u)&=0, \\
\partial_t (\alpha \varrho u) + \mathrm{div}_x(\alpha \varrho u \otimes u)+\alpha \nabla_x p(\varrho, \mathfrak{e})-\Delta_x u - \nabla_x \mathrm{div}_x u&=j_f-\rho_f u, \\
\partial_t (\alpha \varrho \mathfrak{e}) + \mathrm{div}_x(\alpha \varrho \mathfrak{e}u) +p(\varrho, \mathfrak{e})\left(\partial_t \alpha+ \mathrm{div}_x(\alpha u) \right)&=\int_{\R^d}\vert u-v \vert^2 f \, \mathrm{d}v, \\
\alpha&=1-\rho_f.
\end{aligned}
    \right.
\end{equation}
Here, $p : \R^+ \times \R^+ \rightarrow \R$ is a given $\mathscr{C}(\R^+ \times \R^+ ) \cap \mathscr{C}^{\infty}(\R^+ {\setminus} \lbrace 0 \rbrace \times \R^+ {\setminus} \lbrace 0 \rbrace)$ function such that the pressure term is a function $p=p(\varrho, \mathfrak{e})$ depending on the fluid density and internal energy. For instance, a relation of the type $p(\varrho,\mathfrak{e})=b \varrho \mathfrak{e}$ (for some $b>0$) is a perfect gas pressure law. 

The apparent loss of derivatives is still present in such model and occurs between $f$ and $\mathfrak{e}$. In Section \ref{Section:AppendixInternalEnergy}, we shall also justify how to build local in time solutions for \eqref{eq:TSgenFULL2} thanks to an adapted Penrose stability condition.

\subsubsection{Kinetic equation with collisions}
In the thick spray regime, it is physically relevant to take into account collisions between droplets. A collision operator is therefore sometimes added in the kinetic equation, which turns into a Vlasov-Boltzmann type:
$$\partial_t f +v \cdot \nabla_x f +{\rm div}_v \Big[ f (u-v)-f \nabla_x p(\varrho) \Big]=Q(f,f).$$
The quadratic operator $\mathcal{Q}(f,f)=\mathcal{Q}_{\lambda}(f,f)$ stands for some Boltzmann collision operator for (in)elastic hard-spheres. Here $\lambda \in (0,1]$ is given and is called the \textit{restitution coeficcient} which is involved in the microscopic laws defining collisions between particles. We refer to Appendix \ref{Section:AppendixInelasticBoltzmann} for some details about the precise definition of the collision operator $\mathcal{Q}_{\lambda}$ and some of the main basic features of inelastic collisions. Let us mention that the case $\lambda=1$ corresponds to standard perfectly elastic collisions and that the inelastic case $\lambda \in (0,1)$ leads to a loss of kinetic energy along collisions (while mass and momentum are always conserved).

In Section \ref{Section:AppendixInelasticBoltzmann}, we will explain how to include a collision operator in the kinetic equation on $f$ and still obtain an analogue of Theorem~\ref{THM:main}.

\subsubsection{Density-dependent drag force}
In many applications, the force $\Gamma=\Gamma(t,x,v) \in \R^d$ acting on the particle should actually present a density-dependent drag force, for instance of the form
\begin{align*}
\Gamma(t,x,v)=\varrho(t,x)(u(t,x)-v)-\nabla_x [p(\varrho)](t,x).
\end{align*}
Compared to the one of \eqref{eq:TSgenBaro}, this force displays an additional nonlinearity\footnote{Note that a more physical model should deal with a nonlinear drag of the form $\mathrm{C}[\varrho, \vert u-v\vert](u-v)$, which is for the moment out of the scope of a rigorous mathematical analysis. However, our approach can for instance allow the treatment of a drag of the form $\varrho[(u-v)+\gamma(u-v)]$ where $\gamma \in \mathscr{C}^{\infty}(\R^d; \R^d)$ is such that $\gamma(0)=0$ and $\gamma'(0)=0$.}. 
The Brinkman force in the Navier-Stokes equations also becomes
\begin{align*}
-\int_{\R^d} \varrho(u-v)f \, \mathrm{d}v= \varrho(j_f-\rho_f u).
\end{align*}
Because of the modified term $\varrho(t,x) v \cdot \nabla_v f$ in the kinetic equation, this induces a potential growth in velocity which could become out of control. 
In Section \ref{Section:AppendixDragTerm}, we will deal with the case of such density-dependent drag term, up to the additional assumption that the initial data $f^{\mathrm{in}}$ has a compact support in velocity.

\subsubsection{Density-dependent viscosities}
It is also possible to consider more general viscosity coefficients in the Navier-Stokes equations of \eqref{eq:TSgenBaro}, that is replacing the differential operator $-\Delta_x u - \nabla_x \mathrm{div}_x u$ in the equation for $u$ by
\begin{align*}
-\mathrm{div}_x \left(2 \mu[\varrho] \mathrm{D}(u)+\lambda[\varrho] \mathrm{div}_x  u \,  \mathrm{I}_d \right),
\end{align*}
for smooth non-negative coefficients $\mu, \lambda: \mathbb{R}^+ \rightarrow \mathbb{R}^+ $ such that $\mu(0)=\lambda(0)=0$. For the sake of simplicity, we restrict in this article to the case $\mu=1$ and $\lambda=0$. We claim however that our analysis applies \emph{mutatis mutandis} to this more general situation. As a matter of fact, we will consider local in time strong solutions for which $\varrho$ is non-vanishing.

%
%

\subsection{Overview on fluid-kinetic systems and related models}\label{Section:Overview}
Let us provide an overview on couplings between fluid and kinetic equations, both from the modeling and the theoretical points of view. In particular, we want to highlight the differences between the main regimes for the description of sprays (namely, \textit{thick} versus \textit{thin}). We also review some existing works on singular Vlasov equations coming from plasma physics, our strategy for thick sprays being inspired by the study of such systems (see Section \ref{Section:Strat}).

\subsubsection{Fluid-kinetic description}
We can trace back the introduction of fluid-kinetic couplings for the description of sprays involving a large number of particles to \cite{will, oro}. We also refer to \cite{D} for a general overview on the description of multiphase flows, as well as to \cite{ReitzBook}. Note that compared to an Eulerian-Eulerian description for both gas and droplets (where both phases are described at the macroscopic level, with $(t,x)$ variables), a fluid-kinetic point of view seems to be well-suited for polydispersed flows (i.e. when the size of the droplets can vary). Indeed, no average is taken to compute the drag force for instance. Of course, many other physical effects such as coalescence, breakup, vaporization or chemical reactions can be included in the models (see e.g. \cite{Laurent-PHD, Barran}).
\subsubsection{Thick spray case}
Models for \textit{thick sprays} have been explicitly introduced and formally derived in \cite{Duko} and in \cite[Chapter 2]{oro}. In particular, the pressure gradient acting on the dispersed phase is already present in \cite{Duko}, while \cite[Chapter 2]{oro} also considers the additional contribution of collisions. The use of such complex models has then been pursued by O'Rourke's team in the Los Alamos National Laboratory, to develop the Kiva code \cite{Kiva1, Kiva2, Kiva3}. From the numerical and modeling point of view, let us for also refer to the works \cite{BDM,BenjetAL}. 

As explained above, the coupling between both phases 
makes the rigorous study of thick spray equations reputedly challenging. The mathematical study of such type of models is still in its infancy and only a few formal results are available. 

In \cite{DesMathiaud}, the authors consider a formal hydrodynamic limit starting from a thick spray system of the type \eqref{eq:TSgenFULL2} (without fluid dissipation and with an additional energy variable for $f$) with an \textit{inelastic collision operator}. The limit of Knudsen number tending to $0$ allows to derive (at least formally) a two-fluid coupled system. The latter turns out to be a standard model of multiphase flow theory where the volume fraction is now an unknown (see the book \cite{IshiiHibiki}). It formally connects thick spray models to multifluid systems, where the presence of a common pressure term is standard. This somehow \textit{a posteriori} explains the additional pressure gradient in the force field acting in the kinetic equation. A standard feature of this bifluid limiting system seems to be a lack of hyperbolicity \cite{keyfitz2003lack,ndjinga2005influence, de2021thermo}, so that its behavior is \textit{a priori} highly unstable. Note that preliminary computations performed in \cite{Ramos-PHD} tend to indicate that adding some viscous term along some directions makes this type of system better-behaved. 

A new understanding of thick sprays has been obtained in \cite{BDD} where the linear stability for \eqref{eq:TSgenBaro} and \eqref{eq:TSgenFULL2} is investigated. More precisely, the $\Ld^2$ linear stability around a family of particular space-homogeneous profiles (for the kinetic phase) is obtained thanks to a suitable Lyapunov functional. The profiles in velocity are required to satisfy a property of monotonicity, this condition being a special example of the Penrose stability condition \eqref{cond:Penrose} that we shall impose on the initial condition (recall Remark \ref{Rmk:Penrose} above).

Very recently, \cite{BDF} has proposed a new averaged version of thick spray models, where the pressure gradient $-\nabla_x p(\varrho)$ in the kinetic equation and the volume fraction $\alpha$ are regularized by including an extra convolution operator. Local existence in Sobolev spaces for this new version of the original thick spray model is obtained in the Euler case for the fluid, using tools from symmetrisable hyperbolic systems (see \cite{BarD, Mathiaud}).

\subsubsection{Barotropic compressible Navier-Stokes equations} When $f=0$ (and thus $\alpha=1$), the system \eqref{eq:TSgenBaro} reduces to the standard compressible Navier-Stokes equations on $(\varrho,u)$, in the barotropic-type regime. These equations have given rise to an abundant literature for more than half a century. Global weak solutions of finite energy have been built for the first time in \cite{Lions} for constant viscosity coefficients, a result extended in \cite{FNP} for more general pressure power laws. Another notion of weak solutions was also considered in Hoff \cite{Hoff}. For more recent results allowing to include degenerate viscosities and more general pressure laws, see e.g. \cite{BD1,BD2,MV-NScomp, BJ, VY, BVY}.  

In the framework of strong solutions (local in time or global with assumption on the data), we can for instance refer to \cite{serrin1959uniqueness, Nash1962} for classical ones, to \cite{Solo} for mild solutions, to \cite{MN1, MN2} for solutions with high Sobolev regularity near equilibria, and to \cite{HZ} for the fine description of the time asymptotics of the system. Let us also mention the more recent works \cite{D1,D2,D3,D4,D5, charve2010global, DanchinTolk2022critical} for the study of the system in critical spaces.

\subsubsection{Thin spray case: the Vlasov-Navier-Stokes system} Unlike the \textit{thick spray} regime corresponding to \eqref{eq:TSgenBaro}, there exists a rich literature on the so-called \textit{thin spray} models. It corresponds to a regime where the particles volume fraction is small compared to that of the surrounding fluid. Here, the quantity $\alpha$ is set to $1$ and does not appear in the system. The main term of the coupling which is retained is the drag force (that is $\Gamma(t,x,v)=u(t,x)-v$), and its feedback in the fluid equation.

An important fluid-kinetic model in this class is the so-called Vlasov-Navier-Stokes system, describing a monodispersed phase of small particles flowing in an ambient incompressible homogeneous viscous fluid, that takes the form
\begin{equation}\label{eq:VNSthin}
\tag{VNS}
\left\{
      \begin{aligned}
\partial_t f +v\cdot \nabla_{x} f + {\rm div}_{v}[f(u-v)]&=0, \\[1mm]
\partial_t u + (u \cdot\nabla_{x})u- \Delta_{x} u + \nabla_{x} p &=j_f-\rho_f u,   \\[1mm]
{\rm div}_{x} \, u &=0.
\end{aligned}
    \right.
\end{equation}
This system has been for instance shown to provide a good description of medical aerosols in the upper part of the lung (see e.g. \cite{BGLM, BouM}).

From the mathematical point of view, many directions of research have been settled about \eqref{eq:VNSthin} (and its variants) over the past twenty years. 
The Cauchy theory, adressing the existence of global weak solutions for \eqref{eq:VNSthin} on a large class of domains in dimension $2$ or $3$, is by now well-developed (see e.g. \cite{ABdM,Ham, BDGM}), and also allows for more complex physics in the model (see \cite{BGM, BMM}).
It mainly consists in obtaining a Leray weak solution for $u$ and a renormalized weak solution (in the sense of Di-Perna and Lions \cite{DPL}) for $f$, using a remarkable energy-dissipation identity that is satisfied by solutions to the system.
In dimension $2$, the uniqueness of such solutions has been shown in \cite{HKM3}.

More recently, several asymptotics of \eqref{eq:VNSthin} have attracted a lot of attention. The question of the large-time dynamics for \eqref{eq:VNSthin} has received several advances over the past few years. Roughly speaking, it is expected that the cloud of particles aligns its velocity on that of the fluid, that is 
\begin{align*}
u(t) \underset{t \rightarrow +\infty}{\longrightarrow}v^{\infty}, \ \ f(t) \underset{t \rightarrow +\infty}{\longrightarrow} \rho^{\infty} \otimes \delta_{v=v^{\infty}},
\end{align*}
for some asymptotic velocity $v^{\infty} \in \R^3$ and profile $\rho^{\infty} (t,x)$.

The first complete answer justifying such a singular asymptotics has been obtained in \cite{HKMM} for Fujita-Kato type solutions, in the $3d$ torus case. In the whole space $\R^3$, this question is studied in \cite{HK} while the case of a $3d$ bounded domain (with absorption boundary conditions for $f$) is investigated in \cite{EHKM}. We also refer to \cite{E1} for the half-space case, where the additional effect of a gravity force on the particles (combined with absorption at the boundary) leads to decay of the solutions to $0$.

Another asymptotics is the so-called \textit{hydrodynamic limit} starting from \eqref{eq:VNSthin}, related to high-friction regimes. Considering some suitable scalings making a small parameter $\eps$ appear in \eqref{eq:VNSthin}, one wants to obtain a hydrodynamic and effective system when $\eps$ tends to $0$. Here again, this issue is linked to a monokinetic behavior of the form $f_\eps(t) \rightarrow \rho(t) \otimes \delta_{v=u(t)}$ and $u_\eps(t) \rightarrow u(t)$, with $(\rho(t), u(t)$) satisfying Transport-Navier-Stokes or inhomogeneous Navier-Stokes equations. We refer to \cite{HKM} for the more complete and recent results on this question and to \cite{HoferInertia, E2} where the gravity effect is taking into account, leading to macroscopic sedimentation couplings in the limit. 

Let us finally mention the challenging open problem of the derivation of \eqref{eq:VNSthin}, starting from microscopic first principles. We refer to \cite{DGR,H,HMS,CH,H2, HJan} for a partial answer based on homogenization and the justification of the Brinkman force in the fluid equation, but without the complete dynamics of the particles. An alternative (but still formal) program has been proposed in \cite{BDGR1, BDGR2}, starting from a system of coupled Boltzmann equations.  

\subsubsection{Several variants of \eqref{eq:VNSthin}} The Vlasov-Navier-Stokes system can also be considered with inhomogeneous or compressible Navier-Stokes equations \cite{ChKw, choi2017finite, CJ} and additional terms in kinetic equations \cite{CKKK, CJdragBGK}.

Note that the case of compressible Euler equations for the fluid, coupled to a kinetic equation, has also been investigated. We refer to \cite{BarD} (for the thin spray case) and \cite{Mathiaud} (for the so-called moderately thick spray case when collisions between particles are not neglected) where local in time strong solutions are built, thanks to ideas coming from hyperbolic systems. 

\subsubsection{Singular Vlasov equations}

As we shall explain later on (see Section \ref{Section:Strat} on our strategy of proof), we shall take our inspiration from a different problem coming from plasma physics, which is the so-called \textit{quasineutral limit} problem. More specifically, let us look at the dynamics of ions described by the following Vlasov-Poisson system
\begin{equation}\label{eq:VPions}
\tag{VP$_\eps$}
\left\{
      \begin{aligned}
\partial_t f_\eps +v \cdot \nabla_x f_\eps   + E_{\eps} \cdot \nabla_v f_\eps=0,  \\
E_\eps=-\nabla_x U_\eps, \\
(\mathrm{Id}-\eps^2 \Delta_x)U_\eps= \int_{\R^d} f_\eps(t,x,v) \, \mathrm{d}v-1,
\end{aligned}
    \right.
\end{equation}
when $\eps \ll 1$, corresponding to a small Debye length regime for the plasma. The issue at stake  is the validity (or invalidity) of the formal limit $\eps \rightarrow 0$, leading to
\begin{equation}\label{eq:Benney}
\tag{VB}
\left\{
      \begin{aligned}
\partial_t f +v \cdot \nabla_x f   - \nabla_x \rho_f \cdot \nabla_v f=0,  \\
\rho_f(t,x)=\int_{\R^d} f(t,x,v) \, \mathrm{d}v.
\end{aligned}
    \right.
\end{equation}
This system was named as the \textit{Vlasov-(Dirac)-Benney} system by Bardos in \cite{Bar}. As directly seen on \eqref{eq:Benney}, the force field in this Vlasov equation is one derivative less regular than the distribution function $f$ itself, thus displaying an apparent loss of derivative. 

\medskip

The question of the justification of the quasineutral limit (from \eqref{eq:VPions} to \eqref{eq:Benney}) and of the well-posedness of the limiting system \eqref{eq:Benney} has given rise to a wealth of literature for more than twenty years. Preliminary results have investigated the limit, up to to some defect measures \cite{BG,GR95}, and have been followed by a full justification in the analytic regime \cite{Gr96} (see also \cite{HKI, HKI2}). We also refer to \cite{B00, Mas, HK-quasin} for the case of singular monokinetic data leading to fluid equations. However, the quasineutral limit does not hold in general because instabilities for Vlasov-Poisson can take over (see \cite{HKH}).

In general, there also exist unstable homogeneous equilibria of \eqref{eq:Benney} around which the linearized equations have unbounded unstable spectrum (typically two bumps profile in velocity, leading to the so-called \textit{two-stream} instability). Therefore \eqref{eq:Benney} may be ill-posed in the sense of Hadamard  \cite{BN, HKN, baradat2020nonlinear} in Sobolev spaces, even with arbitrary losses of derivatives and arbitrary small time.

A local theory for \eqref{eq:Benney} thus requires more assumptions on the initial data. A Cauchy-Kovalevskaya type theorem can be applied \cite{JN,BFJJ,MV} to show that there is local existence of analytic solutions for analytic initial data. In dimension $d=1$, Sobolev initial data with a one bump profile in velocity (for all $x$) leads to local in time solution as shown in \cite{BB1} (see also \cite{BB2} for more properties). 

In any dimension, the quasineutral limit and the well-posedness in Sobolev spaces of \eqref{eq:Benney} have been obtained in \cite{HKR} under a Penrose stability condition on the initial data $f^{\mathrm{in}}$. In this work, the same type of condition as \eqref{cond:Penrose} was assumed, replacing the function $\mathscr{P}_{\mathrm{f}, \rho}$ by $\mathscr{P}_{\mathrm{f}}^{\mathrm{VP}}$ defined as $$\mathscr{P}_{\mathrm{f}}^{\mathrm{VP}}(x, \gamma, \tau, k)=\int_0^{+\infty} e^{-(\gamma+i \tau) s} \frac{i k}{1+ \vert k \vert^2} \cdot \left( \mathcal{F}_v \nabla_v \mathrm{f} \right)(x,ks)  \, \mathrm{d}s.$$
Up to some prefactor coming from coupling for sprays, our assumption \eqref{cond:Penrose} in Theorem \eqref{THM:main} is closely related to the one of \cite{HKR} and this work is the main inspiration of our analysis. We also refer to the recent work \cite{Chaub} where the existence of local in time solutions for mildly singular Vlasov equations is shown (without assuming any stability condition).

Note that the rescaling $(t,x,v) \mapsto (t/\eps, x/ \eps, v)$ in \eqref{eq:VPions} leads to the same equation with $\eps=1$, hence connecting the quasineutral limit to an issue of large-time dynamics. The Penrose stability condition \eqref{cond:Penrose} appearing on a homogeneous profile $\mathrm{f}(v)$ is actually a necessary condition for its long-time stability in the Vlasov-Poisson equation \eqref{eq:VPions} with $\eps=1$. 
This last issue is also linked to the Landau damping effect, which has been proven to hold in a small (Gevrey) neighborhood of such stable profile. In the torus, we refer to the breakthrough work \cite{MV}, as well as to \cite{BeMM, GNR}. 


\subsection{Strategy of the proof}\label{Section:Strat}

We conclude this introduction by presenting a detailed outline of the proof. This will allow at the same time to explain the structure of the paper. In order to ease readability and highlight the main features of the analysis, we voluntarily state our results without precise assumptions.

As explained above, and for the sake of clarity, we shall focus on \eqref{eq:TSgenBaro}. This system indeed retains the main features and difficulties of this article. Our result and proof will be generalized to the more complete systems presented in Section \ref{Section:GeneralizeIntro} (see the Sections \ref{Section:AppendixInternalEnergy}--\ref{Section:AppendixInelasticBoltzmann}--\ref{Section:AppendixDragTerm}).

\medskip

In the preliminary \textbf{Section~\ref{Section:prelimBootstrap}}, we start by deriving several \textit{a priori} energy estimates on the system \eqref{eq:TSgenBaro}.
We show in Proposition~\ref{prop:energyRHO:eps} that for all $t \in [0,T]$
\begin{equation}
\label{eq:introloss1}\forall t \in [0,T], \ \ \ 
 \Vert \varrho(t) \Vert_{\H^{m}} \leq   \Vert \varrho^{\mathrm{in}} \Vert_{\H^{m}} \Phi \Big( T, \cdots, \Vert u \Vert_{\Ld^{\infty}(0,T;\H^{m+1})},  \Vert  f \Vert_{\Ld^{\infty}(0,T;\mathcal{H}^{m+1}_{r})} \Big),
\end{equation}
where $\Phi$ is a continuous function which is increasing with respect to each of its arguments and where $\cdots$ involves lower order terms. 
On the other hand, we have for all $t \in [0,T]$
\begin{align}\label{eq:introloss3}
\Vert \rho_f (t) \Vert_{\H^{m}} + \Vert j_f (t) \Vert_{\H^m} \lesssim \Vert f(t) \Vert_{\mathcal{H}^{m}_{r}} \leq \Vert f^{\mathrm{in}} \Vert_{\mathcal{H}^{m}_{r}}^2  \Phi \Big(T, \cdots,  \Vert u \Vert_{\Ld^{\infty}(0,T;\H^{m})}),\Vert \varrho \Vert_{\Ld^2(0,T;\H^{m+1})}\Big),
\end{align}
The estimates~\eqref{eq:introloss1} and~\eqref{eq:introloss3} thus yield a loss of two derivatives for the fluid density $\varrho$.This formally prevents the use of standard techniques to obtain a (local in time) solution. The main goal of the analysis is to show that these losses are only apparent when the initial condition $(f^{\mathrm{in}}, \varrho^{\mathrm{in}})$ satisfies the Penrose stability condition \eqref{cond:Penrose}.

To this end, we first consider the following regularization of the system (see also the Remark \ref{rem:RegProcedure}), which includes a parameter $\eps \in (0,1)$ which is bound to go to $0$: 
\begin{equation*}\label{eq:introTSeps}
\left\{
      \begin{aligned}
\partial_t f_\eps +v \cdot \nabla_x f_\eps + {\rm div}_v [f_\eps E_\eps-f_\eps v ]&=0,   \\[2mm]
\partial_t ((1-\rho_{f_\eps})\varrho_{\eps}) +\mathrm{div}_x((1-\rho_{f_\eps})\varrho_\eps u_\eps)&=0, \\[2mm]
(1-\rho_{f_\eps}) \Big( \varrho_\eps\big[ \partial_t u_\eps +(u_\eps \cdot \nabla_x) u_\eps\big]+ \nabla_x p(\varrho_\eps) \Big) -\Delta_x u_\eps  - \nabla_x \mathrm{div}_x \, u_\eps&=j_{f_\eps}-\rho_{f_\eps} u_\eps, \\[2mm]
{f_\eps}_{\mid t=0}=f^{\mathrm{in}}, \ \ {\varrho_\eps}_{\mid t=0}=\varrho^{\mathrm{in}}, \ \ {u_\eps}_{\mid t=0}=u^{\mathrm{in}},
\end{aligned}
    \right.
\end{equation*}
where
\begin{align*}
&\rho_{f_\eps}(t,x)= \int_{\R^d} f_\eps(t,x,v) \, \mathrm{d}v, \ \ j_{f_\eps}(t,x)=\int_{\R^d} f_\eps(t,x,v) v \, \mathrm{d}v, \\[1mm]
&E_\eps = -p'(\varrho_\eps)\nabla_x \left[\mathrm{J}_\eps \varrho_\eps \right]   + u_\eps , \ \ \mathrm{J}_\eps=(\mathrm{Id}-\eps^2 \Delta_x)^{-1}.
\end{align*}
When $\eps>0$, the regularized system can be seen as a non-singular coupling between compressible Navier-Stokes and the Vlasov equation; as a result, classical energy methods allow to build local in time solutions, away from vacuum (this is performed in Appendix~\ref{Appendix:LWPeps}). However, the point is to obtain uniform in $\eps$ estimates on some interval of time which has to be independent of $\eps$. With this goal in mind, we set up a bootstrap argument that starts in the end of Section~\ref{Section:prelimBootstrap}.

We introduce
\begin{align*}
\mathcal{N}_{m,r}(f_\eps,\varrho_\eps,u_\eps,T):= \Vert f_\eps \Vert_{\Ld^{\infty}\left(0,T;\mathcal{H}^{m-1}_{r}\right)}+\Vert \varrho_\eps \Vert_{\Ld^2(0,T;\H^{m})}+\Vert u_\eps \Vert_{\Ld^{\infty}(0,T; \H^m) \cap \Ld^{2}(0,T;\H^{m+1})},
\end{align*}
for $T>0$ and we want to obtain a uniform (in $\eps$) estimate for this quantity. This will pave the way for a compactness argument allowing  to pass to the limit in the previous regularized system, when $\eps \rightarrow 0$. Observe the shift of one derivative between the norm on $f_\eps$ and that on $\varrho_\eps$. 
By \eqref{eq:introloss3}, a control on $\Vert \varrho_\eps \Vert_{\Ld^2(0,T;\H^{m})}$ and $\Vert u_\eps \Vert_{\Ld^{\infty}(0,T;\H^{m})}$ implies a control on $\Vert f_\eps \Vert_{\Ld^{\infty}\left(0,T;\mathcal{H}^{m-1}_{r}\right)}$. Hence, the main challenge is to derive an estimate for $\Vert \varrho_\eps \Vert_{\Ld^2(0,T;\H^{m})}$.

Our main observation is that, using the definition of $\alpha$ in the equation of conservation of mass, the fluid density satisfies a transport equation of the type
\begin{align}\label{eq:introeqRHO}
\partial_t \varrho_\eps +u_\eps \cdot \nabla_x \varrho_\eps +\frac{\varrho_\eps}{1-\rho_{f_\eps}} \mathrm{div}_x \left[ j_{f_\eps} -\rho_{f_\eps} u_\eps \right]=\mathrm{lower \ order  \ terms},
\end{align}
therefore $\varrho_\eps$ depends on ${f_\eps}$ only through its moments in velocity $\rho_{f_\eps}$ and $j_{f_\eps}$. 
The goal of Sections \ref{Section:Lagrangian}-\ref{Section:Averag-lemma}-\ref{Section:Kinetic-moments} is thus to relate these two moments to the fluid density $\varrho$ itself.

In \textbf{Section~\ref{Section:Lagrangian}}, we initiate the study of the Vlasov equation satisfied by $f_\eps$ with a Lagrangian point of view. We study the characteristics curves for the kinetic dynamics with friction
\begin{equation}
\left\{
      \begin{aligned}
        &\frac{\mathrm{d}}{\mathrm{d}s}\mathrm{X}^{s;t}(x,v) =\mathrm{V}^{s;t}(x,v), \qquad \mathrm{X}^{t;t}(x,v)=x,\\[2mm]
&\frac{\mathrm{d}}{\mathrm{d}s}\mathrm{V}^{s;t}(x,v)=-\mathrm{V}^{s;t}(x,v)+E_\eps(s,\mathrm{X}^{s;t}(x,v),\mathrm{V}^{s;t}(x,v) ), \qquad \mathrm{V}^{t;t}(x,v)=v,
      \end{aligned}
    \right.
\end{equation}
stemming from the Vlasov equation in \eqref{eq:introTSeps}. The term $-v$ in the force field $E_\eps$ is responsible for the friction dynamics. 
To simplify its study, we want to straighten the total kinetic operator
\begin{align*}
\partial_t + v \cdot \nabla_x + \mathrm{div}_v ((E_\eps-v) \cdot)
\end{align*}
into
\begin{align}\label{eq:KinOpFric}
\partial_t + v \cdot \nabla_x -v \cdot \nabla_v,
\end{align}
for short times. The operator in \eqref{eq:KinOpFric} corresponds to the free dynamics with friction.

More precisely, we prove (see Lemma~\ref{straight:velocity}) that for $
T$ small enough (independent of $\eps$), $x \in \T^d$ and $s,t \in [0, T]$, there exists a diffeomorphism $\psi_{s,t}(x, \cdot) : \R^d \rightarrow \R^d$ satisfying for all $v \in \R^d$
\begin{align}\label{eq:changeVARintro}
\mathrm{X}^{s;t}\left(x,\psi_{s,t}(x, v) \right)=x+(1-e^{t-s})v.
\end{align}
In addition, we provide several useful Sobolev estimates on $\psi$. We call this diffeomorphism the {\it straightening change of variable} in velocity.

The heart of the proof appears in the remaining sections. 
In \textbf{Section~\ref{Section:Averag-lemma}}, we study some smoothing averaging operators that will be crucial for the subsequent analysis of Section~\ref{Section:Kinetic-moments}. In short, it will enable us to split all the quantities exhibiting a loss of derivatives into a leading term and a good remainder which will be controlled.

Let us consider the following kernel operator which has been considered in \cite{HKR}:
$$
\mathrm{K}_G^{\mathrm{free}}[H](t,x)=\int_0^t \int_{\R^d}   [\nabla_x F](s,x-(t-s)v)\cdot G(t,s,x,v) \, \mathrm{d}v \, \mathrm{d}s.
$$
Despite the apparent loss of derivative, it is proved in \cite{HKR} that this operator is bounded in $\Ld^2_T\Ld^2_x$ as soon as the kernel $G$ is sufficiently smooth and decaying in velocity, a result related to the classical averaging lemmas \cite{GLPSavg}. We refer to the introduction of Section \ref{Section:Averag-lemma} for more references and explanations about this aspect.
We shall provide here extensions of this result to the natural averaging operator for the dynamics with friction (associated with \eqref{eq:KinOpFric}), namely
$$
\mathrm{K}_G^{\mathrm{fric}}[H](t,x)=\int_0^t \int_{\R^d}   [\nabla_x F](s,x+(1-e^{t-s})v)\cdot G(t,s,x,v) \, \mathrm{d}v \, \mathrm{d}s.
$$
We shall see in Proposition~\ref{propo:AveragStandard} that $\mathrm{K}_G^{\mathrm{fric}}$ is also bounded in $\Ld^2_T\Ld^2_x$ under similar smoothness and decay assumptions for the kernel $G$.
It was also observed in \cite{HKR2} that when the kernel cancels out on the diagonal $s=t$, the operator $\mathrm{K}_G^{\mathrm{free}}$ becomes bounded from $\Ld^2_T\Ld^2_x$ to $\Ld^2_T\H^1_{x}$, i.e. we gain one extra derivative in $x$; the same holds as well for $\mathrm{K}_G^{\mathrm{fric}}$, see Proposition~\ref{propo:AveragReg}.

A key result (see Proposition~\ref{propo:AveragDiff}) we prove is the fact that the \emph{difference} between the two latter operators also allows to gain a derivative in $x$, namely  $\mathrm{K}_G^{\mathrm{free}}-\mathrm{K}_G^{\mathrm{fric}}$ is bounded  from $\Ld^2_T\Ld^2_x$ to $\Ld^2_T\H^1_{x}$.  Propositions~\ref{propo:AveragStandard}, \ref{propo:AveragReg} and \ref{propo:AveragDiff} will be used at multiple times in this work. 

\textbf{Section~\ref{Section:Kinetic-moments}} is dedicated to the proper analysis of the kinetic moments $\rho_f$ and $j_f$. The main result provided in Proposition~\ref{coroFInal:D^I:rho-j} is the fact that for all $|I|\leq m$, we can write
\begin{equation}
\label{eq:introrhoj}
\begin{aligned}
 \partial_x^I \rho_{f}(t,x)&=p'(\varrho(t,x))\int_0^t  \int_{\R^d}   \nabla_x [\mathrm{J}_\eps \partial_x^I\varrho](s, x-(t-s)v) \cdot \nabla_v f(t,x,v) \, \mathrm{d}v \, \mathrm{d}s + R, \\
\partial_x^I  j_{f}(t,x)&=p'(\varrho(t,x))\int_0^t  \int_{\R^d}  v\nabla_x  [\mathrm{J}_\eps \partial_x^I \varrho](s, x-(t-s)v) \cdot \nabla_v f(t,x,v) \, \mathrm{d}v \, \mathrm{d}s+R,
\end{aligned}
\end{equation}
when $R$ stands for a well-controlled remainder in $\Ld^2_T \H^1_x$. Combining with the continuity results for the averaging operator $\mathrm{K}^{\mathrm{free}}_{p'(\varrho) \na_v f}$, this proves that the loss of derivative for $\rho_f$ and $j_f$ in~\eqref{eq:introloss3} was only apparent.

To obtain these identities, the first step is to derive a good equation satisfied by $\partial_x^I f$; to this end it is natural to apply the operator $\partial_x^I$ to the Vlasov equation. We readily obtain
$$
\partial_t \partial^I_x f_\eps +  v\cdot \na_x \partial^I_x f_\eps + \operatorname{div}_v (\partial^I_x f_\eps (E_\eps-v)) + \operatorname{div}_v ([\partial^I_x, E_\eps] f_\eps)   =0,
$$
and we observe that the commutator involves
\begin{itemize}
\item the main order term 
\begin{equation}
\label{eq:intromainE}
\operatorname{div}_v (\partial^I_x(E_\eps) f_\eps)
\end{equation}
that will account for the leading term in the identities~\eqref{eq:introrhoj},
\item low order terms that can be controlled,
\item but also terms of the form 
\begin{align*}
\text{(I)} \qquad &\partial_x E_\eps  \cdot  \partial^J_x \na_v f_\eps,  \quad |J| = m-1, \\
\text{(II)} \qquad &\partial^2_x E_\eps  \cdot  \partial^J_x \na_v f_\eps,  \quad |J| = m-2.
\end{align*}
\end{itemize}
The terms of type (I) clearly cannot be considered as remainders since they involve $m$ derivatives of $f_\eps$, which we do not uniformly control.  The terms of type (II) are not remainders either since we expect to plug in the identities~\eqref{eq:introrhoj} in the equation for $\varrho$, and this involves an extra derivative in $x$, thus also resulting in terms with $m$ derivatives of $f_\eps$. To overcome this difficulty, we argue as in \cite{HKR} and consider an augmented unknown $\mathcal{F} = (\partial^I_{x,v} f_\eps)_{|I|=m-1, m}$ which satisfies a system of the form
$$
\partial_t \mathcal{F} +  v\cdot \na_x \mathcal{F}-v\cdot \na_v \mathcal{F} + \operatorname{div}_v (E_\eps \mathcal{F}) + \mathcal{M} \mathcal{F} + \mathcal{L} = \mathcal{R},
$$
where $\mathcal{M}$ is a bounded linear map, $\mathcal{L}$ stands for the terms like~\eqref{eq:intromainE},
 and $\mathcal{R}$ is a well-controlled remainder. Note though that in \cite{HKR}, only the terms of type (I) are relevant and the augmented unknown only involves derivatives of order $m$. 
 
Controlling the averages in velocity of the whole family $(\partial^I_{x,v} f_\eps)_{|I|=m-1, m}$ allows to recover derivatives of $\rho_f$ and $j_f$ in $\H^m$. We finally rely on the Duhamel formula combined with an integration in velocity along the characteristics, on the straightening change of variables in velocity $\psi_{s,t}$ satisfying $\eqref{eq:changeVARintro}$, and on the crucial gain of derivatives provided by the kernel operators $\mathrm{K}_{\nabla_v f}^{\mathrm{free}}$ and  $\mathrm{K}_{\nabla_v f}^{\mathrm{fric}}$ to deduce \eqref{eq:introrhoj}.

We refer to this approach as a \textit{semi-Lagrangian} one in the sense that we we first apply derivatives on the kinetic equation and then integrate along the characteristics to obtain equations bearing on moments.

 \textbf{Section~\ref{Section:FluidDensity-estimate}} is then devoted to the obtention of an estimate for $\Vert \varrho_\eps \Vert_{\Ld^2(0,T;\H^{m})}$. Taking derivatives in the transport equation \eqref{eq:introeqRHO} on $\varrho$ and using \eqref{eq:introrhoj}, one can write an equation on $\partial_x^I \varrho$ for all $\vert I \vert = m$ under the form
 \begin{align*}
 \partial_t  \partial_x^I \varrho+u \cdot \nabla_x \partial_x^I \varrho+ \frac{\varrho}{1-\rho_f} \mathrm{div}_x \Big[ \mathrm{K}_{v p'(\varrho) \nabla_v f}^{\mathrm{free}}(\mathrm{J}_\eps \partial_x^I \varrho)-\mathrm{K}_{p'(\varrho) \nabla_v f}^{\mathrm{free}}(\mathrm{J}_\eps \partial_x^I \varrho)u\Big]&=\mathrm{lower \ order  \ terms}.
 \end{align*}
Based on this equation, and using some commutation properties relating the operators $\mathrm{div}_x$ and $\mathrm{K}_{v p'(\varrho) \nabla_v f}^{\mathrm{free}}$, it is then possible to prove
(see Proposition~\ref{coro:Facto}) that for all $|I|=m$, the function $\partial_x^I \varrho$ satisfies
\begin{align}\label{eq:intropseudo1}
\left( \mathrm{Id}-\frac{\varrho}{1-\rho_f}\mathrm{K}_G^{\mathrm{free}}\circ\mathrm{J}_{\eps} \right)\Big[\partial_t \partial_x^I \varrho + u \cdot \nabla_x \partial_x^I \varrho \Big]=\mathcal{R},\\
G(t,x,v)=  p'( \varrho(t,x)) \nabla_v f(t,x,v),
\end{align}
where $\mathcal{R}$ is a well-controlled remainder. The equality \eqref{eq:intropseudo1} has to be seen as a structural \textit{factorization} of the equation on $\partial_x^I \varrho $, between the operators
\begin{align*}
\mathrm{Id}-\frac{\varrho}{1-\rho_f}\mathrm{K}_G^{\mathrm{free}}\circ\mathrm{J}_{\eps},
\end{align*}
and 
\begin{align*}
\partial_t  + u \cdot \nabla_x.
\end{align*}
This relation is fully based on the coupling with the kinetic part. 

The main goal is then to derive some good $\Ld^2_T\Ld^2_{x}$ estimates on $\partial_x^{I} \varrho$. Again following \cite{HKR}, the idea is to relate $\frac{\varrho}{1-\rho_f}\mathrm{K}_G^{\mathrm{free}}\circ\mathrm{J}_{\eps}$ to a pseudodifferential operator and use pseudodifferential calculus to derive a suitable estimate. This is where the Penrose stability condition steps in and plays a crucial role: it will allow to obtain estimates without loss.

Compared to the analysis of \cite{HKR}, the extra derivative due to the transport operator in~\eqref{eq:intropseudo1} forces us to consider time-dependent symbols; this requires an extension on the whole line $\R$ of all functions, ensuring in the process that the Penrose stability condition still holds globally, see Subsection~\ref{Subsection-Extension}. For any $\gamma>0$, there holds (see Lemma \ref{LM:rewriteEqPseudo})
\begin{align*}
e^{-\gamma t} \mathrm{K}_G^{\mathrm{free}}\Big[e^{\gamma \bullet} H \Big](t,x):=\mathrm{Op}^{\gamma}(a_{f,\varrho})(H)(t,x), \ \ \text{on} \ \ (0,T) \times \T^d,
\end{align*}
with 
\begin{align*}
a_{f,\varrho}(t,x,\gamma,\tau, k)&:=p'(\varrho(t,x)) \int_0^{+\infty} e^{-(\gamma+i \tau) s} i k \cdot \left( \mathcal{F}_v \nabla_v f \right)(t,x,ks)  \, \mathrm{d}s,
\end{align*}
and thus \eqref{eq:intropseudo1} turns into the pseudodifferential equation
\begin{align}\label{eq:intropseudo2}
\left( \mathrm{Id}-\frac{\varrho}{1-\rho_f}\mathrm{Op}^{\gamma}(a_{f,\varrho})\circ\mathrm{J}_{\eps} \right)\Big[\partial_t \partial_x^I \varrho + u \cdot \nabla_x \partial_x^I \varrho \Big]=\mathcal{R}.\end{align}
Here, $\mathrm{Op}^{\gamma}$ refers to a pseudodifferential quantization on $\R \times \T^d$ and with parameter $\gamma>0$ (see Section \ref{Section:Pseudo} in the Appendix for more details).  By observing that
\begin{align*}
\frac{\varrho}{1-\rho_f}a_{f,\varrho}= \mathscr{P}_{f, \varrho},
\end{align*}
where
\begin{align*}
\mathscr{P}_{f(t), \varrho(t)}(x,\gamma, \tau,k):=\frac{p'(\varrho(t,x) )\varrho(t,x)}{1-\rho_f(t,x)}\int_0^{+\infty} e^{-(\gamma+i \tau) s}  \frac{ik}{1+\vert k \vert^2} \cdot \left( \mathcal{F}_v \nabla_v f \right)(t,x,ks)  \, \mathrm{d}s,
\end{align*}
we discover that the Penrose stability condition
\begin{align*}
\forall t \in \R, \ \ \underset{(x,\gamma, \tau,k)}{\inf} \, \vert 1- \mathscr{P}_{f(t), \varrho(t)}(x,\gamma, \tau, k) \vert >0
\end{align*}
thus asserts the \textit{ellipticity} of the symbol involved in the equation \eqref{eq:intropseudo2}. Assuming an $\Ld^2_T\Ld^2_{x}$ estimate on 
$$H=\partial_t \partial_x^{\alpha} \varrho + u \cdot \nabla_x \partial_x^{\alpha} \varrho,$$
a standard hyperbolic energy estimate associated to the transport part $\partial_t  + u \cdot \nabla_x$ thus leads to a suitable estimate on $\partial_x^{\alpha} \varrho$ (see Corollary \ref{Coro:keyHyperb}). Roughly speaking, the Penrose stability condition shows that the equation \eqref{eq:intropseudo1} can be seen as a factorization between an \textit{elliptic} part and an \textit{hyperbolic} part.

Obtaining a control on $H$ solution to the previous pseudodifferential equation \eqref{eq:intropseudo2} is then enough to conclude.
To do so, we rely on a semiclassical (in $\eps$) pseudodifferential calculus (with large parameter $\gamma$) whose aim is to invert the equation on $H$ up to some small remainder\footnote{This part of our analysis (with a large parameter) is reminiscent of the use of Lopatinskii determinant or Evans functions to obtain good estimates in hyperbolic boundary value problems or singular stable boundary layer problems (see e.g. \cite{Met, MZ, Rou}).}. The key is that one can consider the symbol $(1-\mathscr{P}_{f, \varrho})^{-1}$. 
This yields an $\Ld^2_T\Ld^2_{x}$ estimate for $H$ in terms of the remainder $\mathcal{R}$ (see Corollary \ref{Coro:keyL2}).

\textbf{Section~\ref{Section:END}} is eventually dedicated to the conclusion of the proof, gathering all the previous steps and estimates of the bootstrap analysis. We obtain the desired uniform estimate for the quantity $\mathcal{N}_{m,r}(f_\eps,\varrho_\eps,u_\eps,T)$, which is valid for some time $T>0$ independent of $\eps$. The existence part of Theorem \ref{THM:main} is then easily deduced by a compactness argument on $[0,T]$. The uniqueness part requires an additional argument, in the same spirit as the strategy previously devised.

In \textbf{Section \ref{Section:AppendixInternalEnergy}},  we show how one can easily adapt the strategy performed in this article to treat the case of a non-barotropic fluid with an additional equation on the internal energy for the fluid

In \textbf{Section \ref{Section:AppendixInelasticBoltzmann}}, we describe how one can include an inelastic collision operator of Bolzmann type in the kinetic equation (see \eqref{eq:collisionOp}). Note that our method follows an idea used in \cite{Mathiaud}, which allows to overcome the loss of weight in velocity from the collision operator \textit{thanks to}  the friction term in the original Vlasov equation.

In \textbf{Section \textbf{\ref{Section:AppendixDragTerm}}}, we consider the case of a density-dependent drag force, for which one can also prove a local well-posedness result, with the limitation that the initial data $f^{\mathrm{in}}$ has a compact support in velocity.

We refer to the precise statements of the Section \ref{Section:AppendixInternalEnergy}--\ref{Section:AppendixInelasticBoltzmann}--\ref{Section:AppendixDragTerm} for more details about the corresponding existence results.

Let us finally describe the content of the \textbf{Appendices}, gathered in the end of this article. 

\begin{enumerate}
\item[$-$] In Appendix \ref{Section:AppendixDIFF}, we state several useful functional inequalities on commutators, products and composition on $\T^d$ and $\T^d \times \R^d$. 
\item[$-$] In Appendix \ref{Appendix:LWPeps}, we justify the main steps providing the existence of a local in time solution $(f_\eps, \varrho_\eps, u_\eps)$ to the regularized system \eqref{eq:introTSeps} when $\eps >0$ is fixed.
\item[$-$] In Appendix \ref{Section:Pseudo}, we recall and give the main notions on pseudodifferential calculus (with parameter) that we shall need in this article.
\end{enumerate}

\medskip

In the rest of the article, we shall use the standard notation $A \lesssim B$ for $A \leq c B$ for some $c>0$ which is independent of $A,B$ and of $\eps$, but that may change from line to line. Furthermore, $\Lambda$ will stand for a nonnegative continuous function which is independent of $\eps$, nondecreasing with respect to each of its arguments, that may depend implicitly on the initial data and that may change from line to line. 
Finally, we denote by $[P, Q]=PQ-QP$ the commutator between two operators $P$ and $Q$.
\section{Preliminaries}\label{Section:prelimBootstrap}
In this section, we initiate the bootstrap strategy that will be used to prove Theorem \ref{THM:main}.

Throughout this article, we will constantly use the following lemma (which is a straightforward consequence of the Cauchy-Schwarz inequality).
\begin{lem}\label{LM:weightedSob} 
For all nonnegative measurable function $\mathrm{f}(x,v)$ and $k \in \N$, we have
\begin{align*}
\forall \ell \in \N, \ \ \forall r>\frac{d}{2}+k, \ \ \LRVert{\int_{\R^d} \vert v \vert^k f(\cdot,v) \, \mathrm{d}v}_{\H^{\ell}} \lesssim  \Vert \mathrm{f}  \Vert_{\mathcal{H}^{\ell}_r}.
\end{align*}
In particular, we have
\begin{align*}
\begin{split}
\forall \ell \in \N, \ \ \forall r>\frac{d}{2}, \ \ \Vert \rho_{\mathrm{f}} \Vert_{\H^{\ell}} \lesssim \Vert \mathrm{f}  \Vert_{\mathcal{H}^{\ell}_r}, \\
\forall \ell \in \N, \ \ \forall r>\frac{d}{2}+1, \ \ \Vert j_{\mathrm{f}} \Vert_{\H^{\ell}} \lesssim \Vert \mathrm{f}  \Vert_{\mathcal{H}^{\ell}_r}.
\end{split}
\end{align*}
\end{lem}

\subsection{Energy estimates}
Our aim is to obtain some \textit{a priori} estimates for smooth solutions to the system \eqref{eq:TSgenBaro}. We first study the fluid density $\varrho$, which is shown to satisfy an hyperbolic-type equation.
\begin{lem}\label{LM:rewriteEqrho}
Let $T>0$, $c>0$ and $(f, \varrho, u)$ satisfying \eqref{eq:TSgenBaro} on $[0,T]$ with $1-\rho_f \geq c$ on $[0,T]$. Defining the operator $\mathscr{L}^{u,f}$ as
\begin{align*}
\mathscr{L}^{u,f}:=\partial_t +u \cdot\nabla_x+ B^{u,f},
\end{align*}
where
\begin{align*}
B^{u,f}:=\frac{1}{1-\rho_f} \mathrm{div}_x \left[ F+u \right]\mathrm{Id}, \ \ F(t,x):= (j_f-\rho_f u)(t,x),
 \end{align*}
the fluid density $\varrho$ satisfies
\begin{align*}
\mathscr{L}^{u,f} \varrho=0 \quad \text{on  }[0,T].
\end{align*}

\end{lem}

\begin{proof}
The transport equation on $\varrho$ in \eqref{eq:TSgenBaro} can be rewritten
as \begin{align*}
(1-\rho_f)\left(\partial_t \varrho +u \cdot \nabla_x \varrho \right) +\varrho \Big( \partial_t (1-\rho_f) + u \cdot\nabla_x (1-\rho_f) \Big)+ (1-\rho_f) \varrho \mathrm{div}_x \, u=0.
\end{align*}
Integrating the Vlasov equation in the $v$ variable, we obtain the equation of conservation
\begin{align*}
\partial_t \rho_f +\mathrm{div}_x \, j_f=0, 
\end{align*}
therefore $\partial_t (1-\rho_f) = \mathrm{div}_x \, j_f$ and we get
\begin{align*}
0&=\partial_t \varrho +u \cdot \nabla_x \varrho + \frac{\varrho}{1-\rho_f} \left[  \mathrm{div}_x j_f-u \cdot \nabla_x \rho_f \right]+\varrho \mathrm{div}_x \, u \\
&=\partial_t \varrho +u \cdot \nabla_x \varrho + \frac{\varrho}{1-\rho_f} \mathrm{div}_x \left[   j_f- \rho_f u \right]+ \frac{\varrho}{1-\rho_f} \rho_f \mathrm{div}_x \, u +\varrho \mathrm{div}_x \, u \\
&=\partial_t \varrho +u \cdot \nabla_x \varrho + \frac{\varrho}{1-\rho_f} \mathrm{div}_x \left[  j_f- \rho_f u \right]+\frac{\varrho}{1-\rho_f}\mathrm{div}_x \, u.
\end{align*}
We recognize the expression of the Brinkman force
$F:= j_f-\rho_f u$ 
and therefore $\varrho$ satisfies the equation
 \begin{align*}
\partial_t \varrho +u \cdot \nabla_x \varrho +\frac{1}{1-\rho_f} \mathrm{div}_x \left[ F+ u \right]\varrho=0,
\end{align*}
which is the claimed result.
\end{proof}
We are now able to derive a Sobolev estimate bearing on the fluid density $\varrho$, in which we  control $\ell$ derivatives of $\varrho$ by $\ell+1$ derivatives of $f$ and $u$.

\begin{propo}\label{prop:energyRHO:eps}
For all $\ell,r \geq 1+d/2$, $c>0$ and $T>0$ and all smooth functions $(f,\varrho, u)$ such that $1-\rho_f \geq c$ and 
\begin{align*}
\mathscr{L}^{u,f} \varrho=0 \quad \text{on  }[0,T],
\end{align*}
we have the estimate
\begin{align*}
\forall t \in [0,T], \ \ \ 
 \Vert \varrho(t) \Vert_{\H^{\ell}} \leq   \Vert \varrho^{\mathrm{in}} \Vert_{\H^{\ell}} e^{C_T(u,f)T}
 \exp \left[Te^{C_T(u,f)T} \mathcal{Q}_{\ell}(T,u,f) \right],
\end{align*}
where
\begin{align*}
C_T(u,f)&=\Vert \mathrm{div}_x \, u \Vert_{\Ld^{\infty}((0,T) \times \T^d)} +2\Vert B^{u,f} \Vert_{\Ld^{\infty}((0,T) \times \T^d)}, \\
\mathcal{Q}_{\ell}(T,u,f)&=\Vert u \Vert_{\Ld^{\infty}(0,T;\H^{\ell})} +  \left\Vert\frac{1}{1-\rho_f} \right\Vert_{\Ld^{\infty}(0,T; \H^{\ell})} \left( \Vert  f \Vert_{\Ld^{\infty}(0,T;\mathcal{H}^{\ell+1}_{r})}^2+ \Vert u \Vert_{\Ld^{\infty}(0,T;\H^{\ell+1})}^2  \right),
\end{align*}
\end{propo}

\begin{rem}
Note that for the same exponents $\ell$ and $r$, one has
\begin{align*}
C_T(u,f) \lesssim \Vert u \Vert_{\Ld^{\infty}(0,T;\H^{\ell})} +  \left\Vert\frac{1}{1-\rho_f} \right\Vert_{\Ld^{\infty}(0,T; \Ld^{\infty})} \left( \Vert  f \Vert_{\Ld^{\infty}(0,T;\mathcal{H}^{\ell+1}_{r})}^2+ \Vert u \Vert_{\Ld^{\infty}(0,T;\H^{\ell+1})}^2  \right).
\end{align*}
\end{rem}

\begin{rem}\label{rem:energy-alpharho}
Let us mention for any smooth solution $(f, \varrho,u)$ to \eqref{eq:TSgenBaro}, the function $(1-\rho_f) \varrho$ satisfies
\begin{align*}
\mathscr{L}^{u,0} [(1-\rho_f)\varrho]=0.
\end{align*}
As a consequence, the result of Proposition \ref{prop:energyRHO:eps} holds for $(1-\rho_f) \varrho$ instead of $\varrho$, by considering $\mathcal{Q}_{\ell}(T,u,0)$ and $C_T(u,0)$.
\end{rem}

\begin{proof}[Proof of Proposition \ref{prop:energyRHO:eps}]
The proof is standard but for the sake of completeness and for highlighting the dependency in $f$ and $u$, let us write it. 
First, suppose we have a smooth function $h$ satisfying 
\begin{align*}
\mathscr{L}^{u,f}(h)=S,
\end{align*}
where $S$ is a given smooth source term. Performing an $\Ld^2_x$-energy estimate, we get
\begin{align*}
\frac{\mathrm{d}}{\mathrm{d}t} \Vert h \Vert^2_{\Ld^2} 
&= -2  \langle u\cdot \nabla_x h,h \rangle_{\Ld^2} - 2\langle B^{u,f} h,h \rangle_{\Ld^2} +2\langle S,h \rangle_{\Ld^2} \\
&= \int_{\T^d}(\mathrm{div}_x \, u -2 B^{u,f}) \vert h \vert^2 \, \mathrm{d}x+2\langle S,h \rangle_{\Ld^2}.
\end{align*}
By the Cauchy-Schwarz inequality, this yields for all $t \in (0,T)$
\begin{align*}
\frac{\mathrm{d}}{\mathrm{d}t} \Vert h(t) \Vert^2_{\Ld^2}  \leq C_T(u,f) \Vert h (t)\Vert^2_{\Ld^2} + 2\Vert S(t) \Vert_{\Ld^2} \Vert h(t) \Vert_{\Ld^2},
\end{align*}
where
$C_T(u,f):=\Vert \mathrm{div}_x \, u \Vert_{\Ld^{\infty}((0,T) \times \T^d)} +2\Vert B^{u,f} \Vert_{\Ld^{\infty}((0,T) \times \T^d)}.
$
By Gr\"onwall's inequality, we deduce
\begin{align*}
\Vert h(t) \Vert_{\Ld^2} \leq \Vert h(0) \Vert_{\Ld^2}e^{C_T(u,f)t /2}+ \int_0^t e^{C_T(u,f)(t-\tau)/2}\Vert S(\tau) \Vert_{\Ld^2} \, \mathrm{d}\tau.
\end{align*}
Now, let us assume that $\varrho$ is such that $\mathscr{L}^{u,f}(\varrho)=0$. Let $\beta \in \N^d$ such that $\vert \beta \vert \leq \ell$. 
Since
\begin{align*}
\mathscr{L}^{u,f}(\partial_x^{\beta} \varrho)=-\left[\partial_x^{\beta},u \cdot \nabla_x +B^{u,f}\right]\varrho,
\end{align*}
the first part of the proof with $h=\partial_x^{\beta} \varrho $ and $S=-\left[\partial_x^{\beta},u \cdot \nabla_x +B^{u,f}\right]\varrho$ leads to
\begin{align*}
\Vert  \partial_x^{\beta} \varrho(t) \Vert_{\Ld^2} \leq \Vert  \partial_x^{\beta} \varrho(0) \Vert_{\Ld^2}e^{C_T(u,f)t /2} +\int_0^t e^{C_T(u,f)(t-\tau)/2}\left\Vert \left[\partial_x^{\beta},u \cdot \nabla_x +B^{u,f}\right]\varrho(\tau) \right\Vert_{\Ld^2} \, \mathrm{d}\tau.
\end{align*}
We can estimate the commutator
$\left[\partial_x^{\beta},u \cdot \nabla_x +B^{u,f}\right]\varrho=\left[\partial_x^{\beta},u \cdot\right](\nabla_x \varrho)+  \left[\partial_x^{\beta},B^{u,f}\right](\varrho)
$
thanks to Proposition \ref{CommutSOB-KlainBerto}. This gives
\begin{align*}
\sum_{0 \leq \vert \beta \vert \leq \ell}  \Vert\left[\partial_x^{\beta},u \cdot \nabla_x +B^{u,f}\right]\varrho\Vert_{\Ld^2} &\leq 
\sum_{0 \leq \vert \beta \vert \leq \ell}  \Vert \left[\partial_x^{\beta},u \cdot\right](\nabla_x \varrho)\Vert_{\Ld^2}  
+\sum_{0 \leq \vert \beta \vert \leq \ell}  \Vert \left[\partial_x^{\beta},B^{u,f}\right](\varrho)\Vert_{\Ld^2} \\
& \lesssim \Vert \nabla_x u \Vert_{\Ld^{\infty}} \Vert  \nabla_x \varrho \Vert_{\H^{\ell-1}} + \Vert  u \Vert_{\H^{\ell}} \Vert \nabla_x \varrho\Vert_{\Ld^{\infty}}  \\
& \qquad \qquad  \qquad \qquad  + \Vert \nabla_x B^{u,f}\Vert_{\Ld^{\infty}} \Vert  \varrho \Vert_{\H^{\ell-1}} + \Vert  B^{u,f}\Vert_{\H^{\ell}} \Vert \varrho \Vert_{\Ld^{\infty}} \\
& \lesssim  \Vert  u \Vert_{\H^{\ell}} \Vert \varrho \Vert_{\H^{\ell}}+ \Vert  B^{u,f}\Vert_{\H^{\ell}} \Vert \varrho \Vert_{\H^{\ell}},
\end{align*}
by Sobolev embedding, since $\ell >1+d/2$. By summing on $\beta$, we eventually get for all $t \in [0,T]$
\begin{align*}
 \Vert \varrho(t) \Vert_{\H^{\ell}} \leq e^{C_T(u,f)t/2} \left(\Vert \varrho^{\mathrm{in}} \Vert_{\H^{\ell}} +\int_0^t   \left[\Vert  u(\tau) \Vert_{\H^{\ell}} + \Vert  B(\tau) \Vert_{\H^{\ell}} \right]   \Vert \varrho(\tau) \Vert_{\H^{\ell}}  \, \mathrm{d}\tau  \right).
\end{align*}
Again by Grönwall's inequality, we get for all $t \in [0,T]$
\begin{align*}
 \Vert \varrho(t) \Vert_{\H^{\ell}} \leq  e^{C_T(u,f)T/2} \Vert \varrho^{\mathrm{in}} \Vert_{\H^{\ell}} 
 \exp \left[Te^{C_T(u,f)T/2} \left(\Vert  u \Vert_{\Ld^{\infty}(0,T,\H^{\ell})}+\Vert  B^{u,f}\Vert_{\Ld^{\infty}(0,T,\H^{\ell})}  \right)  \right].
\end{align*}
To conclude, we write for all $ \tau \in (0,T)$
\begin{multline*}
\Vert  u \Vert_{\Ld^{\infty}(0,T,\H^{\ell})}+\Vert  B^{u,f}\Vert_{\Ld^{\infty}(0,T,\H^{\ell})}  \\ \leq \Vert  u \Vert_{\Ld^{\infty}(0,T; \H^{\ell})}+ \left\Vert\frac{1}{1-\rho_f} \right\Vert_{\Ld^{\infty}((0,T); \H^{\ell})} \Vert \mathrm{div}_x \left[ j_f-\rho_f u+u \right] \Vert_{\Ld^{\infty}(0,T; \H^{\ell})},
\end{multline*}
and use 
\begin{align*}
\Vert \mathrm{div}_x \left[ j_f-\rho_f u+u \right] \Vert_{ \H^{\ell}} 
 \lesssim \Vert  f \Vert_{\mathcal{H}^{\ell+1}_{r}}+\Vert f \Vert_{\mathcal{H}^{\ell+1}_{r}} \Vert u \Vert_{\H^{\ell+1}}+\Vert u \Vert_{\H^{\ell+1}},
\end{align*}
via $\ell +1>d/2$ and Lemma \ref{LM:weightedSob} with $r>1+d/2$. This concludes the proof.
\end{proof}

Let us now start the study of the kinetic equation satisfied by $f$.

\begin{defi}
For any vector field $u(t,x)$ and function $\varrho(t,x)$, we define the kinetic transport operator $\mathcal{T}^{u,\varrho}$ as
\begin{align*}
\mathcal{T}^{u,\varrho}:=\partial_t +v \cdot\nabla_x -v \cdot \nabla_v + E^{u,\varrho}(t,x)\cdot \nabla_v-d \mathrm{Id},
\end{align*}
where
\begin{align*}
E^{u,\varrho}(t,x):=u(t,x)-p'(\varrho)\nabla_x \varrho(t,x).
\end{align*}
\end{defi}
By developing the divergence (in $v$) term in the kinetic equation, the Vlasov equation on $f$ in \eqref{eq:TSgenBaro} can be rewritten 
\begin{align*}
\mathcal{T}^{u,\varrho}f=0.
\end{align*}

We now state several standard useful estimates to handle the force field $E^{u,\varrho}$.

\begin{lem}\label{LM:Estim-ForceField}
Let $T>0$. If $\varrho \geq c>0$ on $[0,T]$ for some given constant $c$, then the following hold:
\begin{itemize}
\item for all $k \geq 0$ and $t \in [0,T]$, we have 
\begin{align}\label{bound:p'BONY}
\Vert p'(\varrho(t)) \Vert_{\H^{k}} \leq \Lambda \left(\Vert \varrho(t) \Vert_{\Ld^{\infty}} \right)\Vert \varrho(t) \Vert_{\H^{k} }.
\end{align}
\item for all $k>3+d/2$ and $t \in [0,T]$, we have
\begin{align}\label{bound:Sobo2:E}
\Vert  E^{u,\varrho}(t) \Vert_{\H^{k}}  \lesssim \Vert u(t) \Vert_{\H^{k}} 
+ \Lambda \left(\Vert \varrho(t) \Vert_{\H^{k-2}}  \right)   \Vert \varrho (t) \Vert_{\H^{k+1}}.
\end{align} 
\end{itemize}
\end{lem}

\begin{proof} 
To prove \eqref{bound:p'BONY}, we rely on the paralinearization theorem of Bony applied to $p'$ (see Proposition \ref{Bony-Meyer} and Remark \ref{Rem:Bony-Meyer} in the Appendix) thanks to the assumption on the pressure $p$ and the lower bound on $\varrho$. 

To prove \eqref{bound:Sobo2:E}, we only have to estimate the term $p'(\varrho)\nabla_x \varrho $. 
We use the following tame estimate for products (see Proposition \ref{tame:estimate} in the Appendix)
\begin{align*}
 \Vert p'(\varrho(t))  \nabla_x \varrho (t) \Vert_{\H^{k} }
  &\lesssim  \Vert p'(\varrho(t)) \Vert_{\Ld^{\infty} } \Vert \nabla_x \varrho (t) \Vert_{\H^{k} }
 + \Vert p'(\varrho(t)) \Vert_{\H^{k} } \Vert \nabla_x \varrho (t) \Vert_{\Ld^{\infty} } \\
 &\lesssim  \Vert p'(\varrho(t)) \Vert_{\Ld^{\infty} } \Vert \varrho (t) \Vert_{\H^{k+1} }
 + \Vert p'(\varrho(t)) \Vert_{\H^{k} } \Vert \nabla_x \varrho (t) \Vert_{\Ld^{\infty} },
\end{align*}
therefore by Sobolev embedding, we have for $s_1> d/2$ and $s_2>1+d/2$,
\begin{align*}
\Vert E^{u,\varrho}(t) \Vert_{\H^{k} }\lesssim \Vert u(t) \Vert_{\H^{k} } + \Vert p'(\varrho(t)) \Vert_{\H^{s_1} }\Vert \varrho (t) \Vert_{\H^{k+1} }+\Vert p'(\varrho(t)) \Vert_{\H^{k} }\Vert \varrho (t) \Vert_{\H^{s_2}}.
\end{align*}
With $s_1=s_2=k-2$, this yields
\begin{align*}
\Vert  E^{u,\varrho}(t) \Vert_{\H^{k}} 
\lesssim \Vert u(t) \Vert_{\H^{k}} + \Vert p'(\varrho(t)) \Vert_{\H^{k-2}}\Vert \varrho (t) \Vert_{\H^{k+1}}+\Vert p'(\varrho(t)) \Vert_{\H^{k}}\Vert \varrho (t) \Vert_{\H^{k-2}}.
\end{align*}
In view of \eqref{bound:p'BONY}, there exists a continuous nondecreasing function $\mathrm{C}_{k,p'} :\R^+ \rightarrow \R^+$ such that 
\begin{multline*}
\Vert  E^{u,\varrho}(t) \Vert_{\H^{k}}  \lesssim \Vert u(t) \Vert_{\H^{k}} 
+ \mathrm{C}_{k,p'}\left(\Vert \varrho(t) \Vert_{\Ld^{\infty}} \right) \Vert \varrho(t) \Vert_{\H^{k-2}} \Vert \varrho (t) \Vert_{\H^{k+1}} \\
+\mathrm{C}_{k,p'}\left(\Vert \varrho(t) \Vert_{\Ld^{\infty}} \right) \Vert \varrho(t) \Vert_{\H^{k}} \Vert \varrho (t) \Vert_{\H^{k-2} },
\end{multline*}
and finally by using Sobolev embedding (with $k-2>d/2$), we get
\begin{align*}
\Vert  E^{u,\varrho}(t) \Vert_{\H^{k}}  \lesssim \Vert u(t) \Vert_{\H^{k}} 
+ \widetilde{\mathrm{C}}_{k,p'}\left(\Vert \varrho(t) \Vert_{\H^{k-2}}  \right)   \Vert \varrho (t) \Vert_{\H^{k+1}},
\end{align*}
for another function $\widetilde{\mathrm{C}}_{k,p'}$ of the same type. This concludes the proof.
\end{proof}

\begin{defi}\label{def:indices-shift}
For $\beta \in \N^d$, we define $\widehat{\beta} \in \N^d$ and $\overline{\beta} \in (\N\cup\{-1\})^d$ as
\begin{align*}
\widehat{\beta}^k&:=(\beta_1, \cdots, \beta_{k-1}, \beta_k+1, \beta_{k+1}, \cdots, \beta_d), \ \ k \in [\![1,d ]\!],  \\ 
\overline{\beta}^k&:=(\beta_1, \cdots, \beta_{k-1}, \beta_k-1, \beta_{k+1}, \cdots, \beta_d), \ \  k \in [\![1,d ]\!].
\end{align*}
\end{defi}

We have the following straightforward lemma of commutation for the kinetic equation.

\begin{lem}\label{LM:applyD-kin}
For any $\alpha, \beta \in \N^d$ and for any smooth function $f(t,x,v)$, we have
\begin{align*}
\left[\partial_x^{\alpha} \partial_v^{\beta},\mathcal{T}^{u,\varrho} \right] f= \sum_{\substack{i=1\\\beta_i \neq 0}}^d \left(  \partial_x^{\widehat{\alpha}^i} \partial_v^{\overline{\beta}^i} f-\partial_x^{\alpha} \partial_v^{\beta}f \right)  + \left[\partial_x^{\alpha} \partial_v^{\beta}, E^{u,\varrho}(t,x)\cdot \nabla_v\right]f.
\end{align*}
\end{lem}

The Sobolev estimate for the kinetic equation goes as follows, showing that we can control $m$ derivatives of $f$ by $m+1$ derivatives of $\varrho$ and $m$ derivatives of $u$.

\begin{propo}\label{prop:energy:f}
For all $r \geq 0$, $m >3+d/2$, $c>0$, there exists $C>0$ such that for all $T>0$ all smooth functions $(f,\varrho,u)$ satisfying
$$
\mathcal{T}^{u,\varrho} (f) =0 \quad \text{on  } [0,T],
$$
and $\varrho \geq c$ on $[0,T]$, there holds, for all $t \in [0,T]$
\begin{align*}
\Vert f(t) \Vert_{\mathcal{H}^{m}_{r}}^2 \leq \Vert f(0) \Vert_{\mathcal{H}^{m}_{r}}^2 \exp \left[C\left( (1+ \Vert u \Vert_{\Ld^{\infty}(0,T;\H^{m})})T+ {\sqrt{T}}  \Lambda \left( \Vert \varrho \Vert_{\Ld^{\infty}(0,T;\H^{m-2})} \right) \Vert \varrho \Vert_{\Ld^2(0,T;\H^{m+1})} \right) \right].
\end{align*}
\end{propo}
\begin{proof}
By Lemma \ref{LM:applyD-kin}, we have
\begin{align*}
\mathcal{T}^{u,\varrho}(\partial_x^{\alpha} \partial_v^{\beta} f)&=- \sum_{\substack{i=1\\\beta_i \neq 0}}^d  \left(  \partial_x^{\widehat{\alpha}^i} \partial_v^{\overline{\beta}^i} f-\partial_x^{\alpha} \partial_v^{\beta}f \right)  - \left[\partial_x^{\alpha} \partial_v^{\beta}, E^{u,\varrho}(t,x)\cdot \nabla_v\right]f,
\end{align*}
for all $\alpha, \beta \in \N^d$. We take the scalar product of this equality with $(1+\vert v \vert^2)^{r}\partial_x^{\alpha} \partial_v^{\beta} f$, sum for all $\vert \alpha \vert + \vert \beta \vert \leq m$ and then integrate on $\T^d \times \R^d$. For the left-hand side, we observe that
\begin{align*}
&\sum_{\vert \alpha \vert + \vert \beta \vert \leq m}\int_{\T^d \times \R^d}  (1+\vert v \vert^2)^{r}\mathcal{T}^{u,\varrho}(\partial_x^{\alpha} \partial_v^{\beta} f)\partial_x^{\alpha} \partial_v^{\beta} f  \\
&=\frac{1}{2}\frac{\mathrm{d}}{\mathrm{d}t}\Vert f(t) \Vert_{\mathcal{H}^{m}_{r}}^2-d \Vert f(t) \Vert_{\mathcal{H}^{m}_{r}}^2\\
& \qquad \qquad \qquad \qquad+ \sum_{\vert \alpha \vert + \vert \beta \vert \leq m}\int_{\T^d \times \R^d}  (1+\vert v \vert^2)^{r}[v \cdot\nabla_x -v \cdot \nabla_v + E^{u,\varrho}(t,x)\cdot \nabla_v](\partial_x^{\alpha} \partial_v^{\beta} f)\partial_x^{\alpha} \partial_v^{\beta} f  \\
&=\frac{1}{2}\frac{\mathrm{d}}{\mathrm{d}t}\Vert f(t) \Vert_{\mathcal{H}^{m}_{r}}^2-d \Vert f(t) \Vert_{\mathcal{H}^{m}_{r}}^2+  \sum_{\vert \alpha \vert + \vert \beta \vert \leq m}\int_{\T^d \times \R^d}  (1+\vert v \vert^2)^{r} \mathrm{div}_x\left (v \frac{\vert \partial_x^{\alpha} \partial_v^{\beta}f \vert^2}{2} \right) \\
& \qquad \qquad \qquad \qquad \qquad \qquad \qquad+\sum_{\vert \alpha \vert + \vert \beta \vert \leq m}\int_{\T^d \times \R^d}  (1+\vert v \vert^2)^{r} \mathrm{div}_v\left( (E^{u,\varrho}-v) \frac{\vert \partial_x^{\alpha} \partial_v^{\beta} f \vert^2}{2} \right) + d \Vert f(t) \Vert_{\mathcal{H}^{m}_{r}}^2 \\
&= \frac{1}{2}\frac{\mathrm{d}}{\mathrm{d}t}\Vert f(t) \Vert_{\mathcal{H}^{m}_{r}}^2-\sum_{\vert \alpha \vert + \vert \beta \vert \leq m}\int_{\T^d \times \R^d}  \nabla_v(1+\vert v \vert^2)^{r} \cdot (E^{u,\varrho}-v) \frac{\vert \partial_x^{\alpha} \partial_v^{\beta} f \vert^2}{2}, 
\end{align*}
and that the last term satisfies
\begin{align*}
\sum_{\vert \alpha \vert + \vert \beta \vert \leq m}\int_{\T^d \times \R^d}  \nabla_v(1+\vert v \vert^2)^{r} \cdot (E^{u,\varrho}-v) \frac{\vert \partial_x^{\alpha} \partial_v^{\beta} f \vert^2}{2} \leq (1+\Vert E^{u,\varrho}(t) \Vert_{\Ld^{\infty}})\Vert f(t)\Vert_{\mathcal{H}^{m}_{r}}^2.
\end{align*}
We now look at the two terms of the right-hand side. For the first one, we have
\begin{align*}
-\sum_{\vert \alpha \vert + \vert \beta \vert \leq m}\int_{\T^d \times \R^d}  (1+\vert v \vert^2)^{r} \sum_{\substack{i=1\\\beta_i \neq 0}}^d  \left(  \partial_x^{\widehat{\alpha}^i} \partial_v^{\overline{\beta}^i} f-\partial_x^{\alpha} \partial_v^{\beta}f \right)\partial_x^{\alpha} \partial_v^{\beta} f \lesssim \Vert f(t)\Vert_{\mathcal{H}^{m}_{r}}^2,
\end{align*}
while for the second one, we write
\begin{align*}
\sum_{\vert \alpha \vert + \vert \beta \vert \leq m}\int_{\T^d \times \R^d} & (1+\vert v \vert^2)^{r} \left[\partial_x^{\alpha} \partial_v^{\beta}, E^{u,\varrho}(t,x)\cdot \nabla_v\right]f\partial_x^{\alpha} \partial_v^{\beta} f \\
  & \qquad \qquad \qquad \lesssim\Vert (1+\vert v \vert^2)^{r/2} \left[\partial_x^{\alpha} \partial_v^{\beta}, E^{u,\varrho}(t,x)\cdot \nabla_v\right] f   \Vert_{\Ld^2_{x,v}} \Vert f(t)\Vert_{\mathcal{H}^{m}_{r}} \\
& \qquad \qquad \qquad  \lesssim \Vert E^{u,\varrho}(t) \Vert_{\H^{m}} \Vert f(t)\Vert_{\mathcal{H}^{m}_{r}}^2,
\end{align*}
by the Cauchy-Schwarz inequality and the product law \eqref{Sob:estim3} in Sobolev spaces (since $m > 1+d$) of Lemma \ref{LM:ProducLawWeight}. All in all, we obtain
\begin{align}\label{ineq:d/dt:fweight}
\frac{\mathrm{d}}{\mathrm{d}t}\Vert f(t) \Vert_{\mathcal{H}^{m}_{r}}^2 &\lesssim (1+\Vert E^{u,\varrho}(t) \Vert_{\Ld^{\infty}}+\Vert E^{u,\varrho}(t) \Vert_{\H^{m}} ) \Vert f(t)\Vert_{\mathcal{H}^{m}_{r}}^2  \lesssim (1+\Vert E^{u,\varrho}(t) \Vert_{\H^{m}} ) \Vert f(t)\Vert_{\mathcal{H}^{m}_{r}}^2,
\end{align}
if $m > d/2$. Invoking the estimate \eqref{bound:Sobo2:E} of Lemma \ref{LM:Estim-ForceField}, and by integrating in time the inequality \eqref{ineq:d/dt:fweight}, we obtain
\begin{align*}
\Vert f(t) \Vert_{\mathcal{H}^{m}_{r}}^2 \leq \Vert f(0) \Vert_{\mathcal{H}^{m}_{r}}^2 + C\int_0^t \left(1+\Vert u(s) \Vert_{\H^{m}} +  \Lambda\left(\Vert \varrho \Vert_{\Ld^{\infty}(0,T; \H^{m-2})} \right)\Vert \varrho(s) \Vert_{\H^{m+1}} \right) \Vert f(s)\Vert_{\mathcal{H}^{m}_{r}}^2 \, \mathrm{d}s,
\end{align*}
for all $t \in [0,T)$ and for some constant $C>0$ independent of $\eps$. Using Cauchy-Schwarz and Grönwall's inequality, this implies for all $t \in  [0,T)$
\begin{align*}
\Vert f(t) \Vert_{\mathcal{H}^{m}_{r}}^2 \leq \Vert f(0) \Vert_{\mathcal{H}^{m}_{r}}^2 \exp \left[C\left( (1+ \Vert u \Vert_{\Ld^{\infty}(0,T;\H^{m})})T+ \sqrt{T} \Lambda\left( \Vert \varrho \Vert_{\Ld^{\infty}(0,T;\H^{m-2})} \right) \Vert \varrho \Vert_{\Ld^2(0,T;\H^{m+1})} \right) \right],
\end{align*}
and this concludes the proof.
\end{proof}

The estimates given by Proposition \ref{prop:energyRHO:eps} and Proposition \ref{prop:energy:f}  show a possible loss of derivatives between $f$ and $\varrho$. This constitutes the main obstacle of the analysis. 

\subsection{Regularization of the system and setup of the bootstrap}\label{Subsection:EstimateREG}
To (temporarily) bypass this problem, we  introduce the following regularized version of the equations. Since the pressure gradient in the force field of the Vlasov equation seems to  cause estimates with a loss of derivative, we smooth out this precise term. For all $\eps >0$, we consider 
\begin{equation*}\label{S_eps}
\tag{$S_{\eps}$}
\left\{
      \begin{aligned}
\partial_t f_\eps +v \cdot \nabla_x f_\eps   + {\rm div}_v [f_\eps (u_\eps-v)]-p'(\varrho_\eps)\nabla_x \left[(I-\eps^2 \Delta_x)^{-1}\varrho_\eps \right] \cdot \nabla_v f_\eps =0&,   \\[2mm]
\partial_t \varrho_{\eps} +u_{\eps} \cdot \nabla_x \varrho_{\eps} + \frac{\varrho_{\eps}}{1-\rho_{f_\eps}} \mathrm{div}_x \left[  j_{f_\eps} -\rho_{f_\eps} u_{\eps} \right]=-\frac{\varrho_{\eps}}{1-\rho_{f_\eps}} \mathrm{div}_x u_{\eps}&, \\[2mm]
\partial_t u_\eps +(u_\eps \cdot \nabla_x) u_\eps +\frac{1}{\varrho_\eps}\nabla_x p(\varrho_\eps) -\frac{1}{\varrho_\eps(1-\rho_{f_\eps})}\Big[\Delta_x  + \nabla_x \mathrm{div}_x \Big]u_\eps=\frac{1}{\varrho_\eps(1-\rho_{f_\eps})}(j_{f_\eps}-\rho_{f_\eps} u_\eps)&, \\[2mm]
{f_\eps}_{\mid t=0}=f^{\mathrm{in}}, \ \ {\varrho_\eps}_{\mid t=0}=\varrho^{\mathrm{in}}, \ \ {u_\eps}_{\mid t=0}=u^{\mathrm{in}},
\end{aligned}
    \right.
\end{equation*}
where
\begin{align*}
\rho_{f_\eps}(t,x):= \int_{\R^d} f_\eps(t,x,v) \, \mathrm{d}v, \ \ j_{f_\eps}(t,x):=\int_{\R^d} f_\eps(t,x,v) v \, \mathrm{d}v.
\end{align*}
Let us highlight that we have used the rewriting of the transport equation on $\varrho_\eps$ based on Lemma \ref{LM:rewriteEqrho}.
\begin{defi}
For all $\eps>0$, we define the regularized kinetic transport operator $\mathcal{T}_{\reg,\eps}^{u_\eps,\varrho_\eps}$ as
\begin{align*}
\mathcal{T}_{\reg,\eps}^{u_\eps,\varrho_\eps}:=\partial_t +v \cdot\nabla_x -v \cdot \nabla_v + E_{\reg,\eps}^{u_\eps,\varrho_\eps}(t,x)\cdot \nabla_v-d \mathrm{Id},
\end{align*}
where
\begin{align*}
E_{\reg,\eps}^{u_\eps,\varrho_\eps}(t,x):=u_\eps(t,x)-p'(\varrho_\eps)\nabla_x \left[\mathrm{J}_{\eps}\varrho_\eps \right](t,x), \ \ \ \mathrm{J}_{\eps}:=(I-\eps^2 \Delta_x)^{-1}.
\end{align*}
\end{defi}

\begin{rem}\label{rem:RegProcedure}
The regularization through the operator $\mathrm{J}_{\eps}$ could appear as quite arbitrary. In view of the equivalent Penrose condition \eqref{Penrose:ReformuleLAMBDA} appearing in Remark \ref{Rem:Penrose}, it is actually possible to consider a more general Fourier multiplier
\begin{align*}
\mathcal{J}_{\eps}=m(\eps D),
\end{align*}
associated to a smooth function $m : \R^d \rightarrow  (0,1]$ such that 
\begin{align*}
\forall k \in \R^d, \ \ m(k) \leq \frac{C}{(1+ \vert k \vert^2}.
\end{align*}
\end{rem}

For all $\eps >0$, the Vlasov equation satisfied by $f$ in $\eqref{S_eps}$ can be recasted as
\begin{align*}
\mathcal{T}_{\reg,\eps}^{u_\eps,\varrho_\eps}f=0.
\end{align*}
Relying on the elliptic regularity provided by $\mathrm{J}_\eps$, we now have the following estimates for the regularized force field $E_{\reg,\eps}^{u_\eps,\varrho_\eps}$.
 
\begin{lem}
Let $\eps>0$ and $T>0$. If $\varrho \geq c>0$ on $[0,T]$, then for all $k>3+d/2$ and $t \in [0,T]$
\begin{align}\label{bound:Sobo2:Ereg}
\Vert  E_{\reg,\eps}^{u_\eps,\varrho_\eps}(t) \Vert_{\H^{k}}  \lesssim \Vert u(t) \Vert_{\H^{k}} 
+\frac{1}{\eps} \Lambda \left(\Vert \varrho(t) \Vert_{\H^{k-2}}  \right)   \Vert \varrho (t) \Vert_{\H^{k}}.
\end{align} 
\end{lem}

Thanks to the regularization, we can overcome the loss of derivative exhibited by the estimate of Proposition~\ref{prop:energy:f}, up to some factor which is diverging when $\eps \rightarrow0$.

\begin{propo}\label{prop:energy:f:eps}
 For all $r \geq 0$, $m >3+d/2$, $c>0$, there exists $C>0$ such that for all $\eps>0$, $T>0$ and all smooth functions $(f,\varrho,u)$ satisfying
$$
\mathcal{T}_{\reg,\eps}^{u,\varrho} (f) =0 \quad \text{on} \quad  [0,T],
$$
and  $\varrho \geq c$ on $[0,T]$, there holds, for all $t \in [0,T]$
\begin{align*}
\Vert f(t) \Vert_{\mathcal{H}^{m}_{r}}^2 \leq \Vert f(0) \Vert_{\mathcal{H}^{m}_{r}}^2 \exp \left[C\left( (1+ \Vert u \Vert_{\Ld^{\infty}(0,T;\H^{m})})T+ \frac{\sqrt{T}}{\eps}\Lambda \left( \Vert \varrho \Vert_{\Ld^{\infty}(0,T;\H^{m-2})} \right) \Vert \varrho \Vert_{\Ld^2(0,T;\H^{m})} \right) \right].
\end{align*}
\end{propo}

\begin{proof}
The proof is the same as for Proposition~\ref{prop:energy:f}, using~\eqref{bound:Sobo2:Ereg} instead of \eqref{bound:Sobo2:E} to conclude.
\end{proof}


We shall also need the following condition about the pointwise bounds for the densities. 

\begin{defi}
Let $T>0$. For any nonnegative functions $f(t,x,v)$ and $\varrho(t,x)$ on $[0,T]$, we define the property 
\begin{align}\label{def:Bound-HAUTBAS}
\tag{B$^{\mu, \theta}_{\Theta}(T)$}
\forall t \in [0,T], \ \ \rho_{f}(t)\leq \frac{\Theta+1}{2}, \ \ \ \frac{\mu}{2} \leq \varrho(t), \ \ \ \frac{\underline{\theta}}{2} \leq (1-\rho_{f}(t))\varrho(t) \leq 2\overline{\theta},
\end{align}
where $\Theta, \mu, \underline{\theta}, \overline{\theta}$ are given in the statement of Theorem \ref{THM:main}.
\end{defi}

We will be able to propagate the condition \eqref{def:Bound-HAUTBAS} thanks to the following lemmas, giving some rough pointwise control for the local particle density and the fluid density.
\begin{lem}\label{LM:bound:Haut-rho_falpha}
Assume that $\rho_{f^{\mathrm{in}}}<\Theta$. For any $\ell,r >1+d/2$, any smooth solution $f(t,x,v)$ to the Vlasov equation in \eqref{S_eps} satisfies for all $T>0$
\begin{align*}
\Vert \rho_f \Vert_{\Ld^{\infty}(0,T;\Ld^{\infty}(\T^d))} &\lesssim \Theta + T \Vert f \Vert_{\Ld^{\infty}(0,T;\mathcal{H}^{\ell}_{r})}, \\
\Vert 1-\rho_f \Vert_{\Ld^{\infty}(0,T;\Ld^{\infty}(\T^d))} &\lesssim  \Theta  + T \Vert f \Vert_{\Ld^{\infty}(0,T;\mathcal{H}^{\ell}_{r})},
\end{align*}
where $\Theta$ has been introduced in the statement of Theorem \ref{THM:main}. Furthermore, if $\Theta + T \Vert f \Vert_{\Ld^{\infty}(0,T;\mathcal{H}^{\ell}_{r})}<1$, then for all $t \in [0,T]$ and $x \in \T^d$
\begin{align*}
\frac{1}{1- \rho_f(t,x)} &\lesssim \frac{1}{1- \Theta - T \Vert f \Vert_{\Ld^{\infty}(0,T;\mathcal{H}^{\ell}_{r})}}, \\[1mm]
\frac{1}{1- \rho_f(t,x)} &\gtrsim \frac{1 }{\Theta  + T \Vert f \Vert_{\Ld^{\infty}(0,T;\mathcal{H}^{\ell}_{r})}}.
\end{align*}
\end{lem}
\begin{proof}
Integrating the Vlasov equation with respect to velocity, one gets the conservation law $\partial_t \rho_f +\mathrm{div}_x \, j_f=0$. We thus have
\begin{align*}
\rho_f(t)=\rho_{f^{\mathrm{in}}}+ \int_0^t \mathrm{div}_x(j_f)(s) \, \mathrm{d}s,
\end{align*}
therefore by using Sobolev embedding, we get for all $t \in [0,T]$
\begin{align*}
\Vert \rho_f(t) \Vert_{\Ld^{\infty}(\T^d)} \leq  \Vert \rho_{f^{\mathrm{in}}} \Vert_{\Ld^{\infty}(\T^d)} +  \int_0^t \Vert \mathrm{div}_x(j_f)(s)\Vert_{\Ld^{\infty}(\T^d)} \, \mathrm{d}s 
 \leq \Theta +  T  \Vert j_f\Vert_{\Ld^{\infty}(0,T;\H^{\ell}(\T^d))},
\end{align*}
if $\ell >1+d/2$. We obtain the conclusion by using Lemma \ref{LM:weightedSob}. The last estimates stated in the lemma directly follow.
\end{proof}

\begin{lem}\label{LM:bound:HautBas-alpharho}
Assume that $0<\underline{\theta} \leq (1-\rho_{f^{\mathrm{in}}})\varrho^{\mathrm{in}} \leq \overline{\theta}$. Let $T>0$. For $\ell >1+d/2$, any smooth solution $(f(t,x,v),\varrho(t,x), u(t,x))$ to \eqref{S_eps} satisfies for all $t \in [0,T]$
\begin{align*}
\underline{\theta}\exp\left(-T \LRVert{u}_{\Ld^{\infty}(0,T;\H^{\ell})} \right)  \leq (1-\rho_f(t)) \varrho(t) \leq \overline{\theta}\exp\left(T \LRVert{u}_{\Ld^{\infty}(0,T;\H^{\ell})} \right),
\end{align*}
where $\underline{\theta}, \overline{\theta}$ have been introduced in the statement of Theorem \ref{THM:main}.
\end{lem}
\begin{proof}
The proof is a straighforward application of the method of characteristics applied to $(1-\rho_f) \varrho$, as the solution to the continuity equation
\begin{align*}
\partial_t ((1-\rho_f) \varrho) +\mathrm{div}_x \, ((1-\rho_f) \varrho u)=0.
\end{align*}
We obtain the conclusion in view of the assumption on $(1-\rho_{f^{\mathrm{in}}}) \varrho^{\mathrm{in}}$.
\end{proof}

The following proposition then shows that, thanks to the regularization, we can actually build for all $\eps>0$ a local solution to the regularized system \eqref{S_eps}.

\begin{propo}\label{prop:existenceEPS}
There exist $r_0>0$ and $m_0>0$ such that the following holds. For all $\eps>0$, the system \eqref{S_eps} is locally well-posed in Sobolev spaces, that is if $r>r_0$, $m>m_0$ and if
\begin{align*}
(f^{\mathrm{in}},\varrho^{\mathrm{in}},u^{\mathrm{in}}) \in \mathcal{H}^m_{r} \times \H^m \times  \H^{m},
\end{align*}
satisfies
\begin{align*}
0 &\leq f^{\mathrm{in}}, \ \ \ \rho_{f^{\mathrm{in}}}<\Theta<1, \ \ \ 0<\mu \leq \varrho^{\mathrm{in}}, \ \ \ 0<\underline{\theta} \leq (1-\rho_{f^{\mathrm{in}}})\varrho^{\mathrm{in}} \leq \overline{\theta},
\end{align*}
for some fixed constants $\Theta, \mu, \underline{\theta}, \overline{\theta}$, then there exist $T=T(\eps)>0$ and a unique solution $(f_\eps,\varrho_\eps,u_\eps)$ to the regularized system $\eqref{S_eps}$ on $(0,T(\eps)]$ such that
$$ (f_\eps,\varrho_\eps,u_\eps)\in \mathscr{C}(0,T; \mathcal{H}^m_r) \times \mathscr{C}(0,T; \H^m) \times \mathscr{C}(0,T; \H^{m}) \cap \Ld^2(0,T; \H^{m+1}),$$
and starting at $(f^{\mathrm{in}},\varrho^{\mathrm{in}},u^{\mathrm{in}})$. Furthermore, the condition \eqref{def:Bound-HAUTBAS} is satisfied by $(f_\eps,\varrho_\eps)$ with $T=T(\eps)$.
\end{propo}
\begin{proof}The proof is mainly based on the a priori estimates we have just derived, through a classical approximation procedure. Because of the regularization on the gradient of $\varrho_\eps$ in the Vlasov equation, the procedure is fairly standard. For the reader's convenience, we write the proof in Section \ref{Appendix:LWPeps} of the Appendix.
\end{proof}

\textbf{Here ever after and until the end of this article, we consider exponents $r>0$ and $m>0$ which can be taken large enough. They will be chosen later on, in the end of the proof.}

\medskip

We now introduce the following quantity, in view of the expected result of Theorem \ref{THM:main}.
\begin{defi}\label{def:N}
For any functions $f(t,x,v), \varrho(t,x)$ and $u(t,x)$, we set
\begin{align*}
\mathcal{N}_{m,r}(f,\varrho,u,T):= \Vert f \Vert_{\Ld^{\infty}\left(0,T;\mathcal{H}^{m-1}_{r}\right)}+\Vert \varrho \Vert_{\Ld^2(0,T;\H^{m})}+\Vert u \Vert_{\Ld^{\infty}(0,T; \H^m) \cap \Ld^{2}(0,T;\H^{m+1})},
\end{align*}
where $T>0$.
\end{defi}

The proof of Theorem \ref{THM:main} will rely on a bootstrap argument. Let $\eps >0$. From Proposition \ref{prop:existenceEPS}, consider the maximal time of existence  $T^{\ast}_\eps$  to the system \eqref{S_eps}. By definition, Proposition \ref{prop:existenceEPS} ensures that
\begin{align*}
\forall T<T^{\ast}_\eps, \ \ \mathcal{N}_{m,r}(f_\eps,\varrho_\eps,u_\eps,T)<+\infty, \ \ \text{and} \ \ \eqref{def:Bound-HAUTBAS} \ \ \text{holds}.
\end{align*}
So we can consider the time 
\begin{align*}
T_\eps=T_\eps(R):=\sup \big\lbrace T \in [0,T^{\ast}_\eps[, \ \  \mathcal{N}_{m,r}(f_\eps,\varrho_\eps,u_\eps,T) \leq R \ \ \text{and} \ \ \eqref{def:Bound-HAUTBAS} \ \ \text{holds}  \big\rbrace,
\end{align*}
where $R>0$ will be chosen large enough and independent of $\eps$. By continuity, we observe that $T_\eps>0$ if $R$ is taken large enough and independent of $\eps$. In particular, for all $t \in [0,T_\eps]$, we have
\begin{align}\label{bound:HAUT-1/1-rho_f}
0< \frac{1-\Theta}{2} \leq 1-\rho_f(t), \ \ \frac{1}{1-\rho_f(t)} \leq \frac{2}{1-\Theta}.
\end{align}
Our main goal is to prove that $R$ can be chosen large enough so that there exists $T(R)>0$ independent of $\eps$ such that 
\begin{align*}
\forall \eps \in (0,1), \ \ T(R) \leq T_\eps.
\end{align*}
Such a lower bound independent of $\eps$ will pave the way for a compactness argument when $\eps \rightarrow 0$, leading to the existence of a solution for \eqref{eq:TSgenBaro} on $[0,T(R)]$. In what follows, we will work on the interval of time $[0,T_\eps]$.

We have the following trichotomy:
\begin{itemize}
\item either $T_\eps^{\ast}=+\infty$ and $T_\eps=T_\eps^{\ast}$, then there is nothing to do because $\mathcal{N}_{m,r}(f_\eps,\varrho_\eps,u_\eps,T) \leq R$ for all times $T>0$;
\item either $T_\eps^{\ast}<+\infty$ and $T_\eps=T_\eps^{\ast}$,  we shall see soon enough that this is impossible at this leads to a contradiction;
\item else, $T_\eps<T_\eps^{\ast}$ and $\mathcal{N}_{m,r}(f_\eps,\varrho_\eps,u_\eps,T_\eps)=R$.
\end{itemize}
Let us show how to exclude the second case. We need the following lemma.

\begin{lem}\label{LM:rho-pointwise-Hm-2}
Let $\eps>0$. If $m>3+d/2$ and $r \geq 1+d/2$, we have for all 
$T \in [0,T_\eps)$
\begin{align*}
 \Vert \varrho_\eps \Vert_{\Ld^{\infty}(0,T;\Ld^{\infty})} \lesssim \Vert \varrho_\eps \Vert_{\Ld^{\infty}(0,T;\H^{m-2})} \leq   \Lambda\left( T,R,  \Vert \varrho^{\mathrm{in}} \Vert_{\H^{m-2}}, \right).
 \end{align*}
\end{lem}
\begin{proof}
From Proposition \ref{prop:energyRHO:eps}, we know that for all $T \in [0,T_\eps)$ and $t \in [0,T]$
\begin{align*}
 \Vert \varrho_\eps(t) \Vert_{\H^{m-2}} \leq   \Vert \varrho^{\mathrm{in}} \Vert_{\H^{m-2}} e^{C_{m-2}(T,u_\eps,f_\eps)T}
 \exp \left[Te^{C_{m-2}(T,u_\eps,f_\eps)T} \mathcal{Q}_{m-2}(T,u_{\eps},f_{\eps}) \right],
\end{align*}
provided that $m-2, r >1+d/2$, and where
\begin{align*}
C_{m-2}(T,u_{\eps},f_{\eps})\leq R +  2R^2\left\Vert\frac{1}{1-\rho_{f_\eps}} \right\Vert_{\Ld^{\infty}(0,T; \Ld^{\infty})} , \ \ \mathcal{Q}_{m-2}(T,u_{\eps},f_{\eps})\leq R+ 2R^2 \left\Vert\frac{1}{1-\rho_{f_\eps}} \right\Vert_{\Ld^{\infty}(0,T; \H^{m-2})}.
\end{align*}
because $ \mathcal{N}_{m,r}(f_\eps,\varrho_\eps,u_\eps,T) \leq R $ for all $T \in [0,T_\eps)$.
By Sobolev embedding and the bound \eqref{bound:HAUT-1/1-rho_f}, this means that for $r>1+d/2$ and $m>1+d/2 +2$, we have for all $T \in [0,T_\eps)$
\begin{align*}
\Vert \varrho_\eps \Vert_{\Ld^{\infty}(0,T;\Ld^{\infty})} \lesssim \Vert \varrho_\eps \Vert_{\Ld^{\infty}(0,T;\H^{m-2})}\leq \Lambda \left( T,R,  \Vert \varrho^{\mathrm{in}} \Vert_{\H^{m-2}}, \left\Vert\frac{1}{1-\rho_{f_\eps}} \right\Vert_{\Ld^{\infty}(0,T;\H^{m-2})} \right).
\end{align*}
To conclude, we only have to understand the last term in the previous function $\Lambda$: by Lemma \ref{LM:(1-g)-1}, there exists a continuous nonnegative nondecreasing function $\mathrm{C}_m$ such that
\begin{align*}
\left\Vert\frac{1}{1-\rho_{f_\eps}} \right\Vert_{\Ld^{\infty}(0,T;\H^{m-2})} &\leq 1+\mathrm{C}_{m}\left( \LRVert{\rho_{f_\eps}}_{\Ld^{\infty}(0,T;\Ld^{\infty})} \right) \LRVert{\rho_{f_\eps}}_{\Ld^{\infty}(0,T;\H^{m-2})}
 \lesssim \Lambda \left( R \right),
\end{align*}
thanks to Lemma \ref{LM:weightedSob} and the fact that $ \mathcal{N}_{m,r}(f_\eps,\varrho_\eps,u_\eps,T) \leq R $. This concludes the proof. 
\end{proof}

\begin{rem}\label{rem:estim-rho-lowderivative}
A careful inspection of the proof of Proposition \ref{prop:energyRHO:eps} reveals that for all $k \leq m-2$ with $k>3+d/2$, $r>1+d/2$ and
$T \in [0, T_\eps)$
\begin{align*}
 \Vert \varrho_\eps \Vert_{\Ld^{\infty}(0,T;\H^{k})} \lesssim  \Lambda( \Vert \varrho^{\mathrm{in}} \Vert_{\H^{k}})+ T\Lambda\left(T,R \right).
\end{align*}
\end{rem}

As a corollary, we can now exclude the second case written above: if $T_\eps^{\ast}<+\infty$ and $T_\eps=T_\eps^{\ast}$, then according to Proposition \ref{prop:energy:f:eps} and Lemma \ref{LM:rho-pointwise-Hm-2},
\begin{align*}
\underset{t \in [0,T_\eps^{\ast})}{\sup} \, \Vert f_\eps(t) \Vert_{\mathcal{H}^{m}_{r}}^2  \leq  \Vert f^{\mathrm{in}}\Vert_{\mathcal{H}^{m}_{r}}^2 \exp \left[C\left( (1+ R)T_\eps+ \frac{\sqrt{T_\eps}}{\eps}  \Lambda \left( T,R,  \Vert \varrho^{\mathrm{in}} \Vert_{\H^{m-2}} \right) \right) \right] <+ \infty.
\end{align*}  
The previous inequality means that the solution could be continued beyond $T_\eps^{\ast}$  which is impossible by maximality of $T_\eps^{\ast}$. This case is thus impossible.

\medskip
From now on, we assume that $T_\eps<T_\eps^{\ast}$ and $\mathcal{N}_{m,r}(f_\eps,\varrho_\eps,u_\eps,T_\eps)=R$. In view of our bootstrap stategy, we need to estimate $\mathcal{N}_{m,r}(f_\eps,\varrho_\eps,u_\eps,T)$ for all $T<T_\eps$.  

In the end of this section, we  show that the terms $\Vert f_\eps \Vert_{\Ld^{\infty}\left(0,T;\mathcal{H}^{m-1}_{r}\right)}$ and $\Vert u_\eps \Vert_{\Ld^{\infty}(0,T; \H^m) \cap \Ld^{2}(0,T;\H^{m+1})}$ can be handled by energy estimates. The main part of the upcoming analysis will be to provide a uniform control in $\eps$ for the term $\Vert \varrho \Vert_{\Ld^{2}\left(0,T;\mathcal{H}^{m}_{r}\right)}$.

\medskip

In the following lemma, we give an estimate independent of $\eps$ for the term $\Vert f_\eps \Vert_{\Ld^{\infty}\left(0,T;\mathcal{H}^{m-1}_{r}\right)}$ in $\mathcal{N}_{m,r}(f_\eps,\varrho_\eps,u_\eps,T)$, for all $T<T_\eps$.

\begin{lem}\label{energy:f-unif}
For $m>1+d/2+2$ and $r>1+d/2$, the solution $(f_\eps,\varrho_\eps,u_\eps)$ to \eqref{S_eps} satisfies for all
$T \in [0,T_\eps)$
\begin{align*}
 \Vert f_\eps \Vert_{\Ld^{\infty}\left(0,T;\mathcal{H}^{m-1}_{r}\right)} \leq  \Vert f^{\mathrm{in}} \Vert_{\mathcal{H}^{m-1}_{r}}+T^{\frac{1}{4}}\Lambda\left( T,R \right).
\end{align*}
\end{lem}
\begin{proof}
Following the same steps leading to \eqref{ineq:d/dt:fweight} in the proof of Proposition \ref{prop:energy:f}, we have for all $t \in [0,T_\eps)$
\begin{align*}
\frac{\mathrm{d}}{\mathrm{d}t}\Vert f_\eps(t) \Vert_{\mathcal{H}^{m-1}_{r}}^2  \lesssim (1+\Vert E_{\reg,\eps}^{u_\eps,\varrho_\eps}(t) \Vert_{\H^{m-1}} ) \Vert f_\eps(t)\Vert_{\mathcal{H}^{m-1}_{r}}^2,
\end{align*}
since $m-1>d/2$, therefore
\begin{align*}
\Vert f_\eps \Vert_{\Ld^{\infty}\left(0,T;\mathcal{H}^{m-1}_{r}\right)}^2 &\leq \Vert f^{\mathrm{in}} \Vert_{\mathcal{H}^{m-1}_{r}}^2 + \Vert f_\eps \Vert_{\Ld^{\infty}\left(0,T;\mathcal{H}^{m-1}_{r}\right)}^2 \Bigg( T+ \int_0^T \Vert  E_{\reg,\eps}^{u_\eps,\varrho_\eps}(s) \Vert_{\H^{m-1}} \, \mathrm{d}s \Bigg).
\end{align*}
Using now the estimate \eqref{bound:Sobo2:E} from Lemma \ref{LM:Estim-ForceField}, we get
\begin{align*}
\int_0^T \Vert \E_{\reg,\eps}^{u_\eps,\varrho_\eps}(s) \Vert_{\H^{m-1}} \, \mathrm{d}s \lesssim \sqrt{T}\Vert u_\eps\Vert_{\Ld^{2}(0,T;\H^{m-1})}+\sqrt{T}\Lambda\left(\Vert \varrho_\eps \Vert_{\Ld^{\infty}(0,T;\H^{m-2} )} \right) \Vert \varrho_\eps \Vert_{\Ld^2(0,T;\H^{m} )}.
\end{align*}
We thus infer that for all $T \in [0,T_\eps)$
\begin{align*}
\Vert f_\eps \Vert_{\Ld^{\infty}\left(0,T;\mathcal{H}^{m-1}_{r}\right)}^2 &\leq \Vert f^{\mathrm{in}} \Vert_{\mathcal{H}^{m-1}_{r}}^2 +C \Vert f_\eps \Vert_{\Ld^{\infty}\left(0,T;\mathcal{H}^{m-1}_{r}\right)}^2 \Bigg( T+ \sqrt{T}\Vert u_\eps\Vert_{\Ld^{2}(0,T;\H^{m-1})} \\
&\qquad \qquad \qquad \qquad \qquad \qquad  \qquad \quad  + \sqrt{T}\Lambda\left(\Vert \varrho_\eps \Vert_{\Ld^{\infty}(0,T;\H^{m-2} )} \right) \Vert \varrho_\eps  \Vert_{\Ld^2(0,T;\H^{m} )}  \Bigg),
\end{align*}
where $C>0$ is independent of $\eps$. Since $ \mathcal{N}_{m,r}(f_\eps,\varrho_\eps,u_\eps,T) \leq R $ for all $T \in [0,T_\eps)$,
this gives
\begin{align*}
\Vert f_\eps \Vert_{\Ld^{\infty}\left(0,T;\mathcal{H}^{m-1}_{r}\right)}^2 &\leq \Vert f^{\mathrm{in}} \Vert_{\mathcal{H}^{m-1}_{r}}^2 + CR^2 \Bigg( T+\sqrt{T}R +\sqrt{T}\Lambda\left(\Vert \varrho_\eps \Vert_{\Ld^{\infty}(0,T;\H^{m-2} )} \right)  R\Bigg).
\end{align*}
To obtain a uniform bound in $\eps$ for the term $\Lambda\left(\Vert \varrho_{\eps} \Vert_{\Ld^{\infty}(0,T;\H^{m-2})} \right)$,
we use Lemma \ref{LM:rho-pointwise-Hm-2}, leading to the conclusion of the lemma.
\end{proof}

We conclude this section with an estimate for the term $\Vert u_\eps \Vert_{\Ld^{\infty}(0,T; \H^m) \cap \Ld^{2}(0,T;\H^{m+1})}$ in $\mathcal{N}(f_\eps,\varrho_\eps,u_\eps,T)$. To this end we will crucially rely on the smoothing provided by the differential operator $ \Delta_x  + \nabla_x \mathrm{div}_x$ from the Navier-Stokes equation on $u_\eps$. Indeed, we have the following lemma.

\begin{lem}\label{LM:LaméELLIPTIC}
The differential operator $-\Delta_x  - \nabla_x \mathrm{div}_x $ is elliptic.
\end{lem}
\begin{proof}
The operator $-\Delta_x  - \nabla_x \mathrm{div}_x$ is associated to the matrix Fourier multiplier $L(k)=\vert k \vert^2  \mathrm{I}_d+k \otimes k$ ($k \in \Z^d$). One can then prove (see e.g. \cite{DanchinTolk2022critical}) that 
\begin{align*}
2^{-3\frac{d}{2}} \vert k \vert^{2d} \leq \vert \det \,  L(k) \vert,
\end{align*}
which yields the desired ellipticity.
\end{proof}

We are now able to prove the following proposition.

\begin{propo}\label{prop:energy-u_eps}
For $m>2+d/2$ and $ r \geq 0$, for all $\eps>0$, the solution $(f_\eps,\varrho_\eps,u_\eps)$ to \eqref{S_eps} satisfies for all $T \in [0,T_\eps)$
\begin{multline*}
\Vert u_\eps \Vert_{\Ld^{\infty}(0,T;\H^m)} + \Vert u_\eps \Vert_{\Ld^{2}(0,T;\H^{m+1})} \\ \lesssim \left(1+T^{1/2} \Lambda(T,R) \right)\left(\Vert u^{\mathrm{in}} \Vert_{\H^m} + T^{1/2}\Lambda(T,R) + \Lambda(T,R)\Vert \varrho_\eps \Vert_{\Ld^{2}(0,T;\H^{m})} \right).
\end{multline*}
\end{propo}
\begin{proof} First, we rewrite the equation on $u_\eps$ as
\begin{align*}
\partial_t u_\eps - \sigma\big((1-\rho_{f_\eps})\varrho_\eps \big) \Big( \Delta_x  + \nabla_x \mathrm{div}_x \Big)u_\eps=F,
\end{align*}
with
$$ F:=-(u_\eps \cdot \nabla_x )u_\eps -\nabla_x \pi(\varrho_\eps)+ \sigma\big((1-\rho_{f_\eps}) \varrho_\eps \big)(j_{f_\eps}-\rho_{f_\eps} u_\eps),$$
where
$\sigma(s):=\frac{1}{s}, \, \pi(s)=\int_0^s \frac{p'(\tau)}{\tau} \, \mathrm{d}\tau$. 
We then apply $\partial_x ^{\beta}$ in the equation for $\vert \beta \vert \leq m-1$, which gives
\begin{align*}
\partial_t (\partial^{\beta}_x u_\eps) - \sigma\big((1-\rho_{f_\eps})\varrho_\eps \big)\Big( \Delta_x  + \nabla_x \mathrm{div}_x \Big)(\partial^{\beta}_x u_\eps)=\partial^{\beta}_x F+\Big[\partial^{\beta}_x, \sigma\big((1-\rho_{f_\eps})\varrho_\eps \big) \Big]( (\Delta_x  + \nabla_x \mathrm{div}_x )u_\eps),
\end{align*}
and then multiply the equation with $-\Big( \Delta_x  + \nabla_x \mathrm{div}_x \Big)(\partial^{\beta}_x u)$ so that, by integrating on $\T^d$ and with an integration by parts:
\begin{multline*}
\frac{1}{2}\frac{\mathrm{d}}{\mathrm{d}t}  \int_{\T^d} \left( \vert \nabla_x \partial^{\beta}_x u_\eps \vert^2 + \vert \mathrm{div}_x \partial^{\beta}_x u_\eps  \vert^2 \right) \, \mathrm{d}x+\int_{\T^d} \sigma\big((1-\rho_{f_\eps})\varrho_\eps \big)\vert \Big( \Delta_x  + \nabla_x \mathrm{div}_x \Big)(\partial^{\beta}_x u_\eps) \vert^2 \, \mathrm{d}x\\
=-\int_{\T^d} \Big( \Delta_x  + \nabla_x \mathrm{div}_x \Big)(\partial^{\beta}_x u_\eps) \cdot \partial_x^{\beta} F \, \mathrm{d}x- \int_{\T^d}\Big[\partial^{\beta}_x, \sigma\big((1-\rho_{f_\eps})\varrho_\eps \big) \Big](( \Delta_x  + \nabla_x \mathrm{div}_x )u_\eps) \cdot ( \Delta_x  + \nabla_x \mathrm{div}_x )(\partial^{\beta}_x u_\eps) \, \mathrm{d}x.
\end{multline*}
Thanks to the Cauchy-Schwarz and Young inequalities, and after integration in time, we get for all $\eta >0$ and $t \in (0,T)$
\begin{align*}
&\frac{1}{2} \int_{\T^d} \vert \nabla_x \partial^{\beta}_x u_\eps(t) \vert^2  \, \mathrm{d}x + \int_0^t \int_{\T^d}\left( \sigma\big((1-\rho_{f_\eps})\varrho_\eps \big)-\eta \right) \vert ( \Delta_x  + \nabla_x \mathrm{div}_x ) \partial_x^{\beta}u_\eps \vert^2 \, \mathrm{d}x \, \mathrm{d}s \\  
&\leq \frac{1}{2} \int_{\T^d} \left( \vert \nabla_x \partial^{\beta}_x u^{\mathrm{in}} \vert^2 + \vert \mathrm{div}_x \partial^{\beta}_x u^{\mathrm{in}} \vert^2 \right) \, \mathrm{d}x \\ 
& \qquad \qquad \qquad   + \frac{1}{2\eta} \int_0^t \int_{\T^d} \vert \partial_x^{\beta} F \vert^2 \, \mathrm{d}x \, \mathrm{d}s+ \frac{1}{2 \eta }\int_0^t\int_{\T^d}  \left\vert \Big[\partial^{\beta}_x, \sigma\big((1-\rho_{f_\eps})\varrho_\eps \big)\Big](( \Delta_x  + \nabla_x \mathrm{div}_x )u_\eps) \right \vert^2 \, \mathrm{d}x \, \mathrm{d}s.
\end{align*}
Let us deal with the last term: by the commutator inequality from Proposition \ref{CommutSOB-KlainBerto}, we have
\begin{align*}
&\int_{\T^d}  \left\vert \Big[\partial^{\beta}_x, \sigma\big((1-\rho_{f_\eps})\varrho_\eps \big) \Big](( \Delta_x  + \nabla_x \mathrm{div}_x )u_\eps) \right \vert^2 \, \mathrm{d}x  \\
&\leq M \left( \LRVert{\nabla_x \sigma\big((1-\rho_{f_\eps})\varrho_\eps \big)}_{\Ld^{\infty}}^2 \LRVert{( \Delta_x  + \nabla_x \mathrm{div}_x )u_\eps}_{\H^{m-2}}^2 + \LRVert{\sigma\big((1-\rho_{f_\eps})\varrho_\eps \big)}_{\H^{m-1}}^2 \LRVert{( \Delta_x  + \nabla_x \mathrm{div}_x )u_\eps}_{\Ld^{\infty}}^2\right) \\
&\leq M \left( \LRVert{\sigma'\big((1-\rho_{f_\eps})\varrho_\eps \big) \nabla_x ((1-\rho_{f_\eps})\varrho_\eps)}_{\Ld^{\infty}}^2 \LRVert{u}_{\H^{m}}^2 + \LRVert{\sigma\big((1-\rho_{f_\eps})\varrho_\eps \big)}_{\H^{m-1}}^2 \LRVert{( \Delta_x  + \nabla_x \mathrm{div}_x )u_\eps}_{\Ld^{\infty}}^2\right),
\end{align*}
for some constant $M>0$ independent of time. Combining the Sobolev embedding (with $m>2+d/2$) and Remark \ref{Rem:Bony-Meyer}, we get
\begin{multline*}
\frac{1}{2 \eta }\int_0^t \int_{\T^d}  \left\vert \Big[\partial^{\beta}_x, \sigma\big((1-\rho_{f_\eps})\varrho_\eps \big) \Big](( \Delta_x  + \nabla_x \mathrm{div}_x )u_\eps) \right \vert^2 \, \mathrm{d}x \, \mathrm{d}s \\
\leq M \frac{1}{2 \eta } \Lambda \left(\LRVert{(1-\rho_{f_\eps})\varrho_\eps}_{\Ld^{\infty}(0,T;\H^{m-1})} \right)  \int_0^t\LRVert{u_\eps(\tau)}_{\H^{m}}^2 \, \mathrm{d}\tau,
\end{multline*}
for another constant $M>0$. Note that by Remark \ref{rem:energy-alpharho}, we have
\begin{align*}
\LRVert{(1-\rho_{f_\eps})\varrho_\eps}_{\Ld^{\infty}(0,T;\H^{m-1})} \leq  \Lambda(T, R,\Vert u_\eps \Vert_{\Ld^{\infty}(0,T;\H^{m})} )\leq \Lambda(T,R).
\end{align*}
All in all, we get for all $\eta >0$ and $t \in (0,T)$
\begin{multline*}
\frac{1}{2}\LRVert{\nabla_x \partial_x^{\beta} u_\eps(t)}_{\Ld^2}^2 + \int_0^t \int_{\T^d}\left( \sigma\big((1-\rho_{f_\eps})\varrho_\eps \big)-\eta \right) \vert ( \Delta_x  + \nabla_x \mathrm{div}_x ) \partial_x^{\beta}u_\eps \vert^2 \, \mathrm{d}x \, \mathrm{d}s  \\ \leq \Vert u^{\mathrm{in}} \Vert_{\H^m}^2 + \frac{1}{\eta}\Vert F \Vert_{\Ld^{2}(0,T;\H^{m-1})}^2+M \frac{\Lambda(T,R)}{2 \eta } \int_0^t\LRVert{u_\eps(\tau)}_{\H^{m}}^2 \, \mathrm{d}\tau.
\end{multline*}
Thanks to the condition \eqref{def:Bound-HAUTBAS}, we can choose $\eta=1/4\overline{\theta}$ so that 
\begin{align*}
\forall (t,x) \in [0,T] \times \T^d, \ \ \frac{1}{4\overline{\theta}}<\sigma\big((1-\rho_{f_\eps})\varrho_\eps \big)(t,x) -\eta.
\end{align*}
Summing for all $|\beta|=m$ and invoking the elliptic regularity for the operator $-\Delta_x  - \nabla_x \mathrm{div}_x$ given by Lemma \ref{LM:LaméELLIPTIC}, we get for all $t \in (0,T)$
\begin{align*}
\LRVert{u_\eps(t)}_{\H^{m}}^2 \leq \LRVert{u_\eps(t)}_{\H^{m}}^2+\Vert u_\eps\Vert_{\Ld^{2}(0,T;\H^{m+1})}^2 \lesssim  \Vert u^{\mathrm{in}} \Vert_{\H^m}^2 + \Vert F \Vert_{\Ld^{2}(0,T;\H^{m-1})}^2+ \Lambda(T,R) \int_0^t\LRVert{u_\eps(\tau)}_{\H^{m}}^2 \, \mathrm{d}\tau.
\end{align*}
By Grönwall's lemma, we deduce for all $t \in (0,T)$
\begin{align*}
\LRVert{u_\eps(t)}_{\H^{m}}^2 
 \leq  \Lambda(T,R)\left( \Vert u^{\mathrm{in}} \Vert_{\H^m}^2 + \Vert F \Vert_{\Ld^{2}(0,T;\H^{m-1})}^2 \right),
\end{align*} 
which then implies, by using again the previous inequality, that for all $t \in (0,T)$
\begin{align*}
\LRVert{u_\eps}_{\Ld^{\infty}(0,T;\H^{m})}^2 +\Vert u_\eps\Vert_{\Ld^{2}(0,T;\H^{m+1})}^2 \lesssim  (1+ T \Lambda(T,R) ) \left( \Vert u^{\mathrm{in}} \Vert_{\H^m}^2 + \Vert F \Vert_{\Ld^{2}(0,T;\H^{m-1})}^2 \right).
\end{align*}
To conclude, let us now estimate the norm of the source term $F$. We have 
\begin{align*}
\Vert F \Vert_{\Ld^{2}(0,T;\H^{m-1})}^2 &\leq \int_0^T \Bigg(\Vert (u_\eps \cdot \nabla_x )u_\eps(\tau) \Vert_{\H^{m-1}}^2+\Vert \nabla_x \pi (\varrho_\eps(\tau)) \Vert_{\H^{m-1}}^2 \\
& \qquad \qquad \qquad \qquad  \qquad \qquad  \qquad \qquad \qquad \qquad  + \left\Vert \frac{1}{(1-\rho_{f_\eps})\varrho_\eps}(j_{f_\eps}-\rho_{f_\eps} u_\eps)(\tau) \right\Vert_{\H^{m-1}}^2 \Bigg) \, \mathrm{d}\tau \\
& \leq  \int_0^T \left( \Vert u_\eps(\tau) \Vert_{\H^{m-1}}^2 \Vert u_\eps(\tau) \Vert_{\H^{m}}^2 +\Vert \pi(\varrho_\eps(\tau)) \Vert_{\H^{m}}^2  \right)  \, \mathrm{d}\tau \\
& \qquad \qquad \qquad \qquad  \qquad \qquad  \qquad \qquad \qquad \qquad  + T \left\Vert \frac{1}{(1-\rho_{f_\eps})\varrho_\eps}(j_{f_\eps}-\rho_{f_\eps} u_\eps)\right\Vert_{\Ld^{\infty}(0,T;\H^{m-1})}^2.
\end{align*}
In the rest of the proof, we shall make a constant use of the condition \eqref{def:Bound-HAUTBAS}.
By Proposition \ref{Bony-Meyer} in the Appendix, we have
\begin{align*}
\Vert \pi(\varrho_\eps) \Vert_{\H^{m}} &\lesssim \Lambda(\Vert \varrho_\eps \Vert_{\Ld^{\infty}}) \Vert \varrho_\eps \Vert_{\H^{m}},
\end{align*}
from which we infer thanks to Sobolev embedding (taking $m>2 + \frac{d}{2}$)
\begin{align*}
&\Vert F \Vert_{\Ld^{2}(0,T;\H^{m-1})}^2 \\
& \lesssim  T\left\Vert u_\eps\right\Vert_{\Ld^{\infty}(0,T;\H^{m})}^4+\Lambda(\Vert \varrho \Vert_{\Ld^{\infty}(0,T;\Ld^{\infty})}) \Vert \varrho_\eps \Vert_{\Ld^2(0,T;\H^{m})}^2+  T \left\Vert \frac{1}{(1-\rho_{f_\eps})\varrho_\eps}(j_{f_\eps}-\rho_{f_\eps} u_\eps)\right\Vert_{\Ld^{\infty}(0,T;\H^{m-1})}^2  \\
& \lesssim  TR^4+ \Lambda(\Vert \varrho_\eps \Vert_{\Ld^{\infty}(0,T;\H^{m-2})}) \Vert \varrho_\eps \Vert_{\Ld^2(0,T;\H^{m})}^2+  T \left\Vert \frac{1}{(1-\rho_{f_\eps})\varrho_\eps}(j_{f_\eps}-\rho_{f_\eps} u_\eps)\right\Vert_{\Ld^{\infty}(0,T;\H^{m-1})}^2,
\end{align*}
since $ \mathcal{N}_{m,r}(f_\eps,\varrho_\eps,u_\eps,T) \leq R $ for all $T \in [0,T_\eps)$. The first term is then adressed thanks to Lemma \ref{LM:rho-pointwise-Hm-2} and it remains to estimate the last term. For this one, we have
\begin{align*}
\left\Vert \frac{1}{(1-\rho_{f_\eps})\varrho_\eps}(j_{f_\eps}-\rho_{f_\eps} u_\eps)\right\Vert_{\Ld^{\infty}(0,T;\H^{m-1})} &\leq \left\Vert \frac{1}{(1-\rho_{f_\eps})\varrho_\eps}\right\Vert_{\Ld^{\infty}(0,T;\H^{m-1})}\left\Vert (j_{f_\eps}-\rho_{f_\eps} u_\eps)\right\Vert_{\Ld^{\infty}(0,T;\H^{m-1})} \\
& \leq  \Lambda(T,R)\left\Vert \frac{1}{(1-\rho_{f_\eps})\varrho_\eps}\right\Vert_{\Ld^{\infty}(0,T;\H^{m-1})},
\end{align*} 
thanks to Lemma \ref{LM:weightedSob} and $\mathcal{N}_{m,r}(f_\eps,\varrho_\eps,u_\eps,T) \leq R $ for all $T \in [0,T_\eps)$. By Remark \ref{Rem:Bony-Meyer}, we then have by Sobolev embedding
\begin{align*}
\left\Vert \frac{1}{(1-\rho_{f_\eps})\varrho_\eps}(j_f-\rho_f u)\right\Vert_{\Ld^{\infty}(0,T;\H^{m-1})} &\leq \Lambda(T,R,\Vert (1-\rho_{f_\eps})\varrho_\eps \Vert_{\Ld^{\infty}(0,T; \Ld^{\infty})} ) \Vert (1-\rho_{f_\eps})\varrho_\eps \Vert_{\Ld^{\infty}(0,T; \H^{m-1})} \\
& \leq \Lambda(T,R,\Vert (1-\rho_{f_\eps})\varrho_\eps \Vert_{\Ld^{\infty}(0,T; \H^{m-1})} ) \\
& \leq \Lambda(T, R,\Vert u_\eps \Vert_{\Ld^{\infty}(0,T;\H^{m})} ) \\
 & \leq \Lambda(T,R),
\end{align*}
where we have also used Remark \ref{rem:energy-alpharho}. This eventually concludes the proof.

\end{proof}

\begin{rem}\label{rem:estim-u-lowderivative}
By looking at the previous proof, we have for  
$T \in [0,T_\eps)$
\begin{align*}
\Vert u_\eps \Vert_{\Ld^{\infty}(0,T;\H^k)}^2+\Vert u\Vert_{\Ld^{2}(0,T;\H^{k+1})}^2 
 & \leq \Vert u^{\mathrm{in}} \Vert_{\H^k}^2 + T \Lambda(T,R)+ \Lambda(T,R)  \Vert \varrho_{\eps} \Vert_{\Ld^{2}(0,T;\H^{k})}^2 \\
 & \leq \Vert u^{\mathrm{in}} \Vert_{\H^k}^2 + T \Lambda(T,R)+ T\Lambda(T,R)  \Vert \varrho_{\eps} \Vert_{\Ld^{{\infty}}(0,T;\H^{k})}^2,
\end{align*}
for all $k>2+d/2$ such that $k\leq m-2$, therefore by Remark \ref{rem:estim-rho-lowderivative}, we obtain
\begin{align*}
\Vert u_\eps \Vert_{\Ld^{\infty}(0,T;\H^k)} + \Vert u_\eps \Vert_{\Ld^{2}(0,T;\H^{k+1})} \lesssim \left(1+T^{1/2}\Lambda(T,R) \right)\left(\Vert u^{\mathrm{in}} \Vert_{\H^k}+  \Vert \varrho^{\mathrm{in}} \Vert_{\H^k} \right)+  T^{1/2}\Lambda(T,R).
\end{align*}
\end{rem}

\medskip

So far, Lemma \ref{energy:f-unif} and Proposition \ref{prop:energy-u_eps} show that it remains to control the second term in $\mathcal{N}(f_\eps,\varrho_\eps,u_\eps,T)$, that is $\Vert \varrho_\eps \Vert_{\Ld^{2}(0,T; \H^m) }$, to perform a bootstrap argument. This will constitute the heart of our analysis and will be the purpose of the remaining sections.

\section{Trajectories and straightening change of variable}\label{Section:Lagrangian}

In this short section, we study the trajectories associated to a Vlasov equation with friction and  force field $\mathrm{F}(t,x)$. 
We show that for small times, their geometry can be simplified thanks to a straightening change of variable in velocity. Loosely speaking, this allows to boil down the dynamics to that associated with free-transport with friction. This procedure will be useful in Section \ref{Section:Kinetic-moments}.

\medskip

Let $T>0$. Given $\mathrm{F}(s,x) \in \R^d$ a given vector field defined on $[0,T] \times \T^d$ and satisfying
\begin{align*}
\mathrm{F} \in \Ld^2(0,T; \W^{1, \infty}(\T^d)), 
\end{align*}
we can consider, thanks to the Cauchy-Lipschitz theorem, the solution $s \mapsto \pr{\mathrm{X}^{s;t}(x,v),\mathrm{V}^{s;t}(x,v)} \in \T^d \times \R^d$ of the following system of ODE:
\begin{equation*}
\left\{
      \begin{aligned}
        \frac{\mathrm{d}}{\mathrm{d}s}\mathrm{X}^{s;t}(x,v) &=\mathrm{V}^{s;t}(x,v), \\
\frac{\mathrm{d}}{\mathrm{d}s}\mathrm{V}^{s;t}(x,v)&=-\mathrm{V}^{s;t}(x,v)+\mathrm{F}(s,\mathrm{X}^{s;t}(x,v)) ,\\
	\mathrm{X}^{t;t}(x,v)&=x,\\
	\mathrm{V}^{t;t}(x,v)&=v.
      \end{aligned}
    \right.
\end{equation*}

Later on, we will apply this to $\mathrm{F}=E_{\reg,\eps}^{u_\eps,\varrho_\eps}$, which has been defined in the beginning of Section \ref{Subsection:EstimateREG} (see the end of the current section). Integrating the previous system of ODE, we have
\begin{align}
\label{characX}\mathrm{X}^{s;t}(x,v)&=x+(1-e^{t-s})v+\int_t^s (1-e^{\tau-s})\mathrm{F}\left(\tau, \mathrm{X}^{\tau;t}(x,v) \right) \, \mathrm{d}\tau, \\
\label{characV}\mathrm{V}^{s;t}(x,v)&=e^{t-s}v+\int_t^s e^{\tau-s}\mathrm{F}\left(\tau, \mathrm{X}^{\tau;t}(x,v)\right) \, \mathrm{d}\tau.
\end{align}

Considering the full kinetic transport operator 
$$\mathcal{T}_{\mathrm{F}}=\partial_t +v \cdot\nabla_x -v \cdot \nabla_v + \mathrm{F}(t,x)\cdot \nabla_v-d \mathrm{Id},$$ 
the method of characteristics shows that a smooth function $f(t,x,v)$ satisfying
\begin{equation*}
\left\{
      \begin{aligned}
        \mathcal{T}_{\mathrm{F}}f=0,\\[2mm]
f_{\mid t=0}=f^{\mathrm{in}},
      \end{aligned}
    \right.
\end{equation*}
can be represented as
\begin{align*}
f(t,x,v)=e^{dt} f^{\mathrm{in}}\left(\mathrm{X}^{0;t}(x,v),\mathrm{V}^{0;t}(x,v) \right).
\end{align*}
Note also that for all $t,s \in [0,T]$, the map
\begin{align*}
(x,v) \mapsto \pr{\mathrm{X}^{s;t}(x,v),\mathrm{V}^{s;t}(x,v)}
\end{align*}
is a diffeomorphism from $\T^d \times \R^d$ to itself, which Jacobian value is $e^{d(s-t)}$.

\medskip

The main goal of this section is to prove that for short times, and modulo a \textit{straightening change of variable} in velocity, it is possible to come down to the free dynamics with friction associated to the transport operator
\begin{align*}
\mathcal{T}^{\mathrm{fric}}=\partial_t + v\cdot \nabla_x - v \cdot \nabla_v -d \mathrm{Id}.
\end{align*}
This corresponds to the previous system of ODE with $\mathrm{F}=0$, and for which the solution $(\mathrm{X}^{s;t}_{\mathrm{fric}},\mathrm{V}^{s;t}_{\mathrm{fric}})$ is
\begin{align*}
\mathrm{X}^{s;t}_{\mathrm{fric}}=x+(1-e^{t-s})v, \ \ \mathrm{V}^{s;t}_{\mathrm{fric}}=e^{t-s}v.
\end{align*}
Namely, we have the following  lemma. 

\begin{lem}\label{straight:velocity}
Let $T>0$ and $k \geq 1$. Let $\mathrm{F} \in  \Ld^2(0,T; \W^{k, \infty}(\T^d))$ be a vector field such that 
\begin{align*}
\Vert \mathrm{F} \Vert_{\Ld^2(0,T; \W^{k, \infty}(\T^d))} \leq \Lambda(T,\mathrm{R}),
\end{align*}
for some $\mathrm{R}>0$.
There exists $\overline{T}(\mathrm{R})>0$ such that for all $x \in \T^d$ and $s,t \in [0, \min(\overline{T}(\mathrm{R}),T)]$, there exists a diffeomorphism $\psi_{s,t}(x, \cdot) : \R^d \rightarrow \R^d$ satisfying for all $v \in \R^d$
\begin{align*}
\mathrm{X}^{s;t}\left(x,\psi_{s,t}(x, v) \right)=x+(1-e^{t-s})v,
\end{align*}
which furthermore verifies the estimates 
\begin{align}
\label{supDet} \frac{1}{C} \leq \det \left(\mathrm{D}_v  \psi_{s,t}(x,v) \right) &\leq C, \\
\label{bound:perturbId}\underset{s,t \in [0,T]}{\sup} \, \left\Vert \partial_{x,v}^{\beta} \left(\psi_{s,t}(x,v)-v \right) \right\Vert_{\Ld^{\infty}(\T^d \times \R^d)} &\leq \varphi(T) \Lambda(T,\mathrm{R}), \ \ \vert \beta \vert \leq k, \\
\label{bound:perturbId2}\underset{s,t \in [0,T]}{\sup} \, \left\Vert \partial_{x,v}^{\beta} \partial_s \psi_{s,t}(x,v) \right\Vert_{\Ld^{\infty}(\T^d \times \R^d)} &\leq \varphi(T) \Lambda(T,\mathrm{R}), \ \ \vert \beta \vert \leq k-1,
\end{align}
for some $C>0$ and some nondecreasing continuous function $\varphi : \R^+ \rightarrow \R^+$ vanishing at $0$.
\end{lem}

\begin{proof}
We follow the approach of \cite{HKNR-screen}. We observe that if we set 
\begin{align*}
\widetilde{\mathrm{X}}^{s;t}(x,v):=\dfrac{1}{1-e^{t-s}}\left[\mathrm{X}^{s;t}(x,v)-x-(1-e^{t-s})v \right],
\end{align*}
it comes down to prove that for $s,t$ small enough, the mapping $\phi^{s,t,x}:v \mapsto v+ \widetilde{\mathrm{X}}^{s;t}(x,v)$ is a small Lipschitz perturbation of the identity: denoting its inverse $\psi_{s,t}(x, \cdot)$, it will satisfy
\begin{align}\label{eq:Psi:implicit}
v=\psi_{s,t}(x, v)+\widetilde{\mathrm{X}}^{s,t}(x,\psi_{s,t}(x, v)),
\end{align}
and the first conclusion of the lemma will follow. We introduce the remainder
\begin{align*}
\mathrm{Y}^{s;t}(x,v)=\mathrm{X}^{s;t}(x,v)-x-(1-e^{t-s})v,
\end{align*}
which, in view of \eqref{characX}, satisfies for all $s,t \in [0,T]$
\begin{align}\label{formula:Y}
\mathrm{Y}^{s;t}(x,v)&=\int_s^t (e^{\tau-s}-1)\mathrm{F}\left(\tau, \mathrm{X}^{\tau;t}(x,v)\right) \, \mathrm{d}\tau \notag\\
&=\int_s^t (e^{\tau-s}-1)\mathrm{F} \pr{\tau, x+(1-e^{t-\tau})v+\mathrm{Y}^{\tau;t}(x,v)} \, \mathrm{d}\tau.
\end{align}
We now have
\begin{align}\label{formula:Xtilde:int}
\widetilde{\mathrm{X}}^{s;t}(x,v)=\frac{1}{e^{t-s}-1}\int_s^t (e^{\tau-s}-1)\mathrm{F} \pr{\tau, x+(1-e^{t-\tau})v+\mathrm{Y}^{\tau;t}(x,v)} \, \mathrm{d}\tau.
\end{align}
Thus, estimates on $\mathrm{Y}$ and its derivatives obtained thanks to \eqref{formula:Y} shall provide estimates on $\widetilde{\mathrm{X}}$ and its derivatives via \eqref{formula:Xtilde:int}. 

Let us assume that $s \leq t$ (the case $s \geq t$ can be treated similarly). First, we have
\begin{align*}
\Vert \nabla_x \mathrm{Y}^{s;t} \Vert_{\Ld^{\infty}_{x,v}} &\leq \int_s^t (e^{\tau-s}-1) \Vert \nabla_x \mathrm{F} \pr{\tau, x+(1-e^{t-\tau})v+\mathrm{Y}^{\tau;t}(x,v)}\Vert_{\Ld^{\infty}_{x,v}} (1+\Vert \nabla_x \mathrm{Y}^{\tau;t} \Vert_{\Ld^{\infty}_{x,v}} ) \, \mathrm{d}\tau \\
& \leq \int_s^t (e^{\tau-s}-1) \Vert \nabla_x \mathrm{F} (\tau) \Vert_{\Ld^{\infty}_{x}}(1+\Vert \nabla_x \mathrm{Y}^{\tau;t} \Vert_{\Ld^{\infty}_{x,v}} ) \, \mathrm{d}\tau \\
&\leq (e^T-1)T^{1/2} \Vert \nabla_x \mathrm{F} \Vert_{\Ld^2(0,T; \Ld^{\infty})} \pr{1+\underset{\tau \leq t}{\sup} \, \Vert \nabla_x \mathrm{Y}^{\tau;t} \Vert_{\Ld^{\infty}_{x,v}}}.
\end{align*}
By the assumption on the vector field $\mathrm{F}$, we get
\begin{align*}
\Vert \nabla_x \mathrm{Y}^{s;t} \Vert_{\Ld^{\infty}_{x,v}} \leq (e^T-1)T^{1/2} \Lambda(T,\mathrm{R}) \pr{1+\underset{\tau \leq t}{\sup} \, \Vert \nabla_x \mathrm{Y}^{\tau;t} \Vert_{\Ld^{\infty}_{x,v}}}.
\end{align*}
In the following of the proof,  $\varphi : \R^+ \rightarrow \R^+$ will stand for a generic continuous function, vanishing at $0$, that may change from line to line. This yields
\begin{align}\label{bound:DxY}
\Vert \nabla_x \mathrm{Y}^{s;t} \Vert_{\Ld^{\infty}_{x,v}} \leq \underset{\tau \leq t}{\sup} \, \Vert \nabla_x \mathrm{Y}^{\tau;t} \Vert_{\Ld^{\infty}_{x,v}} \leq \frac{(e^T-1)T^{1/2} \Lambda(T,\mathrm{R}) }{1-(e^T-1)T^{1/2} \Lambda(T,\mathrm{R}) } \lesssim \varphi(T) \Lambda(T,\mathrm{R}) ,
\end{align}
for $T$ small enough. In a similar way, we have
\begin{align*}
\Vert \nabla_v \mathrm{Y}^{s;t} \Vert_{\Ld^{\infty}_{x,v}} &\leq \int_s^t (e^{\tau-s}-1) \Vert \nabla_x \mathrm{F} (\tau) \Vert_{\Ld^{\infty}_{x,v}} (1-e^{t-\tau}+\Vert \nabla_v \mathrm{Y}^{\tau;t} \Vert_{\Ld^{\infty}_{x,v}} ) \, \mathrm{d}\tau \\
&\leq (e^T-1)T^{1/2} \Lambda(T,R) \pr{1+\underset{\tau \leq t}{\sup} \, \Vert \nabla_v \mathrm{Y}^{\tau;t} \Vert_{\Ld^{\infty}_{x,v}}},
\end{align*}
therefore 
\begin{align}\label{bound:DvY}
\Vert \nabla_v \mathrm{Y}^{s;t} \Vert_{\Ld^{\infty}_{x,v}}  \lesssim \varphi(T) \Lambda(T,\mathrm{R}),
\end{align}
for $T$ small enough. We then deduce the following estimates. There holds
\begin{align*}
\Vert \nabla_x \widetilde{\mathrm{X}}^{s;t} \Vert_{\Ld^{\infty}_{x,v}} &\leq \frac{1}{e^{t-s}-1}\int_s^t (e^{\tau-s}-1)\Vert \nabla_x \mathrm{F} (\tau) \Vert_{\Ld^{\infty}_{x,v}}(1+\Vert \nabla_x \mathrm{Y}^{\tau;t} \Vert_{\Ld^{\infty}_{x,v}} ) \, \mathrm{d}\tau \\
& \lesssim T^{1/2}\Lambda(T,R)(1+\varphi(T) \Lambda(T,\mathrm{R})),
\end{align*}
thanks to \eqref{bound:DxY}, as well as
\begin{align*}
\Vert \nabla_v \widetilde{\mathrm{X}}^{s;t} \Vert_{\Ld^{\infty}_{x,v}} &\leq \frac{1}{e^{t-s}-1}\int_s^t (e^{\tau-s}-1)\Vert \nabla_x \mathrm{F} (\tau) \Vert_{\Ld^{\infty}_{x,v}}(1-e^{\tau-t}+\Vert \nabla_v \mathrm{Y}^{\tau;t} \Vert_{\Ld^{\infty}_{x,v}} ) \, \mathrm{d}\tau \\
& \lesssim  T^{1/2}\Lambda(T,\mathrm{R})(1+\varphi(T) \Lambda(T,\mathrm{R})),
\end{align*}
thanks to \eqref{bound:DvY}. We have proven that for $T$ small enough, we have
\begin{align}\label{bound:DxvXtilde}
\Vert \nabla_x \widetilde{\mathrm{X}}^{s;t} \Vert_{\Ld^{\infty}_{x,v}}+\Vert \nabla_v \widetilde{\mathrm{X}}^{s;t} \Vert_{\Ld^{\infty}_{x,v}} \leq \varphi(T)\Lambda(T,\mathrm{R}).
\end{align}
For $T$ small enough, we therefore obtain the existence of the desired diffeomorphism $\psi_{s,t}(x,\cdot)$. We also have 
\begin{align*}
0 < \left\vert \det\left( \nabla_v\psi_{s,t}^x(v) \right) \right\vert  =\left\vert \det\left(\mathrm{Id}+\nabla_v \widetilde{\mathrm{X}}^{s,t}(x,\psi_{s,t}^x(v))\right)\right\vert^{-1}. 
\end{align*}
We thus infer the uniform bound \eqref{supDet} from the estimate \eqref{bound:DxvXtilde} and the continuity of $M \mapsto \vert  \det(M) \vert $, reducing $\overline{T}(\mathrm{R})$ if necessary.

Let us finally prove the estimates \eqref{bound:perturbId}--\eqref{bound:perturbId2}. For \eqref{bound:perturbId}, we proceed by induction on the length of $\alpha$. In view of \eqref{eq:Psi:implicit}, we obtain the result for $\vert \alpha \vert=0$. 
For the case $\vert \alpha \vert=1$, we differentiate the identity \eqref{eq:Psi:implicit} and get, with $\nabla=\nabla_{x}$ or $\nabla= \nabla_v$
\begin{align*}
\nabla (v-\psi_{s,t}^x (v))=\nabla \widetilde{\mathrm{X}}^{s;t}(x,\psi_{s,t}^x (v))-\nabla \widetilde{\mathrm{X}}^{s;t}(x,\psi_{s,t}^x (v))\nabla (v-\psi_{s,t}^x (v)),
\end{align*}
therefore thanks to \eqref{bound:DxvXtilde} (reducing again $\overline{T}(\mathrm{R})$ if necessary), we have
\begin{align*}
\left\Vert \nabla \left(\psi_{s,t}^x (v)-v \right) \right\Vert_{\Ld^{\infty}_{x,v}}  \leq \frac{\Vert \nabla \widetilde{\mathrm{X}}^{s;t}\Vert_{\Ld^{\infty}_{x,v}}  }{1-\Vert \nabla \widetilde{\mathrm{X}}^{s;t}\Vert_{\Ld^{\infty}_{x,v}}} \leq \varphi(T)\Lambda(T,\mathrm{R}).
\end{align*}
This yields the result for $\vert \alpha \vert=1$. 
If $1<\vert \alpha \vert \leq k$ and if the result holds for all $\vert \widetilde{\alpha} \vert < \vert \alpha \vert$, we apply $\partial_{x,v}^{\alpha}$ in \eqref{eq:Psi:implicit} and use the Faà di Bruno's formula:
\begin{align*}
\partial_{x,v}^{\alpha} \left(\psi_{s,t}^x (v)-v \right)=\sum_{\mu, \nu} C_{\mu, \nu}\partial_{x,v}^{\mu} \widetilde{\mathrm{X}}^{s,t}(\mathrm{z}(x,v)) \prod_{\substack{1 \leq \vert \beta \vert \leq \vert \alpha \vert \\ 1 \leq j \leq 2d}} (\partial_{x,v}^{\beta}\mathrm{z}(x,v)_j)^{\nu_{\beta_j}}, \ \ \mathrm{z}(x,v):=(x,\psi_{s,t}^x (v)), 
\end{align*}
where the sum is taken on $(\mu, \nu)$ such that $1 \leq \vert \mu \vert \leq \vert \alpha \vert$ and $\nu_k \in \N \setminus{\lbrace 0 \rbrace }$ with
\begin{align*}
\forall 1 \leq j \leq 2d, \  \sum_{1 \leq \vert \beta \vert \leq \vert \alpha \vert}\nu_{\beta_j} =\mu_j,  \ \ {\text{and}}  \ \ \sum_{\substack{1 \leq \vert \beta \vert \leq \vert \alpha \vert \\ 1 \leq j \leq 2d}}\nu_{\beta_j} \beta= \alpha.
\end{align*}
We proceed as in the case $\vert \alpha \vert=1$. We isolate the terms with multi-indices $\mu$ such that $\vert \mu \vert=1$ (giving associated $\nu_{\beta_j}=1$ for all $1 \leq j \leq 2d$): the terms $\partial_{x,v}^{\alpha} \psi_{s,t}$ (given by $\nu_{\alpha_j}=1$) in the product are treated as above, while derivatives of order strictly less than $\vert \alpha \vert$ are bounded thanks to the induction hypothesis. This procedure is allowed provided that uniform bounds (in time) of the same type for $\Vert \widetilde{\mathrm{X}}^{s,t} \Vert_{\mathrm{W}^{k,\infty}_{x,v}}$ ($k \leq \vert \alpha \vert$) hold true. 

Such bounds are obtained by performing the same  induction at the level of $\mathrm{Y}^{s;t}$ first (using the same principle as before with \eqref{formula:Y})
and then for $\widetilde{\mathrm{X}}^{s,t}$ (arguing as before with \eqref{formula:Xtilde:int}). 

Concerning the estimate \eqref{bound:perturbId2}, we have by \eqref{eq:Psi:implicit}
\begin{align*}
\partial_s \psi_{s,t}^x(v)=-\partial_s \widetilde{\mathrm{X}}^{s,t}(x,\psi_{s,t}^x (v)) -\nabla_v \widetilde{\mathrm{X}}^{s,t}(x,\psi_{s,t}^x (v))\partial_s \psi_{s,t}^x (v),
\end{align*}
and by \eqref{formula:Xtilde:int},
\begin{align*}
\partial_s \widetilde{\mathrm{X}}^{s;t}(x,v)=-\frac{e^{t-s}}{1-e^{t-s}}\widetilde{\mathrm{X}}^{s;t}(x,v)-\frac{1}{1-e^{t-s}}\int_s^t e^{\tau-s}\mathrm{F} \pr{\tau, x+(1-e^{t-\tau})v+\mathrm{Y}^{\tau;t}(x,v)} \, \mathrm{d}\tau.
\end{align*}
Since
\begin{align*}
\Vert  \widetilde{\mathrm{X}}^{s;t} \Vert_{\Ld^{\infty}_{x,v}} \leq T\Lambda(T,\mathrm{R}),
\end{align*}
and $\frac{Te^T}{e^T-1}$ is bounded in a neighborhood of $0$, we obtain an estimate on $\Vert \partial_s \widetilde{\mathrm{X}}^{s;t}(x,v) \Vert_{\Ld^{\infty}_{x,v}}$ and then on $\Vert \partial_s \psi_{s,t}^x(v)\Vert_{\Ld^{\infty}_{x,v}}$ as before. Using the same induction procedure as for \eqref{bound:perturbId}, we finally obtain \eqref{bound:perturbId2}.
\end{proof}

\begin{rem}\label{Rmk:BoundsXV}
From the $\W^{k, \infty}_{x,v}$-bounds we have obtained on $\mathrm{Y}^{s,t}$ along the proof, and because
\begin{align*}
\mathrm{Y}^{s;t}(x,v)=\mathrm{X}^{s;t}(x,v)-x-(1-e^{t-s})v,
\end{align*}
we can infer that
\begin{align}\label{bound:Wkinfty-X}
\underset{s,t \in [0,T]}{\sup} \, \left\Vert \partial_{x,v}^{\beta} \left(\X_{s,t}(x,v)-x-(1-e^{t-s})v\right) \right\Vert_{\Ld^{\infty}(\T^d_x \times \R^d_v)} &\leq \varphi(T) \Lambda(T,R), \ \ \vert \beta \vert \leq k.
\end{align}
Likewise, by considering 
\begin{align*}
\mathrm{W}^{s,t}(x,v)=\V^{s,t}(x,v)-e^{t-s}v,
\end{align*}
one can obtain the estimate
\begin{align}\label{bound:Wkinfty-V}
\underset{s,t \in [0,T]}{\sup} \, \left\Vert \partial_{x,v}^{\beta} \left(\V_{s,t}(x,v)-e^{t-s}v \right) \right\Vert_{\Ld^{\infty}(\T^d_x \times \R^d_v)} &\leq \varphi(T) \Lambda(T,R), \ \ \vert \beta \vert \leq k.
\end{align}
\end{rem}

\medskip

Let us conclude this section by showing that in Lemma \ref{straight:velocity}, one can consider 
\begin{align*}
(T, \mathrm{F})=(T_{\eps},E_{\reg,\eps}^{u_\eps,\varrho_\eps}),
\end{align*}
for a given $\eps>0$, where $(T_{\eps},E_{\reg,\eps}^{u_\eps,\varrho_\eps})$ have been defined in the set-up of the bootstrap argument of Section \ref{Subsection:EstimateREG}. Let us prove that the assumptions of Lemma \ref{straight:velocity} indeed hold with $k=m-2-d/2$. By the estimate \eqref{bound:Sobo2:E} from Lemma \ref{LM:Estim-ForceField}, we have for all $t \in (0,T_{\eps})$ and $\ell <m-1-d/2$
\begin{align*}
\LRVert{E_{\reg,\eps}^{u_\eps,\varrho_\eps}(t)}_{\W^{\ell, \infty}_{x}} &\lesssim \LRVert{E_{\reg,\eps}^{u_\eps,\varrho_\eps}(t)}_{H^{m-1}} \\
&\lesssim \Vert u_\eps(t) \Vert_{\H^{m-1}} 
+ \Lambda\left(\Vert \varrho_\eps(t) \Vert_{\H^{m-2}}  \right)   \Vert \varrho_\eps (t) \Vert_{\H^{m}}.
\end{align*}
Appealing to Lemma \ref{LM:rho-pointwise-Hm-2}, and using $\mathcal{N}_{m,r}(f_\eps,\varrho_\eps,u_\eps,T_\eps)\leq R$ for all $T \leq T_{\eps}$, we deduce
\begin{align*}
\Vert E_{\reg,\eps}^{u_\eps,\varrho_\eps} \Vert_{\Ld^2(0,T; \W^{k, \infty}(\T^d))} \leq \Lambda(T,R),
\end{align*}
for $k=m-2-d/2$.

\section{Averaging operators related to the dynamics with friction}\label{Section:Averag-lemma}
For any smooth vector field $G(t,s,x,v) \in \R^d$, we consider the following integral operator $\mathrm{K}^{\mathrm{free}}_G$ acting on functions $H(t,x)$:
\begin{align*}
\mathrm{K}_G^{\mathrm{free}}[H](t,x)&:=\int_0^t \int_{\R^d}   [\nabla_x H](s,x-(t-s)v)\cdot G(t,s,x,v) \, \mathrm{d}v \, \mathrm{d}s.
\end{align*}
This operator, featuring an apparent loss of derivative in space, was systematically studied  in \cite{HKR}. It was proven in \cite[Proposition 5.1 and Remark 5.1]{HKR} that this loss is only apparent, provided that the kernel is sufficiently smooth and decaying in velocity. The statement goes as follows.

\begin{propo}\label{propo:Averag-ORiGINAL}
Let $T>0$. If $p>1+d$ and $\sigma >d/2$ then for all $H  \in \Ld^2(0,T;\Ld^2(\T^d))$
\begin{align*}
\Big\Vert \mathrm{K}_G^{\mathrm{free}}[H]\Big\Vert_{\Ld^2(0,T;\Ld^2(\T^d))} \lesssim \underset{0 \leq s,t \leq T}{\sup} \Vert G(t,s) \Vert_{\mathcal{H}^p_{\sigma}} \LRVert{H}_{\Ld^2(0,T;\Ld^2(\T^d))} .
\end{align*}
\end{propo}
As already noted in \cite{HKR}, this smoothing estimate is reminiscent of (but different from) the so-called kinetic averaging lemmas. Namely, Proposition~\ref{propo:Averag-ORiGINAL} provides the gain of one full derivative. 

Averaging lemmas are well-known to provide powerful regularity and compactness results in the study of kinetic equations. Loosely speaking, moments in velocity of the solutions appear to gain some regularity compared to the solutions themselves, which are just transported along the flow of the equation. We refer to \cite{GPSavg,Agoshavg,GLPSavg} for the introduction of the averaging lemmas, and to \cite{DPLavg, PSavg, GSRavg, JVavg, ASRavg,AMavg, JLT, ALavg} for several extensions of such results. 

A thorough comparison between standard kinetic averaging lemmas and the estimate from Proposition \ref{propo:Averag-ORiGINAL} can be found in \cite{HK-highvlasov}. We finally refer to \cite{HKNR} for the use of Proposition \ref{propo:Averag-ORiGINAL} for a sligthly different purpose, as well as to \cite{Chaub} for an extension of this proposition.

\medskip

In this section, we prove crucial smoothing estimates adaptated to kinetic equations with friction, in the spirit of Proposition \ref{propo:Averag-ORiGINAL}. First, we define the corresponding integral operator.

\begin{defi}
For any smooth vector field $G(t,s,x,v)$, we define the following integral operators acting on functions $H(t,x)$ by
\begin{align*}
\mathrm{K}_G^{\mathrm{fric}}[H](t,x)&:=\int_0^t \int_{\R^d}   [\nabla_x H](s,x+(1-e^{t-s})v)\cdot G(t,s,x,v) \, \mathrm{d}v \, \mathrm{d}s, 
\end{align*}
where $(t,x) \in \R^+ \times \T^d$.
\end{defi}
Assuming that the kernel $G$ is sufficiently smooth and decaying in velocity, we will prove several continuity and regularization estimates for $\mathrm{K}_G^{\mathrm{free}}$ and $\mathrm{K}_G^{\mathrm{fric}}$ (see Propositions \ref{propo:AveragStandard}--\ref{propo:AveragReg}--\ref{propo:AveragDiff} below).

In what follows, $\mathcal{F}_{x,v}$ will refer to the Fourier transform on $\T^d \times \R^d$ defined as
\begin{align*}
\mathcal{F}_{x,v}h(k,\xi)&=\int_{\T^d \times \R^d } e^{-i (k  \cdot x+v \cdot \xi )} h(x,v)   \, \mathrm{d}x \, \mathrm{d}v, \ \ (k, \xi) \in \Z^d \times \R^d.
\end{align*}

Our first result is the following.

\begin{propo}\label{propo:AveragGENERAL}
There exists $C>0$ such that the followings holds. Suppose that $G_{[q]}(t,s,x,v)$ is a kernel of the form
\begin{align*}
G_{[q]}(t,s,x,v)=(t-s)^q \mathcal{G}(t,s,x,v),
\end{align*}
with $q \in \N$. 
 For every $T>0$ satisfying
\begin{align*}
\Vert \mathcal{G} \Vert_{T,s_1,s_2}:=\underset{0\leq t \leq T}{\sup} \left(  \sum_{m \in \Z^d}\underset{0\leq s \leq t}{\sup} \, \underset{\xi \in \R^d }{\sup} \,  \big\lbrace (1+\vert m \vert)^{s_2}(1+\vert \xi \vert)^{s_1} \vert (\mathcal{F}_{x,v} \mathcal{G})(t,s,m,\xi) \vert \big\rbrace^2 \right)^{\frac{1}{2}} < + \infty,
\end{align*}
for $s_1>1+2q$ and $s_2>d/2+2q$, 
\begin{align*}
\Big\Vert \nabla^q_x \mathrm{K}_{G_{[q]}}^{\mathrm{free}}[H] \Big\Vert_{\Ld^2(0,T; \Ld^2(\T^d))} +\Big\Vert \nabla^q_x \mathrm{K}_{G_{[q]}}^{\mathrm{fric}}[H] \Big\Vert_{\Ld^2(0,T; \Ld^2(\T^d))} &\leq C \Vert \mathcal{G} \Vert_{T,s_1,s_2} \Vert H \Vert_{\Ld^2(0,T; \Ld^2(\T^d))}.
\end{align*}
\end{propo}
\begin{proof}
 We only perform the proof in the case of $ \mathrm{K}_{G_{[q]}}^{\mathrm{fric}}$ (the proof is similar for $ \mathrm{K}_{G_{[q]}}^{\mathrm{free}}$). Writing for all $t \geq 0$
\begin{align*}
H(t,x):=\sum_{k \in \Z^d} \widehat{H}_k(t) e^{i k \cdot x} \ \ \text{in} \ \Ld^2(\T^d),
\end{align*}
we have
\begin{align*}
 \mathrm{K}_{G_{[q]}}^{\mathrm{fric}}[H](t,x)&=\int_0^t  \sum_{k \in \Z^d} \widehat{H}_k(s) e^{i k \cdot x} (ik) \cdot  \int_{\R^d} e^{-ik \cdot (e^{t-s}-1)v}  G_{[q]}(t,s,x,v) \, \mathrm{d}v \, \mathrm{d}s \\
 &=\int_0^t  \sum_{k \in \Z^d} \widehat{H}_k(s) e^{i k \cdot x} (ik) \cdot (\mathcal{F}_v G_{[q]})(t,s,x,k(e^{t-s}-1)) \, \mathrm{d}s.
\end{align*}
We now expand $G_{[q]}$ in Fourier series along the $x$ variable so that
\begin{align*}
\mathrm{K}_{G_{[q]}}^{\mathrm{fric}}[H](t,x)&= \sum_{k \in \Z^d}  \sum_{l \in \Z^d}e^{i (k+\ell) \cdot x}   \int_0^t \widehat{H}_k(s)  (ik) \cdot (\mathcal{F}_{x,v} G_{[q]})(t,s,\ell,k(e^{t-s}-1)) \, \mathrm{d}s \\
 &=\sum_{\ell' \in \Z^d}e^{i \ell' \cdot x}  \left\lbrace \sum_{k \in \Z^d}  \int_0^t \widehat{H}_k(s)  (ik) \cdot (\mathcal{F}_{x,v} G_{[q]})(t,s,\ell'-k,k(e^{t-s}-1)) \, \mathrm{d}s \right\rbrace,
\end{align*}
and then
\begin{align*}
 \nabla_x^q \mathrm{K}^{\mathrm{fric}}_{(t-s)^q \mathcal{G}}[H](t,x)=\sum_{\ell \in \Z^d}e^{i \ell \cdot x}  \left\lbrace (i \ell)^q \sum_{k \in \Z^d}  \int_0^t (t-s)^q \widehat{H}_k(s)  (ik) \cdot (\mathcal{F}_{x,v} \mathcal{G})(t,s,\ell-k,k(e^{t-s}-1)) \, \mathrm{d}s \right\rbrace.
\end{align*}
By the Parseval equality and the Cauchy-Schwarz inequality (in frequency and time), we get
\begin{align*}
\Vert \nabla_x^q \mathrm{K}^{\mathrm{fric}}_{(t-s)\mathcal{G}}[F](t) \Vert_{\Ld^2(\T^d)}^2&= \sum_{\ell \in \Z^d} \vert \ell \vert^{2q} \left\vert \sum_{k \in \Z^d}  \int_0^t (t-s)^q \widehat{H}_k(s)  (ik) \cdot (\mathcal{F}_{x,v} \mathcal{G})(t,s,\ell-k,k(e^{t-s}-1)) \, \mathrm{d}s \right\vert^2 \\
& \leq \sum_{\ell \in \Z^d}  \vert \ell \vert^{2q} \left(\sum_{k \in \Z^d}  \int_0^t (t-s)^{2q}  \vert \widehat{H}_k(s) \vert^2 \vert k \cdot (\mathcal{F}_{x,v} \mathcal{G})(t,s,\ell-k,k(e^{t-s}-1)) \vert    \, \dd s \right) \\
& \qquad     \times \left( \sum_{k \in \Z^d}  \int_0^t \vert k \cdot (\mathcal{F}_{x,v} \mathcal{G})(t,s,\ell-k,k(e^{t-s}-1)) \vert  \dd s \right).
\end{align*}
Integrating in time yields 
\begin{align*}
&\Vert \nabla_x^q \mathrm{K}^{\mathrm{fric}}_{(t-s)\mathcal{G}}[F] \Vert_{\Ld^2(0,T ; \Ld^2( \times \T^d))}^2 \\
&\leq \sum_{\ell \in \Z^d} \vert \ell \vert^{2q}  \int_0^T  \int_0^t  \sum_{k \in \Z^d} (t-s)^{2q}  \vert \widehat{H}_k(s) \vert^2 \vert k \cdot (\mathcal{F}_{x,v} \mathcal{G})(t,s,\ell-k,k(e^{t-s}-1)) \vert    \, \dd s \dd t \\
& \quad \times \underset{\ell \in \Z^d}{\sup} \underset{t \in (0,T)}{\sup}  \sum_{k \in \Z^d}  \int_0^t \vert k \cdot (\mathcal{F}_{x,v} \mathcal{G})(t,s,\ell-k,k(e^{t-s}-1)) \vert  \dd s \\
&=(\mathrm{I}) \times (\mathrm{II}).
\end{align*}
A first step is to note that
\begin{align*}
\sum_{k \in \Z^d}  \int_0^t &\vert k \cdot (\mathcal{F}_{x,v} \mathcal{G})(t,s,\ell-k,k(e^{t-s}-1)) \vert  \dd s  \\ 
&\leq \sum_{k \in \Z^d} \underset{\substack{s \in (0,t)\\ \xi \in \R^d}}{\sup} (1+\vert \xi \vert)^{\alpha} \vert \mathcal{F}_{x,v}\mathcal{G}(t,s,\ell-k,\xi)) \vert \int_0^t \frac{\vert k \vert}{(1+\vert k \vert(e^{t-s}-1))^{\alpha}} \dd s \\ 
&\leq \sum_{k \in \Z^d} \underset{\substack{s \in (0,t)\\ \xi \in \R^d}}{\sup} (1+\vert \xi \vert)^{\alpha} \vert \mathcal{F}_{x,v}\mathcal{G}(t,s,\ell-k,\xi)) \vert \int_0^t \frac{\vert k \vert}{(1+\vert k \vert(t-s))^{\alpha}} \dd s,
\end{align*}
and that the change of variable $\tau=\vert k \vert (t-s)$ in the last integral yields
\begin{align*}
\int_0^t \frac{\vert k \vert}{(1+\vert k \vert(t-s))^{\alpha}} \dd s \leq \int_0^{+\infty} \frac{\dd \tau}{(1+\tau)^{\alpha}}<+\infty,
\end{align*}
provided that $\alpha >1$. With this observation, we can treat the term $(\mathrm{II})$ and obtain by the Cauchy-Schwarz inequality in $k$
\begin{align*}
(\mathrm{II}) \leq \underset{\ell \in \Z^d}{\sup} \underset{t \in (0,T)}{\sup} \sum_{k \in \Z^d} \underset{\substack{s \in (0,t)\\ \xi \in \R^d}}{\sup} (1+\vert \xi \vert)^{s_1} \vert \mathcal{F}_{x,v} \mathcal{G}(t,s,\ell-k,\xi)) \vert \leq \Vert \mathcal{G} \Vert_{T,s_1,s_2},
\end{align*}
for $s_1>1$ and $s_2>d/2$. For the term $(\mathrm{I})$, we use Fubini-Tonelli theorem and get
\begin{align*}
(\mathrm{I})&=\sum_{\ell \in \Z^d} \vert \ell \vert^{2q} \int_0^T  \int_0^t  \sum_{k \in \Z^d} (t-s)^{2q}  \vert \widehat{H}_k(s) \vert^2 \vert k \vert \vert (\mathcal{F}_{x,v} \mathcal{G})(t,s,\ell-k,k(e^{t-s}-1))\vert       \, \dd s \, \dd t \\
&= \int_0^T \sum_{k \in \Z^d}   \vert \widehat{H}_k(s) \vert^2 \int_s^T \sum_{\ell \in \Z^d} (t-s)^{2q} \vert \ell \vert^{2q} \vert k \vert  \vert (\mathcal{F}_{x,v} \mathcal{G})(t,s,\ell-k,k(e^{t-s}-1))\vert      \, \dd t \, \dd s \\
&\leq \Vert H \Vert_{\Ld^2(0,T ; \Ld^2(\T^d))}^2  \underset{k \in \Z^d}{\sup} \underset{0\leq s \leq T}{\sup}\int_s^T \sum_{\ell \in \Z^d} (t-s)^{2q}  \vert \ell \vert^{2q} \vert k \vert  \vert (\mathcal{F}_{x,v} \mathcal{G})(t,s,\ell-k,k(e^{t-s}-1)) \vert     \, \dd t.
\end{align*}
Note that the last expression can be taken into account for $k \in \dot \Z^d$ only (indeed, the term corresponding to $k=0$ vanishes in $(\mathrm{I})$). We then have
\begin{align*}
&\underset{k \in \dot \Z^d}{\sup} \underset{0 \leq s \leq T}{\sup}\int_s^T (t-s)^{2q}  \vert k \vert  \sum_{\ell \in \Z^d}\vert \ell \vert^{2q} \vert (\mathcal{F}_{x,v} \mathcal{G})(t,s,\ell-k,k(e^{t-s}-1)) \vert         \, \dd t \\
& \leq \underset{k \in \dot \Z^d}{\sup} \underset{0\leq s \leq T}{\sup}\int_s^T \frac{(t-s)^{2q}   \vert k \vert }{(1+\vert k \vert (e^{t-s}-1))^{\alpha_1}} \sum_{\ell \in \Z^d}\vert \ell \vert^{2q}  (1+\vert k \vert (e^{t-s}-1) )^{\alpha_1})\vert (\mathcal{F}_{x,v} \mathcal{G})(t,s,\ell-k,k(e^{t-s}-1) )\vert    \, \dd t \\
&\leq  \underset{k \in \dot \Z^d}{\sup} \underset{0\leq s \leq T}{\sup}\int_s^T \frac{(t-s)^{2q}   \vert k \vert}{(1+\vert k \vert(e^{t-s}-1))^{\alpha_1}} \, \dd t  \\
& \qquad \qquad \qquad \qquad  \qquad \times  \underset{0\leq s \leq T} {\sup} \underset{s\leq t \leq T}{\sup}  \sum_{\ell \in \Z^d} \vert \ell \vert^{2q} \, \underset{ \xi }{\sup} \,  (1+\vert \xi \vert^{\alpha_1}) \vert (\mathcal{F}_{x,v} \mathcal{G})(t,s,\ell-k,\xi) \vert  \\
&\leq  \underset{k \in \dot \Z^d}{\sup} \, \underset{0\leq s \leq T}{\sup}\int_s^T \frac{(t-s)^{2q}   \vert k \vert}{(1+\vert k \vert(t-s))^{\alpha_1}} \, \dd t  \\
& \qquad \qquad \qquad \qquad  \qquad \times  \underset{0\leq s \leq T} {\sup} \underset{s\leq t \leq T}{\sup}  \sum_{\ell \in \Z^d} \vert \ell+k \vert^{2q} \, \underset{ \xi }{\sup} \,  (1+\vert \xi \vert^{\alpha_1}) \vert (\mathcal{F}_{x,v} \mathcal{G})(t,s,\ell,\xi) \vert,
\end{align*}
therefore
\begin{align*}
\underset{k \in \dot \Z^d}{\sup}& \underset{0 \leq s \leq T}{\sup}\int_s^T (t-s)^{2q}  \vert k \vert  \sum_{\ell \in \Z^d}\vert \ell \vert^{2q} \vert (\mathcal{F}_{x,v} \mathcal{G})(t,s,\ell-k,k(e^{t-s}-1)) \vert         \, \dd t\\
&\leq  \underset{k \in \dot \Z^d}{\sup} \, \underset{0\leq s \leq T}{\sup}\int_s^T \frac{(t-s)^{2q}  \vert k \vert}{(1+\vert k \vert(t-s))^{\alpha_1}} \, \dd t  \\
& \qquad \qquad \qquad \qquad  \qquad \times  \underset{0\leq s \leq T} {\sup} \underset{s\leq t \leq T}{\sup}  \sum_{\ell \in \Z^d} \vert \ell \vert^{2q} \, \underset{ \xi }{\sup} \,  (1+\vert \xi \vert^{\alpha_1}) \vert (\mathcal{F}_{x,v} \mathcal{G})(t,s,\ell,\xi) \vert  \\
& \quad + \underset{k \in \dot \Z^d}{\sup} \, \underset{0\leq s \leq T}{\sup}\int_s^T \frac{(t-s)^{2q}   \vert k \vert^{1+2q}}{(1+\vert k \vert(t-s))^{\alpha_1}} \, \dd t  \\
& \qquad \qquad \qquad \qquad  \qquad \times  \underset{0\leq s \leq T} {\sup} \underset{s\leq t \leq T}{\sup}  \sum_{\ell \in \Z^d} \, \underset{ \xi }{\sup} \,  (1+\vert \xi \vert^{\alpha_1}) \vert (\mathcal{F}_{x,v} \mathcal{G})(t,s,\ell,\xi) \vert  \\
&=\mathrm{S}_1 + \mathrm{S}_2.
\end{align*}
Let us treat these two terms separately. 

$\bullet$ For $\mathrm{S}_1$, we have by the Cauchy-Schwarz inequality
\begin{align*}
\mathrm{S}_1 &\leq \underset{k \in \dot \Z^d}{\sup} \underset{0\leq s \leq T}{\sup}\int_s^T \frac{(t-s)^{2q}   \vert k \vert}{(1+\vert k \vert(t-s))^{\alpha_1}} \, \dd t  \\
& \qquad \qquad \qquad \qquad \qquad \times  \underset{0\leq s \leq T} {\sup} \underset{s\leq t \leq T}{\sup} \left(  \sum_{\ell \in \Z^d} \underset{\xi }{\sup} \,  \left\lbrace (1+\vert \ell \vert^{2q+\alpha_2})(1+\vert \xi \vert^{\alpha_1}) \vert (\mathcal{F}_{x,v} \mathcal{G})(t,s,\ell,\xi) \vert \right\rbrace^2 \right)^{\frac{1}{2}},
\end{align*}
if $\alpha_2 >d/2$. For the integral term, we write for $\alpha_1>3$.
\begin{align*}
\underset{k \in \dot \Z^d}{\sup} \underset{0\leq s \leq T}{\sup}\int_s^T \frac{(t-s)^{2q}   \vert k \vert}{(1+\vert k \vert(t-s))^{\alpha_1}} \, \dd t &=\underset{k \in \dot \Z^d}{\sup}  \frac{1}{\vert k \vert^{2q}}\underset{0\leq s \leq T}{\sup}\int_0^{\vert k \vert(T-s)} \frac{\tau^{2q} }{(1+\tau)^{\alpha_1}} \, \dd \tau \\
 & \leq  \underset{k \in  \dot \Z^d}{\sup}  \frac{1}{\vert k \vert^{2q}} \int_{0}^{+\infty} \frac{\tau^{2q}}{(1+\tau)^{\alpha_1}} \, \dd \tau, \\
\end{align*}
which is a finite constant independent of $k$ and $T$ (since $\alpha_1>1+2q$) therefore 
\begin{align*}
\mathrm{S}_1 \lesssim \underset{0\leq t \leq T}{\sup} \left(  \sum_{\ell \in \Z^d}\underset{0\leq s \leq t}{\sup} \ \underset{\xi \in \R^d }{\sup} \,  \left\lbrace (1+\vert \ell \vert^{2q+\alpha_2})(1+\vert \xi \vert^{\alpha_1}) \vert (\mathcal{F}_{x,v} \mathcal{G})(t,s,\ell,\xi) \vert \right\rbrace^2 \right)^{\frac{1}{2}}.
\end{align*}

$\bullet$ For $\mathrm{S}_2$ , we have for $\alpha_2 >d/2$
\begin{align*}
\mathrm{S}_2 &\leq \underset{k \in \dot \Z^d}{\sup} \underset{0\leq s \leq T}{\sup}\int_s^T \frac{(t-s)^{2q}  \vert k \vert^{1+2q} }{(1+\vert k \vert(t-s))^{\alpha_1}} \, \dd t  \\
& \qquad \qquad \qquad \qquad \qquad \times  \underset{0\leq s \leq T} {\sup} \underset{s\leq t \leq T}{\sup} \left(  \sum_{\ell \in \Z^d} \underset{\xi }{\sup} \,  \left\lbrace (1+\vert m \vert^{\alpha_2})(1+\vert \xi \vert^{\alpha_1}) \vert (\mathcal{F}_{x,v} \mathcal{G})(t,s,\ell,\xi) \vert \right\rbrace^2 \right)^{\frac{1}{2}}.
\end{align*}
The integral term now reads for $\alpha_1>3$
\begin{align*}
\underset{k \in \dot \Z^d}{\sup} \underset{0\leq s \leq T}{\sup}\int_s^T \frac{(t-s)^{2q}   \vert k \vert^{1+2q} }{(1+\vert k \vert(t-s))^{\alpha_1}} \, \dd t &=\underset{k \in \dot \Z^d}{\sup} \underset{0\leq s \leq T}{\sup}\int_0^{\vert k \vert(T-s)} \frac{\tau^{2q} }{(1+\tau)^{\alpha_1}} \, \dd \tau
\leq   \int_{0}^{+\infty} \frac{\tau^{2q}}{(1+\tau)^{\alpha_1}} \, \dd \tau,
\end{align*}
therefore 
\begin{align*}
\mathrm{S}_1 \lesssim \underset{0\leq t \leq T}{\sup} \left(  \sum_{\ell \in \Z^d}\underset{0\leq s \leq t}{\sup} \ \underset{\xi \in \R^d }{\sup} \,  \left\lbrace (1+\vert \ell \vert^{\alpha_2})(1+\vert \xi \vert^{\alpha_1}) \vert (\mathcal{F}_{x,v} \mathcal{G})(t,s,\ell,\xi) \vert \right\rbrace^2 \right)^{\frac{1}{2}}.
\end{align*}
We have thus proven that
\begin{align*}
(\mathrm{I}) \lesssim  \Vert H \Vert_{\Ld^2(0,T ; \Ld^2(\T^d))}^2 \left( \Vert \mathcal{G} \Vert_{T,s_1,s_2}+\Vert \mathcal{G} \Vert_{T,s_1,s_2+2q} \right),
\end{align*}
 for $s_1>1+2q$ and $s_2>d/2+2q$.All in all, we get
\begin{align*}
\Vert \nabla_x^q \mathrm{K}_{G_{[q]}}^{\mathrm{fric}}[H] \Vert_{\Ld^2(0,T; \Ld^2(\T^d))}\leq C \Vert \mathcal{G} \Vert_{T,s_1,s_2} \Vert H \Vert_{\Ld^2(0,T; \Ld^2(\T^d))},
\end{align*}
which ends the proof.
%
\end{proof}
We then deduce the two following propositions. The first one is a direct consequence of Proposition \ref{propo:AveragGENERAL} with $q=0$, and states the continuity of 
$ \mathrm{K}_G^{\mathrm{free}}$ and $ \mathrm{K}_G^{\mathrm{fric}}$ on $\Ld^2_T \Ld^2_x$.

\begin{propo}\label{propo:AveragStandard}
There exists $C>0$ such that for every $T>0$, if $p>1+d$ and $\sigma >d/2$ then for all $H  \in \Ld^2(0,T;\Ld^2(\T^d))$
\begin{align*}
\Big\Vert \mathrm{K}_G^{\mathrm{free}}[H] \Big\Vert_{\Ld^2(0,T; \Ld^2(\T^d))} +\Big\Vert \mathrm{K}_G^{\mathrm{fric}}[H] \Big\Vert_{\Ld^2(0,T; \Ld^2(\T^d))} &\leq C \underset{0 \leq s,t \leq T}{\sup} \Vert G(t,s) \Vert_{\mathcal{H}^p_{\sigma}} \Vert H \Vert_{\Ld^2(0,T; \Ld^2(\T^d))}.
\end{align*}
\end{propo}
\begin{proof}
By Proposition \ref{propo:AveragGENERAL} with $q=0$, we have
\begin{align*}
\Big\Vert \mathrm{K}_{G}^{\mathrm{free}}[H] \Big\Vert_{\Ld^2(0,T; \Ld^2(\T^d))} +\Big\Vert  \mathrm{K}_{G}^{\mathrm{fric}}[H] \Big\Vert_{\Ld^2(0,T; \Ld^2(\T^d))} &\leq C \Vert G \Vert_{T,s_1,s_2} \Vert H \Vert_{\Ld^2(0,T; \Ld^2(\T^d))}.
\end{align*}
for any $s_1>1$ and $s_2>d/2$. Now appealing to \cite[Remark 3]{HKR}, one can prove that for all $p> 1+d$ and $\sigma>d/2$, there exist $s_2>d/2$ and $s_1> 1$ such that
\begin{align*}
\Vert G \Vert_{T,s_1,s_2} &\lesssim \underset{0 \leq s,t \leq T}{\sup} \Vert G(t,s) \Vert_{\mathcal{H}^p_{\sigma}}, 
\end{align*}
hence the result.
\end{proof}

When the kernel $G$ vanishes along the diagonal in time $\lbrace t=s \rbrace$, Proposition \ref{propo:AveragGENERAL} with $q=1$ leads to the following additional regularizing effect of the operators $\mathrm{K}_G$ (as already observed in \cite{HKR2}). Loosely speaking, the operators $\mathrm{K}_G^{\mathrm{free}}$ and $\mathrm{K}_G^{\mathrm{fric}}$ are bounded from $\Ld^2_T \Ld^2_x$ to $\Ld^2_T \dot{ \mathrm{H}}^1_x$ in this case.

\begin{propo}\label{propo:AveragReg}
There exists $C>0$ such that if the kernel $G$ satisfies
\begin{align*}
G(t,t,x,v)=0,
\end{align*}
the following holds. Let $T>0$. If $p>7+d$ and $\sigma >d/2$ then for all $H \in \Ld^2(0,T;\Ld^2(\T^d))$
\begin{align*}
\Big\Vert  \mathrm{K}_G^{\mathrm{free}}[H] \Big\Vert_{\Ld^2(0,T; \H^1(\T^d))}+\Big\Vert   \mathrm{K}_G^{\mathrm{fric}}[H] \Big\Vert_{\Ld^2(0,T; \H^1(\T^d))} \leq C (1+T)\underset{0 \leq s,t \leq T}{\sup} \Vert \partial_s G(t,s) \Vert_{\mathcal{H}^p_{\sigma}} \Vert H \Vert_{\Ld^2(0,T; \Ld^2(\T^d))}.
\end{align*}
\end{propo}

\begin{proof}
Since $G(t,t,x,v)=0$, the Taylor's formula shows that
\begin{align*}
G(t,s,x,v)=(t-s)\widetilde{G}(t,s,x,v), \ \ \widetilde{G}(t,s,x,v):=-\int_0^1\partial_s G(t,t+\tau(s-t),x,v) \, \mathrm{d}\tau.
\end{align*}
By Proposition \ref{propo:AveragGENERAL} with $q=1$, we get for $s_1>3$ and $s_2>d/2+2$
\begin{align*}
\Vert \nabla_x \mathrm{K}_G^{\mathrm{free}}[H] \Vert_{\Ld^2(0,T; \Ld^2(\T^d))}+\Vert \nabla_x  \mathrm{K}_G^{\mathrm{fric}}[H] \Vert_{\Ld^2(0,T; \Ld^2(\T^d))} &\leq C \Vert \widetilde{G} \Vert_{T,s_1,s_2} \Vert H \Vert_{\Ld^2(0,T; \Ld^2(\T^d))}.
\end{align*}

To conclude, we observe that
\begin{align*}
\big\lbrace (1+\vert m \vert)^{s_2}(1+\vert \xi \vert)^{s_1} & \vert (\mathcal{F}_{x,v} \widetilde{G})(t,s,m,\xi) \vert \big\rbrace^2 \\
& \qquad  \qquad \leq \left( (1+\vert m \vert)^{s_2}(1+\vert \xi \vert)^{s_1} \int_0^1 \vert \mathcal{F}_{x,v}(\partial_s G)(t,t+\tau(s-t),m, \xi) \vert \, \mathrm{d}\tau \right)^2 \\
& \qquad  \qquad  \leq \int_0^1 \big\lbrace (1+\vert m \vert)^{s_2}(1+\vert \xi \vert)^{s_1} \vert \mathcal{F}_{x,v}(\partial_s G)(t,t+\tau(s-t),m, \xi) \vert \big\rbrace^2 \, \mathrm{d}\tau,
\end{align*}
therefore
\begin{align*}
\Vert \nabla_x \mathrm{K}_G^{\mathrm{free}}[H] \Vert_{\Ld^2(0,T; \Ld^2(\T^d))}+\Vert \nabla_x  \mathrm{K}_G^{\mathrm{fric}}[H] \Vert_{\Ld^2(0,T; \Ld^2(\T^d))} \leq C \Vert \partial_s G \Vert_{T,s_1,s_2} \Vert H \Vert_{\Ld^2(0,T; \Ld^2(\T^d))}.
\end{align*}
Now appealing to \cite[Remark 3]{HKR}, one can prove that for all $p>2\ell+s + 1+d$ and $\sigma>d/2$ (with $\ell,s \in \R^+$), there exist $s_2>\ell+ d/2$ and $s_1>s + 1$ such that
\begin{align*}
\Vert G \Vert_{T,s_1,s_2} &\lesssim \underset{0 \leq s,t \leq T}{\sup} \Vert G(t,s) \Vert_{\mathcal{H}^p_{\sigma}}. 
\end{align*}
Since $G(t,t,x,v)=0$, we also have 
\begin{align*}
\Vert G \Vert_{T,s_1,s_2} &\lesssim T \underset{0 \leq s,t \leq T}{\sup} \Vert \partial_s G(t,s) \Vert_{\mathcal{H}^p_{\sigma}}.
\end{align*}
By Proposition \ref{propo:AveragStandard}, and taking $\ell=s=2$, we end up with the desired conclusion.
\end{proof}

\begin{rem}\label{rem:Variant-PropoAveragReg}
A variant of Proposition \ref{propo:AveragReg} holds in the following form: there exists $C>0$ such that for $p>7+d$ and $\sigma>d/2$, and
\begin{align*}
G(t,t,x,v)=0,
\end{align*}
we have for all $H \in \Ld^2(0,T; \dot \H^{-1}(\T^d))$
\begin{align*}
\Big\Vert \mathrm{K}_G^{\mathrm{free}}[H] \Big\Vert_{\Ld^2(0,T; \Ld^2(\T^d))}+\Big\Vert  \mathrm{K}_G^{\mathrm{fric}}[H] \Big\Vert_{\Ld^2(0,T; \Ld^2(\T^d))} &\leq C (1+T) \underset{0 \leq s,t \leq T}{\sup} \Vert \partial_s G(t,s) \Vert_{\mathcal{H}^p_{\sigma}} \Vert H \Vert_{\Ld^2(0,T; \dot \H^{-1}(\T^d))}.
\end{align*}
We don't detail the proof, which follows the same lines as the one of Proposition \ref{propo:AveragReg}.
\end{rem}

We finally investigate the smoothing properties of the difference operator $\mathrm{K}_G^{\mathrm{free}}-\mathrm{K}_G^{\mathrm{fric}}$. A somewhat surprising result is the fact that this operator gains one additional derivative. This is the content of the following proposition.

\begin{propo}\label{propo:AveragDiff}
There exists $C>0$ such that for every $T>0$, if $p>8+d$ and $\sigma>1+d/2$ then for all $H \in \Ld^2(0,T;\Ld^2(\T^d))$,
\begin{align*}
\Big\Vert \mathrm{K}_G^{\mathrm{free}}[H]-\mathrm{K}_G^{\mathrm{fric}}[H]  \Big\Vert_{\Ld^2(0,T;\H^1(\T^d))}  &\leq  C \varphi(T)\underset{0 \leq s,t \leq T}{\sup} \Vert G(t,s) \Vert_{\mathcal{H}^p_{\sigma}} \Vert H \Vert_{\Ld^2(0,T;\Ld^2(\T^d))}, 
\end{align*}
where $\varphi : \R^+ \rightarrow \R^+$ is a continuous nondecreasing function.
\end{propo}

\begin{proof} Following the computations performed in the proof of Proposition \ref{propo:AveragReg}, we have
\begin{multline*}
\mathrm{K}_G^{\mathrm{free}}[F](t,x)-\mathrm{K}_G^{\mathrm{fric}}[F](t,x)=\sum_{\ell \in \Z^d}e^{i \ell \cdot x}  \Bigg\lbrace \sum_{k \in \Z^d}  \int_0^t \widehat{F}_k(s)  (ik) \cdot \left[(\mathcal{F}_{x,v} G)(t,s,\ell-k,k(t-s)) \right. \\ \left. - (\mathcal{F}_{x,v} G)(t,s,\ell-k,k(e^{t-s}-1)) \right] \, \mathrm{d}s  \Bigg\rbrace,
\end{multline*}
therefore, if we set
\begin{align*}
\Theta(t,s,\ell,k):=(\mathcal{F}_{x,v} G)(t,s,\ell-k,k(t-s))- (\mathcal{F}_{x,v} G)(t,s,\ell-k,k(e^{t-s}-1)),
\end{align*}
we get 
\begin{align*}
\Vert \nabla_x \pr{\mathrm{K}_G^{\mathrm{free}}[H]-\mathrm{K}_G^{\mathrm{fric}}[H]}(t) \Vert_{\Ld^2(\T^d)}^2 
&\leq  \sum_{\ell \in \Z^d} \vert \ell \vert^2 \left(\sum_{k \in \Z^d}  \int_0^t \vert \widehat{H}_k(s) \vert  \vert k \vert  \vert \Theta(t,s,\ell,k)\vert  \, \mathrm{d}s \right)^2.
\end{align*}
We also have, by setting $\mathcal{G}^{t,s}_{\ell-k}=(\mathcal{F}_{x,v} G)(t,s,\ell-k,\bullet)$
\begin{align*}
\vert \Theta(t,s,\ell,k)\vert &=\vert \mathcal{G}^{t,s}_{\ell-k}(k(e^{t-s}-1))-\mathcal{G}^{t,s}_{\ell-k}(k(t-s))\vert \\
&=\vert \mathcal{G}^{t,s}_{\ell-k}(k(t-s)+k(t-s)^2\varphi(t-s))-\mathcal{G}^{t,s}_{\ell-k}(k(t-s)) \vert \\
& \leq \underset{\theta \in [0,1]}{\sup} \vert \nabla_\xi \mathcal{G}^{t,s}_{\ell-k}\left(\xi_{\theta}^{t,s}(k(t-s)) \right) \vert   \vert  k \vert (t-s)^2 \varphi(t-s),
\end{align*}
where $\varphi(z)=\sum_{i \geq 0} \frac{z^i}{(i+2)!}$ and $\xi_{\theta}^{t,s}(z)=z+ \theta z(t-s)\varphi(t-s)$. By continuity, there exists $\theta^{\star} \in [0,1]$ (which may depend on all the other variables) such that 
\begin{align*}
\vert \Theta(t,s,\ell,k)\vert \leq \vert \nabla_\xi \mathcal{G}^{t,s}_{\ell-k}\left(\xi_{\theta^{\star}}^{t,s}(k(t-s)) \right) \vert   \vert  k \vert (t-s)^2 \varphi(t-s).
\end{align*}
This yields
\begin{align*}
\Vert \nabla_x &\pr{\mathrm{K}_G^{\mathrm{free}}[H]-\mathrm{K}_G^{\mathrm{fric}}[H]}(t) \Vert_{\Ld^2(\T^d)}^2 \\
&\leq   \sum_{\ell \in \Z^d} \vert \ell \vert^2 \left(\sum_{k \in \Z^d}  \int_0^t \vert \widehat{H}_k(s) \vert \vert  k \vert^2 (t-s)^2 \varphi(t-s) \vert \nabla_\xi \mathcal{G}^{t,s}_{\ell-k}\left(\xi_{\theta^{\star}}^{t,s}(k(t-s)) \right) \vert    \, \mathrm{d}s \right)^2 \\
& \leq \sum_{\ell \in \Z^d} \vert \ell \vert^2  \left(\sum_{k \in \Z^d}  \int_0^t \vert \widehat{H}_k(s) \vert^2 \vert k \vert^2 (t-s)^3 \varphi(t-s)^2 \vert \nabla_\xi \mathcal{G}^{t,s}_{\ell-k}\left((\xi_{\theta^{\star}}^{t,s}(k(t-s)) \right) \vert      \, \dd s \right) \\
& \qquad     \times \left( \sum_{k \in \Z^d}  \int_0^t \vert k \vert ^2 (t-s) \vert \nabla_\xi \mathcal{G}^{t,s}_{\ell-k}\left(\xi_{\theta^{\star}}^{t,s}(k(t-s)) \right)\vert   \,  \dd s \right),
\end{align*}
thanks to the Cauchy-Schwarz inequality.
%
As in the proof of Proposition \ref{propo:AveragReg}, we obtain by integrating in time that
\begin{align*}
\Vert \nabla_x &\pr{\mathrm{K}_G^{\mathrm{free}}[H]-\mathrm{K}_G^{\mathrm{fric}}[H]}\Vert_{\Ld^2((0,T) \times \T^d)}^2 \\
&\leq \sum_{\ell \in \Z^d} \vert \ell \vert^2 \int_0^T  \int_0^t  \sum_{k \in \Z^d} \vert \widehat{H}_k(s) \vert^2 \vert k \vert^2 (t-s)^3 \varphi(t-s)^2 \vert \nabla_\xi \mathcal{G}^{t,s}_{\ell-k}\left(\xi_{\theta^{\star}}^{t,s}(k(t-s)) \right) \vert     \, \dd s \, \dd t \\
& \quad \quad  \quad  \times \underset{\ell \in \Z^d}{\sup} \underset{t \in (0,T)}{\sup}  \sum_{k \in \Z^d}  \int_0^t \vert k \vert ^2 (t-s) \vert \nabla_\xi \mathcal{G}^{t,s}_{\ell-k}\left(\xi_{\theta^{\star}}^{t,s}(k(t-s)) \right) \vert    \dd s \\
&=(\mathrm{A}) \times (\mathrm{B}).
\end{align*}
First, we note that for all $\theta \in [0,1]$, $k \in \Z^d$ and $0 \leq s \leq t$
\begin{align*}
\vert \xi_{\theta}^{t,s}(k(t-s)) \vert & =\vert k(t-s)+ \theta k(t-s)^2\varphi(t-s) \vert \\
& = \vert k \vert (t-s) \left[ 1+ \theta (t-s) \varphi(t-s) \right] \\
& \geq \vert k \vert (t-s).
\end{align*}
For $(\mathrm{B})$, we thus have
\begin{align*}
(\mathrm{B}) &\leq \underset{\ell \in \Z^d}{\sup} \underset{t \in (0,T)}{\sup}  \sum_{k \in \Z^d}  \int_0^t  (1+ \vert \xi_{\theta^{\star}}^{t,s}(k(t-s)) \vert )^{\beta_1} \vert \nabla_\xi \mathcal{G}^{t,s}_{\ell-k}\left(\xi_{\theta^{\star}}^{t,s}(k(t-s)) \right) \vert  \frac{\vert k \vert^2(t-s) }{(1+\vert \xi_{\theta^{\star}}^{t,s}(k(t-s)) \vert )^{\beta_1}}  \dd s \\
& \leq  \underset{\ell \in \Z^d}{\sup} \underset{t \in (0,T)}{\sup}  \sum_{k \in \Z^d}   \underset{s, \xi }{\sup} \big\lbrace (1+\vert \xi \vert )^{\beta_1}\vert \nabla_\xi \mathcal{G}^{t,s}_{\ell-k}(\xi) \vert  \big\rbrace \int_0^t \frac{\vert k \vert^2 (t-s)}{(1+\vert k \vert (t-s))^{\beta_1}}  \dd s.
\end{align*}
Since
\begin{align*}
\int_0^t \frac{\vert k \vert^2 (t-s)}{(1+\vert k \vert (t-s))^{\beta_1}}  \dd s=\int_0^{\vert k \vert t} \frac{\tau}{(1+\tau)^{\beta_1}} \dd \tau \leq  \int_0^{+\infty} \frac{\tau}{(1+\tau)^{\beta_1}} \dd \tau < + \infty, 
\end{align*}
if $\beta_1>2$, we get
\begin{align*}
(\mathrm{B}) \lesssim  \underset{\ell \in \Z^d}{\sup} \underset{t \in (0,T)}{\sup}  \sum_{k \in \Z^d}   \underset{s, \xi }{\sup} \big\lbrace (1+\vert \xi \vert )^{\beta_1}\vert \nabla_\xi \mathcal{G}^{t,s}_{\ell-k}(\xi) \vert  \big\rbrace.
\end{align*}
By choosing $\beta_2 >d/2$ and by the Cauchy-Schwarz inequality, this yields
\begin{align*}
(\mathrm{B}) \lesssim  \underset{t \in (0,T)}{\sup} \left( \sum_{k \in \Z^d}   \underset{s, \xi }{\sup} \big\lbrace  (1+\vert k \vert )^{\beta_2}(1+\vert \xi \vert )^{\beta_1}\vert \nabla_\xi \mathcal{G}^{t,s}_{k}(\xi) \vert  \big\rbrace^2 \right)^{\frac{1}{2}}.
\end{align*}
Let us estimate the other term $(\mathrm{A})$. By the Fubini-Tonelli theorem, we have
\begin{align*}
(\mathrm{A})&=\sum_{\ell \in \Z^d} \vert \ell \vert^2 \int_0^T  \int_0^t  \sum_{k \in \Z^d} \vert \widehat{H}_k(s) \vert^2 \vert k \vert^2 (t-s)^3 \varphi(t-s)^2 \vert \nabla_\xi \mathcal{G}^{t,s}_{\ell-k}\left(\xi_{\theta^{\star}}^{t,s}(k(t-s)) \right) \vert       \, \dd s \, \dd t \\
&= \int_0^T \sum_{k \in \Z^d} \vert \widehat{H}_k(s) \vert^2   \int_s^T \sum_{\ell \in \Z^d}  \vert \ell \vert^2 \vert k \vert^2 (t-s)^3 \varphi(t-s)^2 \vert \nabla_\xi \mathcal{G}^{t,s}_{\ell-k}\left(\xi_{\theta^{\star}}^{t,s}(k(t-s)) \right) \vert        \, \dd t \, \dd s \\
&\leq \Vert H \Vert_{\Ld^2(0,T ; \Ld^2(\T^d))}^2  \underset{k \in \Z^d}{\sup} \underset{0\leq s \leq T}{\sup}\int_s^T \sum_{\ell \in \Z^d} \vert \ell \vert^2 \vert k \vert^2 (t-s)^3 \varphi(t-s)^2 \vert \nabla_\xi \mathcal{G}^{t,s}_{\ell-k}\left(\xi_{\theta^{\star}}^{t,s}(k(t-s)) \right) \vert     \, \dd t.
\end{align*}
As in the proof of Proposition \ref{propo:AveragGENERAL}, we have
\begin{align*}
\underset{k \in \Z^d}{\sup}& \underset{0 \leq s \leq T}{\sup}\int_s^T \vert k \vert^2 (t-s)^3 \varphi(t-s)^2 \sum_{\ell \in \Z^d} \vert \ell \vert^2 \vert \nabla_\xi \mathcal{G}^{t,s}_{\ell-k}\left(\xi_{\theta^{\star}}^{t,s}(k(t-s)) \right) \vert       \, \dd t \\
&\leq  \underset{k \in \dot \Z^d}{\sup} \, \underset{0\leq s \leq T}{\sup}\int_s^T \frac{\vert k \vert^2 (t-s)^3 \varphi(t-s)^2}{(1+\vert k \vert(t-s))^{\alpha_1}} \, \dd t  \\
& \qquad \qquad \qquad \qquad  \qquad \times  \underset{0\leq s \leq T} {\sup} \underset{s\leq t \leq T}{\sup}  \sum_{\ell \in \Z^d} \vert \ell \vert^2 \, \underset{ \xi }{\sup} \,  (1+\vert \xi \vert^{\alpha_1}) \vert \nabla_\xi \mathcal{G}^{t,s}_{\ell}(\xi) \vert  \\
& \quad + \underset{k \in \dot \Z^d}{\sup} \, \underset{0\leq s \leq T}{\sup}\int_s^T \frac{\vert k \vert^4 (t-s)^3 \varphi(t-s)^2}{(1+\vert k \vert(t-s))^{\alpha_1}} \, \dd t  \\
& \qquad \qquad \qquad \qquad  \qquad \times  \underset{0\leq s \leq T} {\sup} \underset{s\leq t \leq T}{\sup}  \sum_{\ell \in \Z^d} \, \underset{ \xi }{\sup} \,  (1+\vert \xi \vert^{\alpha_1}) \vert \nabla_\xi \mathcal{G}^{t,s}_{\ell}(\xi) \vert  \\
&=\mathrm{T}_1 + \mathrm{T}_2.
\end{align*}
 
We treat these two terms in a  separate way.

$\bullet$ In $\mathrm{T}_1$, the integral term can be bounded via
\begin{align*}
\underset{k \in \dot \Z^d }{\sup}  \underset{0\leq s \leq T}{\sup}\int_s^T \frac{\vert k \vert^2 (t-s)^3 \varphi(t-s)^2}{(1+\vert k \vert(t-s))^{\alpha_1}} \, \dd t 
 &\leq \underset{k \in  \dot \Z^d}{\sup}  \frac{1}{\vert k \vert^2} \underset{0\leq s \leq T}{\sup} \varphi(T-s)^2  \int_{0}^{+\infty} \frac{\tau^3}{(1+\tau)^{\alpha_1}} \, \dd \tau. \\
 & \leq \varphi(T)^2 \underset{k \in  \dot \Z^d}{\sup}  \frac{1}{\vert k \vert^2} \int_{0}^{+\infty} \frac{\tau^3}{(1+\tau)^{\alpha_1}} \, \dd \tau \\
 & \lesssim \varphi(T)^2,
\end{align*}
provided that $\alpha_1 >4$. This implies that for any $\alpha_2>d/2$
\begin{align*}
\mathrm{T}_1 \lesssim\varphi(T)^2 \underset{0\leq s \leq T} {\sup} \underset{s\leq t \leq T}{\sup} \left(  \sum_{m \in \Z^d} \underset{\xi }{\sup} \,  \left\lbrace (1+\vert m \vert^{2+\alpha_2})(1+\vert \xi \vert^{\alpha_1}) \vert \nabla_\xi \mathcal{G}^{t,s}_{m}(\xi) \vert \right\rbrace^2 \right)^{\frac{1}{2}}.
\end{align*}

$\bullet$ In $\mathrm{T}_2$, the integral term can be bounded in a similar way via
\begin{align*}
 \underset{k \in \dot \Z^d}{\sup} \, \underset{0\leq s \leq T}{\sup}\int_s^T \frac{\vert k \vert^4 (t-s)^3 \varphi(t-s)^2}{(1+\vert k \vert(t-s))^{\alpha_1}} \, \dd t  
 &\leq  \underset{0\leq s \leq T}{\sup} \varphi(T-s)^2  \int_{0}^{+\infty} \frac{\tau^3}{(1+\tau)^{\alpha_1}} \, \dd \tau. \\
 & \leq \varphi(T)^2  \int_{0}^{+\infty} \frac{\tau^3}{(1+\tau)^{\alpha_1}} \, \dd \tau \\
 & \lesssim \varphi(T)^2,
\end{align*}
provided that $\alpha_1 >4$. This implies that for any $\alpha_2>d/2$
\begin{align*}
\mathrm{T}_2 \lesssim\varphi(T)\underset{0\leq s \leq T} {\sup} \underset{s\leq t \leq T}{\sup} \left(  \sum_{m \in \Z^d} \underset{\xi }{\sup} \,  \left\lbrace (1+\vert m \vert^{\alpha_2})(1+\vert \xi \vert^{\alpha_1}) \vert \nabla_\xi \mathcal{G}^{t,s}_{m}(\xi) \vert \right\rbrace^2 \right)^{\frac{1}{2}}.
\end{align*}
All in all, we get for $\alpha_1 >4$ and $\alpha_2>2+d/2$
\begin{align*}
(\mathrm{A}) &\lesssim \varphi(T)^2 \Vert H \Vert_{\Ld^2(0,T ; \Ld^2(\T^d))}^2 \\
& \qquad \qquad  \times  \underset{0\leq s \leq T} {\sup} \underset{s\leq t \leq T}{\sup} \left(  \sum_{m \in \Z^d} \underset{\xi }{\sup} \,  \left\lbrace (1+\vert m \vert^{\alpha_2})(1+\vert \xi \vert^{\alpha_1}) \vert \nabla_\xi \mathcal{G}^{t,s}_{m}(\xi) \vert \right\rbrace^2 \right)^{\frac{1}{2}} \\
&\lesssim \varphi(T)^2 \Vert H \Vert_{\Ld^2(0,T ; \Ld^2(\T^d))}^2 \\
& \qquad \qquad  \times  \underset{0\leq t \leq T}{\sup} \left(  \sum_{m \in \Z^d}\underset{0\leq s \leq t}{\sup} \underset{\xi }{\sup} \,  \left\lbrace (1+\vert m \vert^{\alpha_2})(1+\vert \xi \vert^{\alpha_1}) \vert \nabla_\xi \mathcal{G}^{t,s}_{m}(\xi) \vert \right\rbrace^2 \right)^{\frac{1}{2}},
\end{align*}
We have thus proven that for $s_1>4$ and $s_2>2+d/2$
\begin{align*}
 \LRVert{\nabla_x \left(\mathrm{K}_G^{\mathrm{free}}[H]-\mathrm{K}_G^{\mathrm{fric}}[H] \right)}_{\Ld^2(0,T ; \Ld^2(\T^d))} \lesssim \varphi(T)\LRVert{v \otimes G}_{T,s_1,s_2} \Vert H \Vert_{\Ld^2(0,T ; \Ld^2(\T^d))},
\end{align*}
where we have used the semi-norm $\Vert \cdot \Vert_{T,s_1,s_2}$ from Proposition \ref{propo:AveragGENERAL}. We can now conclude as in the proof of Proposition \ref{propo:AveragReg}. We observe that for all $p>2\ell+s+1+d$ and $\sigma>1+d/2$ (with $\ell, s \in \R^+$), there exist $s_1>s+1$ and $s_2>\ell+d/2$ such that
\begin{align*}
\Vert v \otimes G \Vert_{T,s_1,s_2} &\lesssim \underset{0 \leq s,t \leq T}{\sup} \Vert G(t,s) \Vert_{\mathcal{H}^p_{\sigma}}. 
\end{align*}
By taking $\ell=2$ ans $s=3$, and by finally using Proposition \ref{propo:AveragGENERAL}, we reach the desired conclusion.

\end{proof}

\section{Analysis of the kinetic moments}\label{Section:Kinetic-moments}

Following the bootstrap procedure initiated in Section \ref{Subsection:EstimateREG}, we want to control $\LRVert{\varrho_\eps}_{\Ld^{2}(0,T; \H^m)}$ uniformly in $\eps$ and for $T<T_\eps$. In view of the transport equation bearing on $\varrho_\eps$ (see Lemma \ref{LM:rewriteEqrho}), we will relate the kinetic moments $\rho_{f_\eps}$ and $j_{f_\eps}$ to the fluid density $\varrho_\eps$  itself.

In this section, to ease readability, we drop out the subscripts $\eps$ when we refer to the solution. Recall that $\Lambda$ will always stand for a nonnegative continuous function which is independent of $\eps$, nondecreasing with respect to each of its argument, that may depend implicitly on the initial data and that may change from line to line.

\medskip

For all $T \in [0,T_\eps)$ small enough, the goal of this section is thus to prove the following result.

\begin{propo}\label{coroFInal:D^I:rho-j}
Let $T \in (0, \min \left(T_\eps(R),\overline{T}(R)\right)$. For all $|I|\leq m$, we have for any $t \in (0,T)$,
\begin{align*}
 \partial_x^I \rho_{f}(t,x)&=p'(\varrho(t,x))\int_0^t  \int_{\R^d}   \nabla_x [\mathrm{J}_\eps \partial_x^I\varrho](s, x-(t-s)v) \cdot \nabla_v f(t,x,v) \, \mathrm{d}v \, \mathrm{d}s + R^I[\rho_f](t,x), \\
\partial_x^I  j_{f}(t,x)&=p'(\varrho(t,x))\int_0^t  \int_{\R^d}  v\nabla_x  [\mathrm{J}_\eps \partial_x^I \varrho](s, x-(t-s)v) \cdot \nabla_v f(t,x,v) \, \mathrm{d}v \, \mathrm{d}s+R^I[j_f](t,x),
\end{align*}
where the remainders $R^I[\rho_f]$ and $R^I[j_f]$ satisfy
\begin{align*}
\left\Vert R^I[\rho_f] \right\Vert_{\Ld^2(0,T, \H^1_x)} \leq \Lambda(T,R), \ \ \left\Vert R^I[j_f] \right\Vert_{\Ld^2(0,T, \H^1_x)} \leq \Lambda(T,R).
\end{align*}
\end{propo}

We recall the definition of the time $T_\eps(R)$ from our bootstrap procedure settled in Section \ref{Subsection:EstimateREG}, as well as the definition of the time $\overline{T}(R)$ from Lemma \ref{straight:velocity} in Section \ref{Section:Lagrangian}. Note that it is independent of $\eps$. In the rest of this section, we will always implicitly consider times $T>0$ such that
\begin{align*}
T< \min \left(T_\eps(R),\overline{T}(R) \right).
\end{align*}

From Proposition \ref{coroFInal:D^I:rho-j}, we can immediately infer the following corollary.
\begin{coro}\label{Coro:endUseAVERAGING}
For $m>2+d$, $\sigma >1+d/2$ and $|I|\leq m$, we have
\begin{align*}
\Vert \partial_x^I \rho_f \Vert_{\Ld^2(0,T;\Ld^2)}  &\leq \Lambda(T,R), \\
\Vert \partial_x^I j_f \Vert_{\Ld^2(0,T;\Ld^2)}  &\leq \Lambda(T,R).
\end{align*}
\end{coro}
\begin{proof}
By Proposition \ref{coroFInal:D^I:rho-j}, we can write
\begin{align*}
 \partial_x^I \rho_{f}= p'(\varrho)\mathrm{K}_{G}^{\mathrm{free}}[\mathrm{J}_\eps \partial_x^I\varrho]+R^I[\rho_f],
\end{align*}
with $G(t,x,v)= \nabla_v f(t,x,v)$ and $\left\Vert R^I[\rho_f] \right\Vert_{\Ld^2(0,T, \H^1_x)} \leq \Lambda(T,R)$. Since the kernel $G$ satisfies for $p>1+d$
\begin{align*}
 \underset{0 \leq t \leq T}{\sup} \Vert G\Vert_{\mathcal{H}^p_{\sigma}} \leq \Vert  f\Vert_{\Ld^{\infty}(0,T;\mathcal{H}^{m-1}_{\sigma})} \leq \Lambda(T,R), 
\end{align*}
 we can use the estimate from Proposition \ref{propo:Averag-ORiGINAL} to get
\begin{align*}
\Vert \partial_x^I \rho_f \Vert_{\Ld^2(0,T;\Ld^2)} &\lesssim \Vert p'(\varrho) \Vert_{\Ld^{2}(0,T; \Ld^{\infty})} \Lambda(T,R)\Vert \partial_x^I \varrho \Vert_{\Ld^2(0,T;\Ld^2)}  + \Vert R^I[\rho_f] \Vert_{\Ld^2(0,T;\Ld^2)}  \\
& \lesssim \mathrm{C}(\Vert \varrho \Vert_{\Ld^{\infty}(0,T; \H^{m-2})}) \Lambda(T,R)  +\Lambda(T,R) \\
& \leq \Lambda(T,R),
\end{align*}
by Sobolev embedding, Proposition \ref{Bony-Meyer} and Lemma \ref{LM:rho-pointwise-Hm-2}. The same argument applies for $\Vert \partial_x^I j_f \Vert_{\Ld^2(0,T;\Ld^2)}$.
\end{proof}
Our strategy to prove Proposition \ref{coroFInal:D^I:rho-j} goes as follows:
\begin{itemize}
\item[$-$] first, we take derivatives in the Vlasov equation to obtain a system of coupled kinetic equations satisfied by the augmented unknown $(\partial^I_{x} \partial^J_{v} f_\eps)_{|I|+|J|=m-1, m}$;
\item[$-$] next, we study the average in velocity of $\mathcal{F}$ by relying on Duhamel formula and the Lagrangian point of view of Section \ref{Section:Lagrangian}. We isolate the leading terms and prove estimates for the remainders, using crucially the techniques developed in Sections \ref{Section:Lagrangian} and \ref{Section:Averag-lemma}.
\end{itemize}

\subsection{The integro-differential system for derivatives of moments}\label{Subsection-Integrodiff}

\subsubsection{Applying derivatives}
We start with the following algebraic lemma, where we apply $\partial_x^{I} \partial_v^{J}$ to the Vlasov equation. 
Let us recall the notation $\widehat{\alpha}^k$ and $\overline{\alpha}^k$ for shifted indices (see Definition \ref{def:indices-shift}).

\begin{lem}\label{LM:ALLapplyD}
For any $I=(i_1, \cdots, i_d), J=(j_1, \cdots, ,j_d) \in \N^d$ such that $ \vert I \vert+\vert J \vert \in \lbrace m-1,m \rbrace $ and for any smooth function $f(t,x,v)$, we have
\begin{align*}
\left[\partial_x^{I} \partial_v^{J},\mathcal{T}_{\reg,\eps}^{u,\varrho} \right]f= \partial_x^I \partial_v^J E_{\reg,\eps}^{u,\varrho} \cdot \nabla_v f + \mathcal{M}^{I,J} \mathcal{F} +\mathcal{R}^{I,J}_1+ \mathcal{R}^{I,J}_0,
\end{align*}
where
\begin{align*}
\mathcal{F}&:=\left(\partial_x^{I} \partial_v^{J}f \right)_{\substack{I,J \in \N^d, \\  \vert I \vert+\vert J \vert \in \lbrace m-1,m \rbrace} }, \\
\mathcal{M}^{I,J} \mathcal{F}&:=\sum_{p=1}^d \mathbf{1}_{j_p \neq 0}  \left( \partial_x^{\widehat{I}^p} \partial_v^{\overline{J}^p} f-\partial_x^{I} \partial_v^{J}f \right)+\mathbf{1}_{\vert  I \vert > 2} \sum_{\substack{0 <\alpha < I \\ \vert \alpha  \vert \in \lbrace 1,2 \rbrace}} \binom{I}{\alpha}\partial_x^{\alpha} E_{\reg,\eps}^{u,\varrho} \cdot \nabla_v \partial_x^{I-\alpha} \partial_v^{J}  f \\
& \quad +\mathbf{1}_{\vert I \vert =1,2}\partial_x^I E_{\reg,\eps}^{u,\varrho} \cdot \partial_v^J \nabla_v f, \\[2mm]
\mathcal{R}^{I,J}_1&:=\mathbf{1}_{\substack{\vert I \vert > 2 \\ \vert J \vert \neq 0}} \partial_x^I E_{\reg,\eps}^{u,\varrho} \cdot  \nabla_v \partial_v^Jf+\mathbf{1}_{\vert I \vert>1} \sum_{\substack{0 <\alpha < I \\ \vert \alpha  \vert = m-1}} \binom{I}{\alpha}\partial_x^{\alpha} E_{\reg,\eps}^{u,\varrho} \cdot \nabla_v \partial_x^{I-\alpha} \partial_v^{J}  f,
\end{align*}
where
\begin{align}\label{estim1L2:R_1}
\left\Vert \mathcal{R}^{I,J}_1 \right\Vert_{\Ld^2(0,T;\mathcal{H}^0_{r})} \leq  \Lambda(T,R),
\end{align}
and where $\mathcal{R}^{I,J}_0$ is a remainder satisfying,
\begin{align}\label{estim1:R_0}
\left\Vert \mathcal{R}^{I,J}_0 \right\Vert_{\Ld^2(0,T;\mathcal{H}^1_{r})} \leq  \Lambda(T,R).
\end{align}
\end{lem}

\begin{proof}
Using Lemma \ref{LM:applyD-kin}, we have 
\begin{align*}
\partial_x^{I} \partial_v^{J}(\mathcal{T}_{\reg,\eps}^{u,\varrho} f)&=\mathcal{T}_{\reg,\eps}^{u,\varrho}(\partial_x^{I} \partial_v^{J}f)+ \sum_{p=1}^d  \mathbf{1}_{j_p \neq 0} \left( \partial_x^{\widehat{I}^p} \partial_v^{\overline{J}^p} f-\partial_x^{I} \partial_v^{J}f \right)+ \left[\partial_x^{I} \partial_v^{J}, E_{\reg,\eps}^{u,\varrho}\cdot \nabla_v\right]f.
\end{align*}
Since the force $E_{\reg,\eps}^{u,\varrho}(t,x)$ does not depend on $v$, we expand the commutator as
\begin{align*}
\left[\partial_x^{I} \partial_v^{J}, E(t,x)\cdot \nabla_v\right]f
= \mathbf{1}_{I \neq 0}\partial_x^I E_{\reg,\eps}^{u,\varrho} \cdot \partial_v^J \nabla_v f   +\mathbf{1}_{\vert I \vert> 1} \sum_{0<\alpha < I } \binom{I}{\alpha}\partial_x^{\alpha} E_{\reg,\eps}^{u,\varrho} \cdot  \partial_x^{I-\alpha} \partial_v^{J} \nabla_v f.
\end{align*}
Note that if $J \neq 0$ then $\vert I \vert \leq m-1$, and if $\vert I \vert =1,2$ then $J \neq 0$. The terms that cannot be considered as remainders in the last sum are those such that $\vert I \vert-\vert \alpha \vert + \vert J \vert+1 \in \lbrace m-1,m \rbrace$, that is for $\vert \alpha \vert  \in \lbrace 1,2 \rbrace$. We thus have
\begin{align*}
\left[\partial_x^{I} \partial_v^{J}, E_{\reg,\eps}^{u,\varrho}\cdot \nabla_v\right]f&=\mathbf{1}_{\substack{\vert I \vert \neq 0 \\ J=0}}\partial_x^I \partial_v^J E_{\reg,\eps}^{u,\varrho} \cdot \nabla_v f+\mathbf{1}_{\vert I \vert =1,2}\partial_x^I E_{\reg,\eps}^{u,\varrho} \cdot \partial_v^J \nabla_v f   \\
& \quad +\mathbf{1}_{\vert I \vert> 2}\sum_{\substack{0 <\alpha < I \\ \vert \alpha  \vert \in \lbrace 1,2 \rbrace}} \binom{I}{\alpha}\partial_x^{\alpha} E_{\reg,\eps}^{u,\varrho} \cdot \nabla_v \partial_x^{I-\alpha} \partial_v^{J}  f \\
& \quad +  \mathcal{R}^{I,J}_1+\mathcal{R}^{I,J}_0,
\end{align*}
where
\begin{align*}
 \mathcal{R}^{I,J}_1&:=\mathbf{1}_{\substack{\vert I \vert > 2 \\ \vert J \vert \neq 0}}\partial_x^I E_{\reg,\eps}^{u,\varrho} \cdot \partial_v^J \nabla_v f 
+\mathbf{1}_{\vert I \vert> 1}\sum_{\substack{0 <\alpha < I \\ \vert \alpha  \vert = m-1}} \binom{I}{\alpha}\partial_x^{\alpha} E_{\reg,\eps}^{u,\varrho} \cdot \nabla_v \partial_x^{I-\alpha} \partial_v^{J}  f, \\
 \mathcal{R}^{I,J}_0&:=\mathbf{1}_{\vert I \vert> 1} \sum_{\substack{0 <\alpha < I \\ 3 \leq \vert \alpha  \vert \leq m-2}} \binom{I}{\alpha}\partial_x^{\alpha} E_{\reg,\eps}^{u,\varrho} \cdot \nabla_v \partial_x^{I-\alpha} \partial_v^{J}  f.
\end{align*}
Let us estimate the remainder $\mathcal{R}^{I,J}_0$ in $\mathcal{H}_r^1$: setting $\chi(v)=(1+\vert v \vert^2)^{r/2}$, we have
\begin{align*}
\Vert \mathcal{R}^{I,J}_0 \Vert_{\mathcal{H}_r^1} &\lesssim \sum_{\substack{0 <\alpha < I \\ 3 \leq \vert \alpha  \vert \leq m-2}} \Vert \chi \partial_x^{\alpha} E_{\reg,\eps}^{u,\varrho} \cdot \nabla_v \partial_x^{I-\alpha} \partial_v^{J}  f \Vert_{\Ld^2_{x,v}} \\
& \quad +\sum_{\substack{0 <\alpha < I \\ 3 \leq \vert \alpha  \vert \leq m-2}}  \sum_{k=1}^d\Big( \Vert \chi \partial_x^{\widehat{\alpha}^k} E_{\reg,\eps}^{u,\varrho} \cdot \nabla_v \partial_x^{I-\alpha} \partial_v^{J}  f \Vert_{\Ld^2_{x,v}}
+\Vert \chi \partial_x^{\alpha} E_{\reg,\eps}^{u,\varrho} \cdot \partial_{x_k} \nabla_v \partial_x^{I-\alpha} \partial_v^{J}  f \Vert_{\Ld^2_{x,v}} \\
& \quad \qquad  \qquad \qquad \qquad \qquad  \qquad \qquad  \qquad \qquad \qquad \qquad  +\Vert \chi \partial_x^{\alpha} E_{\reg,\eps}^{u,\varrho} \cdot \partial_{v_k} \nabla_v \partial_x^{I-\alpha} \partial_v^{J}  f \Vert_{\Ld^2_{x,v}} \Big).
\end{align*}
$\bullet$ If $\frac{m+1}{2}\leq\vert \alpha \vert \leq m-2$, then for all $k$
\begin{align*}
 \Vert \chi \partial_x^{\alpha} E_{\reg,\eps}^{u,\varrho} \cdot \nabla_v \partial_x^{I-\alpha} \partial_v^{J}  f \Vert_{\Ld^2_{x,v}}^2&+\Vert \chi \partial_x^{\widehat{\alpha}^k} E_{\reg,\eps}^{u,\varrho} \cdot \nabla_v \partial_x^{I-\alpha} \partial_v^{J}  f \Vert_{\Ld^2_{x,v}}^2 \\
 &\leq \left(\Vert \partial_x^{\alpha} E_{\reg,\eps}^{u,\varrho} \Vert_{\Ld^2}^2+ \Vert \partial_x^{\widehat{\alpha}^k} E_{\reg,\eps}^{u,\varrho} \Vert_{\Ld^2}^2 \right)\int_{\R^d} \chi(v)^2 \Vert \nabla_v \partial_x^{I-\alpha} \partial_v^{J}  f \Vert_{\Ld^{\infty}}^2 \, \mathrm{d}v \\
 & \lesssim \Vert  E_{\reg,\eps}^{u,\varrho} \Vert_{\H^{m-1}}^2 \int_{\R^d} \chi(v)^2 \Vert \nabla_v \partial_v^{J}  f \Vert_{\H^{\sigma}}^2 \, \mathrm{d}v,
\end{align*}
provided that $\sigma>\vert I-\alpha \vert+d/2$. Since $\vert J \vert +1 + \vert I \vert -\vert \alpha \vert+d/2 \leq m+1-\vert \alpha \vert +d/2\leq \frac{m+1+d}{2}$ and since $m>4+d$, we can find such a $\sigma$ so that 
\begin{align*}
\Vert \chi \partial_x^{\alpha} E_{\reg,\eps}^{u,\varrho} \cdot \nabla_v \partial_x^{I-\alpha} \partial_v^{J}  f \Vert_{\Ld^2_{x,v}}+ \Vert \chi \partial_x^{\widehat{\alpha}^k} E_{\reg,\eps}^{u,\varrho} \cdot \nabla_v \partial_x^{I-\alpha} \partial_v^{J}  f \Vert_{\Ld^2_{x,v}} \lesssim \Vert  E_{\reg,\eps}^{u,\varrho} \Vert_{\H^{m-1}} \Vert f \Vert_{\mathcal{H}_r^{m-1}}.
\end{align*}
Likewise, we have for all $k$
\begin{align*}
\Vert \chi \partial_x^{\alpha} E_{\reg,\eps}^{u,\varrho} \cdot \partial_{x_k} \nabla_v \partial_x^{I-\alpha} \partial_v^{J}  f \Vert_{\Ld^2_{x,v}} ^2 
\leq \Vert  E_{\reg,\eps}^{u,\varrho} \Vert_{\H^{m-1}}^2 \int_{\R^d} \chi(v)^2 \Vert \nabla_v \partial_v^{J}  f \Vert_{\H^{\sigma}}^2 \, \mathrm{d}v,
\end{align*}
provided that $\sigma>1+\vert I-\alpha \vert+d/2$. Since $\vert J \vert +1 + 1+\vert I \vert -\vert \alpha \vert+d/2 \leq m+2-\vert \alpha \vert +d/2\leq \frac{m+3+d}{2}$ and since $m>6+d$, there exists such a $\sigma$ so that 
\begin{align*}
 \Vert \chi \partial_x^{\alpha} E \cdot \partial_{x_k} \nabla_v \partial_x^{I-\alpha} \partial_v^{J}  f \Vert_{\Ld^2_{x,v}} \lesssim \Vert  E \Vert_{\H^{m-1}} \Vert f \Vert_{\mathcal{H}_r^{m-1}}.
\end{align*}
The same procedure can be applied for the terms $\Vert \chi \partial_x^{\alpha} E_{\reg,\eps}^{u,\varrho} \cdot \partial_{v_k} \nabla_v \partial_x^{I-\alpha} \partial_v^{J}  f \Vert_{\Ld^2_{x,v}}$.

$\bullet$ If $3 \leq \vert \alpha \vert < \frac{m+1}{2}$, then
\begin{align*}
\Vert \chi \partial_x^{\alpha} E_{\reg,\eps}^{u,\varrho} \cdot \nabla_v \partial_x^{I-\alpha} \partial_v^{J}  f \Vert_{\Ld^2_{x,v}}&+ \Vert \chi \partial_x^{\widehat{\alpha}^k} E_{\reg,\eps}^{u,\varrho} \cdot \nabla_v \partial_x^{I-\alpha} \partial_v^{J}  f \Vert_{\Ld^2_{x,v}} \\
&\leq \left(  \Vert \partial_x^{\alpha} E_{\reg,\eps}^{u,\varrho} \Vert_{\Ld^{\infty}} + \Vert \partial_x^{\widehat{\alpha}^k} E_{\reg,\eps}^{u,\varrho} \Vert_{\Ld^{\infty}} \right) \Vert \chi   \nabla_v \partial_x^{I-\alpha} \partial_v^{J}  f \Vert_{\Ld^2_{x,v}} \\
&\lesssim \Vert  E_{\reg,\eps}^{u,\varrho} \Vert_{\H^{\sigma}} \Vert f \Vert_{\mathcal{H}_r^{m-1}},
\end{align*}
provided that $\sigma> 1+ \vert \alpha \vert+d/2$. Since $1+ \vert \alpha \vert+d/2 \leq \frac{m+3+d}{2}$ and $m>5+d$, we can find such a $\sigma$ so that 
\begin{align*}
\Vert \chi \partial_x^{\widehat{\alpha}^k} E_{\reg,\eps}^{u,\varrho} \cdot \nabla_v \partial_x^{I-\alpha} \partial_v^{J}  f \Vert_{\Ld^2_{x,v}}  \lesssim \Vert  E_{\reg,\eps}^{u,\varrho} \Vert_{\H^{m-1}} \Vert f \Vert_{\mathcal{H}_r^{m-1}}.
\end{align*}
Likewise, we have for all $k$
\begin{align*}
\Vert \chi \partial_x^{\alpha} E_{\reg,\eps}^{u,\varrho} \cdot \partial_{x_k} \nabla_v \partial_x^{I-\alpha} \partial_v^{J}  f \Vert_{\Ld^2_{x,v}} ^2 
\leq \Vert \partial_x^{\alpha} E_{\reg,\eps}^{u,\varrho} \Vert_{\Ld^{\infty}} \Vert \chi \partial_x  \nabla_v \partial_x^{I-\alpha} \partial_v^{J}  f \Vert_{\Ld^2_{x,v}} \lesssim \Vert  E_{\reg,\eps}^{u,\varrho} \Vert_{\H^{\sigma}} \Vert f \Vert_{\mathcal{H}_r^{m-1}},
\end{align*}
provided that $\sigma>\vert I-\alpha \vert+d/2$. Since $\vert \alpha \vert+d/2 \leq \frac{m+2+d}{2}$ and $m>4+d$, there exists such a $\sigma$ so that 
\begin{align*}
\Vert \chi \partial_x^{\alpha} E_{\reg,\eps}^{u,\varrho} \cdot \nabla_v \partial_x \partial_x^{I-\alpha} \partial_v^{J}  f \Vert_{\Ld^2_{x,v}}  \lesssim \Vert  E_{\reg,\eps}^{u,\varrho} \Vert_{\H^{m-1}} \Vert f \Vert_{\mathcal{H}_r^{m-1}}.
\end{align*}
The same procedure can be applied for the terms $\Vert \chi \partial_x^{\alpha} E_{\reg,\eps}^{u,\varrho} \cdot \partial_{v_k} \nabla_v \partial_x^{I-\alpha} \partial_v^{J}  f \Vert_{\Ld^2_{x,v}}$.

\medskip

All in all, we have proven that for all $t \in [0,T]$
\begin{align*}
\Vert \mathcal{R}^{I,J}_0 (t)\Vert_{\mathcal{H}_r^1} \lesssim \Vert  E_{\reg,\eps}^{u,\varrho}(t) \Vert_{\H^{m-1}} \Vert f(t) \Vert_{\mathcal{H}_r^{m-1}} \leq \Vert f \Vert_{\Ld^{\infty}(0,T;\mathcal{H}_r^{m-1})} \Vert  E_{\reg,\eps}^{u,\varrho}(t) \Vert_{\H^{m-1}} \leq R \Vert  E(t) \Vert_{\H^{m-1}}.
\end{align*}

By the estimate \eqref{bound:Sobo2:E} from Lemma \ref{LM:Estim-ForceField} and by Lemma \ref{LM:rho-pointwise-Hm-2}, we finally have 
\begin{align*}
\left\Vert \mathcal{R}^{I,J}_0 \right\Vert_{\Ld^2(0,T;\mathcal{H}^1_{r})} \leq  \Lambda(T,R).
\end{align*}
With the same exact arguments, we easily obtain the fact that
\begin{align*}
\left\Vert \mathcal{R}^{I,J}_1 \right\Vert_{\Ld^2(0,T;\mathcal{H}^0_{r})} \leq  \Lambda(T,R),
\end{align*}
and this concludes the proof.

\end{proof}

\begin{rem}
We will actually obtain an improved $\Ld^2(0,T;\mathcal{H}^1_{r})$ estimate for the term $\mathcal{R}^{I,J}_1$ (or, more preciely, related terms) in the end of the current section. 
\end{rem}

We can see $\mathcal{M}^{I,J} \mathcal{F}$ appearing in Lemma \ref{LM:ALLapplyD} as a linear combination of $\mathcal{F}^{K,L}= \partial_x^{K} \partial_v^{L}f$. More precisely, we can write for all $(I,J)$,
\begin{align*}
\mathcal{M}^{I,J}\mathcal{F} &=\sum_{K,L} \mathcal{M}_{(I,J),(K,L)}\mathcal{F}^{K,L}, \\
 \mathcal{M}_{(I,J),(K,L)}&:=\sum_{p=1}^d \mathbf{1}_{j_p \neq 0} (\mathbf{1}_{(K,L)=(\widehat{I}^p,\overline{J}^p)}-\mathbf{1}_{(K,L)=(I,J)} )
 +\mathbf{1}_{ \vert I \vert=1,2} \sum_{p=1}^d  \mathbf{1}_{(K,L)=(0,\widehat{J}^p)} \partial_x^{I} (E_{\reg,\eps}^{u,\varrho})_p  \\
 & \quad +\sum_{p=1}^d \sum_{\substack{0 <\alpha < I \\ \vert \alpha  \vert \in \lbrace 1,2 \rbrace}}  \binom{I}{\alpha} \mathbf{1}_{(K,L)=(I-\alpha,\widehat{J}^p)} \partial_x^{\alpha} (E_{\reg,\eps}^{u,\varrho})_p .
\end{align*}
Let us observe that the coefficient involved in the operator $\mathcal{M}$ involve only $0$, $1$ or $2$ derivatives of the force field $E_{\reg,\eps}^{u,\varrho}(t,x)$, but nothing coming from $f$.

\medskip

We finally introduce the following additional notations which will allow us to reformulate Lemma \ref{LM:ALLapplyD} in a compact way.

\begin{defi}\label{def:AugmentedVar+Coupling}
We consider the following quantities:
\begin{enumerate}
\item $\mathcal{R}_0$ and $\mathcal{R}_1$ are the vectors defined by
\begin{align*}
 \mathcal{R}_0=\left(\mathcal{R}^{I,J}_0 \right)_{\substack{I,J \in \N^d, \\  \vert I \vert+\vert J \vert \in \lbrace m-1,m \rbrace }}, \ \ \mathcal{R}_1=\left(\mathcal{R}^{I,J}_0 \right)_{\substack{I,J \in \N^d, \\  \vert I \vert+\vert J \vert \in \lbrace m-1,m \rbrace }};
\end{align*}
\item $\mathcal{M}$ is the linear map defined by
\begin{align*}
\mathcal{M}=\left(\mathcal{M}_{(I,J),(K,L)} \right)_{\substack{ (I,J),(K,L)\\  \vert I \vert+\vert J \vert, \vert K \vert+\vert L \vert \in \lbrace m-1,m \rbrace }};
\end{align*}
\item $\mathcal{L}$ is the vector defined by
\begin{align*}
\mathcal{L}=\left( \partial_x^I \partial_v^J E_{\reg,\eps}^{u,\varrho} \cdot \nabla_v f \right)_{\substack{I,J \in \N^d, \\  \vert I \vert+\vert J \vert \in \lbrace m-1,m \rbrace }}.
\end{align*}
\end{enumerate}
\end{defi}

\subsubsection{The semi-Lagrangian approach}
If $f$ satisfies the Vlasov equation (in a strong sense), we have $\partial_x^{I} \partial_v^{J}(\mathcal{T}_{\reg,\eps}^{u,\varrho} f)=0$ for any $I,J$. We can use Lemma \ref{LM:ALLapplyD} to obtain the following coupled system of equations satisfied by the family $\mathcal{F}=(\partial_x^{I} \partial_v^{J}f)_{I,J}$:
\begin{align}\label{eq:augmentedVarF}
\mathcal{T}_{\reg,\eps}^{u,\varrho} \mathcal{F}+\mathcal{M} \mathcal{F}+\mathcal{L}=-\mathcal{R}_0-\mathcal{R}_1.
\end{align}
For any function $g(t,x,v)$, we set 
\begin{align*}
\widetilde{g}(t,x,v)=g(t,\X^{t;0}(x,v),\V^{t;0}(x,v)),
\end{align*}
where $$s \mapsto \mathrm{Z}^{s;t}(x,v)=\pr{\X^{s;t}(x,v),\V^{s;t}(x,v)}$$ is the solution to
\begin{equation}\label{EDO-charac}
\left\{
      \begin{aligned}
        \frac{\mathrm{d}}{\mathrm{d}s} \mathrm{X}^{s;t}(x,v) &=\mathrm{V}^{s;t}(x,v), \ \ \mathrm{X}^{t;t}(x,v)=x \in \T^d, \\[2mm]
\frac{\mathrm{d}}{\mathrm{d}s}\mathrm{V}^{s;t}(x,v)&=-\mathrm{V}^{s;t}(x,v)+E_{\reg,\eps}^{u,\varrho}(s,\mathrm{X}^{s;t}(x,v)), \ \ \mathrm{V}^{t;t}(x,v)=v \in \R^d.
      \end{aligned}
    \right.
\end{equation}
After the composition by $ (t,x,v) \mapsto (t,\mathrm{X}^{t;0}(x,v),\mathrm{V}^{t;0}(x,v))$, we thus obtain by the method of characteristics
\begin{align}\label{eq:Ftilde}
\partial_t \widetilde{\mathcal{F}}+\widetilde{\mathcal{M}} \widetilde{\mathcal{F}}+\widetilde{\mathcal{L}}=d\widetilde{\mathcal{F}} -\widetilde{\mathcal{R}}_0-\widetilde{\mathcal{R}}_1.
\end{align}

To deal with the coupling matrix $\mathcal{M}$, we introduce the following object. 
\begin{defi}
For all $(x,v)$ and $s,t \geq 0$, we define the resolvant operator $\mathfrak{N}^{s,t}(x,v)$ as the solution $s \mapsto \mathfrak{N}^{s,t}(x,v)$ of 
\begin{equation}\label{eq:resolvant}
\left\{
      \begin{aligned}
        \partial_s \mathfrak{N}^{s;t}+\left[\mathcal{M} \circ \mathrm{Z}^{s;0}-d \mathrm{Id} \right] \mathfrak{N}^{s;t}&=0,\\
	\mathfrak{N}^{t;t}&= \mathrm{Id}.
      \end{aligned}
    \right.
\end{equation}
\end{defi}
The resolvant is well-defined thanks to the Cauchy-Lipschitz theorem. We also have
\begin{align*}
\mathfrak{N}^{s;t}(x,v)=e^{d(s-t)} \mathrm{N}^{s;t}(x,v),
\end{align*}
where
\begin{equation}\label{eq:resolvant2}
\left\{
      \begin{aligned}
        \partial_s \mathrm{N}^{s;t}+\mathcal{M} \circ \mathrm{Z}^{s;0} \mathrm{N}^{s;t}&=0,\\
	\mathrm{N}^{t;t}&= \mathrm{Id}.
      \end{aligned}
    \right.
\end{equation}

For the upcoming analysis, we need the following bounds on the resolvant.

\begin{lem}\label{LM:estimResolvante}
For all $0 \leq k<m-3-d/2$, we have
\begin{align*}
\underset{0 \leq s,t \leq T}{\sup} \Vert \mathrm{N}^{s;t} \Vert_{\W^{k, \infty}_{x,v}} + \underset{0 \leq s,t \leq T}{\sup} \Vert \partial_s \mathrm{N}^{s;t} \Vert_{\W^{k, \infty}_{x,v}}+ \underset{0 \leq s,t \leq T}{\sup} \Vert \partial_t \mathrm{N}^{s;t} \Vert_{\W^{k, \infty}_{x,v}} \leq \Lambda(T,R).
\end{align*}
\end{lem}
\begin{proof}
We have 
$$\mathrm{N}^{s;t}= \mathrm{Id}- \int_t^s [\mathcal{M} \circ \mathrm{Z}^{\tau,0}] \mathrm{N}^{\tau;t} \, \mathrm{d}\tau.$$
By the definition of coefficients of the matrix $\mathcal{M}$, we also have
\begin{align*}
\Vert  \mathcal{M} \Vert_{\Ld^{2}(0,T;\Ld^{\infty})} \lesssim \underset{ \vert \alpha \vert \leq 2}{\sup} \Vert \partial_x^{\alpha} E \Vert_{\Ld^{2}(0,T;\Ld^{\infty})} \leq \Lambda(T,R),
\end{align*}
thanks to the estimate \eqref{bound:Sobo2:E} from Lemma \ref{LM:Estim-ForceField} and by Sobolev embedding (with $m-1>2+d/2$). Grönwall's lemma leads to the conclusion.
\end{proof}

We first obtain the following decomposition for the kinetic moments of the vector $\mathcal{F}$.

\begin{lem}\label{LM:decompo:Moment1}
We have
\begin{multline*}
\int_{\R^d} \mathcal{F}(t,x,v) \, \mathrm{d}v=\mathcal{I}_{in}^0(t,x)+\mathcal{I}_{\mathcal{R}_0}^0(t,x)+\mathcal{I}_{\mathcal{R}_1}^0(t,x)\\
  -\int_0^t  e^{d(t-s)} \int_{\R^d} \mathrm{N}^{t;s}(\mathrm{Z}^{0;t}(x,v)) \mathcal{L}(s, \mathrm{Z}^{s;t}(x,v)) \, \mathrm{d}v \, \mathrm{d}s,
\end{multline*}
and 
\begin{multline*}
\int_{\R^d} v \otimes \mathcal{F}(t,x,v) \, \mathrm{d}v=\mathcal{I}_{in}^{1}(t,x)+\mathcal{I}_{\mathcal{R}_0}^{1}(t,x)+\mathcal{I}_{\mathcal{R}_1}^{1}(t,x)\\
 -\int_0^t  e^{d(t-s)} \int_{\R^d} v\otimes \mathrm{N}^{t;s}(\mathrm{Z}^{0;t}(x,v)) \mathcal{L}(s, \mathrm{Z}^{s;t}(x,v)) \, \mathrm{d}v \, \mathrm{d}s,
\end{multline*}
where 
\begin{align*}
\mathcal{I}_{in}^0(t,x)&:= e^{dt}\int_{\R^d}  \mathrm{N}^{t;0}(\mathrm{Z}^{0;t}(x,v))\mathcal{F}_{\mid t=0}(\mathrm{Z}^{0;t}(x,v)) \, \mathrm{d}v, \\
\mathcal{I}_{\mathcal{R}_j}^0(t,x)&:=-\int_0^t e^{d(t-s)}  \int_{\R^d} \mathrm{N}^{t;s}(\mathrm{Z}^{0;t}(x,v))  \mathcal{R}_j(s, \mathrm{Z}^{s;t}(x,v)) \, \mathrm{d}v \, \mathrm{d}s, \ \ j=0,1, \\
\mathcal{I}_{in}^{1}(t,x)&:= e^{dt} \int_{\R^d} v \otimes \mathrm{N}^{t;0}(\mathrm{Z}^{0;t}(x,v))\mathcal{F}_{\mid t=0}(\mathrm{Z}^{0;t}(x,v)) \, \mathrm{d}v, \\
\mathcal{I}_{\mathcal{R}_j}^{1}(t,x)&:=-\int_0^t e^{d(t-s)} \int_{\R^d}   v \otimes \mathrm{N}^{t;s}(\mathrm{Z}^{0;t}(x,v))  \mathcal{R}_j(s, \mathrm{Z}^{s;t}(x,v)) \, \mathrm{d}v \, \mathrm{d}s; \ \ j=0,1.
\end{align*}

\end{lem}
\begin{proof}
We only explain the case of the moment of order $0$, the other being similar. Starting from the equation \eqref{eq:Ftilde} and using the resolvant operators $\mathfrak{N}$ and $\mathrm{N}$ defined in \eqref{eq:resolvant} and \eqref{eq:resolvant2}, we have
\begin{align*}
\widetilde{\mathcal{F}}(t)&=\mathfrak{N}^{t,0} \widetilde{\mathcal{F}}_{\mid t=0}-\int_0^t \mathfrak{N}^{t,s}\left[\widetilde{\mathcal{L}}(s)+\widetilde{\mathcal{R}}_0(s)+\widetilde{\mathcal{R}}_1(s) \right] \, \mathrm{d}s \\
&=e^{dt}\mathrm{N}^{t;0} \widetilde{\mathcal{F}}_{\mid t=0}-\int_0^t e^{d(t-s)} \mathrm{N}^{t;s}\left[\widetilde{\mathcal{L}}(s)+\widetilde{\mathcal{R}}_0(s)+\widetilde{\mathcal{R}}_1(s) \right] \, \mathrm{d}s.
\end{align*}
After a composition by the map $ (t,x,v) \mapsto (t,\mathrm{X}^{0;t}(x,v),\mathrm{V}^{0;t}(x,v))$, we obtain
\begin{multline*}
\mathcal{F}(t)=e^{dt}\left[\mathrm{N}^{t;0} \circ \mathrm{Z}^{0;t}\right]\mathcal{F}_{\mid t=0} \circ \mathrm{Z}^{0;t}-\int_0^t e^{d(t-s)} \left[\mathrm{N}^{t;s} \circ \mathrm{Z}^{0;t} \right]\left(\mathcal{R}_0(s,\mathrm{Z}^{s;t}) +\mathcal{R}_1(s,\mathrm{Z}^{s;t}) \right) \, \mathrm{d}s \\
-\int_0^t  e^{d(t-s)} \left[\mathrm{N}^{t;s} \circ \mathrm{Z}^{0;t} \right] \mathcal{L}(s, \mathrm{Z}^{s;t}) \, \mathrm{d}s.
\end{multline*}
We reach the desired conclusion by integrating in velocity.

\end{proof}

\subsection{First remainders}

Let us show straightaway that some of the previous terms can be considered as remainders.

\begin{lem}\label{LM:estimateI_0R}
We have
\begin{align}
\label{estimate:I_0-H1}\Vert \mathcal{I}_{in}^0 \Vert_{\Ld^2(0,T;\H^1)} \Vert +\Vert \mathcal{I}_{in}^{1} \Vert_{\Ld^2(0,T;\H^1)}   &\leq \Lambda(T,R) \Vert f^{\mathrm{in}} \Vert_{\mathcal{H}^{m+1}_r}, \\
\label{estimate:I_R0-H1}\Vert \mathcal{I}_{\mathcal{R}_0}^0\Vert_{\Ld^2(0,T;\H^1)} + \Vert \mathcal{I}_{\mathcal{R}_0}^{1} \Vert_{\Ld^2(0,T;\H^1)}  &\leq  \Lambda(T,R), \\
\label{estimate:I_R1-L2}\Vert \mathcal{I}_{\mathcal{R}_1}^0\Vert_{\Ld^2(0,T;\Ld^2)} + \Vert \mathcal{I}_{\mathcal{R}_1}^{1} \Vert_{\Ld^2(0,T;\Ld^2)}  &\leq  \Lambda(T,R).
\end{align}
\end{lem}

\begin{proof}
We shall first estimate the term $\mathcal{I}_{in}^0$. We have
\begin{align*}
\vert \mathcal{I}_{in}^0(t,x) \vert \lesssim e^{dT} \underset{0 \leq t \leq T}{\sup} \Vert \mathrm{N}^{t;s} \Vert_{\Ld^{\infty}_{x,v}} \int_{\R^d} \vert \mathcal{F}_{0}(\mathrm{Z}^{0;t}(x,v)) \vert \, \mathrm{d}v \lesssim \Lambda(T,R) \sum_{I,J} \int_{\R^d} \vert \partial_x^{I} \partial_v^{J} f^{\mathrm{in}} (\mathrm{Z}^{0;t}(x,v)) \vert \, \mathrm{d}v,
\end{align*}
thanks to Lemma \ref{LM:estimResolvante} with $k=0$. Using the generalized Minkowski inequality and the Cauchy-Schwarz inequality, we get
\begin{align*}
\Vert \mathcal{I}_{in}^0 \Vert_{\Ld^2(0,T;\Ld^2)} &\lesssim \Lambda(T,R)  \sum_{I,J} \left\Vert \int_{\R^d} \Vert \partial_x^{I} \partial_v^{J} f^{\mathrm{in}} (\mathrm{Z}^{0;t}(\cdot,v)) \Vert_{\Ld^2} \, \mathrm{d}v \right\Vert_{\Ld^2(0,T)} \\
&\lesssim\Lambda(T,R)  \sum_{I,J} \left\Vert \left( \int_{\T^d \times \R^d} (1+ \vert v \vert^2)^r \vert \partial_x^{I} \partial_v^{J} f^{\mathrm{in}} (\mathrm{Z}^{0;t}(x,v)) \vert^2 \, \mathrm{d}x \, \mathrm{d} v \right)^{1/2} \right\Vert_{\Ld^2(0,T)},
\end{align*}
since $2r>d$. We then perform the change of variable $(x,v) \mapsto \mathrm{Z}^{0;t}(x,v)$ which yields
\begin{align*}
\int_{\T^d \times \R^d} (1+ \vert v \vert^2)^r \vert \partial_x^{I} \partial_v^{J} f^{\mathrm{in}} (\mathrm{Z}^{0;t}(x,v)) \vert^2 \, \mathrm{d}x \, \mathrm{d} v  \lesssim \Lambda(T,R) \int_{\T^d \times \R^d} (1+ \vert \mathrm{V}^{t;0}(x,v) \vert^2)^r \vert \partial_x^{I} \partial_v^{J} f^{\mathrm{in}} (x,v) \vert^2 \, \mathrm{d}x \, \mathrm{d} v .
\end{align*}
Since
\begin{align*}
\mathrm{V}^{t;0}(x,v)=e^{-t}v + \int_0^t e^{\tau-t} E_{\reg,\eps}^{u,\varrho}(\tau, \X^{\tau;t}(x,v)) \, \mathrm{d}\tau,
\end{align*}
we have by Sobolev embedding 
\begin{align*}
\vert \mathrm{V}^{t;0}(x,v) \vert^2  \leq \vert v \vert^2 + \left\vert \int_0^t \Vert E_{\reg,\eps}^{u,\varrho}(\tau) \Vert_{\Ld^{\infty}} \, \mathrm{d}\tau \right\vert^2 \lesssim \vert v \vert^2+ T \Vert E_{\reg,\eps}^{u,\varrho} \Vert_{\Ld^{2}(0,T; \H^{m-1})}^2.
\end{align*}
By the estimate \eqref{bound:Sobo2:E} from Lemma \ref{LM:Estim-ForceField}, we get
\begin{align*}
\vert \mathrm{V}^{t;0}(x,v) \vert^2   \leq\vert v \vert^2 \Lambda(T,R),
\end{align*}
and then
\begin{align*}
\int_{\T^d \times \R^d} (1+ \vert \mathrm{V}^{t;0}(x,v) \vert^2)^r \vert \partial_x^{I} \partial_v^{J} f^{\mathrm{in}} (x,v) \vert^2 \, \mathrm{d}x \, \mathrm{d} v &\lesssim \Lambda(T,R)\int_{\T^d \times \R^d} (1+ \vert v\vert^2)^r \vert \partial_x^{I} \partial_v^{J} f^{\mathrm{in}} (x,v) \vert^2 \, \mathrm{d}x \, \mathrm{d} v \\
& \lesssim \Lambda(T,R) \Vert f^{\mathrm{in}} \Vert_{\mathcal{H}^m_r}^2.
\end{align*}
This implies 
\begin{align*}
\Vert \mathcal{I}_{in}^0 \Vert_{\Ld^2(0,T;\Ld^2)}  \leq \Lambda(T,R) \Vert f^{\mathrm{in}} \Vert_{\mathcal{H}^m_r}.
\end{align*}
Likewise, using $2r>d+1$, we have
\begin{align*}
\Vert \mathcal{I}_{in}^{1} \Vert_{\Ld^2(0,T;\Ld^2)} \lesssim\Lambda(T,R) \Vert f^{\mathrm{in}} \Vert_{\mathcal{H}^m_r}.
\end{align*}
We also have
\begin{align*}
\left[\mathcal{I}_{in}^0 \right]_{(I,J)}(t,x)=e^{dt}\int_{\R^d} \sum_{(K,L)} \mathrm{N}^{t;0}_{(I,J),(K,L)}(\mathrm{Z}^{0;t}(x,v))\left[\mathcal{F}_{\mid t=0} \right]_{(K,L)}(\mathrm{Z}^{0;t}(x,v)) \, \mathrm{d}v,
\end{align*}
and
\begin{align*}
\nabla_x\left[\mathcal{I}_{in}^0 \right]_{(I,J)}(t,x)&=e^{dt}\int_{\R^d} \sum_{(K,L)} \nabla_x \mathrm{Z}^{0;t}(x,v) \nabla_x\mathrm{N}^{t;0}_{(I,J),(K,L)}(\mathrm{Z}^{0;t}(x,v))\left[\mathcal{F}_{\mid t=0} \right]_{(K,L)}(\mathrm{Z}^{0;t}(x,v)) \, \mathrm{d}v \\
& \quad + e^{dt}\int_{\R^d} \sum_{(K,L)} \mathrm{N}^{t;0}_{(I,J),(K,L)}(\mathrm{Z}^{0;t}(x,v)) \nabla_x \mathrm{Z}^{0;t}(x,v) \nabla_x \left[\mathcal{F}_{\mid t=0} \right]_{(K,L)}(\mathrm{Z}^{0;t}(x,v)) \, \mathrm{d}v.
\end{align*}
The same procedure as before, using Lemma \ref{LM:estimResolvante} with $k=1$ and the pointwise bounds \eqref{bound:Wkinfty-X}-\eqref{bound:Wkinfty-V} from Remark \ref{Rmk:BoundsXV}, gives
\begin{align*}
\Vert \nabla_x \mathcal{I}_{in}^0 \Vert_{\Ld^2(0,T;\Ld^2)}+\Vert \nabla_x\mathcal{I}_{in}^{1} \Vert_{\Ld^2(0,T;\Ld^2)} \lesssim  \Lambda(T,R) \Vert f^{\mathrm{in}} \Vert_{\mathcal{H}^{m+1}_r}.
\end{align*}

We now estimate the terms $ \mathcal{I}_{\mathcal{R}_0}^0$ and $\mathcal{I}_{\mathcal{R}_0}^{1}$. We use the same arguments as before, with an additional Cauchy-Schwarz inequality in time leading to
\begin{align*}
\Vert  \mathcal{I}_{\mathcal{R}_0}^0  \Vert_{\Ld^2(0,T;\Ld^2)} &\leq \Lambda(T,R) \sum_{I,J} \left\Vert \int_0^t \int_{\R^d} \left\Vert  \mathcal{R}_0^{I,J}(s,\mathrm{Z}^{s;t}(\cdot,v)  \right\Vert_{\Ld^2} \, \mathrm{d}s \, \mathrm{d}v  \right \Vert_{\Ld^2(0,T)} \\
& \leq  \Lambda(T,R) \left\Vert \int_0^t  \left\Vert  \mathcal{R}_0(s)  \right\Vert_{\mathcal{H}^0_r} \, \mathrm{d}s   \right \Vert_{\Ld^2(0,T)} \\
& \leq   \Lambda(T,R) T \left\Vert \mathcal{R}_0 \right\Vert_{\Ld^2(0,T; \mathcal{H}^0_r)} \\
& \leq   \Lambda(T,R),
\end{align*}
thanks to \eqref{estim1:R_0} in Lemma \ref{LM:ALLapplyD}. We obtain the same result for $\mathcal{I}_{\mathcal{R}_0}^{1}$. Using again \eqref{estim1:R_0} for the first order derivative, we also have
\begin{align*}
\Vert \mathcal{I}_{\mathcal{R}_0} \Vert_{\Ld^2(0,T;\H^1)}+\Vert \mathcal{I}_{\mathcal{R}_0}^{1} \Vert_{\Ld^2(0,T;\H^1)} \lesssim  \Lambda(T,R).
\end{align*}
For the estimate in $\Ld^2(0,T; \Ld^2)$ of the last term $\mathcal{I}_{\mathcal{R}_1}^0$, we end up with the conclusion thanks to the inequality \eqref{estim1L2:R_1} in Lemma \ref{LM:ALLapplyD}, combined with the same arguments as before.
\end{proof}
Note that the previous lemma does not give any control on $\Vert  \nabla_x\mathcal{I}_{\mathcal{R}_1}^0\Vert_{\Ld^2(0,T;\Ld^2)} + \Vert \nabla_x \mathcal{I}_{\mathcal{R}_1}^{1} \Vert_{\Ld^2(0,T;\Ld^2)} $. 
The treatment of these terms requires additional arguments that we will develop in the next subsections. For now, we merely state the result and postpone the proof to the end of Subsection \ref{Subsec:LastRemainders}.

\begin{lem}\label{LM:estimateGrad-I_R1}
We have
\begin{align}
\label{estimate:I_R1-H1}\Vert \mathcal{I}_{\mathcal{R}_1}^0\Vert_{\Ld^2(0,T;\H^1)} + \Vert \mathcal{I}_{\mathcal{R}_1}^{1} \Vert_{\Ld^2(0,T;\H^1)}  \leq \Lambda(T,R).
\end{align}
\end{lem}

\subsection{The leading terms and the conclusion}

In this section, we focus on the following two terms:
\begin{align*}
\mathfrak{L}^0(t,x):=-\int_0^t e^{d(t-s)} \int_{\R^d}  \mathrm{N}^{t;s}(\mathrm{Z}^{0;t}(x,v))  \mathcal{L}(s, \mathrm{Z}^{s;t}(x,v)) \, \mathrm{d}v \, \mathrm{d}s,
\end{align*}
and
\begin{align*}
\mathfrak{L}^{1}(t,x):=-\int_0^t e^{d(t-s)} \int_{\R^d}   v \otimes \mathrm{N}^{t;s}(\mathrm{Z}^{0;t}(x,v))  \mathcal{L}(s, \mathrm{Z}^{s;t}(x,v)) \, \mathrm{d}v \, \mathrm{d}s.
\end{align*}
The goal is to prove that $\mathfrak{L}^0$ and $\mathfrak{L}^1$ can be decomposed as a sum of a leading term and a remainder in $\Ld^2_T \H^1$. This will imply the result stated in Proposition \ref{coroFInal:D^I:rho-j}. Since the treament of $\mathfrak{L}^1$ is similar, we focus on $\mathfrak{L}^0$. 

First, we have the following decomposition, which introduces several remainder terms that we shall estimate later on. Recall the definition of the straightening diffeomorphism $\psi_{s,t}(x,v)$ from Lemma \ref{straight:velocity}.

\begin{lem}\label{LM:decompoLeading}
For $\vert I \vert \leq m$, we have
\begin{multline*}
[\mathfrak{L}^0]_{(I,0)}(t,x)=-\int_0^t  \int_{\R^d}   \partial_x^I E_{\reg,\eps}^{u,\varrho}(s, x-(t-s)v) \cdot \nabla_v f(t,x,v)  \, \mathrm{d}v \, \mathrm{d}s \\
+ \mathscr{R}_{I}^{\mathrm{Diff}}(t,x)+\mathscr{R}_{I,1}^{\mathrm{Duha}}(t,x)+\mathscr{R}_{I,2}^{\mathrm{Duha}}(t,x),
\end{multline*}
with
\begin{align*}
\mathscr{R}_{I}^{\mathrm{Diff}}(t,x)&:=-\int_0^t  \int_{\R^d}  \left[ \partial_x^I E_{\reg,\eps}^{u,\varrho}(s, x+(1-e^{t-s})v)-\partial_x^I E_{\reg,\eps}^{u,\varrho}(s, x-(t-s)v) \right] \cdot \nabla_v f(t,x,v) \, \mathrm{d}v \, \mathrm{d}s, \\
\mathscr{R}_{I,1}^{\mathrm{Duha}}(t,x)&:=-\sum_{K}\int_0^t  \int_{\R^d}   \partial_x^K E_{\reg,\eps}^{u,\varrho}(s, x+(1-e^{t-s})v) \cdot\mathrm{H}^{K,I}(s,t,x,v) \, \mathrm{d}v \, \mathrm{d}s, \\
\mathscr{R}_{I,2}^{\mathrm{Duha}}(t,x)&:=-\sum_{K}\int_0^t \int_{\R^d} \partial_x^K E_{\reg,\eps}^{u,\varrho}(s, x+(1-e^{t-s}v)) \cdot \mathfrak{H}^{K,I}(t,s,x,v)\, \mathrm{d}v \, \mathrm{d}s,
\end{align*}
where $\mathrm{H}^{K,I}$ and $\mathfrak{H}^{K,I}$ are vector fields defined by
\begin{align}\label{def:KernelH1}
\mathrm{H}^{K,I}(t,s,x,v):=\mathrm{N}^{t;s}(\mathrm{Z}^{0;t}(x,\psi_{s,t}(x,v)))_{(I,0),(K,0)}\mathrm{J}^{s,t}(x,v)   \nabla_v f(t,x,\psi_{s,t}(x,v)) -     \nabla_v f(t,x,v),
\end{align}
and 
\begin{multline}\label{def:KernelH2}
 \mathfrak{H}^{K,I}(t,s,x,v) :=\int_s^t e^{d(t-\tau)}   \mathrm{N}^{t;s}(\mathrm{Z}^{0;t}(x,\psi_{s,t}(x,v)))_{(I,0),(K,0)}   \\
  \big( \nabla_x f(\tau, \mathrm{Z}^{\tau;t}(x,\psi_{s,t}(x,v))) -\nabla_v f(\tau, \mathrm{Z}^{\tau;t}(x,\psi_{s,t}(x,v))) \big) \, \mathrm{J}^{s,t}(x,v) \, \mathrm{d}\tau,
\end{multline}
with 
\begin{align*}
\mathrm{J}^{s,t}(x,w)=\vert \det \left(\mathrm{D}_w  \psi_{s,t}(x,w) \right) \vert.
\end{align*}
\end{lem}
\begin{proof}
Let $T \in (0, \min \left(T_\eps(R),\overline{T}(R)\right)$.
We have for $\vert I \vert \leq m$
\begin{align}\label{exp:L^0-beforeCHARACchange}
\begin{split}
&[\mathfrak{L}^0]_{(I,0)}(t,x)\\
&\qquad    =-\sum_{(K,L)}\int_0^t e^{d(t-s)} \int_{\R^d}  \mathrm{N}^{t;s}(\mathrm{Z}^{0;t}(x,v))_{(I,0),(K,L)} \partial^{K}_x \partial^{L}_v E_{\reg,\eps}^{u,\varrho}(s, \mathrm{X}^{s;t}(x,v)) \cdot \nabla_v f(s,\mathrm{Z}^{s;t}(x,v)) \, \mathrm{d}v \, \mathrm{d}s \\
&\qquad    =-\sum_{K}\int_0^t e^{d(t-s)} \int_{\R^d}  \mathrm{N}^{t;s}(\mathrm{Z}^{0;t}(x,v))_{(I,0),(K,0)} \partial^{K}_x E_{\reg,\eps}^{u,\varrho}(s, \mathrm{X}^{s;t}(x,v)) \cdot \nabla_v f(s,\mathrm{Z}^{s;t}(x,v)) \, \mathrm{d}v \, \mathrm{d}s,
\end{split}
\end{align}
since $E_{\reg,\eps}^{u,\varrho}$ does not depend on the $v$ variable.
By Lemma \ref{LM:applyD-kin}, we know that
\begin{align*}
\mathcal{T}_{\reg,\eps}^{u,\varrho}(\nabla_v f)=\nabla_v f -\nabla_x f.
\end{align*}
Invoking Duhamel formula, we get
\begin{align*}
\nabla_v f(t,x,v)=e^{d(t-s)} \nabla_v f(s, \mathrm{Z}^{s;t}(x,v)) + \int_s^t e^{d(t-\tau)} \left( \nabla_v f(\tau, \mathrm{Z}^{\tau;t}(x,v))-\nabla_x f(\tau, \mathrm{Z}^{\tau;t}(x,v)) \right) \, \mathrm{d}\tau,
\end{align*}
and therefore
\begin{align*}
e^{d(t-s)} \nabla_v f(s, \mathrm{Z}^{s;t}(x,v)) =\nabla_v f(t,x,v)+\int_s^t e^{d(t-\tau)} \left( \nabla_x f(\tau, \mathrm{Z}^{\tau;t}(x,v))-\nabla_v f(\tau, \mathrm{Z}^{\tau;t}(x,v)) \right) \, \mathrm{d}\tau.
\end{align*}
Inserting this expression in \eqref{exp:L^0-beforeCHARACchange} yields
\begin{align}\label{decompo:L}
\begin{split}
[\mathfrak{L}^0]_{(I,0)}(t,x)&=\mathscr{L}_I^1(t,x)+\mathscr{L}_I^2(t,x), \\
\end{split}
\end{align}
where 
\begin{align*}
\mathscr{L}_I^1(t,x)&:=-\sum_{K}\int_0^t  \int_{\R^d}  \mathrm{N}^{t;s}(\mathrm{Z}^{0;t}(x,v))_{(I,0),(K,0)} \partial^{K}_x E_{\reg,\eps}^{u,\varrho}(s, \mathrm{X}^{s;t}(x,v)) \cdot\nabla_v f(t,x,v) \, \mathrm{d}v \, \mathrm{d}s, \\
\mathscr{L}_I^2(t,x)&:=-\sum_{K}\int_0^t  \int_{\R^d} \int_s^t e^{d(t-\tau)} \mathrm{N}^{t;s}(\mathrm{Z}^{0;t}(x,v))_{(I,0),(K,0)} \partial^{K}_x E_{\reg,\eps}^{u,\varrho}(s, \mathrm{X}^{s;t}(x,v)) \\
& \qquad \qquad \qquad \qquad \qquad \qquad \qquad \qquad  {\cdot}   \left( \nabla_x f(\tau, \mathrm{Z}^{\tau;t}(x,v))-\nabla_v f(\tau, \mathrm{Z}^{\tau;t}(x,v)) \right) \, \mathrm{d}\tau \, \mathrm{d}v \, \mathrm{d}s.
\end{align*}
Let us transform these two expressions in order to make the terms $\mathscr{R}_I^1(t,x)$, $\mathscr{R}_I^2(t,x)$ and $\mathscr{R}_I^3(t,x)$ appear.

$\bullet$ We first focus on $\mathscr{L}_I^1$, which will produce the leading term in the result. Using the change of variable $v=\psi_{s,t}(x,w)$ coming from Lemma \ref{straight:velocity}, we have (since $t \leq \overline{T}(R)$)
\begin{align*}
\mathscr{L}_I^1(t,x)
&=-\sum_{K}\int_0^t  \int_{\R^d}  \mathrm{N}^{t;s}(\mathrm{Z}^{0;t}(x,\psi_{s,t}(x,v)))_{(I,0),(K,0)} \\
& \quad \qquad \qquad \qquad \qquad \qquad \qquad\partial_x^K E_{\reg,\eps}^{u,\varrho}(s, x+(1-e^{t-s})v) \cdot \nabla_v f(t,x,\psi_{s,t}(x,v))  \mathrm{J}^{s,t}(x,v) \, \mathrm{d}v \, \mathrm{d}s,
\end{align*}
where $\mathrm{J}^{s,t}(x,w)=\vert \det \left(\mathrm{D}_w  \psi_{s,t}(x,w) \right) \vert$. We obtain
\begin{align*}
\mathscr{L}_I^1(t,x)&=-\sum_{K}\int_0^t  \int_{\R^d}   \partial_x^K E_{\reg,\eps}^{u,\varrho}(s, x+(1-e^{t-s})v) \cdot \nabla_v f(t,x,\psi_{t,t}(x,v)) \\
& \quad \qquad \qquad \qquad \qquad \qquad \qquad \qquad \mathrm{N}^{t;t}(\mathrm{Z}^{0;t}(x,\psi_{t,t}(x,v)))_{(I,0),(K,0)} \mathrm{J}^{t,t}(x,v) \, \mathrm{d}v \, \mathrm{d}s \\
& \quad -\sum_{K} \int_0^t  \int_{\R^d} \partial_x^K E_{\reg,\eps}^{u,\varrho}(s, x+(1-e^{t-s})v) \cdot \mathrm{H}^{K,I}(t,s,x,v) \, \mathrm{d}v \, \mathrm{d}s,
\end{align*}
with
\begin{multline*}
\mathrm{H}^{K,I}(t,s,x,v):=\mathrm{N}^{t;s}(\mathrm{Z}^{0;t}(x,\psi_{s,t}(x,v)))_{(I,0),(K,0)}\mathrm{J}^{s,t}(x,v)   \nabla_v f(t,x,\psi_{s,t}(x,v)) \\[2mm]
 \quad   - \mathrm{N}^{t;t}(\mathrm{Z}^{0;t}(x,\psi_{t,t}(x,v)))_{(I,0),(K,0)} \mathrm{J}^{t,t}(x,v)       \nabla_v f(t,x,v).
\end{multline*}
Now observe that since $\mathrm{N}^{t;t}= \mathrm{Id}$ and $\psi_{t,t}=\mathrm{Id}$, we have
\begin{align*}
\nabla_v f(t,x,\psi_{t,t}(x,v))\mathrm{N}^{t;t}(\mathrm{Z}^{0;t}(x,\psi_{t,t}(x,v)))_{(I,0),(K,0)} \mathrm{J}^{t,t}(x,v)=\mathbf{1}_{I=K} \nabla_v f(t,x,v),
\end{align*}
therefore 
\begin{align*}
\mathscr{L}_I^1(t,x)&=-\int_0^t  \int_{\R^d}   \partial_x^I E_{\reg,\eps}^{u,\varrho}(s, x+(1-e^{t-s})v) \cdot \nabla_v f(t,x,v)  \, \mathrm{d}v \, \mathrm{d}s \\
 & \quad-\sum_{K}\int_0^t  \int_{\R^d}   \partial_x^K E_{\reg,\eps}^{u,\varrho}(s, x+(1-e^{t-s})v) \cdot\mathrm{H}^{K,I}(s,t,x,v) \, \mathrm{d}v \, \mathrm{d}s \\
 &=-\int_0^t  \int_{\R^d}   \partial_x^I E_{\reg,\eps}^{u,\varrho}(s, x-(t-s)v) \cdot \nabla_v f(t,x,v)  \, \mathrm{d}v \, \mathrm{d}s + \mathscr{R}_{I}^{\mathrm{Diff}}(t,x)+\mathscr{R}_{I,1}^{\mathrm{Duha}}(t,x).
\end{align*}

$\bullet$ To deal with the term $\mathscr{L}_I^2$, we apply again the change of variable $v=\psi_{s,t}(x,w)$ from Lemma \ref{straight:velocity} and get (since $T \leq \overline{T}(R)$)
\begin{align*}
\mathscr{L}_I^2(t,x)
&=-\sum_{K}\int_0^t \int_{\R^d} \partial_x^K E_{\reg,\eps}^{u,\varrho}(s, x+(1-e^{t-s}v)) \cdot \mathfrak{H}^{K,I}(t,s,x,v)\, \mathrm{d}v \, \mathrm{d}s,
\end{align*}
where
\begin{multline*}
 \mathfrak{H}^{K,I}(t,s,x,v) :=\int_s^t e^{d(t-\tau)}   \mathrm{N}^{t;s}(\mathrm{Z}^{0;t}(x,\psi_{s,t}(x,v)))_{(I,0),(K,0)}   \\
  \big( \nabla_x f(\tau, \mathrm{Z}^{\tau;t}(x,\psi_{s,t}(x,v))) -\nabla_v f(\tau, \mathrm{Z}^{\tau;t}(x,\psi_{s,t}(x,v))) \big) \, \mathrm{J}^{s,t}(x,v) \, \mathrm{d}\tau,
\end{multline*}
which means that $\mathscr{L}_I^2(t,x)=\mathscr{R}_{I,2}^{\mathrm{Duha}}(t,x)$. Combining the previous decompositions eventually yields the conclusion.

\end{proof}

We now have the following lemma, which is the continuation of Lemma \ref{LM:decompoLeading}, in which we express $[\mathfrak{L}^0]_{(I,0)}$ as a sum of a leading term and a remainder which is controlled in $\Ld^2_T \H^1$. The proof, which is rather lengthy and technical, is based on the smoothing estimates of Section \ref{Section:Averag-lemma}. We refer to Subsection \ref{Subsec:LastRemainders} where the proof is postponed.

\begin{lem}\label{LM:estimateL0}
We have for all $|I|=m$,
\begin{align*}
[\mathfrak{L}^0]_{(I,0)}(t,x)=p'(\varrho(t,x)) \int_0^t  \int_{\R^d}   \nabla_x [\mathrm{J}_\eps \partial_x^I\varrho](s, x-(t-s)v) \cdot \nabla_v f(t,x,v) \, \mathrm{d}v \, \mathrm{d}s + \mathscr{R}_I(t,x),
\end{align*}
with
\begin{align*}
\Vert \mathscr{R}_I \Vert_{\Ld^2(0,T, \H^1)} \leq \Lambda(T,R).
\end{align*}
\end{lem}

We can finally proceed with the proof of Proposition~\ref{coroFInal:D^I:rho-j}.
\begin{proof}[Proof of Proposition~\ref{coroFInal:D^I:rho-j}]
We only treat the case of $ \partial_x^I \rho_{f}$, the one of $ \partial_x^I j_{f}$ being similar. First invoking Lemma \ref{LM:decompo:Moment1}, we have
\begin{align*}
 \partial_x^I \rho_{f}&=[\mathcal{I}_{in}^{0}]_{(I,0)}+[\mathcal{I}_{\mathcal{R}_0}^0]_{(I,0)}+[\mathcal{I}_{\mathcal{R}_1}^0]_{(I,0)}+[\mathfrak{L}^0]_{(I,0)}.
 \end{align*}
Thanks to Lemmas \ref{LM:estimateI_0R}--\ref{LM:estimateGrad-I_R1}--\ref{LM:estimateL0}, we infer
 \begin{align*}
 \partial_x^I \rho_{f}(t,x)&=p'(\varrho(t,x)) \int_0^t  \int_{\R^d}   \nabla_x [\mathrm{J}_\eps \partial_x^I\varrho](s, x-(t-s)v) \cdot \nabla_v f(t,x,v) \, \mathrm{d}v \, \mathrm{d}s \\
 & \quad +[\mathcal{I}_0^{0}]_{(I,0)}(t,x)+[\mathcal{I}_{\mathcal{R}_0}^0]_{(I,0)}(t,x)+[\mathcal{I}_{\mathcal{R}_1}^0]_{(I,0)}(t,x)+ \mathscr{R}_I(t,x),
\end{align*}
where the previous remainders are estimated as
\begin{align*}
\Vert \mathcal{I}_{in}^0 \Vert_{\Ld^2(0,T;\H^1)} \Vert  &\leq \Lambda(T,R) \Vert f^{\mathrm{in}} \Vert_{\mathcal{H}^{m+1}_r}, \\
\Vert \mathcal{I}_{\mathcal{R}_0}^0\Vert_{\Ld^2(0,T;\H^1)}+\Vert \mathcal{I}_{\mathcal{R}_1}^0\Vert_{\Ld^2(0,T;\H^1)}   &\leq  \Lambda(T,R), \\
\Vert \mathscr{R}_I \Vert_{\Ld^2(0,T, \H^1)} &\leq \Lambda(T,R).
\end{align*}
This concludes the proof.
\end{proof}

\subsection{Estimates of the last remainders}\label{Subsec:LastRemainders}

In this subsection, we mainly aim at giving a proof for Lemma \ref{LM:estimateL0} and Lemma \ref{LM:estimateGrad-I_R1}, that we have previously stated. We shall rely on the crucial smoothing estimates derived in Section \ref{Section:Averag-lemma} to treat the different remainders. Broadly speaking, there are three types of terms requiring a gain of regularity:
\begin{itemize}
\item[$-$] \textbf{Type I}: terms that will be treated thanks to the continuity estimate of Proposition \ref{propo:Averag-ORiGINAL}, as in \cite{HKR}, and of Proposition \ref{propo:AveragStandard}. It will be used for the terms $ \nabla_x \mathcal{I}_{\mathcal{R}_1}^0$ and $ \nabla_x \mathcal{I}_{\mathcal{R}_1}^1$ in the proof of Lemma \ref{LM:estimateGrad-I_R1}. 
\item[$-$] \textbf{Type II}: terms involving a kernel vanishing on the diagonal in time. They will be treated by the regularization estimate of Proposition \ref{propo:AveragReg}. Such terms will appear in the proof of Lemma \ref{LM:estimateL0}, as well as for the remainders $\mathscr{R}_{I}^{\mathrm{Diff}}, \mathscr{R}_{I,1}^{\mathrm{Duha}}$ and $\mathscr{R}_{I,2}^{\mathrm{Duha}}$.
\item[$-$] \textbf{Type III}: terms involving the difference between the integral operators $\mathrm{K}^{\mathrm{free}}$ and $\mathrm{K}^{\mathrm{fric}}$ (see Section \ref{Section:Averag-lemma}. They will be handled thanks to the regularization estimate of Proposition \ref{propo:AveragDiff}. They will also appear in the treatment of the different remainders.
\end{itemize}
Let us mention that the previous gains require to control a fixed number of derivatives of the kernels that are involved (see Section \ref{Section:Averag-lemma}).

\medskip

Recall also the expression
\begin{align*}
E_{\reg,\eps}^{u,\varrho}(t,x)=u(t,x)-p'(\varrho)\nabla_x \left[\mathrm{J}_{\eps}\varrho \right](t,x),
\end{align*}
as well as Definition \ref{def:indices-shift} for shifted indices. In order to check that the assumptions of the smoothing estimates of Section \ref{Section:Averag-lemma} are satisfied, it is convenient to have the following result.

\begin{lem}\label{LM:decompoForce-gradient}
For any $K \in \N^d$ such that $\vert K \vert>0$, we have
\begin{multline*}
\partial_x^K E_{\reg,\eps}^{u,\varrho}=-p'(\varrho) \nabla_x [\partial_x^K\mathrm{J}_{\eps}\varrho]+\partial_x^K u
-\sum_{\ell=0}^{\lfloor \frac{\vert K \vert-1}{2} \rfloor} \sum_{\substack{i=1,{\cdots},d \\ \beta \in \mathbb{B}_K(i ,\ell)}} \binom{K}{\widehat{\beta}^i} \nabla_x(\partial_x^{\overline{K}^i-\beta} \mathrm{J}_{\eps}\varrho)\nabla_x (\partial_x^{\beta}p'(\varrho) ) \cdot e_i  \\
-\sum_{\ell=\lfloor \frac{\vert K \vert-1}{2} \rfloor+1}^{\vert K \vert-1} \sum_{\substack{i=1,{\cdots},d \\ \beta \in \mathbb{B}_K(i ,\ell)}} \binom{K}{\widehat{\beta}^i} \nabla_x (\partial_x^{\beta}p'(\varrho) ) \cdot e_i \nabla_x(\partial_x^{\overline{K}^i-\beta}\mathrm{J}_{\eps} \varrho),
\end{multline*}
where
\begin{align*}
\mathbb{B}_K(i ,\ell):=\left\lbrace \beta \in \N^d \mid \vert \beta \vert=\ell, \ \ 0<\widehat{\beta}^i \leq K \right\rbrace, \ \ i=1,{\cdots},d, \ \ \ell =0,{\cdots}, \vert K \vert-1 .
\end{align*}
\end{lem}

\begin{proof}
The proof directly follows from Leibniz formula, which provides
\begin{align*}
\partial_x^K E_{\reg,\eps}^{u,\varrho}=-p'(\varrho) \nabla_x \partial_x^K\varrho+\partial_x^K u-\sum_{\ell=0}^{\vert K \vert-1} \sum_{\substack{i=1,{\cdots},d \\ \beta \in \mathbb{B}_K(i ,\ell)}}\binom{K}{\widehat{\beta}^i} \nabla_x (\partial_x^{\beta}p'(\varrho) ) \cdot e_i \nabla_x(\partial_x^{\overline{K}^i-\beta} \mathrm{J}_\eps \varrho),
\end{align*}
and from which we infer the conclusion.
\end{proof}

In the next lemma, we show how to obtain the leading term of Lemma \ref{LM:estimateL0} (up to some good remainder) from the integral term of Lemma \ref{LM:decompoLeading}.

\begin{lem}\label{LM:makeLeadingappear}
We have
\begin{multline*}
-\int_0^t  \int_{\R^d}   \partial_x^I E_{\reg,\eps}^{u,\varrho}(s, x-(t-s)v) \cdot \nabla_v f(t,x,v)  \, \mathrm{d}v \, \mathrm{d}s \\=p'(\varrho(t,x)) \int_0^t  \int_{\R^d}   \nabla_x [ \mathrm{J}_\eps\partial_x^I\varrho](s, x-(t-s)v) \cdot \nabla_v f(t,x,v) \, \mathrm{d}v \, \mathrm{d}s +\widetilde{\mathscr{R}}_{I}(t,x),
\end{multline*}
where the remainder $\widetilde{\mathscr{R}}_{I}$ satisfies
\begin{align*}
\LRVert{\widetilde{\mathscr{R}}_{I}}_{\Ld^2(0,T; \H^1)} \leq \Lambda(T,R).
\end{align*}
\end{lem}
\begin{proof}
Let us introduce the following vector fields
\begin{align*}
G_1(s,t,x,v)=[p'(\varrho(s,x-(t-s)v)-p'(\varrho(t,x))] \nabla_v f(t,x,v),
\end{align*}
and 
\begin{align*}
G_{3,i}^{\beta}(s,t,x,v)&:=\Big(\nabla_x (\partial_x^{\beta}p'(\varrho) )(s, x-(t-s)v) \cdot e_i \Big) \nabla_v f(t,x,v), \\[2mm]
G_{4,i }^{K, \beta}(s,t,x,v) &:=\Big( \nabla_x(\partial_x^{\overline{K}^i-\beta} \mathrm{J}_\eps \varrho)(s, x-(t-s)v)\cdot \nabla_v f(t,x,v)  \Big)e_i,
\end{align*}
for
\begin{align*}
\beta \in \mathbb{B}_K(i ,\ell)=\left\lbrace \beta \in \N^d \mid \vert \beta \vert=\ell, \ \ 0<\widehat{\beta}^i \leq K \right\rbrace, \ \ i=1,{\cdots},d, \ \ \ell =0,{\cdots}, \vert K \vert-1 .
\end{align*}
Thanks to Lemma \ref{LM:decompoForce-gradient}, we can write
\begin{multline*}
-\int_0^t  \int_{\R^d}   \partial_x^I E_{\reg,\eps}^{u,\varrho}(s, x-(t-s)v) \cdot \nabla_v f(t,x,v)  \, \mathrm{d}v \, \mathrm{d}s \\ 
=p'(\varrho(t,x))\int_0^t  \int_{\R^d} \nabla_x [\mathrm{J}_\eps \partial_x^I \varrho](s, x-(t-s)v) \cdot \nabla_v f(t,x,v) \, \mathrm{d}v \, \mathrm{d}s +\mathbf{S_1}(t,x)+\mathbf{S_2}(t,x) + \mathbf{S_3}(t,x)+\mathbf{S_4}(t,x),
\end{multline*}
where
\begin{align*}
\mathbf{S_1}(t,x)
&:=\mathrm{K}_{G_1}^{\mathrm{free}}[\partial_x^I \mathrm{J}_\eps \varrho],\\
\mathbf{S_2}(t,x)&:= -\int_0^t  \int_{\R^d}  \partial_x^I u(s, x-(t-s)v) \cdot \nabla_v f(t,x,v) \, \mathrm{d}v \, \mathrm{d}s , \\
\mathbf{S_3}(t,x)&:=\sum_{\ell=0}^{\lfloor \frac{\vert I \vert-1}{2} \rfloor} \sum_{\substack{i=1,{\cdots},d \\ \beta \in \mathbb{B}_I(i ,\ell)}} \binom{I}{\widehat{\beta}^i} \mathrm{K}_{G_{3,i}^{\beta}}^{\mathrm{free}}\left[\partial_x^{\overline{I}^i-\beta} \mathrm{J}_\eps \varrho \right], \\
\mathbf{S_4}(t,x)&:=\sum_{\ell=\lfloor \frac{\vert I \vert-1}{2} \rfloor+1}^{\vert I \vert-1} \sum_{\substack{i=1,{\cdots},d \\ \beta \in \mathbb{B}_I(i ,\ell)}} \binom{I}{\widehat{\beta}^i} \mathrm{K}_{G_{4,i}^{\overline{I}^i,\beta}}^{\mathrm{free}}\Big[\partial_x^{\beta}p'(\varrho)\Big].
\end{align*}
The treatment of the term $\mathbf{S_2}$ will follow from a straightforward estimate. The terms $\mathbf{S_3}$ and $\mathbf{S_4}$ are terms of \textbf{Type I} and we will use the continuity estimates provided by Proposition \ref{propo:Averag-ORiGINAL}. The term $\mathbf{S_1}$, which already contains $I$ derivatives of $\varrho$, is a term of \textbf{Type II} (since $G_1(t,t,x,v)=0$) and we will rely on the regularization estimate of Proposition \ref{propo:AveragReg}.

\medskip

$\bullet$ \textbf{Estimate of $\mathbf{S_1}$}: Since $G_1(t,t,x,v)=0$, we use Proposition \ref{propo:AveragReg} to get
\begin{align*}
\LRVert{\mathbf{S_1}}_{\Ld^2(0,T; \H^1)} &\lesssim (1+T)\underset{0\leq s,t \leq T}{\sup} \,  \Vert \partial_s G_1(s,t) \Vert_{\mathcal{H}^{\ell}_{\sigma}}\Vert \partial_x^I\varrho \Vert_{\Ld^2(0,T; \Ld^2)} \\
& \lesssim (1+T)\underset{0\leq s,t \leq T}{\sup} \,  \Vert \partial_s G_1(s,t) \Vert_{\mathcal{H}^{\ell}_{\sigma}}\Vert \varrho \Vert_{\Ld^2(0,T; \H^m)} ,
\end{align*}
for $\ell >7+d$ and $\sigma >d/2$. A direct computation gives
\begin{align*}
\partial_s G_1(s,t,x,v)=p''(\varrho)\left[ \partial_s \varrho + v \cdot\nabla_x \varrho\right](s,x-(t-s)v) \nabla_v f(t,x,v).
\end{align*}
We thus have for all $0 \leq s,t \leq T$
\begin{align*}
\Vert \partial_s G_1(s,t,x,v) \Vert_{\mathcal{H}^{\ell}_{\sigma}}^2 &\lesssim C_T\sum_{\vert \mu \vert + \vert \nu \vert \leq \ell}\sum_{\gamma=0}^{\mu + \nu}\left( \Vert  \partial_x^{\gamma}(p''(\varrho)\partial_s \varrho (s))\Vert_{\Ld^{\infty}}^2+\Vert  \partial_x^{\gamma} (p''(\varrho)\nabla_x \varrho (s))\Vert_{\Ld^{\infty}}^2 \right)\\
& \qquad \qquad \qquad \qquad \qquad \qquad  \times  \int_{\T^d \times \R^d} \langle v \rangle^{2\sigma+2}  \vert\partial^{\mu+ \nu-\gamma}_{x,v} \nabla_v f(t,x,v) \vert^2 \, \mathrm{d}x \, \mathrm{d}v \\
& \lesssim C_T \left( \LRVert{p''(\varrho) \partial_s \varrho(s)}_{\H^k}^2 +\LRVert{p''(\varrho) \nabla_x \varrho(s)}_{\H^k}^2\right)  \Vert f(t) \Vert_{\mathcal{H}^{m-1}_{\sigma+1}}^2 \\
& \lesssim C_T \LRVert{p''(\varrho(s))}_{\H^k}^2\left( \LRVert{\partial_s \varrho(s)}_{\H^k}^2 +\LRVert{ \nabla_x \varrho(s)}_{\H^k}^2\right)  \Vert f(t) \Vert_{\mathcal{H}^{m-1}_{\sigma+1}}^2,
\end{align*}
for $k>\frac{d}{2}+\ell>\frac{d}{2}+1+d$ and $m-1>\ell>1+d$. Using the equation satisfied by $\varrho$, we get
\begin{align*}
\LRVert{p''(\varrho(s))}_{\H^k} &\leq \Lambda \left( \LRVert{\varrho(s)}_{\H^k} \right) \LRVert{\varrho(s)}_{\H^k}, \\
\LRVert{\partial_s \varrho(s)}_{\H^k} +\LRVert{ \nabla_x \varrho(s)}_{\H^k} &\lesssim \LRVert{u}_{\H^k}\LRVert{ \nabla_x \varrho}_{\H^{k}} +\left\Vert \frac{\varrho}{1-\rho_f} \right\Vert_{\H^k}  \LRVert{ \mathrm{div}_x \left(j_f-\rho_fu+u \right)}_{\H^{k}} + \LRVert{\varrho}_{\H^{k+1}},
\end{align*}
thanks to the algebra property of $\H^k$. Taking $k+1<m-3$ and using the bootstrap assumption combined with (the proof) of Lemma \ref{LM:rho-pointwise-Hm-2}, we obtain
\begin{align*}
\underset{0\leq s,t \leq T}{\sup} \,  \Vert \partial_s G_1(s,t) \Vert_{\mathcal{H}^{\ell}_{\sigma}} \lesssim \Lambda(T,R).
\end{align*}

\medskip

$\bullet$ \textbf{Estimate of $\mathbf{S_2}$}:
Using the generalized Minkowski inequality followed by the Cauchy-Schwarz inequality, we have for $r>d/2$
\begin{align*}
\LRVert{\mathbf{S_2}}_{\Ld^2(0,T;\Ld^2)} &\leq \left\Vert \int_0^t  \left(\int_{\T^d \times \R^d}  (1+ \vert v \vert^2)^r \vert \partial_x^I u(s, x-(t-s)v) \vert^2 \vert \nabla_v f(t,x,v) \vert^2 \, \mathrm{d}x \, \mathrm{d}v  \right)^{1/2}\, \mathrm{d}s  \right\Vert_{\Ld^2(0,T)} \\
& \leq  \left\Vert \int_0^t  \left(\int_{\T^d \times \R^d}  (1+ \vert v \vert^2)^r \vert \partial_x^I u(s, y) \vert^2 \vert \nabla_v f(t,y+(t-s)v,v) \vert^2 \, \mathrm{d}y \, \mathrm{d}v  \right)^{1/2}\, \mathrm{d}s  \right\Vert_{\Ld^2(0,T)} \\
& \leq \underset{t \in [0,T]}{\sup} \left( \int_{\R^d} (1+ \vert v \vert^2)^r \Vert \nabla_v f(t,\cdot,v) \Vert^2_{\Ld^{\infty}} \, \mathrm{d}v \right)^{1/2} \left\Vert \int_0^t  \Vert \partial_x^I u(s) \Vert_{\Ld^2}  \mathrm{d}s  \right\Vert_{\Ld^2(0,T)} \\
& \lesssim \Vert f \Vert_{\Ld^{\infty}(0,T; \mathcal{H}^{m-1}_{r})} \Vert u \Vert_{\Ld^{2}(0,T;\H^m)},
\end{align*}
by Sobolev embedding in the last line, since $m-1>1+d/2$. This yields
\begin{align*}
\LRVert{\mathbf{S_2}}_{\Ld^2(0,T;\Ld^2)} \lesssim \Lambda(T,R).
\end{align*}
Likewise, we have since $m-1>2+d/2$
\begin{align*}
\LRVert{ \nabla_x \mathbf{S_2}}_{\Ld^2(0,T;\Ld^2)} &\lesssim \Vert f \Vert_{\Ld^{\infty}(0,T; \mathcal{H}^{m-1}_{r})} \Vert u \Vert_{\Ld^{2}(0,T;\H^m)}+ \Vert f \Vert_{\Ld^{\infty}(0,T; \mathcal{H}^{m-1}_{r})} \Vert u \Vert_{\Ld^{2}(0,T;\H^{m+1})} \\
& \leq \Lambda(T,R),
\end{align*}
which gives the conclusion.

\medskip

$\bullet$ \textbf{Estimate of $\mathbf{S_3}$ and $\mathbf{S_4}$}: for all $0 \leq \vert \beta \vert \leq \lfloor \frac{\vert I \vert-1}{2} \rfloor$ and $i=1, \cdots, d$, we use Proposition \ref{propo:Averag-ORiGINAL} (after taking one derivative in space) to get
\begin{align*}
\left\Vert \mathrm{K}^{\mathrm{free}}_{G_{3,i}^{\beta}}\left[\partial_x^{\overline{I}^i-\beta} \mathrm{J}_\eps \varrho \right] \right\Vert_{\Ld^2(0,T; \H^1)} \lesssim \underset{0 \leq s,t \leq T}{\sup} \Vert G_{3,i}^{\beta}(t,s) \Vert_{\mathcal{H}^{\ell+1}_{\sigma}} \left\Vert  \partial_x^{\overline{I}^i-\beta} \varrho \right\Vert_{\Ld^2(0,T; \H^1)},
\end{align*}
for $\ell>1+d$ and $\sigma>d/2$ therefore
\begin{align*}
\left\Vert \mathrm{K}^{\mathrm{free}}_{G_{3,i}^{\beta}}\left[\partial_x^{\overline{I}^i-\beta}\mathrm{J}_\eps \varrho \right] \right\Vert_{\Ld^2(0,T; \H^1)} \lesssim  \underset{0 \leq s,t \leq T}{\sup} \Vert G_{3,i}^{\beta}(t,s) \Vert_{\mathcal{H}^{\ell+1}_{\sigma}} \left\Vert   \varrho \right\Vert_{\Ld^2(0,T; \H^m)}.
\end{align*}
We have for all $0 \leq s,t \leq T$
\begin{align*}
& \Vert G_{3,i}^{\beta}(t,s) \Vert_{\mathcal{H}^{\ell+1}_{\sigma}}^2 \\
& \leq C_T\sum_{\vert \mu \vert + \vert \nu \vert \leq \ell+1}\sum_{\gamma=0}^{\mu + \nu} \Vert  \partial_x^{\gamma}\nabla_x (\partial_x^{\beta}p'(\varrho) )(s)\Vert_{\Ld^{\infty}}^2 \int_{\T^d \times \R^d} \langle v \rangle^{2\sigma}  \vert\partial^{\mu+ \nu-\gamma}_{x,v} \nabla_v f(t,x,v) \vert^2 \, \mathrm{d}x \, \mathrm{d}v \\
& \lesssim C_T  \Vert p'(\varrho) \Vert_{\Ld^{\infty}(0,T;\H^k)}^2 \Vert f(t) \Vert_{\mathcal{H}^{m-1}_{\sigma}}^2,
\end{align*}
provided that $m-1 \geq \ell +2$ and $k>\frac{d}{2	}+ \vert \gamma \vert+1 + \vert \beta\vert$. Since $\ell>1+d$ and $\frac{d}{2}+ \vert \gamma \vert+1 + \vert \beta\vert \leq \frac{d}{2}+ \ell +2 + \frac{m-1}{2}$, a condition such as $3d+9<m$ ensures that 
\begin{align*}
\Vert G_{3,i}^{\beta}(t,s) \Vert_{\mathcal{H}^{\ell+1}_{\sigma}} &\lesssim C_T \Vert p'(\varrho) \Vert_{\Ld^{\infty}(0,T;\H^{m-2})} \Vert f \Vert_{\Ld^{\infty}(0,T;\mathcal{H}^{m-1}_{\sigma})} \leq \Lambda(T,R).
\end{align*}
thanks to \eqref{bound:p'BONY} from Lemma \ref{LM:Estim-ForceField}, Sobolev embedding, and  Lemma \ref{LM:rho-pointwise-Hm-2}. We thus obtain
\begin{align*}
\left\Vert \mathrm{K}^{\mathrm{free}}_{G_{3,i}^{\beta}}\left[\partial_x^{\overline{I}^i-\beta} \mathrm{J}_\eps \varrho \right] \right\Vert_{\Ld^2(0,T; \H^1)} \leq  \Lambda(T,R).
\end{align*}
We proceed in the same way for $\mathbf{S_4}$: for all $\lfloor \frac{\vert I \vert-1}{2} \rfloor+1 \leq \vert \beta \vert \leq \vert I \vert-1 $ and $i=1, \cdots, d$, we take one derivative in space and use Proposition \ref{propo:Averag-ORiGINAL} to write for $\ell>1+d$ and $\sigma>d/2$
\begin{align*}
\left\Vert \mathrm{K}^{\mathrm{free}}_{G_{4,i}^{\overline{I}^i,\beta}}\left[\partial_x^{\beta}p'(\varrho)\right] \right\Vert_{\Ld^2(0,T; \H^1)} 
 &\leq \underset{0 \leq s,t \leq T}{\sup} \Vert G_{4,i}^{\overline{I}^i,\beta}(t,s) \Vert_{\mathcal{H}^{\ell+1}_{\sigma}} \left\Vert  p'(\varrho) \right\Vert_{\Ld^2(0,T; \H^m)}.
\end{align*}
The kernel $G_{4,i}^{\overline{I}^i,\beta}$ is estimated as before: using Leibniz rule, we have for all $0 \leq s,t \leq T$
\begin{align*}
\Vert G_{4,i}^{\overline{I}^i,\beta}(t,s) \Vert_{\mathcal{H}^{\ell+1}_{\sigma}}^2 
& \lesssim C_T \Vert \varrho\Vert_{\Ld^{\infty}(0,T;\H^k)}^2 \Vert f(t) \Vert_{\mathcal{H}^{m-1}_{\sigma}}^2,
\end{align*}
provided that $m-1>2+\ell$ and $k>\frac{d}{2}+\vert \gamma \vert +1 +\vert \overline{I}^i \vert- \vert \beta \vert$. Since $\ell>1+d$ and 
$$\frac{d}{2}+\vert \gamma \vert +1 +\vert \overline{I}^i \vert- \vert \beta \vert \leq \frac{d+2 \ell+m+1}{2},$$
a condition such as $3d+9<m$ ensures that 
\begin{align*}
\Vert G_{4,i}^{\overline{I}^i,\beta}(t,s) \Vert_{\mathcal{H}^{\ell+1}_{\sigma}} &\lesssim C_T \Vert \varrho \Vert_{\Ld^{\infty}(0,T;\H^{m-2})} \Vert f \Vert_{\Ld^{\infty}(0,T;\mathcal{H}^{m-1}_{\sigma})}.
\end{align*}
We thus obtain by Sobolev embedding, the bound \eqref{bound:p'BONY} in Lemma \ref{LM:Estim-ForceField} and Lemma \ref{LM:rho-pointwise-Hm-2}.
\begin{align*}
\left\Vert \mathrm{K}^{\mathrm{free}}_{G_{4,i}^{\overline{I}^i,\beta}}\left[\partial_x^{\beta}p'(\varrho)\right] \right\Vert_{\Ld^2(0,T; \H^1)} \leq \Lambda(T,R).
\end{align*}
\end{proof}

The remaining task is to show that the remainder terms $\mathscr{R}_{I}^{\mathrm{Diff}}, \mathscr{R}_{I,1}^{\mathrm{Duha}}$ and $\mathscr{R}_{I,2}^{\mathrm{Duha}}$ introduced in Lemma \ref{LM:decompoLeading} are well-controlled in $\Ld^2(0,T; \H^1)$. For the first one, we have the following lemma.

\begin{lem}
We have 
\begin{align*}
\LRVert{\mathscr{R}_{I}^{\mathrm{Diff}}}_{\Ld^2(0,T; \H^1)} \leq \Lambda(T,R).
\end{align*}
\end{lem}
\begin{proof}
In view of Lemma \ref{LM:decompoForce-gradient}, let us write for $0\leq \vert \mu \vert \leq \lfloor \frac{\vert K \vert-1}{2} \rfloor $ and $ \lfloor \frac{\vert K \vert-1}{2} \rfloor+1 \leq \vert \nu \vert \leq \vert K \vert-1 $
\begin{align*}
& \big(\nabla_x [\partial_x^{\overline{K}^i-\mu} \mathrm{J}_\eps \varrho]\nabla_x[ \partial_x^{\mu}p'(\varrho)] {\cdot} e_i+ \nabla_x [\partial_x^{\mu}p'(\varrho)]{\cdot} e_i \nabla_x [\partial_x^{\overline{K}^i-\nu} \mathrm{J}_\eps \varrho] \big)(s,x-(t-s)v) \cdot \nabla_v f(t,x,v)\\[2mm]
& \qquad  -\big(\nabla_x [\partial_x^{\overline{K}^i-\mu} \mathrm{J}_\eps \varrho]\nabla_x[ \partial_x^{\mu}p'(\varrho)] {\cdot} e_i+ \nabla_x [\partial_x^{\mu}p'(\varrho)]{\cdot} e_i \nabla_x [\partial_x^{\overline{K}^i-\nu} \mathrm{J}_\eps \varrho]  \big)(s,x+(1-e^{t-s})v) \cdot \nabla_v f(t,x,v) \\[2mm]
& =\Big(\nabla_x [\partial_x^{\overline{K}^i-\mu} \mathrm{J}_\eps \varrho] (s,x-(t-s)v)  -\nabla_x [\partial_x^{\overline{K}^i-\mu} \mathrm{J}_\eps \varrho](s,x+(1-e^{t-s})v) \Big) \cdot G_{8,\mu,i}(s,t,x,v) \\
& \quad + \nabla_x [\partial_x^{\overline{K}^i-\mu} \mathrm{J}_\eps \varrho](s,x+(1-e^{t-s})v) \cdot G_{9,\mu,i}(s,t,x,v)\\
& \quad +\nabla_x [\partial_x^{\nu} p'(\varrho)](s,x-(t-s)v) \cdot G_{10,\nu,i}^K(s,t,x,v) \\
& \quad +  \Big( \nabla_x [\partial_x^{\nu}p'(\varrho)] (s,x-(t-s)v)-\nabla_x[\partial_x^{\nu}p'(\varrho)] (s,x+(1-e^{t-s})v) \Big) \cdot G_{11,\nu,i}^K(s,t,x,v).
\end{align*}
where 
\begin{align*}
G_{8,\mu,i}(s,t,x,v)&:=\Big(\nabla_x [\partial_x^{\mu}p'(\varrho)] (s,x-(t-s)v) {\cdot} e_i \nabla_v f(t,x,v) \Big), \\
G_{9,\mu,i}(s,t,x,v)&:=\Big((\nabla_x [\partial_x^{\mu}p'(\varrho)] {\cdot} e_i-\nabla_x[\partial_x^{\mu}p'(\varrho)](s,x+(1-e^{t-s})v) {\cdot} e_i ) \nabla_v f(t,x,v) \Big), \\
G_{10,\nu,i}^K(s,t,x,v)&:=\Big((\nabla_x [\partial_x^{\overline{K}^i-\nu} \mathrm{J}_\eps \varrho](s,x-(t-s)v)-\nabla_x [\partial_x^{\overline{K}^i-\nu} \mathrm{J}_\eps \varrho](s,x+(1-e^{t-s})v))  {\cdot} \nabla_v f(t,x,v) e_i\Big), \\
G_{11,\nu,i}^K(s,t,x,v)&:=\Big(\nabla_x [\partial_x^{\overline{K}^i-\nu} \mathrm{J}_\eps \varrho](s,x+(1-e^{t-s})v) {\cdot} \nabla_x f(t,x,v) e_i \Big).
\end{align*}
By Lemma \ref{LM:decompoForce-gradient}, we can thus rewrite
\begin{align*}
\mathscr{R}_{I}^{\mathrm{Diff}}=\mathbf{S_5}+\mathbf{S_6}+\mathbf{S_7}+\mathbf{S_8}+\mathbf{S_9}+\mathbf{S_{10}}+\mathbf{S_{11}},
\end{align*}
where
\begin{align*}
\mathbf{S_5}(t,x)&:=-\int_0^t  \int_{\R^d}  \left[ \partial_x^I u(s, x+(1-e^{t-s})v)-\partial_x^I u(s, x-(t-s)v) \right] \cdot \nabla_v f(t,x,v) \, \mathrm{d}v \, \mathrm{d}s,
\end{align*}
and 
\begin{align*}
\mathbf{S_6}&
:= \mathrm{K}_{G_6}^{\mathrm{free}}[\partial_x^I \mathrm{J}_\eps \varrho]-\mathrm{K}_{G_6}^{\mathrm{fric}}[\partial_x^I \mathrm{J}_\eps \varrho], \ \ \ ÷  G_6(t,s,x,v):=p'(\varrho(s,x-(t-s)v))\nabla_v f(t,x,v), \\[2mm]
\mathbf{S_7}
&:= \mathrm{K}_{G_7}^{\mathrm{fric}}[\partial_x^I \mathrm{J}_\eps \varrho],  \ \ G_7(t,s,x,v):=\left[ (p'(\varrho)(s, x-(t-s)v)- p'(\varrho)(s, x+(1-e^{t-s})v))\nabla_v f(t,x,v) \right], \\ 
\mathbf{S_8}&:=  \sum_{\ell=0}^{\lfloor \frac{\vert I \vert-1}{2} \rfloor} \sum_{\substack{i=1,{\cdots},d \\ \beta \in \mathbb{B}_I(i ,\ell)}} \binom{I}{\widehat{\beta}^i} \left( \mathrm{K}_{G_{8,\beta,i}}^{\mathrm{free}}[\partial_x^{\overline{I}^i-\beta} \mathrm{J}_\eps \varrho] -\mathrm{K}_{G_{8,\beta,i}}^{\mathrm{fric}}[\partial_x^{\overline{I}^i-\beta} \mathrm{J}_\eps \varrho] \right), \\
\mathbf{S_9}&:= \sum_{\ell=0}^{\lfloor \frac{\vert I \vert-1}{2} \rfloor} \sum_{\substack{i=1,{\cdots},d \\ \beta \in \mathbb{B}_I(i ,\ell)}} \binom{I}{\widehat{\beta}^i} \mathrm{K}_{G_{9,\beta,i}}^{\mathrm{fric}}\left[\partial_x^{\overline{I}^i-\beta} \mathrm{J}_\eps \varrho \right], \\
\mathbf{S_{10}}&:=\sum_{\ell=\lfloor \frac{\vert I \vert-1}{2} \rfloor+1}^{\vert I \vert-1} \sum_{\substack{i=1,{\cdots},d \\ \beta \in \mathbb{B}_I(i ,\ell)}} \binom{I}{\widehat{\beta}^i} \mathrm{K}_{G_{10,\beta,i}^I }^{\mathrm{free}}\left[\partial_x^{\beta}p'(\varrho)\right], \\
\mathbf{S_{11}}&:=\sum_{\ell=\lfloor \frac{\vert I \vert-1}{2} \rfloor+1}^{\vert I \vert-1} \sum_{\substack{i=1,{\cdots},d \\ \beta \in \mathbb{B}_I(i ,\ell)}} \binom{I}{\widehat{\beta}^i} \left( \mathrm{K}_{G_{11,\beta,i}^I}^{\mathrm{free}}[\partial_x^{\beta}p'(\varrho)] -\mathrm{K}^{\mathrm{fric}}_{G_{11,\beta,i}^I}[\partial_x^{\beta}p'(\varrho)] \right).
\end{align*}
Let us explain how to estimate each of these terms. The term $\mathbf{S_5}$ will be estimated by a direct proof. All the other terms actually require the use of the smoothing estimates coming either from Proposition \ref{propo:AveragReg} or Proposition \ref{propo:AveragDiff}:
\begin{itemize}
\item[$-$] for $\mathbf{S_6}, \mathbf{S_8}$ and $\mathbf{S_{11}}$, we have the difference of the operators $\mathrm{K}^{\mathrm{free}}$ and $\mathrm{K}^{\mathrm{fric}}$ which appears: these terms are therefore of \textbf{Type III} and we will use Proposition \ref{propo:AveragDiff};
\item[$-$] since the kernels $G_7$, $G_9$ and $G_{10}$ vanish on the diagonal $\lbrace s=t \rbrace$, the terms $\mathbf{S_7}, \mathbf{S_9}$ and $\mathbf{S_{10}}$ are of \textbf{Type II} and we will appeal to Proposition \ref{propo:AveragReg}. 
\end{itemize}
Let us now turn to the estimates.

\medskip

$\bullet$ \textbf{Estimate of $\mathbf{S_5}$}: the argument is the same as for $\mathbf{S}_2 $ above. We obtain 
$$\LRVert{\mathbf{S_5}}_{\Ld^2(0,T;\Ld^2)} \lesssim \Lambda(T,R).$$

\medskip

$\bullet$ \textbf{Estimate of $\mathbf{S_6}, \mathbf{S_8}$ and $\mathbf{S_{11}}$}: by Proposition \ref{propo:AveragDiff}, we have for all $\ell>8+d$ and $\sigma>1+d/2$
\begin{align*}
 \LRVert{\mathbf{S_6}}_{\Ld^2(0,T; \H^1(\T^d))} &\lesssim \underset{0 \leq s,t \leq T}{\sup} \Vert G_6(t,s) \Vert_{\mathcal{H}^{\ell}_{\sigma}} \Vert \partial_x^I \varrho \Vert_{\Ld^2(0,T; \Ld^2(\T^d))} \\
 & \leq \Lambda(T,R),
\end{align*}
thanks to the bound \eqref{bound:p'BONY} in Lemma \ref{LM:Estim-ForceField}, Lemma \ref{LM:rho-pointwise-Hm-2} and provided that we can take $m -2\geq \frac{d}{2}+8+d$. 

Likewise, for $\mathbf{S_8}$, we use Proposition \ref{propo:AveragDiff} to get
\begin{align*}
\LRVert{\mathbf{S_8}}_{\Ld^2(0,T;\H^1)}  \lesssim \underset{\substack{0\leq s,t \leq T\\ 0 \leq \vert \beta \vert \leq \lfloor \frac{\vert I \vert-1}{2} \rfloor \\i=1, \cdots, d}}{\sup} \,  \Vert  G_{8, \beta, i} (s,t) \Vert_{\mathcal{H}^{\ell}_{\sigma}} \Vert  \varrho \Vert_{\Ld^2(0,T; \H^{m-1})},
\end{align*}
for all $\ell >8+d$ and $\sigma >1+d/2$. As in the treatment of $\mathbf{S_3}$ above, we deduce 
\begin{align*}
\LRVert{\mathbf{S_8}}_{\Ld^2(0,T;\H^1)}  \lesssim C_T \Vert p'(\varrho) \Vert_{\Ld^{\infty}(0,T;\H^{m-2})} \Vert f \Vert_{\Ld^{\infty}(0,T;\mathcal{H}^{m-1}_{\sigma})} \Vert \varrho \Vert_{\Ld^2(0,T; \H^{m-1})} \leq \Lambda(T,R),
\end{align*}
since $m>3d+21$. We argue in the same way for $\mathbf{S_{11}}$ (see the treatment of $\mathbf{S_4}$ above) and obtain
\begin{align*}
\LRVert{\mathbf{S_{11}}}_{\Ld^2(0,T;\H^1)}  \leq  \Lambda(T,R).
\end{align*}

\medskip

$\bullet$ \textbf{Estimate of $\mathbf{S_7},\mathbf{S_9}$ and $\mathbf{S_{10}}$}: we proceed exactly as in the estimate of $\textbf{S}_1$ above, since $G_7(t,t,x,v)=0$. Here, we have
\begin{multline*}
\partial_s G_7(s,t,x,v)=\Big[ p''(\varrho)\left[ \partial_s \varrho + v \cdot\nabla_x \varrho\right](s,x-(t-s)v)\\
-p''(\varrho)\left[ \partial_s \varrho + e^{t-s}v \cdot\nabla_x \varrho\right](s,x+(1-e^{t-s})v) \Big]\nabla_v f(t,x,v).
\end{multline*}
Using triangle inequality with Proposition \ref{propo:AveragReg}, we end up with
\begin{align*}
\LRVert{\mathbf{S_7}}_{\Ld^2(0,T; \H^1)} \leq \Lambda(T,R).
\end{align*}
For $\mathbf{S_9}$, using $G_{9, \beta, i}(t,t,x,v)=0$ for $0 \leq \vert \beta \vert \leq \lfloor \frac{\vert I \vert-1}{2} \rfloor$, we also have by  Proposition \ref{propo:AveragReg}
\begin{align*}
\LRVert{\mathrm{K}^{\mathrm{fric}}_{G_{9,\beta,i}}\left[\partial_x^{\overline{I}^i-\beta} \mathrm{J}_\eps \varrho \right]}_{\Ld^2(0,T;\H^1)} 
&\lesssim (1+T)\underset{0\leq s,t \leq T}{\sup} \,  \Vert \partial_s G_{9,\beta,i}(s,t) \Vert_{\mathcal{H}^{\ell}_{\sigma}}\Vert  \varrho \Vert_{\Ld^2(0,T; \H^{m-1})},
\end{align*}
for all $\ell >7+d$ and $\sigma >d/2$. We then proceed as in the estimate of $\mathbf{S}_1$ to deal with the time derivative, combined  with what we have done for the estimate of $\mathbf{S}_3$ (since $m>3d+19$ for instance) and get
\begin{align*}
\LRVert{\mathbf{S_9}}_{\Ld^2(0,T;\H^1)} \leq \Lambda(T,R) \Vert f \Vert_{\Ld^{\infty}(0,T;\mathcal{H}^{m-1}_{\sigma})}\Vert  \varrho \Vert_{\Ld^2(0,T; \H^{m-1})} \leq \Lambda(T,R).
\end{align*}
Finally, we use the same exact arguments as before for $\mathbf{S_{10}}$ (see the estimate of $\mathbf{S}_4$ above for instance) to get
\begin{align*}
\LRVert{\mathbf{S_{10}} }_{\Ld^2(0,T;\H^1)} \leq \Lambda(T,R) \Vert \varrho \Vert_{\Ld^2(0,T; \H^{m-1})} \leq \Lambda(T,R).
\end{align*}
\end{proof}
The second term $\mathscr{R}_{I,1}^{\mathrm{Duha}}$ from Lemma \ref{LM:decompoLeading} is estimated thanks to the following lemma.

\begin{lem}
We have 
\begin{align*}
\LRVert{\mathscr{R}_{I,1}^{\mathrm{Duha}}}_{\Ld^2(0,T; \H^1)} \leq \Lambda(T,R).
\end{align*}
\end{lem}
\begin{proof}
Let us introduce the following vector fields
\begin{align*}
G_{14, i, \beta}^{K,I}(s,t,x,v)&:= \Big(\nabla_x (\partial_x^{\beta}p'(\varrho) )(s, x+(1-e^{t-s}v)) \cdot e_i \Big)\mathrm{H}^{K,I}(s,t,x,v), \\[2mm]
G_{15,i, \beta}^{K,I}(s,t,x,v)&:= \Big(\nabla_x(\partial_x^{\overline{K}^i-\beta} \mathrm{J}_\eps \varrho)(s,x+(1-e^{t-s}v)) \cdot \mathrm{H}^{K,I}(s,t,x,v) \Big) e_i,
\end{align*}
for
\begin{align*}
\beta \in \mathbb{B}_K(i ,\ell)=\left\lbrace \beta \in \N^d \mid \vert \beta \vert=\ell, \ \ 0<\widehat{\beta}^i \leq K \right\rbrace, \ \ i=1,{\cdots},d, \ \ \ell =0,{\cdots}, \vert K \vert-1,
\end{align*}
and where we recall the expression of the kernel $H$ defined in \eqref{def:KernelH1} by
\begin{align*}
\mathrm{H}^{K,I}(t,s,x,v):=\mathrm{N}^{t;s}(\mathrm{Z}^{0;t}(x,\psi_{s,t}(x,v)))_{(I,0),(K,0)}\mathrm{J}^{s,t}(x,v)   \nabla_v f(t,x,\psi_{s,t}(x,v)) -    \nabla_v f(t,x,v).
\end{align*}
By Lemma \ref{LM:decompoForce-gradient}, we now decompose $\mathscr{R}_{I,1}^{\mathrm{Duha}}$ as
\begin{align*}
\mathscr{R}_{I,1}^{\mathrm{Duha}}=\mathbf{S_{12}}+\mathbf{S_{13}}+\mathbf{S_{14}}+\mathbf{S_{15}},
\end{align*}
where
\begin{align*}
\mathbf{S_{12}}(t,x)&:=- \sum_{K} \int_0^t  \int_{\R^d}  \partial_x^K u(s, x+(1-e^{t-s})v) \cdot \mathrm{H}^{K,I}(t,s,x,v) \, \mathrm{d}v \, \mathrm{d}s,
\end{align*}
and
\begin{align*}
\mathbf{S_{13}}&:= \sum_{K} \mathrm{K}^{\mathrm{fric}}_{G_{13}} \left[\partial_x^K  \mathrm{J}_\eps\varrho \right], \ \ \ \  G_{13}(t,s,x,v):=p'(\varrho)(s,x+(1-e^{t-s})v)\mathrm{H}^{K,I}(t,s,x,v),\\
\mathbf{S_{14}}&:=\sum_{K}  \sum_{\ell=0}^{\lfloor \frac{\vert K \vert-1}{2} \rfloor} \sum_{\substack{i=1,{\cdots},d \\ \beta \in \mathbb{B}_K(i ,\ell)}} \binom{I}{\widehat{\beta}^i}\mathrm{K}^{\mathrm{free}}_{G_{14, i, \beta}^{K,I}}\left[\partial_x^{\overline{K}^i-\beta}\mathrm{J}_\eps \varrho \right], \\
\mathbf{S_{15}}&:= \sum_{K} \sum_{\ell=\lfloor \frac{\vert K \vert-1}{2} \rfloor+1}^{\vert K \vert-1} \sum_{\substack{i=1,{\cdots},d \\ \beta \in \mathbb{B}_K(i ,\ell)}} \binom{K}{\widehat{\beta}^i} \mathrm{K}^{\mathrm{free}}_{G_{15,i, \beta}^{K,I}} \left[\partial_x^{\beta} p'(\varrho) \right].
\end{align*}
Let us estimate each of theses terms. Note that $\mathrm{H}^{K,I}(t,t,x,v)=0$ so that the kernels appearing in $\mathbf{S_{13}}, \mathbf{S_{14}}$ and $\mathbf{S_{15}}$ vanish in the diagonal in time: these terms are therefore of \textbf{Type II} and we will use the regularization property from Proposition \ref{propo:AveragReg} to handle them.

\medskip

$\bullet$ \textbf{Estimate of $\mathbf{S_{12}}$}: We proceed as in the estimate for $\mathbf{S}_2$ above and first get for $k>1+\frac{d}{2}$
\begin{align*}
\LRVert{J_8}_{\Ld^2(0,T;\H^1)} \lesssim  \sum_{K} \left\Vert \mathrm{H}^{K,I} \right\Vert_{\Ld^{\infty}(0,T; \mathcal{H}^{k}_{r})} \Vert u \Vert_{\Ld^{2}(0,T;\H^{m+1})}.
\end{align*}
Now observe that for $s,t$, we have 
\begin{align*}
&\left\Vert \mathrm{H}^{K,I}(s,t) \right\Vert_{\mathcal{H}^{k}_{r}} \\
&\lesssim  \LRVert{J^{s,t}}_{\mathrm{W}^{k, \infty}_{x,v}} \LRVert{\mathrm{N}^{s;t}(\mathrm{Z}^{0;t}(\cdot, \psi_{s,t})) \nabla_vf(t,\cdot, \psi_{s,t})}_{\mathcal{H}^{k}_{r}} + \left\Vert f(t)\right\Vert_{\mathcal{H}^{k+1}_{r}} \\
& \lesssim \LRVert{J^{s,t}}_{\mathrm{W}^{k, \infty}_{x,v}}\left(1+\LRVert{\mathrm{Z}^{0;t}}_{\mathrm{W}^{k, \infty}_{x,v}} \right)\left( 1+\LRVert{\psi_{s,t}}_{ \dot{\mathrm{W}}^{k, \infty}_{x,v}} \right) \LRVert{\mathrm{N}^{t;s}}_{\mathrm{W}^{k, \infty}_{x,v}} \sum_{\vert \gamma \vert \leq k} \LRVert{ \partial_{x,v}^{\gamma}(\nabla_v f)(t,\cdot, \psi_{s,t})}_{\mathcal{H}^{0}_{r}} \\
& \quad + \left\Vert f(t)\right\Vert_{\mathcal{H}^{k+1}_{r}} \\
&\lesssim \Lambda(T,R) \sum_{\vert \gamma \vert \leq k} \LRVert{ \partial_{x,v}^{\gamma}(\nabla_v f)(t,\cdot, \psi_{s,t})}_{\mathcal{H}^{0}_{r}} +\left\Vert f(t)\right\Vert_{\mathcal{H}^{k+1}_{r}},
\end{align*}
thanks to the estimate \eqref{bound:perturbId2} of Lemma \ref{straight:velocity}, Remark \ref{Rmk:BoundsXV} and Lemma \ref{LM:estimResolvante}, and since $m>4+d$. To handle the last sum, we write 
\begin{multline*}
\int_{\T^d \times \R^d} \vert \partial_{x,v}^{\gamma} \nabla_v f(t,x, \psi_{s,t}(x,v)) \vert^2 (1+\vert v \vert^2)^r \, \mathrm{d}x \, \mathrm{d}v \\
 \leq  \int_{\T^d \times \R^d} \vert \partial_{x,v}^{\gamma} \nabla_v f(t,x, \psi_{s,t}(x,v)) \vert^2 (1+\vert v-\psi_{s,t}(x,v) \vert^2+\vert \psi_{s,t}(x,v) \vert^2)^r \, \mathrm{d}x \, \mathrm{d}v,
\end{multline*}
and use the change of variable $v\mapsto w= \psi_{s,t}(x,v)$ from Lemma \ref{straight:velocity}, combined with the bounds \eqref{bound:perturbId} and \eqref{supDet} to get (choosing $k$ such that $k+1 \leq m-1$)
\begin{align*}
\left\Vert \mathrm{H}^{K,I}(s,t) \right\Vert_{\mathcal{H}^{k}_{r}} \leq \Lambda(T,R),
\end{align*}
Taking a supremum in time we obtain
\begin{align*}
\LRVert{\mathbf{S_{12}}}_{\Ld^2(0,T;\H^1)} \leq \Lambda(T,R).
\end{align*}
which yields the result.

\medskip

$\bullet$ \textbf{Estimate of $\mathbf{S_{13}}, \mathbf{S_{14}}$ and $\mathbf{S_{15}}$}: Since $\mathrm{H}^{K,I}(t,t,x,v)=0$, we observe that $G_{13}(t,t,x,v)=0$. By Proposition \ref{propo:AveragReg}, we therefore have for $\ell>7+d$ and $\sigma >d/2$
\begin{align*}
\LRVert{\mathbf{S_{13}}}_{\Ld^2(0,T;\H^1)} &\lesssim (1+T) \sum_{K} \underset{0\leq s,t \leq T}{\sup} \,  \Vert \partial_s\left[ p'(\varrho(s,x+(1-e^{t-s})v))\mathrm{H}^{K,I}(s,t) \right] \Vert_{\mathcal{H}^{\ell}_{\sigma}} \Vert \partial_x^{K} \varrho \Vert_{\Ld^2(0,T; \Ld^2)}\\
& \lesssim C_T \sum_{K} \underset{0\leq s,t \leq T}{\sup} \,  \Vert  p''(\varrho)\partial_s \varrho (s,x+(1-e^{t-s})v))\mathrm{H}^{K,I}(s,t)\Vert_{\mathcal{H}^{\ell}_{\sigma}} \Vert \partial_x^{K} \varrho \Vert_{\Ld^2(0,T; \Ld^2)} \\
& \quad + C_T \sum_{K} \underset{0\leq s,t \leq T}{\sup} \, \Vert p''(\varrho)\nabla_x \varrho (s,x+(1-e^{t-s})v))\mathrm{H}^{K,I}(s,t) \Vert_{\mathcal{H}^{\ell}_{\sigma+1}} \Vert \partial_x^{K} \varrho \Vert_{\Ld^2(0,T; \Ld^2)} \\
& \quad +C_T \sum_{K} \underset{0\leq s,t \leq T}{\sup} \,  \Vert  p'(\varrho(s,x+(1-e^{t-s})v))  \partial_s\mathrm{H}^{K,I}(s,t)  \Vert_{\mathcal{H}^{\ell}_{\sigma}} \Vert \partial_x^{K} \varrho \Vert_{\Ld^2(0,T; \Ld^2)}.
\end{align*}
The two first terms can be handled by similar arguments to the ones used for $\mathbf{S}_1$ and $\mathbf{S_{12}}$, a fixed number of derivatives being involved. For the last one, we proceed as for the other terms, combined with the arguments used for $\mathbf{S}_{12}$, and write for all $t,s$
\begin{align*}
 &\Vert  p'(\varrho(s,x+(1-e^{t-s})v))  \partial_s\mathrm{H}^{K,I}(s,t)  \Vert_{\mathcal{H}^{\ell}_{\sigma}} \\
 &\leq C_T \Vert p'(\varrho) \Vert_{\Ld^{\infty}(0,T;\H^{m-2})} \Vert  \partial_s\mathrm{H}^{K,I}(s,t)  \Vert_{\mathcal{H}^{\ell}_{\sigma}}\\
  &\leq \Lambda(T,R)  \Big( \Vert J^{s,t} \Vert_{\W^{\ell,\infty}_{x,v}} \Vert \nabla_v f(t) \Vert_{\mathcal{H}^{\ell}_{\sigma}} +\Vert \partial_s J^{s,t} \Vert_{\W^{\ell, \infty}_{x,v}} \Vert \nabla_v f(t,\cdot, \psi_{s,t}) \Vert_{\mathcal{H}^{\ell}_{\sigma}} +\Vert \nabla_v f(t) \Vert_{\mathcal{H}^{\ell+1}_{\sigma}}\Big)\\
  & \leq \Lambda(T,R),
\end{align*}
since $m>3d/2+11$. This yields
\begin{align*}
\LRVert{\mathbf{S_{13}}}_{\Ld^2(0,T;\H^1)}  \leq \Lambda(T,R).
\end{align*}
Next, for $\mathbf{S_{14}}$, we use the fact that $G_{14, i, \beta}^{K,I}(t,t,x,v)=0$ for $0 \leq \vert \beta \vert \leq \lfloor \frac{\vert K \vert-1}{2} \rfloor$ and $i=1, \cdots, d$ so that by Proposition \ref{propo:AveragReg}, we have
\begin{align*}
\left\Vert \mathrm{K}_{G_{14, i, \beta}^{K,I}}\left[\partial_x^{\overline{K}^i-\beta} \mathrm{J}_\eps \varrho \right] \right\Vert_{\Ld^2(0,T;\H^1)} 
& \lesssim (1+T) \underset{0 \leq t,s \leq T}{\sup} \, \left\Vert  \partial_s G_{14, i, \beta}^{K,I}(s,t) \right\Vert_{\mathcal{H}^{\ell}_{\sigma}} \left\Vert  \varrho \right\Vert_{\Ld^2(0,T;\H^m)} ,
\end{align*}
for $\ell >7+d$ and $\sigma>d/2$. Next, we have for all $t,s$
\begin{multline*}
\left\Vert  \partial_s G_{14, i, \beta}^{K,I}(s,t) \right\Vert_{\mathcal{H}^{\ell}_{\sigma}} \lesssim \left\Vert \partial_s[\nabla_x(\partial_x^{\beta} p'(\varrho))(s, x+(1-e^{t-s})v)]\mathrm{H}^{K,I}(s,t) \right\Vert_{\mathcal{H}^{\ell}_{\sigma}} \\ + \left\Vert \nabla_x(\partial_x^{\beta} p'(\varrho))(s, x+(1-e^{t-s})v) \partial_s \mathrm{H}^{K,I}(s,t) \right\Vert_{\mathcal{H}^{\ell}_{\sigma}}.
\end{multline*}
The first term can be handled by arguments similar to the ones used for $\mathbf{S_9}$, $\mathbf{S}_1$ and for $\mathbf{S_{12}}$. The second can be adressed by the same procedure where one relies on the arguments used for $\mathbf{S_{13}}$. Likewise, we obtain
\begin{align*}
\LRVert{\mathbf{S_{14}}}_{\Ld^2(0,T;\H^1)}  \leq \Lambda(T,R).
\end{align*}

\end{proof}

We eventually estimate the third term $\mathscr{R}_{I,2}^{\mathrm{Duha}}$ from Lemma \ref{LM:decompoLeading}.

\begin{lem}
We have 
\begin{align*}
\LRVert{\mathscr{R}_{I,2}^{\mathrm{Duha}}}_{\Ld^2(0,T; \H^1)} \leq \Lambda(T,R).
\end{align*}
\end{lem}
\begin{proof}
We proceed as before by introducing the vector fields
\begin{align*}
G_{18, i, \beta}^{K,I}(t,s,x,v)&:= \left(\nabla_x(\partial_x^{\beta} p'(\varrho))(s,x+(1-e^{t-s}v)) \cdot e_i \right)\mathfrak{H}^{K,I}(s,t,x,v), \\[2mm]
G_{19,i, \beta}^{K,I}(t,s,x,v)&:= \Big(\nabla_x(\partial_x^{\overline{K}^i-\beta} \mathrm{J}_\eps \varrho)(s,x+(1-e^{t-s}v)) \cdot \mathfrak{H}^{K,I}(s,t,x,v) \Big) e_i,
\end{align*}
where
\begin{align*}
\beta \in \mathbb{B}_K(i ,\ell)=\left\lbrace \beta \in \N^d \mid \vert \beta \vert=\ell, \ \ 0<\widehat{\beta}^i \leq K \right\rbrace, \ \ i=1,{\cdots},d, \ \ \ell =0,{\cdots}, \vert K \vert-1 .
\end{align*}
Let us also recall the expression of the kernel $\mathfrak{H}$ defined in \eqref{def:KernelH2} by
\begin{multline*}
 \mathfrak{H}^{K,I}(t,s,x,v):=\int_s^t e^{d(t-\tau)}   \mathrm{N}^{t;s}(\mathrm{Z}^{0;t}(x,\psi_{s,t}(x,v)))_{(I,0),(K,0)}   \\
  \big( \nabla_x f(\tau, \mathrm{Z}^{\tau;t}(x,\psi_{s,t}(x,v))) -\nabla_v f(\tau, \mathrm{Z}^{\tau;t}(x,\psi_{s,t}(x,v))) \big) \, \mathrm{J}^{s,t}(x,v) \, \mathrm{d}\tau.
\end{multline*}
 Thanks to Lemma \ref{LM:decompoForce-gradient}, we can write
\begin{align*}
\mathscr{R}_{I,2}^{\mathrm{Duha}}=\mathbf{S}_{16}+\mathbf{S}_{17}+\mathbf{S}_{18}+\mathbf{S}_{19},
\end{align*}
where
\begin{align*}
\mathbf{S}_{16}(t,x) &:=-\sum_{K}\int_0^t \int_{\R^d} \partial_x^K u(s,x+(1-e^{t-s}v))\cdot \mathfrak{H}^{K,I}(t,s,x,v)\, \mathrm{d}v \, \mathrm{d}s,
\end{align*}
and
\begin{align*}
\mathbf{S}_{17} &:=\mathrm{K}^{\mathrm{fric}}[\partial_x^K \mathrm{J}_\eps \varrho], \ \ \ G_{17}(t,s,x,v):=p'(\varrho(s, x+(1-e^{t-s})v))\mathfrak{H}^{K,I}(t,s,x,v), \\
\mathbf{S}_{18} &:= \sum_{K}\sum_{\ell=0}^{\lfloor \frac{\vert K \vert-1}{2} \rfloor} \sum_{\substack{i=1,{\cdots},d \\ \beta \in \mathbb{B}_K(i ,\ell)}} \binom{K}{\widehat{\beta}^i} \mathrm{K}_{G_{18, i, \beta}^{K,I}}\left[ \partial_x^{\overline{K}^i-\beta} \mathrm{J}_\eps \varrho \right], \\
\mathbf{S}_{19} &:= \sum_{K} \sum_{\ell=\lfloor \frac{\vert K \vert-1}{2} \rfloor+1}^{\vert K \vert-1} \sum_{\substack{i=1,{\cdots},d \\ \beta \in \mathbb{B}_K(i ,\ell)}} \binom{K}{\widehat{\beta}^i}\mathrm{K}_{G_{19, i, \beta}^{K,I}} \left[ \partial_x^{\beta}p'(\varrho) \right].
\end{align*}
Let us estimate these terms, as we have done previously. Let us observe that $ \mathfrak{H}^{K,I}(t,t,x,v)=0$ therefore the terms $\mathbf{S}_{17}, \mathbf{S}_{18}$ and $\mathbf{S}_{19}$ are of \textbf{Type II} and we can rely on Proposition \ref{propo:AveragReg}.

\medskip

$\bullet$ \textbf{Estimate of $\mathbf{S}_{16}$}: We mainly proceed as for $\mathbf{S}_{12}$, the kernel being changed from $\mathrm{H}$ to $\mathfrak{H}$. Hence, we only have to give an estimate for $\left\Vert \mathfrak{H}^{K,I}(s,t) \right\Vert_{\mathcal{H}^{k}_{r}}$ (for $k$ and $r$ large enough).
As in the estimate of $\mathbf{S}_{12}$, we use $\W^{k,\infty}_{x,v}$ bounds on $\psi_{s,t},  \mathrm{Z}^{s;t}$ and $ \mathrm{N}^{s;t}$ from Lemmas \ref{straight:velocity}, Remark \ref{Rmk:BoundsXV} and Lemma \ref{LM:estimResolvante} to write
\begin{align*}
&\left\Vert \mathfrak{H}^{K,I}(s,t) \right\Vert_{\mathcal{H}^{k}_{r}}^2 \\
 &\lesssim \Lambda(T,R) \LRVert{J^{s,t}}_{\W^{k, \infty}_{x,v}}^2 \left( 1+\LRVert{\mathrm{N}^{s;t}}_{\W^{k, \infty}_{x,v}}^2 \right)   \left\Vert \int_s^t \left[\nabla_x f(\tau, \mathrm{Z}^{\tau;t}(\cdot, \psi_{s,t}))-\nabla_v f(\tau, \mathrm{Z}^{\tau;t}(\cdot, \psi_{s,t})) \right] \, \mathrm{d}\tau \right\Vert_{\mathcal{H}^{k}_r}^2\\
& \lesssim \Lambda(T,R) \sum_{\vert \gamma \vert \leq k} \int_{\T^d \times \R^d}  \langle v \rangle^{2r} \left\vert \int_s^t \left[ \partial_{x,v}^{\gamma}(\nabla_x f)(\tau, \mathrm{Z}^{\tau;t}(\cdot, \psi_{s,t}))-\partial_{x,v}^{\gamma}(\nabla_v f)(\tau,\mathrm{Z}^{\tau;t}(\cdot, \psi_{s,t})) \right] \mathrm{d}\tau \right\vert^2 \, \mathrm{d}x \, \mathrm{d}v.
\end{align*}
By the the generalized Minkowski inequality, the last expression is bounded by
\begin{align*}
& \sum_{\vert \gamma \vert \leq k} \left( \int_s^t \left( \int_{\T^d \times \R^d} \langle v \rangle^{2r} \vert \partial_{x,v}^{\gamma}(\nabla_x f)(\tau,\mathrm{Z}^{\tau;t}(\cdot, \psi_{s,t}))-\partial_{x,v}^{\gamma}(\nabla_v f)(\tau, \mathrm{Z}^{\tau;t}(\cdot, \psi_{s,t})) \vert^2 \, \mathrm{d}x \, \mathrm{d}v \right)^{1/2} \, \mathrm{d}\tau \right)^2 \\
& \leq \Lambda(T,R)\sum_{\vert \gamma \vert \leq k} \left( \int_s^t \left( \int_{\T^d \times \R^d} \langle v \rangle^{2r} \vert \partial_{x,v}^{\gamma}(\nabla_x f)(\tau,x,v)-\partial_{x,v}^{\gamma}(\nabla_v f)(\tau,x,v) \vert^2 \, \mathrm{d}x \, \mathrm{d}v \right)^{1/2} \, \mathrm{d}\tau \right)^2,
\end{align*}
where we have performed the change of variable $v \mapsto w=\psi_{s,t}(x,v)$ from Lemma \ref{straight:velocity} followed by $(x,w) \mapsto \mathrm{Z}^{0;t}(x,w)$. Here, we have used the bounds on the Jacobian from \eqref{supDet}, as well as the one on $\vert v-\psi_{s,t}(x,v) \vert$ via \eqref{bound:perturbId}. It follows that
\begin{align*}
\left\Vert \mathfrak{H}^{K,I}(s,t) \right\Vert_{\mathcal{H}^{k}_{r}}^2 &\lesssim\Lambda(T,R) \vert t-s \vert^2  \sup_{0 \leq \tau \leq T} \, \Big\lbrace \Vert \nabla_x f(\tau) \Vert_{\mathcal{H}^k_r}^2+\Vert \nabla_v f(\tau) \Vert_{\mathcal{H}^k_r}^2 \Big \rbrace,
\end{align*}
 and therefore 
\begin{align*}
\LRVert{\mathbf{S_{16}}}_{\Ld^2(0,T;\H^1)} \leq \Lambda(T,R).
\end{align*}

\medskip

$\bullet$ \textbf{Estimate of $\mathbf{S}_{17}, \mathbf{S}_{18}$ and $\mathbf{S}_{19}$}: We mainly proceed as for $\mathbf{S}_{13}, \mathbf{S}_{14}$ and $\mathbf{S}_{15}$, using Proposition \ref{propo:AveragReg}. As before, the kernel has just been changed from $\mathrm{H}$ to $\mathfrak{H}$. Hence, we only have to provide an estimate for $\Vert  \partial_s\mathfrak{H}^{K,I}(s,t)  \Vert_{\mathcal{H}^{k}_{r}} $ (for $k$ and $r$ large enough).
%
We have
\begin{align*}
\partial_s\mathfrak{H}^{K,I}(s,t,x,v)&=\partial_s J^{s,t} \mathrm{N}^{t;s}(\mathrm{Z}^{0;t}(x,\psi_{s,t}(x,v)))_{(I,0),(K,0)} \\
 &   \qquad  \qquad \qquad     \int_s^t e^{d(t-\tau)}  \big[ \nabla_x f(\tau, \mathrm{Z}^{\tau;t}(x,\psi_{s,t}(x,v))) -\nabla_v f(\tau, \mathrm{Z}^{\tau;t}(x,\psi_{s,t}(x,v))) \big] \, \mathrm{d}\tau \\
 & \quad + J^{s,t}    \partial_s \left\lbrace \mathrm{N}^{t;s}(\mathrm{Z}^{0;t}(x,\psi_{s,t}(x,v)))_{(I,0),(K,0)} \right\rbrace   \\
 &   \qquad  \qquad \qquad     \int_s^t e^{d(t-\tau)} \big[ \nabla_x f(\tau, \mathrm{Z}^{\tau;t}(x,\psi_{s,t}(x,v))) -\nabla_v f(\tau, \mathrm{Z}^{\tau;t}(x,\psi_{s,t}(x,v))) \big] \, \mathrm{d}\tau \\
 & \quad +J^{s,t} \mathrm{N}^{t;s}(\mathrm{Z}^{0;t}(x,\psi_{s,t}(x,v)))_{(I,0),(K,0)}   \\
 &   \qquad  \qquad \qquad     \int_s^t e^{d(t-\tau)} \partial_s \big[ \nabla_x f(\tau, \mathrm{Z}^{\tau;t}(x,\psi_{s,t}(x,v))) -\nabla_v f(\tau, \mathrm{Z}^{\tau;t}(x,\psi_{s,t}(x,v))) \big] \, \mathrm{d}\tau \\
 & \quad -e^{d(t-s)} \J^{s,t} \mathrm{N}^{t;s}(\mathrm{Z}^{0;t}(x,\psi_{s,t}(x,v)))_{(I,0),(K,0)}  \\
 &  \qquad  \qquad \qquad  \qquad \qquad\big[ \nabla_x f(s, \mathrm{Z}^{s;t}(x,\psi_{s,t}(x,v))) -\nabla_v f(s, \mathrm{Z}^{s;t}(x,\psi_{s,t}(x,v))) \big].
\end{align*}
Using the same $\W^{k,\infty}_{x,v}$ bounds on $J^{s,t}$, $\mathrm{N}^{s;t}$ and $\dot{\W}^{k,\infty}_{x,v}$ bounds on $\psi_{s,t}$ as before, as well as on their time derivatives, we obtain
\begin{align*}
\Vert  \partial_s\mathfrak{H}^{K,I}(s,t)  \Vert_{\mathcal{H}^{k}_{r}} \leq  \Lambda(T,R).
\end{align*}
This eventually yields
\begin{align*}
\LRVert{\mathbf{S_{17}}}_{\Ld^2(0,T;\H^1)}+\LRVert{\mathbf{S_{18}}}_{\Ld^2(0,T;\H^1)} +\LRVert{\mathbf{S_{19}}}_{\Ld^2(0,T;\H^1)}  \leq \Lambda(T,R).
\end{align*}
This ends the proof.
\end{proof}

We end this section by eventually giving a proof of Lemma \ref{LM:estimateGrad-I_R1}. Let us mention that it only requires the use of the continuity estimate coming from Proposition \ref{propo:AveragStandard} and not the ones of Proposition \ref{propo:AveragReg}-\ref{propo:AveragDiff}.
\begin{proof}[Proof of Lemma \ref{LM:estimateGrad-I_R1}]
In view of the estimate \eqref{estimate:I_R1-L2} of Lemma \ref{LM:estimateI_0R}, we only have to prove that 
\begin{align*}
\Vert \nabla_x \mathcal{I}_{\mathcal{R}_1}^0\Vert_{\Ld^2(0,T;\Ld^2)} + \Vert \nabla_x \mathcal{I}_{\mathcal{R}_1}^{1} \Vert_{\Ld^2(0,T;\Ld^2)} \leq \Lambda(T,R).
\end{align*}
We only write the proof for $\nabla_x \mathcal{I}_{\mathcal{R}_1}^0$. Let us recall that 
\begin{align*}
\mathcal{I}_{\mathcal{R}_1}^0(t,x)=-\int_0^t e^{d(t-s)}  \int_{\R^d} \mathrm{N}^{t;s}(\mathrm{Z}^{0;t}(x,v))  \mathcal{R}_1(s, \mathrm{Z}^{s;t}(x,v)) \, \mathrm{d}v \, \mathrm{d}s
\end{align*}
where
\begin{align*}
\mathcal{R}_1=\left(\mathcal{R}^{K,L}_1 \right)_{\vert K \vert + \vert L \vert \in \lbrace m-1,m \rbrace },
\end{align*}
with
\begin{align*}
\mathcal{R}^{K,L}_1=\mathbf{1}_{\substack{\vert K \vert > 2 \\ L \neq 0}} \partial_x^K E_{\reg,\eps}^{u,\varrho} \cdot  \nabla_v \partial_v^Lf+\mathbf{1}_{\vert K \vert>1} \sum_{\substack{0 <\alpha < K \\ \vert \alpha  \vert = m-1}} \binom{L}{\alpha}\partial_x^{\alpha} E_{\reg,\eps}^{u,\varrho} \cdot \nabla_v \partial_x^{L-\alpha} \partial_v^{K}  f.
\end{align*}
We have for all $(I,J)$
\begin{align*}
\left[ \mathcal{I}_{\mathcal{R}_1}^0\right]_{(I,J)}(t,x)=-\int_0^t e^{d(t-s)}  \int_{\R^d}  \sum_{(K,L)} \mathrm{N}^{t;s}_{(I,J),(K,L)}(\mathrm{Z}^{0;t}(x,v))\left[\mathcal{R}_1 \right]_{(K,L)}(s,\mathrm{Z}^{s;t}(x,v)) \, \mathrm{d}v \, \mathrm{d}s,
\end{align*}
therefore
\begin{align*}
&\nabla_x \left[ \mathcal{I}_{\mathcal{R}_1}^0 \right]_{(I,J)}(t,x)\\
&=-\int_0^t e^{d(t-s)}  \int_{\R^d}  \sum_{(K,L)} \nabla_x \mathrm{Z}^{0;t}(x,v) \nabla_x\mathrm{N}^{t;s}_{(I,J),(K,L)}(\mathrm{Z}^{0;t}(x,v))\left[\mathcal{R}_1 \right]_{(K,L)}(s,\mathrm{Z}^{s;t}(x,v)) \, \mathrm{d}v \, \mathrm{d}s \\
& \quad -\int_0^t e^{d(t-s)}  \int_{\R^d}  \sum_{(K,L)} \mathrm{N}^{t;s}_{(I,J),(K,L)}(\mathrm{Z}^{0;t}(x,v)) \nabla_x \mathrm{Z}^{0;t}(x,v) \nabla_x \left[\mathcal{R}_1 \right]_{(K,L)}(s,\mathrm{Z}^{s;t}(x,v)) \, \mathrm{d}v \, \mathrm{d}s.
\end{align*}

$\bullet$ For the first term, there is no derivative on $\mathcal{R}_1$. We can proceed exactly as for $\mathcal{R}_0$ in Lemma \ref{LM:estimateI_0R}, relying on the estimate \eqref{estimate:I_R1-L2}.

$\bullet$ The second term is more involved, since $\nabla_x \left[\mathcal{R}_1 \right]_{(K,L)}$ contains several types of terms:
\begin{itemize}
\item[$-$] some are of the form
$$\mathbf{1}_{\substack{\vert K \vert > 2 \\ L \neq 0}} \nabla_x [\nabla_v \partial_v^L  f] \partial_x^K E_{\reg,\eps}^{u,\varrho}   .$$ 
They involve at most $m$ derivatives of $\varrho$ and at most $2+(m-3)=m-1$ derivatives of $f$. We can bound this term in $\Ld^2(0,T; \mathcal{H}^0_r)$, following the argument used for $\mathcal{R}_0$ in the proof of Lemma \ref{LM:ALLapplyD}. Its contribution to $\nabla_x  \mathcal{I}_{\mathcal{R}_1}^0 $ is handled as in the proof of Lemma \ref{LM:estimateI_0R} (for the term $\nabla_x  \mathcal{I}_{\mathcal{R}_0}^0 $);
\item[$-$] some are of the form
$$ \nabla_x[ \nabla_v \partial_x^{L-\alpha} \partial_v^{K}  f]\partial_x^{\alpha} E_{\reg,\eps}^{u,\varrho}  ,$$
 with $\vert K \vert>1$ and $\vert \alpha  \vert = m-1$. They involve at most $m$ derivatives of $\varrho$ and at most $3$ derivatives of $f$. We can also bound this term in $\Ld^2(0,T; \mathcal{H}^0_r)$, following the argument used for $\mathcal{R}_0$ in the proof of Lemma \ref{LM:ALLapplyD}. As before, we can follow the proof of Lemma \ref{LM:estimateI_0R} to control its contribution to $\nabla_x  \mathcal{I}_{\mathcal{R}_1}^0 $;
\item[$-$] some are of the form
$$\mathbf{1}_{\substack{\vert K \vert > 2 \\ L \neq 0}} \nabla_x [\partial_x^K  E_{\reg,\eps}^{u,\varrho}]   \nabla_v \partial_v^Lf,$$
and of the form
$$\nabla_x  [\partial_x^{\alpha} E_{\reg,\eps}^{u,\varrho}]   \nabla_v \partial_x^{L-\alpha} \partial_v^{K}  f,$$ 
with $\vert K \vert>1$ and $\vert \alpha  \vert = m-1$. These terms involve at most $m+1$ derivatives of $\varrho$ and cannot be directly treated as the previous ones. We need to rely on the smoothing estimates from Section \ref{Section:Averag-lemma} (by making some terms of \textbf{Type I} appear).
\end{itemize}

We then focus on the two last types of terms, and we have to deal with
\begin{multline*}
\mathbb{W}_1^{K,L}(t,x):=-\int_0^t e^{d(t-s)}  \int_{\R^d}  \mathrm{N}^{t;s}_{(I,J),(K,L)}(\mathrm{Z}^{0;t}(x,v))\\
 \nabla_x \mathrm{Z}^{0;t}(x,v) \mathbf{1}_{\substack{\vert K \vert > 2 \\ L \neq 0}} \nabla_x [\partial_x^K  E_{\reg,\eps}^{u,\varrho}]   \nabla_v \partial_v^Lf(s,\mathrm{Z}^{s;t}(x,v)) \, \mathrm{d}v \, \mathrm{d}s,
 \end{multline*}
 and
 \begin{multline*}
\mathbb{W}_2^{K,L}(t,x):=-\int_0^t e^{d(t-s)}  \int_{\R^d}  \mathrm{N}^{t;s}_{(I,J),(K,L)}(\mathrm{Z}^{0;t}(x,v)) \\ \nabla_x \mathrm{Z}^{0;t}(x,v) \mathbf{1}_{\vert K \vert>1} 
\sum_{\substack{0 <\alpha < K \\ \vert \alpha  \vert = m-1}} \binom{L}{\alpha} \nabla_x  [\partial_x^{\alpha} E_{\reg,\eps}^{u,\varrho} ]  \nabla_v \partial_x^{L-\alpha} \partial_v^{K}  f(s,\mathrm{Z}^{s;t}(x,v)) \, \mathrm{d}v \, \mathrm{d}s,
\end{multline*}
for $\vert K \vert + \vert L \vert \leq m$. Let us turn to the estimates of these terms in $\Ld^2(0,T; \Ld^2)$. 

\medskip

$\bullet$ \textbf{Estimate of $\mathbb{W}_1^{K,L}$ when $\frac{m-2}{2} \leq \vert L \vert  <m-2$}: since $\vert K \vert + \vert L \vert \leq m$, this implies that $\vert K \vert \leq 1+m/2$. As in the proof of Lemma \ref{LM:estimateI_0R}, we have by Cauchy-Schwarz inequality
\begin{align*}
\left\Vert  \mathbb{W}_1^{K,L}  \right\Vert_{\Ld^2(0,T;\Ld^2)} \leq   \Lambda(T,R)  \left \Vert \nabla_x \partial_x^K  E_{\reg,\eps}^{u,\varrho} \cdot  \nabla_v \partial_v^Lf \right\Vert_{\Ld^2(0,T; \mathcal{H}^0_r)},
\end{align*}
and, by setting $\chi(v)=(1+\vert v \vert^2)^{r/2}$, we have
\begin{align*}
\Vert \chi \nabla_x \partial_x^K  E_{\reg,\eps}^{u,\varrho} \cdot  \nabla_v \partial_v^Lf \Vert_{\Ld^2_{x,v}} \lesssim \Vert \nabla_x \partial_x^K  E_{\reg,\eps}^{u,\varrho} \Vert_{\Ld^{\infty}_{x}} \Vert \chi  \nabla_v \partial_v^Lf\Vert_{\Ld^2_{x,v}} \lesssim \Vert  E \Vert_{\H^{k}} \Vert   f\Vert_{\mathcal{H}^{m-1}_r},
\end{align*}
with $k>\frac{d}{2}+1+ \vert K \vert$. Since $d+6 \leq m$, we can choose $k$ such that 
\begin{align*}
\left\Vert  \mathbb{W}_1^{K,L}  \right\Vert_{\Ld^2(0,T;\Ld^2)} \leq   \Lambda(T,R)  \Vert  E_{\reg,\eps}^{u,\varrho} \Vert_{\Ld^2(0,T;\H^{m-1})} \Vert   f\Vert_{\Ld^{\infty}(0,T;\mathcal{H}^{m-1}_r)} \leq \Lambda(T,R),
\end{align*}
and conclude as in the proof of Lemma \ref{LM:ALLapplyD}. We obtain in that case
\begin{align*}
\LRVert{ \mathbb{W}_1^{K,L}}_{\Ld^2(0,T; \Ld^2)} \leq \Lambda(T,R).
\end{align*}

\medskip

$\bullet$ \textbf{Estimate of $\mathbb{W}_1^{K,L}$ when $1 \leq \vert L \vert \leq \frac{m-2}{2} -1$}: in that case, we only have $\vert K \vert \leq m-1$. We shall rely on the smoothing estimate of Proposition \ref{propo:AveragStandard} (to treat terms of \textbf{Type I}) and to recover the loss of the extra derivative on $\varrho$. To do so, we first write
\begin{align*}
&\vert \mathbb{W}_1^{K,L}(t,x) \vert \\
&\leq \Lambda(T,R)  \LRVert{\nabla_x \mathrm{Z}^{0;t} }_{\Ld^{\infty}_{x,v}} \underset{0 \leq s \leq T}{\sup}\LRVert{\mathrm{N}^{t;s}}_{\Ld^{\infty}_{x,v}}\mathbf{1}_{\substack{\vert K \vert > 2 \\ L \neq 0}} \left\vert \int_0^t  \int_{\R^d}  \nabla_x [\partial_x^K  E_{\reg,\eps}^{u,\varrho}]   \nabla_v \partial_v^Lf(s,\mathrm{Z}^{s;t}(x,v)) \, \mathrm{d}v \, \mathrm{d}s \right\vert,
\end{align*}
and we perform the change of variable $v \mapsto \psi_{s,t}(x,w)$ from Lemma \ref{straight:velocity} in the last integral (since $t \leq \overline{T}(R)$) to get
\begin{align*}
&\vert \mathbb{W}_1^{K,L}(t,x) \vert \\
&\leq \Lambda(T,R) \mathbf{1}_{\substack{\vert K \vert > 2 \\ L \neq 0}} \left\vert \int_0^t  \int_{\R^d}  \nabla_x [\partial_x^K  E_{\reg,\eps}^{u,\varrho}](s,x+(1-e^{t-s})w)   \nabla_v \partial_v^Lf(s,\mathrm{Z}^{s;t}(x, \psi_{s,t}(x,w))) \, \mathrm{d}w \, \mathrm{d}s \right\vert,
\end{align*}
thanks to the bounds \eqref{supDet} of Lemma \ref{straight:velocity}, Remark \ref{Rmk:BoundsXV} and Lemma \ref{LM:estimResolvante}. 

Relying on the decomposition of Lemma \ref{LM:decompoForce-gradient}, we can make the operator $\mathrm{K}^{\mathrm{fric}}$ appear, hence dealing with terms of \textbf{Type I}, and use Proposition \ref{propo:AveragStandard} (combined with the bounds on $\mathrm{Z}^{s;t}$ and $ \psi_{s,t}$) to estimate the last integral in $\Ld^2(0,T; \Ld^2)$. The proof follows the same lines as the ones of Lemma \ref{LM:makeLeadingappear}. Note that Proposition \ref{propo:AveragStandard} only requires an estimate of $\nabla_v \partial_v^Lf$ in $\Ld^{\infty}(0,T; \mathcal{H}_r^{s})$ with $s>1+d$. Since $m > 2d+2$, we obtain in that case
\begin{align*}
\LRVert{ \mathbb{W}_1^{K,L}}_{\Ld^2(0,T; \Ld^2)} \leq \Lambda(T,R).
\end{align*}

%

\medskip

$\bullet$ \textbf{Estimate of $\mathbb{W}_2^{K,L}$}: to treat this term, we observe that it only involves at most $2$ derivatives of $f$ and the gradient of $m-1$ derivatives of the force field  $E_{\reg,\eps}^{u,\varrho}$. Hence, we can exactly perform the same as previously for $\mathbb{W}_1^{K,L}$. We likewise obtain 
\begin{align*}
\LRVert{ \mathbb{W}_2^{K,L}}_{\Ld^2(0,T; \Ld^2)} \leq \Lambda(T,R).
\end{align*}

This concludes the proof of Lemma \ref{LM:estimateGrad-I_R1}.
\end{proof}

\section{Analysis of the fluid density}\label{Section:FluidDensity-estimate}

We pursue our goal which is to obtain an uniform control on $\Vert \varrho \Vert_{\Ld^2(0,T;\H^{m})}$. 
In the previous section, we have  related the kinetic moments $\rho_{f}$ and $j_{f}$ to the fluid density $\varrho$, up to some well controlled remainders.  In this section, we build on these relations to further analyze $\varrho$.

We start by taking derivatives in the transport equation satisfied by $\varrho$. By using the key Proposition \ref{coroFInal:D^I:rho-j} (which was precisely the main outcome of Section \ref{Section:Kinetic-moments}), we obtain a factorization of the equation on the derivatives between
\begin{enumerate}
\item[$-$] a purely hyperbolic part, which is the transport operator $\partial_t + u \cdot \nabla_x$;
\item[$-$] an integro-differential operator part.
\end{enumerate}
This is where the crucial Penrose condition \eqref{cond:Penrose} steps in and allows to justify that this last operator is actually elliptic in space-time and therefore can provide $\Ld^2_T \Ld^2_x$ estimates without loss. This relies on a semiclassical pseudodifferential analysis, in the spirit of \cite{HKR}.

In this section, we will use the notation $M_{\mathrm{in}}$, which stands for a positive constant depending only on the initial data.
\subsection{Equation on the derivatives of the fluid density}\label{Subsection:DerivativeRhoFluid}
For $T \in [0,\min \left(T_\eps(R),\overline{T}(R)\right)$, the aim of this section is to prove the following proposition.

\begin{propo}\label{coro:Facto}
Setting $h=\partial_x^{\alpha} \varrho$ for $\vert \alpha \vert \leq m$, one has 
\begin{align}\label{finaleq:h}
\left( \mathrm{Id}-\frac{\varrho}{1-\rho_f}\mathrm{K}_G^{\mathrm{free}}\circ\mathrm{J}_{\eps} \right)\Big[\partial_t h + u \cdot \nabla_x h\Big]=\mathcal{R}, \ \ t \in (0,T),
\end{align}
with $G(t,x,v)=  p'( \varrho(t,x)) \nabla_v f(t,x,v)$ and 
$$
\| \mathcal{R} \|_{\Ld^2(0,T; \Ld^2(\T^d))} \leq \Lambda (T,R, \|h(0)\|_{\H^1(\T^d)}).
$$
\end{propo}
We start with a commutation result when one takes derivatives in the equation on $\varrho$.

\begin{lem}\label{LMestim:ResteRalpha}
For all $|\alpha|\leq m$, $\pa^\alpha_x \varrho$ satisfies the equation
\begin{align*}
\partial_t (\partial_x^{\alpha}\varrho )+u \cdot \nabla_x (\partial_x^{\alpha}\varrho)+ \frac{\varrho}{1-\rho_f} \mathrm{div}_x \left[ j_{\partial_x^{\alpha} f}-\rho_{\partial_x^{\alpha}f} u\right]=R^{\alpha},
\end{align*}
with 
\begin{align*}
\Vert R^{\alpha} \Vert_{\Ld^2(0,T;\Ld^2)} \leq \Lambda(T,R).
\end{align*}
\end{lem}

\begin{proof}
Recall that by Lemma \ref{LM:rewriteEqrho}, the transport equation on $\varrho$ reads as
\begin{align*}
\partial_t \varrho +u \cdot \nabla_x \varrho +\frac{\varrho}{1-\rho_f} \mathrm{div}_x \left[ j_f -\rho_f u \right]=S, \ \ S=-\frac{\varrho}{1-\rho_f} \mathrm{div}_x \, u.
\end{align*}
We get, for all $\alpha \in \N^d$,
\begin{align*}
\partial_t (\partial_x^{\alpha}\varrho )+u \cdot \nabla_x (\partial_x^{\alpha}\varrho)+[\partial_x^{\alpha}, u \cdot \nabla_x]\varrho +\frac{\varrho}{1-\rho_f} \mathrm{div}_x \partial_x^{\alpha}\left( j_f -\rho_f u \right)+\left[\partial_x^{\alpha}, \frac{\varrho \mathrm{div}_x}{1-\rho_f} \right](j_f-\rho_f u)=\partial_x^{\alpha}S,
\end{align*}
and therefore $\partial_x^{\alpha} \varrho$ satisfies the equation
\begin{align*}
\partial_t (\partial_x^{\alpha}\varrho )+u \cdot \nabla_x (\partial_x^{\alpha}\varrho)+ \frac{\varrho}{1-\rho_f} \mathrm{div}_x \left[ j_{\partial_x^{\alpha} f}-\rho_{\partial_x^{\alpha}f} u\right]=R^{\alpha},
\end{align*}
where $R^{\alpha}$ is a remainder defined by
\begin{align*}
R^{\alpha}&:=\partial_x^{\alpha}S-[\partial_x^{\alpha}, u \cdot \nabla_x]\varrho-\left[\partial_x^{\alpha}, \frac{\varrho \mathrm{div}_x}{1-\rho_f} \right](j_f-\rho_f u)+\frac{\varrho }{1-\rho_f} \mathrm{div}_x \left( [\partial_x^{\alpha},u]\rho_f \right) \\
&:=\partial_x^{\alpha}S +\mathcal{C}_1+\mathcal{C}_2+\mathcal{C}_3.
\end{align*}

Let us estimate each of these terms in $\Ld^2(0,T;\Ld^2)$, for all $\vert \alpha \vert \leq m$.

\medskip

$\bullet$ \textbf{Estimate of $\partial_x^{\alpha}S$}: we use the tame estimate from Proposition \ref{tame:estimate} to write
\begin{align*}
\LRVert{\partial_x^{\alpha}S}_{\Ld^2} &\lesssim
 \LRVert{\frac{\varrho}{1-\rho_f}}_{\Ld^{\infty}}\LRVert{\mathrm{div}_x \, u }_{\H^m}+\LRVert{\frac{\varrho}{1-\rho_f}}_{\H^m}\LRVert{\mathrm{div}_x \, u}_{\Ld^{\infty}} = \mathfrak{S}_1+\mathfrak{S}_2.
\end{align*}
Since $m-2>d/2$, we have by Lemma \ref{LM:rho-pointwise-Hm-2} and \eqref{bound:HAUT-1/1-rho_f}
\begin{align*}
\LRVert{ \mathfrak{S}_1 }_{\Ld^2(0,T)} &\lesssim \LRVert{\frac{1}{1-\rho_f}}_{\Ld^{\infty}(0,T;\Ld^{\infty})}\LRVert{\varrho}_{\Ld^{\infty}(0,T;\H^{m-2})} \LRVert{u}_{\Ld^2(0,T;\H^{m+1})} \leq \Lambda \left( T,R\right).
\end{align*}
 For $ \mathfrak{S}_2$, we combine the tame estimate from Proposition \ref{tame:estimate} with Lemma \ref{LM:(1-g)-1} which provides
\begin{align*}
\LRVert{ \mathfrak{S}_2 }_{\Ld^2(0,T)}& \lesssim \Vert \varrho \Vert_{\Ld^{\infty}(0,T;\H^{m-2})}\LRVert{u}_{\Ld^2(0,T;\H^{m})}  +\Vert \varrho \Vert_{\Ld^2(0,T;\H^{m})} \left\Vert  \frac{1}{1-\rho_f}  \right\Vert_{\Ld^{\infty}(0,T;\Ld^{\infty})}\LRVert{u}_{\Ld^{\infty}(0,T;\H^{m})}  \\
& \quad +\Vert \varrho \Vert_{\Ld^{\infty}(0,T;\H^{m-2})}\left( \LRVert{\rho_f}_{\Ld^{\infty}(0,T;\Ld^{\infty})} \right) \LRVert{\rho_f}_{\Ld^2(0,T;\H^m)} \LRVert{u}_{\Ld^{\infty}(0,T;\H^{m})}  \\
& \lesssim \Lambda(T,R)+\Lambda(T,R)\LRVert{\rho_f}_{\Ld^2(0,T;\H^m)},
\end{align*}
since $m-2>d/2$ and again thanks to Lemma \ref{LM:rho-pointwise-Hm-2} and \eqref{bound:HAUT-1/1-rho_f}. We obtain 
\begin{align*}
\Vert \partial_x^{\alpha}S \Vert_{\Ld^2(0,T;\Ld^2)} \leq \Lambda(T,R),
\end{align*}
 by using Corollary \ref{Coro:endUseAVERAGING}.

\medskip

$\bullet$  \textbf{Estimate of $\mathcal{C}_1$}: we have $-\mathcal{C}_1=[\partial_x^{\alpha}, u \cdot] (\nabla_x \varrho)$ therefore the commutator estimate from Proposition \ref{CommutSOB-KlainBerto} yields
\begin{align*}
\Vert \mathcal{C}_1 \Vert_{\Ld^2} \lesssim \Vert \nabla_x u \Vert_{\Ld^{\infty}} \Vert \nabla_x \varrho \Vert_{\H^{m-1}} + \Vert u \Vert_{\H^m} \Vert \nabla_x \varrho \Vert_{\Ld^{\infty}} \lesssim  \Vert u \Vert_{\H^m} \Vert  \varrho \Vert_{\H^{m}},
\end{align*}
since $m>1+d/2$. We then obtain by Corollary \ref{Coro:endUseAVERAGING}
\begin{align*}
\Vert \mathcal{C}_1 \Vert_{\Ld^2(0,T;\Ld^2)} \lesssim \Vert u \Vert_{\Ld^{\infty}(0,T;\H^m)} \Vert  \varrho \Vert_{\Ld^2(0,T;\H^{m})} \leq \Lambda(T,R).
\end{align*}

\medskip

$\bullet$  \textbf{Estimate of $\mathcal{C}_2$}: we have
\begin{align*}
\mathcal{C}_2=\left[ \partial_x^{\alpha}, \frac{\varrho}{1-\rho_f}\right](\mathrm{div}_x (j_f-\rho_f u)).
\end{align*}
Applying the commutator estimate from Proposition \ref{CommutSOB-KlainBerto}, we get
\begin{align*}
\Vert \mathcal{C}_{2} \Vert_{\Ld^2} &\lesssim \left\Vert \nabla_x \frac{\varrho}{1-\rho_f}  \right\Vert_{\Ld^{\infty}} \Vert j_f-\rho_f u \Vert_{\H^{m}} + \left\Vert  \frac{\varrho}{1-\rho_f}  \right\Vert_{\H^m}   \Vert \mathrm{div}_x (j_f-\rho_f u) \Vert_{\Ld^{\infty}}= \mathcal{C}_{2,1}+ \mathcal{C}_{2,2}.
\end{align*}
For $\mathcal{C}_{2,1}$, we infer from Sobolev embedding (since $m-2>1+d/2$)
\begin{align*}
\left\Vert \nabla_x \frac{\varrho}{1-\rho_f}  \right\Vert_{\Ld^{\infty}} &\leq \left\Vert  \frac{1}{1-\rho_f}  \right\Vert_{\Ld^{\infty}} \left\Vert \nabla_x \varrho  \right\Vert_{\Ld^{\infty}} + \left\Vert  \frac{\varrho}{(1-\rho_f)^2}  \right\Vert_{\Ld^{\infty}} \left\Vert \nabla_x \rho_f  \right\Vert_{\Ld^{\infty}} \\
& \lesssim \left\Vert  \frac{1}{1-\rho_f}  \right\Vert_{\Ld^{\infty}} \left\Vert  \varrho  \right\Vert_{\H^{m-2}} + \left\Vert  \frac{1}{(1-\rho_f)}  \right\Vert_{\Ld^{\infty}}^2 \LRVert{\rho}_{\H^{m-2}} \left\Vert \nabla_x \rho_f  \right\Vert_{\Ld^{\infty}} \\
& \leq \Lambda(T,R),
\end{align*}
thanks to Lemma \ref{LM:rho-pointwise-Hm-2} and Corollary \ref{Coro:endUseAVERAGING}, therefore the tame estimate coming from Proposition \ref{tame:estimate}
entails
\begin{align*}
\LRVert{\mathcal{C}_{2,1}}_{\Ld^2(0,T)} 
&\lesssim\Lambda(T,R) \Big(\Vert j_f\Vert_{\Ld^2(0,T;\H^{m})} +\Vert f  \Vert_{\Ld^{\infty}(0,T;\mathcal{H}^{m-1}_r)}\Vert u \Vert_{\Ld^2(0,T;\H^{m})} \\
& \quad \qquad \qquad \qquad \qquad \qquad \qquad \qquad \qquad + \Vert u \Vert_{\Ld^{\infty}(0,T;\H^{m})} \Vert \rho_f  \Vert_{\Ld^2(0,T;\H^{m})} \Big),
\end{align*}
since $m-1>d/2$. Invoking Lemma \ref{LM:weightedSob} and Corollary \ref{Coro:endUseAVERAGING}, we get 
\begin{align*}
\LRVert{\mathcal{C}_{2,1}}_{\Ld^2(0,T)} \leq \Lambda(T,R).
\end{align*}
For $\mathcal{C}_{2,2}$, we observe that since $m-1>1+d/2$
\begin{align*}
\Vert \mathrm{div}_x (j_f-\rho_f u) \Vert_{\Ld^{\infty}} \lesssim \Vert j_f \Vert_{\H^{m-1}}+ \Vert \rho_f \Vert_{\H^{m-1}}\Vert u \Vert_{\H^{m-1}},
\end{align*}
therefore we obtain 
\begin{align*}
\LRVert{\mathcal{C}_{2,2}}_{\Ld^2(0,T)} \leq \Lambda(T,R),
\end{align*}
by using what we have done for $\mathfrak{S}_2$ above.

\medskip

$\bullet$  \textbf{Estimate of $\mathcal{C}_3$}: we have 
\begin{align*}
\mathcal{C}_3=\frac{\varrho }{1-\rho_f} \Big([\partial_x^{\alpha}, \mathrm{div}_x \, u \cdot ](\rho_f) + [\partial_x^{\alpha}, u \cdot ](\nabla_x \rho_f) \Big)=\frac{\varrho }{1-\rho_f} \left(\mathcal{C}_{3,1}+\mathcal{C}_{3,2} \right).
\end{align*}
Applying Proposition \ref{CommutSOB-KlainBerto}, we get for $m+1>2+d/2$
\begin{align*}
\Vert \mathcal{C}_{3,1} \Vert_{\Ld^2} &\lesssim \Vert \nabla_x \mathrm{div}_x \, u \Vert_{\Ld^{\infty}} \Vert \rho_f \Vert_{\H^{m-1}} + \Vert \mathrm{div}_x \, u \Vert_{\H^m} \Vert  \rho_f \Vert_{\Ld^{\infty}} \lesssim \Vert u \Vert_{ \H^{m+1}} \Vert \rho_f \Vert_{\H^{m-1}},\\[2mm]
\Vert \mathcal{C}_{3,2} \Vert_{\Ld^2} &\lesssim \Vert \nabla_x u \Vert_{\Ld^{\infty}} \Vert \nabla_x \rho_f \Vert_{\H^{m-1}} + \Vert  u \Vert_{\H^m} \Vert \nabla_x \rho_f \Vert_{\Ld^{\infty}} \lesssim \Vert u \Vert_{\H^{m}} \Vert \rho_f \Vert_{\H^m} + \Vert u \Vert_{\H^{m}} \Vert \rho_f \Vert_{\H^{m-1}}.
\end{align*}
Taking the $\Ld^2$-norm in time and using Lemma \ref{LM:weightedSob}, we have for $ \mathcal{C}_{3,1}$
\begin{align*}
\Vert \mathcal{C}_{3,1} \Vert_{\Ld^2(0,T;\Ld^2)} &\lesssim \Vert u \Vert_{ \Ld^2(0,T;\H^{m+1})} \Vert \rho_f \Vert_{\Ld^{\infty}(0,T;\H^{m-1})}
\leq \Lambda(T,R),
\end{align*}
 while for $ \mathcal{C}_{3,2}$, we have 
 \begin{align*}
 \Vert \mathcal{C}_{3,2} \Vert_{\Ld^2(0,T;\Ld^2)} &\lesssim \Vert u \Vert_{\Ld^{\infty}\left(0,T,\H^{m} \right)} \Vert \rho_f \Vert_{\Ld^2(0,T;\H^m)} + \Vert u \Vert_{\Ld^2(0,T;\H^{m+1})} \Vert \rho_f \Vert_{\Ld^{\infty}(0,T;\H^{m-1})} \\
 & \leq  \Lambda(T,R),
 \end{align*}
thanks to Corollary \ref{Coro:endUseAVERAGING}.
\end{proof}

Let us now transform the equation on the derivatives of $\varrho$ obtained in Lemma \ref{LMestim:ResteRalpha}. To ease readibility let us momentaneously set
\begin{align}\label{eqdef:K_1G}
\mathrm{K}_{1,G}^{\mathrm{free}}[F](t,x,):=\int_0^t \int_{\R^d}   v  [\nabla_x F](s,x-(t-s)v) \cdot G(t,s,x,v) \, \mathrm{d}v \, \mathrm{d}s.
\end{align}

\begin{lem}\label{Rewrite-Eqh}
For $h=\partial_x^{\alpha} \varrho$ with $\vert \alpha \vert \leq m$, one has 
\begin{align*}
\partial_t  h+u \cdot \nabla_x h+ \frac{\varrho}{1-\rho_f} \mathrm{div}_x \Big[ \mathrm{K}_{1,G}^{\mathrm{free}}(\mathrm{J}_\eps h)-\mathrm{K}_G^{\mathrm{free}}(\mathrm{J}_\eps h)u\Big]&=\mathcal{R},
\end{align*}
with
$$G(t,x,v):=  p'( \varrho(t,x)) \nabla_v f(t,x,v),$$
and where the remainder $\mathcal{R}$ satisfies
$$
\| \mathcal{R} \|_{\Ld^2(0,T; \Ld^2(\T^d))} \leq \Lambda (T,R).
$$
\end{lem}
\begin{proof}
By Lemma \ref{LMestim:ResteRalpha}, we have
\begin{align*}
\partial_t  h+u \cdot \nabla_x h+ \frac{\varrho}{1-\rho_f} \mathrm{div}_x \left[ j_{\partial_x^{\alpha} f}-\rho_{\partial_x^{\alpha}f} u\right]&=R^{\alpha},
\end{align*}
with
\begin{align*}
\Vert R^{\alpha} \Vert_{\Ld^2(0,T; \Ld^2)} &\leq \Lambda(T,R).
\end{align*}
Thanks to Proposition \ref{coroFInal:D^I:rho-j}, we can write 
\begin{align*}
\rho_{\partial_x^{\alpha}f}= \mathrm{K}_G^{\mathrm{free}}(\mathrm{J}_\eps \partial^{\alpha}_x \varrho)+ \mathrm{R}^{\alpha}[\rho_f], \ \ j_{\partial_x^{\alpha} f}=\mathrm{K}_{1,G}^{\mathrm{free}}(\mathrm{J}_\eps \partial^{\alpha}_x \varrho)+ \mathrm{R}^{\alpha}[j_f],
\end{align*}
with $G(t,x,v):=  p'( \varrho(t,x)) \nabla_v f(t,x,v)$ and where 
\begin{align*}
\left\Vert \mathrm{R}^{\alpha}[\rho_f] \right\Vert_{\Ld^2(0,T, \H^1)} \leq \Lambda(T,R), \ \ \left\Vert  \mathrm{R}^{\alpha}[j_f] \right\Vert_{\Ld^2(0,T, \H^1)} \leq \Lambda(T,R).
\end{align*}
We obtain
\begin{align*}
\partial_t  h+u \cdot \nabla_x h+ \frac{\varrho}{1-\rho_f} \mathrm{div}_x \Big[ \mathrm{K}_{1,G}^{\mathrm{free}}(\mathrm{J}_\eps h)-\mathrm{K}_G^{\mathrm{free}}(\mathrm{J}_\eps h)u\Big]&=R^{\alpha}-\mathrm{div}_x\Big[ \mathrm{R}^{\alpha}[j_f]- \mathrm{R}^{\alpha}[\rho_f]u \Big].
\end{align*}
Thanks to the aforementioned estimates in $\Ld^2_T \H^1_x$, we obtain the desired estimate on the remainder.
\end{proof}

\begin{rem}\label{rem-estim:G}
In the sequel, we shall rely on the following estimate: for all $\ell>0$ and $\sigma>0$ such that $\ell<m-d/2-2$, we have
\begin{align}\label{estim:G}
\underset{0 \leq t \leq T}{\sup} \Vert G(t) \Vert_{\mathcal{H}^{\ell}_{\sigma}} \leq \Lambda(T,R).
\end{align}
Indeed, we have
\begin{align*}
\Vert p'( \varrho(t)) \nabla_v f(t) \Vert_{\mathcal{H}^{\ell}_{\sigma}}^2 & \lesssim \sum_{\vert \mu \vert + \vert \nu \vert \leq \ell} \sum_{\gamma=0}^{\mu + \nu}\Vert  \partial_x^{\gamma}(p'( \varrho(t)))\Vert_{\Ld^{\infty}}^2    \int_{\T^d \times \R^d} \langle v \rangle^{2\sigma}  \vert\partial^{\mu+ \nu-\gamma}_{x,v} \nabla_v f(t,x,v) \vert^2 \, \mathrm{d}x \, \mathrm{d}v \\
& \lesssim  \LRVert{p'( \varrho(t))}_{\H^k}^2 \Vert f(t) \Vert_{\mathcal{H}^{m-1}_{\sigma}}^2,
\end{align*}
if $m-1 \geq \ell$ and $k>\frac{d}{2}+\ell$. Invoking Lemma \ref{LM:rho-pointwise-Hm-2} by choosing also $k\leq m-2$, we obtain the estimate \eqref{estim:G}.
\end{rem}

Our goal is now to understand the term
$$\mathrm{div}_x \Big[ \mathrm{K}_{1,G}^{\mathrm{free}}(\mathrm{J}_\eps h)-\mathrm{K}_G^{\mathrm{free}}(\mathrm{J}_\eps h)u\Big].$$
We will show that up to good remainders, it is related to the transport part $\partial_t  h+u \cdot \nabla_x h$ appearing in the equation of Lemma \ref{Rewrite-Eqh}. This is the object of the following Lemmas \ref{Commut:K0}--\ref{Commut:K1} which are crucial commutation results.

\begin{lem}\label{Commut:K0}
For all $|\alpha|\leq m$ and $h=\partial_x^{\alpha} \varrho$, there holds
$$
 \mathrm{div}_x (\mathrm{K}_G^{\mathrm{free}}[\mathrm{J}_\eps h]u) =  \mathrm{K}_G^{\mathrm{free}}[ \mathrm{J}_\eps (u \cdot \na_x h)]  + \mathcal{R},
$$
with 
$$
\| \mathcal{R} \|_{\Ld^2(0,T; \Ld^2(\T^d))} \leq \Lambda (T,R).
$$

\end{lem}

\begin{proof}
First, let us prove that for all smooth function $\mathfrak{h}(t,x)$, we have
\begin{align}\label{commutK0-div}
 \mathrm{div}_x (\mathrm{K}_G^{\mathrm{free}}[\mathfrak{h}]u) =  \mathrm{K}_G^{\mathrm{free}}[  (u \cdot \na_x \mathfrak{h})]  + \mathcal{R},
\end{align}
with 
\begin{align}\label{estimate:KO-div}
\| \mathcal{R} \|_{\Ld^2(0,T; \Ld^2(\T^d))} \leq \Lambda (T,R,\Vert \mathfrak{h} \Vert_{\Ld^2(0,T; \Ld^2(\T^d))} ).
\end{align}
We have
\begin{align*}
 \mathrm{div}_x (u \mathrm{K}_G^{\mathrm{free}}[\mathfrak{h}]) &=  u \cdot \na_x \mathrm{K}_G^{\mathrm{free}}[\mathfrak{h}] +   (\mathrm{div}_x u) \mathrm{K}_G^{\mathrm{free}}[\mathfrak{h}] \\
 &=  \mathrm{K}_G^{\mathrm{free}}[ (u \cdot \na_x \mathfrak{h})] +   (\mathrm{div}_x u) \mathrm{K}_G^{\mathrm{free}}[\mathfrak{h}]  + [u\cdot \na_x, \mathrm{K}_G^{\mathrm{free}}] [\mathfrak{h}].
\end{align*}

Using the notation $\partial_i=\partial_{x_i}$, we have
\begin{align*}
 [u\cdot \na_x, \mathrm{K}_G^{\mathrm{free}} ] [\mathfrak{h}](t,x) &=\sum_{i=1}^d \int_{\R^d}u_i(t,x) \partial_i (\mathrm{K}_G^{\mathrm{free}}[\mathfrak{h}])(t,x)-\sum_{i=1}^d \int_{\R^d}\mathrm{K}_G^{\mathrm{free}}[(u_i \partial_i \mathfrak{h})](t,x) \\
&=\sum_{i=1}^d \int_{\R^d}u_i(t,x)  \int_0^t  \int_{\R^d} \nabla_x ( \partial_i \mathfrak{h})(s,x-(t-s)v)\cdot G(t,x,v)  \, \mathrm{d}v \, \mathrm{d}s \\
& \quad +\sum_{i=1}^d \int_{\R^d}u_i(t,x)  \int_0^t  \int_{\R^d} \nabla_x  \mathfrak{h}(s,x-(t-s)v)\cdot \partial_i G(t,x,v)  \, \mathrm{d}v \, \mathrm{d}s \\
 & \quad -\sum_{i=1}^d \int_{\R^d}\int_0^t  \int_{\R^d} \partial_i \mathfrak{h}(s,x-(t-s)v) \nabla_x u_i(s,x-(t-s)v)\cdot  G(t,x,v)  \, \mathrm{d}v \, \mathrm{d}s \\
 & \quad -\sum_{i=1}^d \int_{\R^d}\int_0^t  \int_{\R^d} u_i(s,x-(t-s)v) \nabla_x(\partial_i \mathfrak{h})(s,x-(t-s)v) \cdot  G(t,x,v)  \, \mathrm{d}v \, \mathrm{d}s \\
 &=\mathrm{K}^{\mathrm{free}}_{(u\cdot \na_x) G} [\mathfrak{h}](t,x)  - \mathrm{K}^{\mathrm{free}}_{(G\cdot \na_x) \widetilde{u}} [\mathfrak{h}](t,x) \\
 & \quad + \sum_{i=1}^d \int_{\R^d} \nabla_x (\partial_i \mathfrak{h})(s,x-(t-s)v)\cdot \Big((u_i(t,x)-\widetilde{u}_i(s,t,x,v)) G(t,x,v) \Big)  \, \mathrm{d}v \, \mathrm{d}s,
\end{align*}
where we have set $\widetilde{u}(s,t,x,v)=u(s,x-(t-s)v)$. We thus get
\begin{align*}
 \mathrm{div}_x (u \mathrm{K}_G^{\mathrm{free}}[\mathfrak{h}]) =\mathrm{K}_G^{\mathrm{free}}[ (u \cdot \na_x \mathfrak{h})] +   (\mathrm{div}_x u) \mathrm{K}_G^{\mathrm{free}}[\mathfrak{h}] +\mathrm{K}^{\mathrm{free}}_{(u\cdot \na_x) G} [\mathfrak{h}]  - \mathrm{K}^{\mathrm{free}}_{(G\cdot \na_x) \widetilde{u}} [\mathfrak{h}] + \sum_{i=1}^d \mathrm{K}^{\mathrm{free}}_{(u_i-\widetilde{u}_i) G} [\partial_i \mathfrak{h}],
\end{align*}
which gives a decomposition as \eqref{commutK0-div}. In the sequel, we shall constantly use the estimate \eqref{estim:G} of Remark \ref{rem-estim:G}, that is, for all $0<p<m-d/2-2$ and $\sigma>0$
\begin{align}
\underset{0 \leq t \leq T}{\sup} \Vert G(t) \Vert_{\mathcal{H}^{p}_{\sigma}} \leq \Lambda(T,R).
\end{align}
Let us estimate the different terms of the previous decomposition in order to prove \eqref{estimate:KO-div}. For the first one, we use the smoothing estimate of Proposition \ref{propo:Averag-ORiGINAL} to directly get for $\ell>1+d$ and $\sigma>d/2$
\begin{align*}
\left\Vert (\mathrm{div}_x u) \mathrm{K}_G^{\mathrm{free}}[\mathfrak{h}] \right\Vert_{\Ld^2(0,T; \Ld^2)} &\leq \Lambda(T,R) \underset{0 \leq t \leq T}{\sup} \Vert G(t) \Vert_{\mathcal{H}^{\ell}_{\sigma}} \LRVert{\mathfrak{h}}_{\Ld^2(0,T; \Ld^2)} \leq \Lambda(T,R)\LRVert{\mathfrak{h}}_{\Ld^2(0,T; \Ld^2)},
\end{align*}
since $m>3+3d/2$. For the second and third ones, Proposition \ref{propo:Averag-ORiGINAL} yields
\begin{multline*}
\left\Vert \mathrm{K}^{\mathrm{free}}_{(u\cdot \na_x) G} [\mathfrak{h}] \right\Vert_{\Ld^2(0,T; \Ld^2)} + \left\Vert \mathrm{K}^{\mathrm{free}}_{(G\cdot \na_x) \widetilde{u}} [\mathfrak{h}] \right\Vert_{\Ld^2(0,T; \Ld^2)} \\
\lesssim  \left( \underset{0 \leq s,t \leq T}{\sup} \Vert (u\cdot \na_x) G(t,s) \Vert_{\mathcal{H}^p_{\sigma}}.+\underset{0 \leq s,t \leq T}{\sup} \Vert (G\cdot \na_x) \widetilde{u}(t,s) \Vert_{\mathcal{H}^p_{\sigma}}\right)\LRVert{\mathfrak{h}}_{\Ld^2(0,T; \Ld^2)},
\end{multline*}
if $\sigma >d/2$ and $p>1+d$. With the same arguments as in Remark \ref{rem-estim:G} to estimate the terms inside the parentheses, we get (since $m>4+3d/2$)
\begin{align*}
\left\Vert \mathrm{K}^{\mathrm{free}}_{(u\cdot \na_x) G} [\mathfrak{h}] \right\Vert_{\Ld^2(0,T; \Ld^2)} + \left\Vert \mathrm{K}^{\mathrm{free}}_{(G\cdot \na_x) \widetilde{u}} [\mathfrak{h}] \right\Vert_{\Ld^2(0,T; \Ld^2)} \lesssim \Lambda(T,R) \LRVert{\mathfrak{h}}_{\Ld^2(0,T; \Ld^2)} .
\end{align*}
For the last term, we observe that the kernel $(u_i(t)-\widetilde{u}_i(t,s,x,v)) G(t,x,v)$ vanishes at $s=t$, therefore Remark \ref{rem:Variant-PropoAveragReg} implies
\begin{align*}
\left\Vert \mathrm{K}^{\mathrm{free}}_{(u_i-\widetilde{u}_i) G} [\partial_i \mathfrak{h}]\right\Vert_{\Ld^2(0,T; \Ld^2)}  &\lesssim \Lambda(T) \underset{0 \leq s,t \leq T}{\sup} \Vert \partial_s(u_i-\widetilde{u}_i) G(t,s) \Vert_{\mathcal{H}^p_{\sigma}}\LRVert{\mathfrak{h}}_{\Ld^2(0,T; \Ld^2)} \\
&=\Lambda(T) \underset{0 \leq s,t \leq T}{\sup} \Vert \partial_s\widetilde{u}_i(t,s) \Vert_{\H^k} \underset{0 \leq t \leq T}{\sup} \Vert  G(t) \Vert_{\mathcal{H}^p_{\sigma}}\LRVert{\mathfrak{h}}_{\Ld^2(0,T; \Ld^2)},
\end{align*}
for $p>7+d/2$, $k>p+d/2$ and $\sigma>d/2$. Using the equation on $u$, we obtain (since $m>9+d$)
\begin{align*}
\left\Vert \mathrm{K}^{\mathrm{free}}_{(u_i-\widetilde{u}_i) G} [\partial_i \mathfrak{h}]\right\Vert_{\Ld^2(0,T; \Ld^2)} \leq \Lambda(T,R)\LRVert{\mathfrak{h}}_{\Ld^2(0,T; \Ld^2)}.
\end{align*}
All in all, this yields the claimed estimate \eqref{estimate:KO-div}.

Now applying \eqref{commutK0-div} and \eqref{estimate:KO-div} with  $\mathfrak{h}= \mathrm{J}_{\eps} h$, we get
$$ \mathrm{div}_x (u \mathrm{K}_G^{\mathrm{free}}[\mathrm{J}_{\eps} h]) =  \mathrm{K}_G^{\mathrm{free}}[  (u \cdot \na_x \mathrm{J}_{\eps} h)]  + \mathcal{R},
$$
with 
$$
\| \mathcal{R} \|_{\Ld^2(0,T; \Ld^2(\T^d))} \leq \Lambda (T,R,\Vert \mathrm{J}_{\eps} h\Vert_{\Ld^2(0,T; \Ld^2(\T^d))} ) \leq \Lambda (T,R,\Vert h\Vert_{\Ld^2(0,T; \Ld^2(\T^d))} ).
$$
Finally observe that 
\begin{align*}
\mathrm{K}_G^{\mathrm{free}}[  (u \cdot \na_x \mathrm{J}_{\eps} h)] =\mathrm{K}_G^{\mathrm{free}}[ \mathrm{J}_\eps (u \cdot \na_x h)]+ \mathrm{K}_G^{\mathrm{free}} \big[[u \cdot \nabla_x, \mathrm{J}_\eps]h\big].
\end{align*}
Relying once again on Proposition \ref{propo:Averag-ORiGINAL}, we thus have
\begin{align*}
\LRVert{\mathrm{K}_G^{\mathrm{free}} \big[[u \cdot \nabla_x, \mathrm{J}_\eps]h\big]}_{\Ld^2(0,T; \Ld^2(\T^d))} \leq \Lambda(T,R) \LRVert{[u \cdot \nabla_x, \mathrm{J}_\eps]h\big]}_{\Ld^2(0,T; \Ld^2(\T^d))}.
\end{align*}
Invoking (a variant of the proof of) \cite[Theorem C.14]{BS} about the commutator between a differential operator of order $1$ and a regularizing operator, we get
\begin{align*}
 \LRVert{[u \cdot \nabla_x, \mathrm{J}_\eps]h\big]}_{\Ld^2(\T^d)} \lesssim \LRVert{u}_{\mathrm{W}^{1,\infty}(\T^d)} \LRVert{h}_{\Ld^2(\T^d)},
\end{align*}
where this estimate is independent of $\eps$. We obtain
\begin{align*}
\LRVert{\mathrm{K}_G^{\mathrm{free}} \big[[u \cdot \nabla_x, \mathrm{J}_\eps]h\big]}_{\Ld^2(0,T; \Ld^2(\T^d))} \leq \Lambda(T,R),
\end{align*}
which concludes the proof.

\end{proof}

\begin{lem}\label{Commut:K1}
For all $|\alpha|\leq m$ and $h=\partial_x^{\alpha} \varrho$, and with $G(s,t,x,v)=p'(\varrho(t,x))\nabla_v f(t,x,v)$, there holds
$$
 \mathrm{div}_x \mathrm{K}_{1,G}^{\mathrm{free}}[\mathrm{J}_\eps h] = - \mathrm{K}_G^{\mathrm{free}}[\mathrm{J}_\eps \pa_s h]  + \mathcal{R},
$$
with 
$$
\| \mathcal{R} \|_{\Ld^2(0,T; \Ld^2(\T^d))} \leq \Lambda (T,R, \|h(0)\|_{\H^1(\T^d)}).
$$

\end{lem}

\begin{proof}
Recall the definition \eqref{eqdef:K_1G} for $\mathrm{K}_{1,G}^{\mathrm{free}}$. We first write
$$
 \mathrm{div}_x \mathrm{K}_{1,G}^{\mathrm{free}}[\mathrm{J}_\eps h] = p'(\varrho(t,x))\int_0^t  \int_{\R^d}  \nabla_x    \mathrm{div}_x  \left[v \mathrm{J}_\eps h](s, x-(t-s)v)\right] \cdot \nabla_v f(t,x,v) \, \mathrm{d}v \, \mathrm{d}s  + \mathcal{R},
$$
with 
\begin{align*}
\mathcal{R}(t,x)&:=\int_{\R^d}  \nabla_x   [\mathrm{J}_\eps  h](s, x-(t-s)v) \cdot \big[ (v \cdot \nabla_x)( p'(\varrho(t,x)) \nabla_v f(t,x,v)) \big] \, \mathrm{d}v \, \mathrm{d}s  \\
&=\mathrm{K}^{\mathrm{free}}_{(v \cdot \nabla_x)( p'(\varrho) \nabla_v f)}[\mathrm{J}_\eps  h].
\end{align*}
Thanks to Proposition \ref{propo:Averag-ORiGINAL}, we thus have for $p>1+d$ and $\sigma>d/2$
\begin{align*}
\| \mathcal{R} \|_{\Ld^2(0,T; \Ld^2(\T^d))} & \lesssim \LRVert{\mathrm{J}_\eps  h}_{\Ld^2(0,T; \Ld^2(\T^d))} \underset{0 \leq s,t \leq T}{\sup} \Vert (v \cdot \nabla_x p'(\varrho)) \nabla_v f(t,s) \Vert_{\mathcal{H}^p_{\sigma}} \\
&\leq  \Lambda (T,R).
\end{align*}
Now observe the identity
$$
\pa_s \Big[  \mathrm{J}_\eps h(s, x-(t-s)v) \Big] = (\pa_s \mathrm{J}_\eps h)(s, x-(t-s)v)  +  \mathrm{div}_x (v  \mathrm{J}_\eps h) (s, x-(t-s)v).
$$
Since
$$\int_{\R^d} \na_x \mathrm{J}_\eps h(t, x)\cdot  \na_v f(t,x,v) \, \dd v = 0,$$
we have
\begin{align*}
\int_0^t \int_{\R^d} \pa_s \Big[ \na_x \mathrm{J}_\eps h(s, x-(t-s)v) \Big]\cdot  \na_v f(t,x,v) \, \dd v \dd s &=   - \int_{\R^d} \na_x \mathrm{J}_\eps h(0, x-tv)\cdot  \na_v f(t,x,v) \, \dd v.
\end{align*}
Using the generalized Minkowski inequality and the Sobolev embedding, 
the last term is estimated in $\Ld^2(0,T,\Ld^2)$ by
\begin{align*}
\Vert \na_x \mathrm{J}_\eps h(0) \Vert_{\Ld^2(\T^d)} \LRVert{\int_{\R^d} \underset{x \in \T^d}{\sup} \vert \nabla_v f(\cdot, x,v) \vert  \, \mathrm{d}v}_{\Ld^2(0,T)} \leq \Lambda(T,R)\|h(0)\|_{\H^1(\T^d)}.
\end{align*}
We deduce that we can write
$$
 \mathrm{div}_x \mathrm{K}_{1,G}^{\mathrm{free}}[\mathrm{J}_\eps h]  = - \mathrm{K}_G^{\mathrm{free}}[\pa_s \mathrm{J}_\eps h] +  \mathcal{R},
$$
with 
$$
\| \mathcal{R} \|_{\Ld^2(0,T; \Ld^2(\T^d))} \leq \Lambda (T,R, \|h(0)\|_{\H^1(\T^d)}).
$$ 
\end{proof}

Combining the results of Lemmas \ref{Rewrite-Eqh}--\ref{Commut:K0}--\ref{Commut:K1} leads to the proof of Proposition~\ref{coro:Facto}.

\subsection{Propagation of the Penrose condition for short times}\label{Subsec:PropagPenrose}
We now show how to propagate the Penrose stability condition \eqref{cond:Penrose} for short times. This will allow to study the operator
$$\mathrm{Id}-\frac{\varrho}{1-\rho_f}\mathrm{K}_G^{\mathrm{free}}\circ\mathrm{J}_{\eps}$$
in the next sections, the outset being its ellipticity. 

First, we shall need several estimates on the time derivatives of the solutions. We have the following basic lemma.

\begin{lem}\label{LMestim:Dt-f-rhof-rho}
Let $s \geq 0$ and $\sigma\geq 0$. For all $T \in (0,T_\eps)$, the following holds.
\begin{itemize}
\item If $s>d$ and $s+1>d/2$, we have
\begin{align*}
\Vert \partial_t f \Vert_{\Ld^{\infty}(0,T;\mathcal{H}^{s}_{\sigma})} \lesssim \Vert  f \Vert_{\Ld^{\infty}(0,T;\mathcal{H}^{s+1}_{\sigma+1})} + \Vert  f \Vert_{\Ld^{\infty}(0,T;\mathcal{H}^{s+1}_{\sigma})}\left( \Vert u \Vert_{\Ld^{\infty}(0,T;\H^{s}) } + \Lambda\left(\Vert \varrho \Vert_{\Ld^{\infty}(0,T;\H^{s+1}) } \right) \right).
\end{align*}
\item If  $s>d/2$ and $\sigma >1+d/2$, we have
\begin{align*}
\Vert \partial_t \rho_f \Vert_{\Ld^{\infty}(0,T;\Ld^{\infty})} \lesssim \Vert \partial_t \rho_f \Vert_{\Ld^{\infty}(0,T;\H^s)}  \lesssim \Vert  f \Vert_{\Ld^{\infty}(0,T;\mathcal{H}^{1+s}_{\sigma})}.
\end{align*}
\item If $s>d/2$ and $\sigma >1+d/2$, we have
\begin{align*}
\Vert \partial_t \varrho \Vert_{\Ld^{\infty}(0,T;\Ld^{\infty})} &\lesssim \Vert u \Vert_{\Ld^{\infty}(0,T;\mathrm{H}^{s})} \Vert  \varrho \Vert_{\Ld^{\infty}(0,T;\mathrm{H}^{1+s})} \\
& \quad +\LRVert{\frac{1}{1-\rho_f}}_{\Ld^{\infty}(0,T;\Ld^{\infty})} \Vert \varrho \Vert_{\Ld^{\infty}(0,T;\mathrm{H}^{s})} 
\Big( \Vert f \Vert_{\Ld^{\infty}(0,T;\mathcal{H}^{1+s}_{\sigma})}+\Vert u \Vert_{\Ld^{\infty}(0,T;\mathrm{H}^{1+s})} \\
& \qquad \qquad \qquad \qquad \qquad \qquad  \qquad \qquad \qquad \qquad \qquad + \Vert f \Vert_{\Ld^{\infty}(0,T;\mathcal{H}^{1+s}_{\sigma})} \Vert u \Vert_{\Ld^{\infty}(0,T;\mathrm{H}^{1+s})}  
  \Big).
\end{align*}
\end{itemize}
\end{lem}
\begin{proof}
$\bullet$ Using the Vlasov equation satisfied by $f$, we get
\begin{align*}
\Vert \partial_t f \Vert_{\mathcal{H}^{s}_{\sigma}} & \lesssim \Vert  f \Vert_{\mathcal{H}^{s}_{\sigma}}+  \Vert v \cdot \nabla_x  f \Vert_{\mathcal{H}^{s}_{\sigma}} + \Vert v \cdot \nabla_v f \Vert_{\mathcal{H}^{s}_{\sigma}} + \Vert E_{\reg,\eps}^{u,\varrho} \cdot  \nabla_v f  \Vert_{\mathcal{H}^{s}_{\sigma}} \lesssim  \Vert  f \Vert_{\mathcal{H}^{s+1}_{\sigma+1}}  + \Vert E_{\reg,\eps}^{u,\varrho} \cdot  \nabla_v f  \Vert_{\mathcal{H}^{s}_{\sigma}}.
\end{align*}
Next, combining the estimate \eqref{Sob:estim2} of Lemma \ref{LM:ProducLawWeight} and the estimate \eqref{bound:Sobo2:E} of Lemma \ref{LM:Estim-ForceField}, we have for $s>d$ such that $s>3+d/2$
\begin{align*}
\Vert E_{\reg,\eps}^{u,\varrho} \cdot  \nabla_v f  \Vert_{\mathcal{H}^{s}_{\sigma}} \lesssim\Vert E_{\reg,\eps}^{u,\varrho} \Vert_{\mathrm{H}^{s}} \Vert  f  \Vert_{\mathcal{H}^{s+1}_{\sigma}} \lesssim \Vert  f  \Vert_{\mathcal{H}^{s+1}_{\sigma}} \left( \Vert u(t) \Vert_{\H^{s}} + \Lambda ( \Vert \varrho (t) \Vert_{\H^{s+1}} )\right) ,
\end{align*}
hence providing the first estimate.

$\bullet$ To estimate $\partial_t \rho_f$, we use the fact that $\partial_t \rho_f=-\mathrm{div}_x j_f$ so that we have by Sobolev embedding
\begin{align*}
\Vert \partial_t \rho_f \Vert_{\Ld^{\infty}(0,T;\Ld^{\infty})}\lesssim \Vert \partial_t \rho_f \Vert_{\Ld^{\infty}(0,T;\H^s)}  \lesssim \Vert j_f \Vert_{\Ld^{\infty}(0,T;\mathrm{H}^{1+s})} \lesssim \Vert  f \Vert_{\Ld^{\infty}(0,T;\mathcal{H}^{1+s}_{\sigma})},
\end{align*}
thanks to Lemma \ref{LM:weightedSob}. This gives the second estimate.

$\bullet$ To estimate $\partial_t \varrho$, we use the equation
\begin{align*}
\partial_t \varrho=-u \cdot \nabla_x \varrho-\frac{1}{1-\rho_f} \, \mathrm{div}_x(j_f-\rho_f u+u) \varrho,
\end{align*}
from which we infer 
\begin{multline*}
\Vert \partial_t \varrho \Vert_{\Ld^{\infty}(0,T;\Ld^{\infty})} \lesssim \Vert u \Vert_{\Ld^{\infty}(0,T;\mathrm{H}^{s})} \Vert \nabla_x \varrho \Vert_{\Ld^{\infty}(0,T;\mathrm{H}^{s})} \\ +\LRVert{\frac{1}{1-\rho_f}}_{\Ld^{\infty}(0,T;\Ld^{\infty})} \Vert \varrho \Vert_{\Ld^{\infty}(0,T;\mathrm{H}^{s})} \Vert j_f-\rho_f u +u \Vert_{\Ld^{\infty}(0,T;\mathrm{H}^{1+s})}.
\end{multline*}
We conclude by using Lemma \ref{LM:weightedSob}.
\end{proof}

The propagation of the Penrose condition \eqref{cond:Penrose} for short times is obtained in the following lemma.

\begin{lem}\label{LM:propagPENROSE}
There exists $\widetilde{T}_0(R)>0$ independent of $\eps$ such that the following holds: if $(f^{\mathrm{in}},\varrho^{\mathrm{in}})$ satisfies the $c$-Penrose stability condition $\eqref{cond:Penrose}_{c}$ for some $c>0$ then $(f(t),\varrho(t))$ satisfies the $\frac{c}{2}$-Penrose stability condition $\eqref{cond:Penrose}_{c/2}$ for all $t \in \left[0, \min\left(\widetilde{T}_0(R),T^{\eps} \right) \right]$.
\end{lem}
\begin{proof}
Let $T<T^{\eps}$ and recall the Definition \ref{def:PenroseSymbol} of the Penrose function $\mathscr{P}$. We start writing for all $t \in [0,T]$
\begin{align*}
1-\mathscr{P}_{f(t), \varrho(t)}(x,\gamma, \tau, k)&=1-\frac{p'(\varrho(t,x))\rho(t,x)}{1-\rho_f(t,x)}\int_0^{+\infty} e^{-(\gamma+i \tau) s} \frac{i k}{1+ \vert k \vert^2} \cdot \left( \mathcal{F}_v \nabla_v f \right)(t,x,ks)  \, \mathrm{d}s \\
&=1-\mathscr{P}_{f^{\mathrm{in}}, \varrho^{\mathrm{in}}}(x,\gamma, \tau, k) +\mathbf{A}_0(t,x,\gamma, \tau,k)+ \mathbf{B}_0(t,x,\gamma, \tau,k),
\end{align*}
where
\begin{multline*}
\mathbf{A}_0(t,x,\gamma, \tau,k)\\:= -\frac{p'(\varrho^{\mathrm{in}}(x))\varrho^{\mathrm{in}}(x)}{1-\rho_{f^{\mathrm{in}}}(x)} \left( \int_0^{+\infty} e^{-(\gamma+i \tau) s} \frac{i k}{1+ \vert k \vert^2} \cdot \left[ \left(\mathcal{F}_v \nabla_v f \right)(t,x,ks) -\left( \mathcal{F}_v \nabla_v f^{\mathrm{in}} \right)(x,ks)  \right] \, \mathrm{d}s \right),
\end{multline*}
and 
\begin{align*}
\mathbf{B}_0(t,x,\gamma, \tau,k)&:=\left( \frac{p'(\varrho^{\mathrm{in}}(x))\varrho^{\mathrm{in}}(x)}{1-\rho_{f^{\mathrm{in}}}(x)}-\frac{p'(\varrho(t,x))\varrho(t,x)}{1-\rho_f(t,x)}\right)\int_0^{+\infty} e^{-(\gamma+i \tau) s} \frac{i k}{1+ \vert k \vert^2} \cdot \left( \mathcal{F}_v \nabla_v f \right)(t,x,ks)  \, \mathrm{d}s.
\end{align*}
$\bullet$  \textbf{Estimate of $\mathbf{A}_0$}: using Taylor's formula, we have
\begin{align*}
\vert \mathbf{A}_0(t,x,\gamma, \tau,k) \vert &\leq \left\Vert \frac{p'(\varrho^{\mathrm{in}})\varrho^{\mathrm{in}}}{1-\rho_{f^{\mathrm{in}}}} \right\Vert_{\Ld^{\infty}}   \int_0^{+\infty}  \vert k \vert \int_0^T  \vert \left(\mathcal{F}_v \nabla_v \partial_t f \right)(\theta,x,ks) \vert \, \mathrm{d}\theta   \, \mathrm{d}s \\
&\leq \left\Vert \frac{p'(\varrho^{\mathrm{in}})\varrho^{\mathrm{in}}}{1-\rho_{f^{\mathrm{in}}}} \right\Vert_{\Ld^{\infty}}  \int_0^T  \int_0^{+\infty}    \left\vert \left(\mathcal{F}_v \nabla_v \partial_t f \right)\left(\theta,x,\frac{k}{\vert k \vert}s \right) \right\vert    \, \mathrm{d}s \, \mathrm{d}\theta \\
&\lesssim \left\Vert \frac{p'(\varrho^{\mathrm{in}})\varrho^{\mathrm{in}}}{1-\rho_{f^{\mathrm{in}}}} \right\Vert_{\Ld^{\infty}}  \int_0^T     \underset{s \in \R^+}{\sup} (1+s^2) \left\vert \left(\mathcal{F}_v \nabla_v \partial_t f \right)\left(\theta,x,\frac{k}{\vert k \vert}s \right) \right\vert  \, \mathrm{d}\theta,
\end{align*}
therefore we get for $\sigma>d/2$
\begin{align*}
\vert \mathbf{A}_0(t,x,\gamma, \tau,k) \vert &\lesssim \left\Vert \frac{p'(\varrho^{\mathrm{in}})\varrho^{\mathrm{in}}}{1-\rho_{f^{\mathrm{in}}}} \right\Vert_{\Ld^{\infty}}  \int_0^T   \underset{s \in \R^+}{\sup} \sum_{\vert \beta \vert \leq 2} \left\vert \left(\mathcal{F}_v \partial_v^{\beta} \nabla_v \partial_t f \right)\left(\theta,x,\frac{k}{\vert k \vert}s \right) \right\vert  \, \mathrm{d}\theta \\
&\lesssim \left\Vert \frac{p'(\varrho^{\mathrm{in}})\varrho^{\mathrm{in}}}{1-\rho_{f^{\mathrm{in}}}} \right\Vert_{\Ld^{\infty}}  \int_0^T  \left(\sum_{\vert \beta \vert \leq 2} \int_{\R^d}(1+ \vert v \vert^2)^{\sigma} \left\vert \partial_v^{\beta} \nabla_v \partial_t f \left(\theta,x,v\right) \right\vert^2 \, \mathrm{d}v \right)^{\frac{1}{2}} \, \mathrm{d}\theta \\
& \lesssim \left\Vert \frac{p'(\varrho^{\mathrm{in}})\varrho^{\mathrm{in}}}{1-\rho_{f^{\mathrm{in}}}} \right\Vert_{\Ld^{\infty}}  T \Vert \partial_t f \Vert_{\Ld^{\infty}(0,T;\mathcal{H}^{3+s}_{\sigma})},
\end{align*}
for all $s>d/2$, thanks to the Sobolev embedding. Invoking Lemma \ref{LMestim:Dt-f-rhof-rho} (taking $3+s>d$ and $3+s+1>d/2$), we have
\begin{align*}
\Vert \partial_t f \Vert_{\Ld^{\infty}(0,T;\mathcal{H}^{3+s}_{\sigma})} \leq \Lambda\left( T,R \right),
\end{align*}
thanks to Lemma \ref{LM:rho-pointwise-Hm-2}, choosing $s+6 \leq m $.
 Therefore there exists a universal constant $C>0$ such that 
\begin{align*}
\vert \mathbf{A}_0(t,x,\gamma, \tau,k)\vert \leq C \left\Vert \frac{p'(\varrho^{\mathrm{in}})\varrho^{\mathrm{in}}}{1-\rho_{f^{\mathrm{in}}}} \right\Vert_{\Ld^{\infty}} T \Lambda\left( T,R\right).
\end{align*}
$\bullet$ \textbf{Estimate of $\mathbf{B}_0$}: Likewise, we have for all $s>d/2$ such that $3+s<m-1$
\begin{align*}
\vert \mathbf{B}_0(t,x,\gamma, \tau,k) \vert &\lesssim \Vert f \Vert_{\Ld^{\infty}(0,T;\mathcal{H}^{3+s}_{\sigma})} \left\vert \frac{p'(\varrho(t,x))\rho(t,x)}{1-\varrho_f(t,x)} -\frac{p'(\varrho^{\mathrm{in}}(x))\varrho^{\mathrm{in}}(x)}{1-\rho_{f^{\mathrm{in}}}(x)}\right\vert \\
& \leq R \int_0^T \left\vert \partial_t \left\lbrace \frac{p'(\varrho(\theta,x))\varrho(\theta,x)}{1-\rho_f(\theta,x)} \right\rbrace \right\vert \, \mathrm{d}\theta \\
& \leq  RT \left\Vert \partial_t \left\lbrace \frac{p'(\varrho)\varrho}{1-\rho_f} \right\rbrace \right\Vert_{\Ld^{\infty}(0,T;\Ld^{\infty})},
\end{align*}
Now observe that
\begin{align*}
\partial_t \left\lbrace \frac{p'(\varrho)\varrho}{1-\rho_f} \right\rbrace=\frac{1}{1-\rho_f}\left(\varrho \partial_{t} \varrho  \, p''(\varrho)+p'(\varrho) \partial_{t} \varrho \right)+ p'(\varrho) \varrho \frac{ \partial_{t} \rho_f}{(1-\rho_f)^2},
\end{align*}
therefore by the Sobolev embedding, we get for all $s>d/2$
\begin{align*}
\left\Vert \partial_t \left\lbrace \frac{p'(\varrho)\varrho}{1-\rho_f} \right\rbrace \right\Vert_{\Ld^{\infty}(0,T;\Ld^{\infty})} 
&\leq \left\Vert \frac{1}{1-\rho_f} \right\Vert_{\Ld^{\infty}(0,T;\Ld^{\infty})} \Vert \partial_t \varrho \Vert_{\Ld^{\infty}(0,T;\Ld^{\infty})} \\
& \qquad \qquad \qquad \qquad \quad  \times\left(\left\Vert \varrho    \right\Vert_{\Ld^{\infty}(0,T;\H^s)} \left\Vert  p''(\varrho)   \right\Vert_{\Ld^{\infty}(0,T;\H^s)}+\left\Vert  p'(\varrho)   \right\Vert_{\Ld^{\infty}(0,T;\H^s)}
 \right)  \\
 & \quad +\left\Vert \frac{1}{1-\rho_f} \right\Vert_{\Ld^{\infty}(0,T;\Ld^{\infty})}^2  \Vert p'(\varrho) \Vert_{\Ld^{\infty}(0,T;\H^s)}  \Vert  \varrho \Vert_{\Ld^{\infty}(0,T;\H^s)}  \Vert \partial_t \rho_f \Vert_{\Ld^{\infty}(0,T;\Ld^{\infty})}.
\end{align*}
Using Lemma \ref{LMestim:Dt-f-rhof-rho} with Proposition \ref{LM:rho-pointwise-Hm-2}, we obtain as before
\begin{align*}
\vert \mathbf{B}_0(t,x,\gamma, \tau,k) \vert \leq C T\Lambda \left( T,R\right).
\end{align*}
All in all, we have for all $t \in [0,T]$
\begin{align*}
\vert \mathbf{A}_0(t,x,\gamma, \tau,k) \vert + \vert \mathbf{B}_0(t,x,\gamma, \tau,k) \vert 
\leq  C \left(\left\Vert \frac{p'(\varrho^{\mathrm{in}})\varrho^{\mathrm{in}}}{1-\rho_{f^{\mathrm{in}}}} \right\Vert_{\Ld^{\infty}}+1 \right) T \Lambda\left( T,R \right).
\end{align*}
Consequently, there exists $\eta_0>0$ independent of $\eps$ (and depending on $c$) such that for $\widetilde{T}_0=\widetilde{T}_0(R)$ satisfying
\begin{align*}
C \left( \left\Vert \frac{p'(\varrho^{\mathrm{in}})\varrho^{\mathrm{in}}}{1-\rho_{f^{\mathrm{in}}}} \right\Vert_{\Ld^{\infty}} +1 \right) \widetilde{T}_0 \Lambda\left( \widetilde{T}_0,R \right) \leq \eta_0,
\end{align*}
the $\frac{c}{2}$-Penrose stability condition $\eqref{cond:Penrose}_{c/2}$ holds true for $(f(t), \varrho(t))$ whenever $t \in \left[ 0 ,  \left(\widetilde{T}_0(R),T^{\eps}  \right)\right]$, provided that $(f^{\mathrm{in}},\varrho^{\mathrm{in}})$ satisfies the $c$-Penrose stability condition $\eqref{cond:Penrose}_{c}$.

\end{proof}

\subsection{Extension of the solution}\label{Subsection-Extension}
In  this section, our goal is to construct a suitable extension in time of the solution $(f_\eps, \varrho_\eps, u_\eps)$ on the whole line $\R$. This technical step is required in view of the subsequent pseudodifferential analysis (in time-space) of Sections \ref{Section:Symbols-PropagPenrose}--\ref{Subsection:EllipticEstimate} -- the symbols being dependent on our solutions. 
A main issue is to obtain an extension still satisfying the Penrose stability condition (for all times): we refer to the later Proposition \ref{Prop:PenroseForalltimes}.

\medskip

The parameter $R$ being fixed, there exists by continuity a nonnegative time $\widehat{T}(R) \leq 1$ (independent of $\eps$) such that
\begin{align*}
\forall T \in [0, \widehat{T}(R)], \ \ T \Lambda(T,R) \leq T^{1/2} \Lambda(T,R) \leq T^{1/4} \Lambda(T,R) \leq 1.
\end{align*}
Define $T^{\star}_\eps$ (depending on $R$) as 
\begin{align}\label{def:Tstareps-Sectio -estimates}
T^{\star}_\eps:=\min \left(T_\eps(R),\overline{T}(R),\widetilde{T}_0(R),\widehat{T}(R) \right).
\end{align}
In particular, the Penrose stability condition holds on $[0,T^{\star}_\eps]$ thanks to Lemma \ref{LM:propagPENROSE}.

Consider two nonnegative nonincreasing cutoffs $\chi, \underline{\chi} \in \mathscr{C}^{\infty}(\R)$ such that 
\begin{align*}
\forall t \in \R, \ \ \chi(t)=\left\{
      \begin{aligned}
        &1, \ \ && t \leq 0, \\
		&0, \ \ &&t \geq 1,
 \end{aligned}
    \right.
, \ \ \ \
\underline{\chi}(t)=\left\{
      \begin{aligned}
        &1, \ \ && t \leq 0, \\
		&1/2, \ \ &&t \geq 1.
 \end{aligned}
    \right. 
\end{align*}
We set for $\delta>0$ to be fixed later, $\chi_{\delta}(t):=\chi(t /\delta)$.

Given a solution $(f, \varrho, u)$ to the system \eqref{S_eps}, we consider its extension $(\widetilde{f}, \widetilde{\varrho}, \widetilde{u})$ as follows. Given $(N_u, N_f, N_{\varrho}) \in \N^3$ to be determined later on (by the number of derivatives we will use), we define:

$\bullet$  \textbf{Extension in time for $u$}: we set
\begin{align*}
\widetilde{u}(t)=\left\{
      \begin{aligned}
        &u(t), \ \ &&t \in [0, T^{\star}_\eps] ,\\[2mm]
		&\chi(t-T^{\star}_\eps) \sum_{k=0}^{N_u} \partial_t^{k} u(T^{\star}_\eps) \frac{(t-T^{\star}_\eps)^k}{k!}, \ \ &&t \geq T^{\star}_\eps, \\
		&\chi(-t) \sum_{k=0}^{N_u} \partial_t^{k} u(0) \frac{t^k}{k!}, \ \ && t \leq 0.
 \end{aligned}
    \right.    
\end{align*}
In particular, the extension is $0$ after $t=T^{\star}_\eps+1$ and before $t=-1$. 

$\bullet$  \textbf{Extension in time for $f$}: we set
\begin{align*}
\widetilde{f}(t)=\left\{
      \begin{aligned}
        &f(t), \ \ &&t \in [0, T^{\star}_\eps] ,\\[2mm]
		&\chi_{\delta}(t-T^{\star}_\eps)f(T^{\star}_\eps) + \chi_{\delta}(t-T^{\star}_\eps) \sum_{k=1}^{N_f} \partial_t^{k} f(T^{\star}_\eps) \frac{(t-T^{\star}_\eps)^k}{k!}, \ \ &&t \geq T^{\star}_\eps, \\
		&\chi_{\delta}(-t)f^{\mathrm{in}}+\chi_{\delta}(-t) \sum_{k=1}^{N_f} \partial_t^{k} f(0) \frac{t^k}{k!}, \ \ &&t \leq 0.
 \end{aligned}
    \right.    
\end{align*}
In particular, the extension is $0$ after $t=T^{\star}_\eps+\delta$ and before $t=-\delta$. 

$\bullet$  \textbf{Extension in time for $\varrho$}: we set
\begin{align*}
\widetilde{\varrho}(t)=\left\{
      \begin{aligned}
        &\varrho(t), \ \ &&t \in [0, T^{\star}_\eps] ,\\[2mm]
		&\underline{\chi}(t-T^{\star}_\eps)\varrho(T^{\star}_\eps) + \chi_{\delta}(t-T^{\star}_\eps) \sum_{k=1}^{N_{\varrho}} \partial_t^{k} \varrho(T^{\star}_\eps) \frac{(t-T^{\star}_\eps)^k}{k!}, \ \ &&t \geq T^{\star}_\eps, \\
		&\underline{\chi}(-t) \varrho^{\mathrm{in}}+\chi_{\delta}(-t) \sum_{k=1}^{N_{\varrho}} \partial_t^{k} \varrho(0) \frac{t^k}{k!}, \ \ &&t \leq 0.
 \end{aligned}
    \right.    
\end{align*}
In particular, the extension is constant in time equal to $\varrho(T^{\star}_\eps)/2$ after $t=T^{\star}_\eps+1$ and equal to $\varrho^{\mathrm{in}}/2$ before $t=-1$

The bounds from above and below \eqref{def:Bound-HAUTBAS} on $\varrho$ and $\rho_f$  are still valid for $\widetilde{\varrho}$ and $\rho_{\widetilde{f}}$, provided that we choose the parameter $\delta$ small enough.
\begin{rem}
Since $T^{\star}_\eps \leq 1$, we observe that the Penrose function $\mathscr{P}_{\widetilde{f}(t), \widetilde{\varrho}(t)}$ has a compact support in time included in $[0,2]$.
\end{rem}
Here ever after, we drop out the tilde notation and we shall always consider the extension of our solutions. Let us conclude this section by explaining how we will deal with such an extension:
\begin{itemize}
\item Replacing the former solution (defined on $[0,T^{\star}_\eps]$) by its extension $(f, \varrho,u)$ on $\R$, we observe that $(f, \varrho,u)$ satisfies \eqref{S_eps} with the addition of a new source term $S^{new}$ in the r.h.s which has a support included in $\R {\setminus} [0,T^{\star}_\eps]$.
\item The results of Section \ref{Section:Kinetic-moments} and Subsection \ref{Subsection:DerivativeRhoFluid} remain true on $ [0, T^{\star}_\eps)$.
\end{itemize}
We also refer to Proposition \ref{Prop:PenroseForalltimes} below where we will prove that the extension $(f, \varrho,u)$ satisfies a Penrose stability condition for all times (the proof requiring some technical estimates from the upcoming Section \ref{Section:Symbols-PropagPenrose}).

\subsection{Bounds on the symbols}\label{Section:Symbols-PropagPenrose}
The aim of this section is twofold:
\begin{itemize}
\item obtain some bounds in terms of the initial data for some symbol seminorms of the Penrose function introduced in \eqref{def:PenroseSymbol} (depending on the extension $(f, \varrho)$);
\item propagate the Penrose stability condition \eqref{cond:Penrose} on the whole line in time for the extension $(f, \varrho,u)$. 
\end{itemize}
These two ingredients are required to obtain crucial elliptic estimates in Section \ref{Subsection:EllipticEstimate}.

Before stating the next lemma, consider the symbol seminorms \eqref{seminorm0}--\eqref{seminormN}--\eqref{seminormA4} introduced in the Section \ref{Section:Pseudo} in the Appendix: for any $\mathrm{M} \geq 0$ and for any symbol $a(t,x,\eta)$ with $\eta=(\gamma, \tau, k) \in (0, + \infty) \times \R \times \R^d {\setminus} {\lbrace 0 \rbrace} $, we set
\begin{align*}
\omega[  a ]&:= \underset{\substack{\alpha \in \N^d \\ \alpha_i \in \lbrace 0,1 \rbrace}}{\sup} \,  \left\Vert (1+t)\partial_{x}^{\alpha}   a \right\Vert_{\Ld^{\infty}_{t,x,\eta}}+\underset{\substack{\alpha \in \N^d \\ \alpha_i \in \lbrace 0,1 \rbrace}}{\sup} \,  \left\Vert (1+t)\partial_{t}\partial_x^{\alpha}   a \right\Vert_{\Ld^{\infty}_{t,x,\eta}}, \\
\Omega[a]&:=\underset{\substack{\alpha \in \N^d \\ \alpha_i \in \lbrace 0,1 \rbrace}}{\sup} \,   \left\lbrace \left\Vert  \vert \eta \vert \partial_{x}^{\alpha}   \nabla_{\tau,k} a \right\Vert_{\Ld^{\infty}_{t,x,\eta}}+\left\Vert  \vert \eta \vert \partial_{x}^{\alpha}   \nabla_{\tau,k} \partial_t a \right\Vert_{\Ld^{\infty}_{t,x,\eta}} \right \rbrace \\
& \notag \quad +\underset{\substack{\alpha \in \N^d \\ \alpha_i \in \lbrace 0,1 \rbrace}}{\sup} \,   \left\lbrace \left\Vert  \vert \eta \vert \partial_{x}^{\alpha}   \partial_{\tau}\nabla_{\tau,k} a \right\Vert_{\Ld^{\infty}_{t,x,\eta}}+\left\Vert  \vert \eta \vert \partial_{x}^{\alpha}   \partial_{\tau}\nabla_{\tau,k} \partial_t a \right\Vert_{\Ld^{\infty}_{t,x,\eta}} \right \rbrace, \\
\Xi[a]_{\mathrm{M}}&:= \underset{\substack{1 \leq \vert \alpha \vert \leq 1+\mathrm{M}\\ \beta =0,1,2,3,4 } }{\sup} \, \left\Vert \partial_{x}^{\alpha} \partial_t^{\beta} a  \right\Vert_{\Ld^{\infty}_{t,x,\eta}}.
\end{align*}

\begin{lem}\label{LM:estim-af}
For $(t,x,\gamma, \tau, k) \in \R \times \T^d \times (0,+\infty) \times \R \times \R^d {\setminus} \lbrace 0 \rbrace$, set 
\begin{align}\label{def-af}
a_{f}(t,x,\gamma,\tau,k)&:=\int_0^{+\infty} e^{-(\gamma+i \tau) s} i k \cdot \left( \mathcal{F}_v \nabla_v f \right)(t,x,ks)  \, \mathrm{d}s.
\end{align}
The symbol $a_{f}$ is (positively) homogeneous of degree zero in $(\gamma, \tau, k)$ in the sense that
\begin{align*}
\forall (t,x), \ \ \forall \eta=(\gamma, \tau, k), \ \ \forall \lambda >0  \ \ 
a_{f} (t,x, \lambda \eta)=a_{f}\left(t,x,\eta \right).
\end{align*} 
Furthermore, for any $A>0$ and $r>d/2+4$, we have
\begin{align}
\label{estim:af-0}\omega[  a_f ] &\lesssim \underset{i=0,1}{\sup} \Vert (1+t) \partial_t^i f \Vert_{\Ld^{\infty}(\R;\mathcal{H}^{\ell}_r)},  \ \ 3+ \frac{3d}{2} <\ell, \\[2mm]
\label{estim:af-4A}\Xi[a]_{M} &\lesssim \underset{i=0,1,2,3,4}{\sup} \,  \Vert \partial_t^{i} f \Vert_{\Ld^{\infty}(\R;\mathcal{H}^{\ell}_r)},  \ \ 4+ M+ \frac{d}{2} <\ell, \\[2mm]
\label{estim:af-N} \Omega[a_f] & \lesssim \underset{i=0,1}{\sup} \Vert\partial_t^i f \Vert_{\Ld^{\infty}(\R;\mathcal{H}^{\ell}_r)}, \ \ 7+ \frac{3d}{2} <\ell.
\end{align}
\end{lem}
\begin{proof}
The homogeneity is obtained by performing the change of variable $s=\frac{s'}{\lambda}$ (for $\lambda>0$) in the integrals in $s$ defining $a_{f}(t,x, \eta)$. In what follows, we will rely on the estimate
\begin{align}\label{ineq:derivativesSymbolNablaf}
 \vert \partial_t^{\delta} \partial_{x}^{\alpha} \partial_{\xi}^{\beta}  \mathcal{F}_v \nabla_v f(t,x, \xi) \vert \lesssim \frac{1}{1+ \vert \xi \vert^q} \left( \int_{\R^d} (1+ \vert v \vert^2)^{\sigma + \vert \beta \vert} \vert \nabla_v \partial_t^{\delta} \partial_{x}^{\alpha} (I-\Delta_v)^{\frac{q}{2}} f(t,x,v)\vert^2 \ \mathrm{d}v \right)^{\frac{1}{2}},
\end{align}
which is valid for all $\sigma>d/2$, $q>0$ and any $(\delta, \alpha, \beta) \in \N \times \N^{d} \times \N^d$. 

Consequently, for any $\alpha \in \N^d $ with $\alpha _i=0,1$, we apply \eqref{ineq:derivativesSymbolNablaf} with $\delta=0,1$, $q=2$ and $\beta= 0$, and obtain for $\chi(v)=(1+\vert v \vert^2)^{\frac{\sigma}{2}}$
\begin{align*}
\vert \partial_t ^{\delta} \partial_x^{\alpha} a_f(t,x, \eta) \vert
& \lesssim \left( \int_{\R^d} (1+ \vert v \vert^2)^{\sigma} \vert \nabla_v \partial_t^{\delta} \partial_{x}^{\alpha} (I-\Delta_v) f(t,x,v)\vert^2 \ \mathrm{d}v \right)^{\frac{1}{2}} \int_0^{+ \infty} \frac{\vert k \vert}{1+ s^2 \vert k\vert^2 } \, \mathrm{d}s \\
& \lesssim \LRVert{\chi \partial_t^{\delta} \partial_x^{\alpha} f(t,x)}_{\H^3_v(\R^d)} \int_0^{+ \infty} \frac{1}{1+ s^2 } \, \mathrm{d}s \\
& \lesssim \Vert \partial_t^{\delta} f(t) \Vert_{\mathcal{H}_{\sigma}^{3+\vert \alpha \vert + \frac{d}{2}+\kappa}},
\end{align*}
for all $\kappa>0$, thanks to the Sobolev embedding. We then deduce
\begin{align*}
\omega[ a_f ] \lesssim \Vert (1+t) f \Vert_{\Ld^{\infty}\big(\R;\mathcal{H}_{\sigma}^{3+d +\frac{d}{2}+\kappa} \big)}+ \Vert(1+t) \partial_t f \Vert_{\Ld^{\infty}\big(\R;\mathcal{H}_{\sigma}^{3+d +\frac{d}{2}+\kappa}\big)},
\end{align*}
and hence the claimed inequality \eqref{estim:af-0}. The inequality \eqref{estim:af-4A} can be obtained in the same way.

Let us now turn to the proof of \eqref{estim:af-N}. First, observe by the homogeneity of $a_f$ in $\eta=(\gamma, \tau, k)$, it is enough to estimates the quantities
$\left\Vert  \partial_{x}^{\alpha}   \nabla_{\tau,k} \partial_t^{\delta} a_f \right\Vert_{\Ld^{\infty}_{t,x}\Ld^{\infty}_{\eta}(S^+)}$ and 
$\left\Vert \partial_{x}^{\alpha}   \partial_{\tau}\nabla_{\tau,k} \partial_t^{\delta} a \right\Vert_{\Ld^{\infty}_{t,x}\Ld^{\infty}_{\eta}(S^+)}$ with $\delta=0,1$, $\alpha \in \N^d $ with $\alpha _i=0,1$ and where
$$S^+:= \left\lbrace \widetilde{\eta}=(\widetilde{\gamma}, \widetilde{\tau}, \widetilde{k}) \in (0,+\infty) \times \R \times \R^d {\setminus} \lbrace 0 \rbrace \mid \widetilde{\gamma}^2+ \widetilde{\tau}^2+ \widetilde{k}^2=1 \right\rbrace.
$$
We thus need to estimate the following symbols
\begin{align*}
I^{\alpha, \delta}_1(t,x,\widetilde{\eta})&=\int_0^{+\infty} e^{-(\widetilde{\gamma}+i \widetilde{\tau}) s}  n \cdot \left( \partial_x^{\alpha} \partial_t^{\delta} \mathcal{F}_v  \nabla_v f \right)(t,x,\widetilde{k}s)  \, \mathrm{d}s, \ \ n \in \R^d, \  \ \vert n \vert=1,  \\
I^{\alpha, \delta}_2(t,x,\widetilde{\eta})&=\int_0^{+\infty} e^{-(\widetilde{\gamma}+i \widetilde{\tau}) s}  s \widetilde{k} \cdot \left( \partial_{\xi}^{\beta} \partial_x^{\alpha} \partial_t^{\delta}\mathcal{F}_v  \nabla_v f \right)(t,x,\widetilde{k}s)  \, \mathrm{d}s, \ \ \vert \beta \vert \in \lbrace 0,1 \rbrace.
\end{align*}
and
\begin{align*}
J^{\alpha, \delta}_{1,q_1}(t,x,\widetilde{\eta})&=\int_0^{+\infty} e^{-(\widetilde{\gamma}+i \widetilde{\tau}) s} s^{q_1} n \cdot \left( \partial_x^{\alpha} \partial_t^{\delta} \mathcal{F}_v  \nabla_v f \right)(t,x,\widetilde{k}s)  \, \mathrm{d}s, \ \ n \in \R^d, \  \ \vert n \vert=1, \ \ q_1 \in \lbrace 1,2 \rbrace,  \\
J^{\alpha, \delta}_{2,q_2}(t,x,\widetilde{\eta})&=\int_0^{+\infty} e^{-(\widetilde{\gamma}+i \widetilde{\tau}) s}  s^{q_2} \widetilde{k} \cdot \left( \partial_{\xi} \partial_x^{\alpha} \partial_t^{\delta}\mathcal{F}_v  \nabla_v f \right)(t,x,\widetilde{k}s)  \, \mathrm{d}s, \ \ q_2 \in \lbrace 2,3 \rbrace,
\end{align*}
for $\widetilde{\eta} \in S^+$.

We focus on the terms $J^{\alpha, \delta}_{1,q_1}$ and $J^{\alpha, \delta}_{2,q_2}$ by following \cite[Lemma 16]{HKR}, the treatment of the symbols $I^{\alpha, \delta}_1$ and $I^{\alpha, \delta}_2$ being similar and involving fewer derivatives.

If $\vert \widetilde{k}\vert \geq 1/2$, invoking \eqref{ineq:derivativesSymbolNablaf} with $q=q_1+3$ and $\beta=0$ yields as before
\begin{align*}
\vert J^{\alpha, \delta}_{1,q_1}(t,x,\widetilde{\eta}) \vert &\lesssim \left( \int_{\R^d} (1+ \vert v \vert^2)^{\sigma} \vert \nabla_v \partial_t^{\delta} \partial_{x}^{\alpha} (I-\Delta_v)^{\frac{q_1+3}{2}} f(t,x,v)\vert^2 \ \mathrm{d}v \right)^{\frac{1}{2}} \int_0^{+ \infty} \frac{s^{q_1}}{1+ s^{q_1+3} \vert \widetilde{k} \vert^{q_1+3} } \, \mathrm{d}s \\
& \lesssim \Vert \partial_t^{\delta} f(t) \Vert_{\mathcal{H}_{\sigma}^{4+q_1+\vert \alpha \vert + {\frac{d}{2}}^+}}\int_0^{+ \infty} \frac{s^{q_1}}{1+ s^{q_1+3}} \, \mathrm{d}s,
\end{align*}
since $\vert \widetilde{k} \vert$ is bounded from below. We thus obtain a uniform estimate in this case. Otherwise, if $\vert \widetilde{k}\vert \leq 1/2$, then $\widetilde{\gamma}^2+ \widetilde{\tau}^2 \geq 3/4$ and we can therefore rely on the exponential to integrate by parts in $s$ in the integral defining $J^{\alpha, \delta}_{1,q_1}(t,x,\widetilde{\eta})$. If $q_1=1$, we first get
\begin{align*}
J^{\alpha, \delta}_{1,1}(t,x,\widetilde{\eta})&=\frac{1}{\widetilde{\gamma}+i \widetilde{\tau}}\int_0^{+\infty} e^{-(\widetilde{\gamma}+i \widetilde{\tau}) s} n \cdot \left( \partial_x^{\alpha} \partial_t^{\delta} \mathcal{F}_v  \nabla_v f \right)(t,x,\widetilde{k}s)  \, \mathrm{d}s \\
& \quad +\frac{1}{\widetilde{\gamma}+i \widetilde{\tau}}\int_0^{+\infty} e^{-(\widetilde{\gamma}+i \widetilde{\tau}) s} s n \cdot \left( \mathrm{D}_{\xi} \partial_x^{\alpha} \partial_t^{\delta} \mathcal{F}_v  \nabla_v f \right)(t,x,\widetilde{k}s)\widetilde{k}  \, \mathrm{d}s,
\end{align*}
and integrating by parts once again yields the estimate (since $\widetilde{\gamma}^2+ \widetilde{\tau}^2 \geq 3/4$)
\begin{align*}
\vert J^{\alpha, \delta}_{1,1}(t,x,\widetilde{\eta}) \vert &\lesssim \vert \partial_x^{\alpha} \partial_t^{\delta} \mathcal{F}_v  \nabla_v f (t,x,0)\vert  + \int_0^{+\infty}   \vert \widetilde{k} \vert  \vert \mathrm{D}_{\xi} \partial_x^{\alpha} \partial_t^{\delta} \mathcal{F}_v  \nabla_v f (t,x,\widetilde{k}s) \vert \, \mathrm{d}s \\
& \quad +\int_0^{+\infty}   s\vert \widetilde{k} \vert^2  \vert \mathrm{D}^2_{\xi} \partial_x^{\alpha} \partial_t^{\delta} \mathcal{F}_v  \nabla_v f (t,x,\widetilde{k}s) \vert \, \mathrm{d}s.
\end{align*}
Using \eqref{ineq:derivativesSymbolNablaf} with $q=2$, $\vert \beta \vert =0,1$, or $q=3$, $\vert \beta \vert =2$  now provides for all $\kappa>0$
\begin{align*}
\vert J^{\alpha, \delta}_{1,1}(t,x,\widetilde{\eta}) \vert &\lesssim \Vert \partial_t^{\delta} f(t) \Vert_{\mathcal{H}_{\sigma}^{4+\vert \alpha \vert + \frac{d}{2}+ \kappa}}
\left( 1+ \int_0^{+ \infty} \frac{\vert \widetilde{k} \vert }{1+ \vert \widetilde{k} \vert^2 s^{2}} \, \mathrm{d}s+\int_0^{+ \infty} \frac{\vert \widetilde{k} \vert^2 s}{1+ \vert \widetilde{k} \vert^3 s^{3}} \, \mathrm{d}s \right) \\
&\lesssim \Vert \partial_t^{\delta} f(t) \Vert_{\mathcal{H}_{\sigma+2}^{4+\vert \alpha \vert + \frac{d}{2}+ \kappa}}
\left( 1+ \int_0^{+ \infty} \frac{1}{1+ s^{2}} \, \mathrm{d}s+\int_0^{+ \infty} \frac{ s}{1+ s^3} \, \mathrm{d}s \right),
\end{align*}
and hence the uniform estimate for this symbol. If $q_1=2$, we use the same strategy with additional integration by parts and \eqref{ineq:derivativesSymbolNablaf} with $q=2$, $\vert \beta \vert =0,1$, or $q=3$, $\vert \beta \vert =2$, or $q=4$, $\vert \beta \vert =3$ to get
\begin{align*}
\vert J^{\alpha, \delta}_{1,2}(t,x,\widetilde{\eta}) \vert &\lesssim \vert J^{\alpha, \delta}_{1,1}(t,x,\widetilde{\eta}) \vert + \int_0^{+\infty}   \vert \widetilde{k} \vert  \vert \mathrm{D}_{\xi} \partial_x^{\alpha} \partial_t^{\delta} \mathcal{F}_v  \nabla_v f (t,x,\widetilde{k}s) \vert \, \mathrm{d}s \\
& \quad +\int_0^{+\infty}   s\vert \widetilde{k} \vert^2  \vert \mathrm{D}^2_{\xi} \partial_x^{\alpha} \partial_t^{\delta} \mathcal{F}_v  \nabla_v f (t,x,\widetilde{k}s) \vert \, \mathrm{d}s \\
& \quad +\int_0^{+\infty}   s^2\vert \widetilde{k} \vert^3  \vert \mathrm{D}^3_{\xi} \partial_x^{\alpha} \partial_t^{\delta} \mathcal{F}_v  \nabla_v f (t,x,\widetilde{k}s) \vert \, \mathrm{d}s. \\
&\lesssim \Vert \partial_t^{\delta} f(t) \Vert_{\mathcal{H}_{\sigma+3}^{5+\vert \alpha \vert + {\frac{d}{2}}^+}}
\left( 1+ \int_0^{+ \infty} \frac{\vert \widetilde{k} \vert }{1+ \vert \widetilde{k} \vert^2 s^{2}} \, \mathrm{d}s+\int_0^{+ \infty} \frac{\vert \widetilde{k} \vert^2 s}{1+ \vert \widetilde{k} \vert^3 s^{3}} \, \mathrm{d}s+\int_0^{+ \infty} \frac{\vert \widetilde{k} \vert^3 s^2}{1+ \vert \widetilde{k} \vert^4 s^{4}} \, \mathrm{d}s \right) \\
&\lesssim \Vert \partial_t^{\delta} f(t) \Vert_{\mathcal{H}_{\sigma+3}^{5+\vert \alpha \vert + {\frac{d}{2}}^+}}
\left( 1+ \int_0^{+ \infty} \frac{1}{1+ s^{2}} \, \mathrm{d}s+\int_0^{+ \infty} \frac{ s}{1+ s^3} \, \mathrm{d}s+\int_0^{+ \infty} \frac{ s^2}{1+ s^4} \, \mathrm{d}s \right).
\end{align*}
Gathering the two cases together, we deduce the estimate
\begin{align*}
\LRVert{ J^{\alpha, \delta}_{1,q_1}}_{\Ld^{\infty}_{t,x}\Ld^{\infty}_{\eta}(S^+)} \lesssim \Vert \partial_t^{\delta} f \Vert_{\Ld^{\infty}\big(\R; \mathcal{H}_{\sigma+3}^{6+\vert \alpha \vert + \frac{d}{2}+ \kappa} \big)},
\end{align*}
for all $\kappa>0$. Likewise, we can apply these arguments to $J^{\alpha, \delta}_{2,q_2}$, by invoking again \eqref{ineq:derivativesSymbolNablaf} with $q=q_2+3$, $\beta=1$ if $\vert \widetilde{k}\vert \leq 1/2$, or with at most $q=5$, $\beta=4$ if $\vert \widetilde{k}\vert \geq 1/2$. All in all, we obtain
\begin{align*}
\LRVert{ J^{\alpha, \delta}_{2,q_1}}_{\Ld^{\infty}_{t,x}\Ld^{\infty}_{\eta}(S^+)} \lesssim \Vert \partial_t^{\delta} f \Vert_{\Ld^{\infty}\big(\R; \mathcal{H}_{\sigma+4}^{7+\vert \alpha \vert + \frac{d}{2}+ \kappa} \big)}.
\end{align*}
We have eventually proven the estimate \eqref{estim:af-N}.
\end{proof}

In view of Lemma \ref{LM:estim-af}, it will be useful to have the following estimates on $f$.

\begin{lem}\label{LM:highDtf}
For $k>d$ and $\sigma>0$, we have
\begin{align}
\label{estim:Partial_t01f}\underset{i=0,1}{\sup} \, \LRVert{(1+t)\partial_t^{i}f}_{\Ld^{\infty}(\R; \mathcal{H}^k_{\sigma})} &\leq \Lambda(1+M_{\mathrm{in}}), \ \ \sigma+4<r, \ \ k+3<m, \\
\label{estim:Partial_t01234f}\underset{i=0,1,2,3,4}{\sup} \, \LRVert{\partial_t^{i}f}_{\Ld^{\infty}(\R; \mathcal{H}^k_{\sigma})} &\leq \Lambda(1+M_{\mathrm{in}}), \ \ \sigma+4<r, \ \  k+6<m.
\end{align}
\end{lem}
\begin{proof}
Thanks to our choice of extension for $(f, \varrho, u)$, and picking $N_f=4, N_{\varrho}=3$ and $N_u=3$, it is sufficient to study the estimates on $[0,T]$ with $T \in [0,T^{\star}_\eps]$. Here, we shall constantly use 
Remarks \ref{rem:estim-rho-lowderivative}
--\ref{rem:estim-u-lowderivative}.
Taking all the exponants $k$ large enough in the following,  we can always assume that the Sobolev spaces that we use are algebras.

 We proceed inductively, relying on the equations satisfied by $f$, $\varrho$ and $u$. For $i=0$, we directly use Lemma \ref{energy:f-unif}
 with $k \leq m-1$ to get
\begin{align*}
\LRVert{f}_{\Ld^{\infty}(0,T; \mathcal{H}^k_{\sigma})} \leq  \Vert f^{\mathrm{in}} \Vert_{\mathcal{H}^{m-1}_{\sigma}}+T^{\frac{1}{4}}\Lambda\left( T,R \right)\leq M_{\mathrm{in}}+1.
\end{align*}
For $i=1$, we use Lemma \ref{LMestim:Dt-f-rhof-rho} to get for $k>d$ and $k+1>d/2$
\begin{align*}
\Vert \partial_t f \Vert_{\Ld^{\infty}(0,T;\mathcal{H}^{k}_{\sigma})} \lesssim \Vert  f \Vert_{\Ld^{\infty}(0,T;\mathcal{H}^{k+1}_{\sigma+1})} + \Vert  f \Vert_{\Ld^{\infty}(0,T;\mathcal{H}^{k+1}_{\sigma})}\left( \Vert u \Vert_{\Ld^{\infty}(0,T;\H^{k}) } + \Lambda \left(\Vert \varrho \Vert_{\Ld^{\infty}(0,T;\H^{k+1}) } \right) \right),
\end{align*}
therefore using Remarks \ref{rem:estim-rho-lowderivative}--\ref{rem:estim-u-lowderivative} and the previous estimate on $f$, we deduce
\begin{align*}
\Vert \partial_t f \Vert_{\Ld^{\infty}(0,T;\mathcal{H}^{k}_{\sigma})} \leq \Lambda(1+M_{\mathrm{in}}),
\end{align*}
if $k \leq m-3$. For $i=2$, we take one derivative in time in the Vlasov equation and get
\begin{align*}
\Vert \partial_t^2 f \Vert_{\Ld^{\infty}(0,T;\mathcal{H}^{k}_{\sigma})} &\lesssim \Vert \partial_t f \Vert_{\Ld^{\infty}(0,T;\mathcal{H}^{k+1}_{\sigma+1})}\\
& \quad +\Vert E_{\mathrm{reg},\eps}^{\varrho,u} \Vert_{\Ld^{\infty}(0,T;\H^{k})}\Vert \partial_t f \Vert_{\Ld^{\infty}(0,T;\mathcal{H}^{k+1}_{\sigma})}
+\Vert \partial_t E_{\mathrm{reg},\eps}^{\varrho,u} \Vert_{\Ld^{\infty}(0,T;\H^{k})}\Vert f \Vert_{\Ld^{\infty}(0,T;\mathcal{H}^{k+1}_{\sigma})} \\
& \leq \Lambda(1+M_{\mathrm{in}})+ \Lambda(1+M_{\mathrm{in}})\Vert \partial_t E_{\mathrm{reg},\eps}^{\varrho,u} \Vert_{\Ld^{\infty}(0,T;\mathcal{H}^{k})},
\end{align*}
if we take $k \leq m-4$. Since
\begin{align*}
\partial_t E_{\mathrm{reg},\eps}^{\varrho,u} =\partial_t u-\partial_t \varrho p''(\varrho) \mathrm{J}_\eps \nabla_x \varrho-p'(\varrho)\mathrm{J}_{\eps} \nabla_x \partial_t \varrho,
\end{align*}
it is enough to estimate $\LRVert{\partial_t u}_{\H^k}$ and $\LRVert{\partial_t \varrho}_{\H^{k+1}}$. Thanks to the equation on $u$ and on $\varrho$, we easily get the fact that this involves $\rho_f, j_f, \varrho$ and $u$ in $\Ld^{\infty}(0,T ; \H^{k+2})$. Using Lemma \ref{energy:f-unif} and \ref{rem:estim-rho-lowderivative}--\ref{rem:estim-u-lowderivative} with $k \leq m-4$ shows that
\begin{align*}
\Vert \partial_t^2 f \Vert_{\Ld^{\infty}(0,T;\mathcal{H}^{k}_{\sigma})} \lesssim \Lambda(1+M_{\mathrm{in}}).
\end{align*}
Now for $i=3$, we observe that $\Vert \partial_t^3 f \Vert_{\Ld^{\infty}(0,T;\H^{k})}$ is estimated by, at most, $\Vert \partial_t^2 f \Vert_{\Ld^{\infty}(0,T;\mathcal{H}^{k+1}_{\sigma+1})}$ and $\Vert \partial_t^2 E_{\mathrm{reg},\eps}^{\varrho,u} \Vert_{\Ld^{\infty}(0,T;\mathcal{H}^{k}_{\sigma})}$, thus requiring to estimate at most $\LRVert{\partial_t^2 u}_{\H^k}$ and $\LRVert{\partial_t^2 \varrho}_{\H^{k+1}}$. Using again the equation on $u$ and $\varrho$, this now involves $u$ in $\Ld^{\infty}(0,T ; \H^{k+4})$ and $\rho_f, j_f, \varrho$ in $\Ld^{\infty}(0,T ; \H^{k+3})$. Using Lemma \ref{energy:f-unif} and Remarks \ref{rem:estim-rho-lowderivative}--\ref{rem:estim-u-lowderivative} with $k \leq m-5$ shows that 
\begin{align*}
\Vert \partial_t^3 f \Vert_{\Ld^{\infty}(0,T;\mathcal{H}^{k}_{\sigma})} \leq \Lambda(1+M_{\mathrm{in}}).
\end{align*}
The same can be done for $i=4$, implying for $k \leq m-6$
\begin{align*}
\Vert \partial_t^4 f \Vert_{\Ld^{\infty}(0,T;\mathcal{H}^{k}_{\sigma})} \leq \Lambda(1+M_{\mathrm{in}}).
\end{align*}
This allows us to conclude the proof.
\end{proof}

We are now in position to prove the following result on the Penrose symbol $\mathscr{P}_{f, \varrho} $, which is, as we recall here:
\begin{align*}
\mathscr{P}_{f, \varrho}(t,x, \gamma, \tau, \eta)=\frac{p'(\varrho(t,x)) \varrho(t,x)}{1-\rho_f(t,x)} \int_0^{+\infty} e^{-(\gamma+i \tau) s} \frac{i k}{1+ \vert k \vert^2} \cdot \left( \mathcal{F}_v \nabla_v f \right)(t,x,\eta s)  \, \mathrm{d}s,
\end{align*}
and is defined for all $t \in \R$. The symbol seminorms \eqref{seminorm0}--\eqref{seminormN}--\eqref{seminormA4} of $\mathscr{P}_{f, \varrho}$ are first estimated as follows.

\begin{lem}\label{LM:estim-techn}
For any $k>d/2$ and $\mathrm{M} \geq 0$, we have
\begin{align}
\label{estim-techn:P-0}\omega[  \mathscr{P}_{f, \varrho} ] &\leq \Lambda\left(\underset{i=0,1}{\sup} \LRVert{\partial_t^i \varrho}_{\Ld^{\infty}(\R;\H^{d+k})},\underset{i=0,1}{\sup} \LRVert{\partial_t^i\rho_f}_{\Ld^{\infty}(\R;\H^{d+k})}  \right)\omega[a_f], \\[1mm]
\label{estim-techn:P-A4} \Xi[{P}_{f, \varrho}]_{\mathrm{M}} &\leq \Lambda\left(\underset{i=0,1,2,3,4}{\sup} \LRVert{\partial_t^i \varrho}_{\Ld^{\infty}(\R;\H^{1+M+k})},\underset{i=0,1,2,3,4}{\sup} \LRVert{\partial_t^i\rho_f}_{\Ld^{\infty}(\R;\H^{1+\mathrm{M}+k})}  \right)\Xi[a_f]_{\mathrm{M}}, \\[1mm]
\label{estim-techn:P-N}\Omega[\mathscr{P}_{f, \varrho}] &\leq  \Lambda\left(\underset{i=0,1}{\sup} \LRVert{\partial_t^i \varrho}_{\Ld^{\infty}(\R;\H^{d+k})},\underset{i=0,1}{\sup} \LRVert{\partial_t^i\rho_f}_{\Ld^{\infty}(\R;\H^{d+k})}  \right) \Omega[a_f].
\end{align}
\end{lem}
\begin{proof}
Since the symbol $\mathscr{P}_{f, \varrho}$ depends only on $(\gamma, \tau,k)$ through $a_f$, we only prove the estimate \eqref{estim-techn:P-0}, the treatment of \eqref{estim-techn:P-A4} and \eqref{estim-techn:P-N} being similar.

First, we write the symbol $\mathscr{P}_{f, \varrho}$ as 
$$\mathscr{P}_{f, \varrho}(t,x,\gamma, \tau,k)=\mathfrak{m}(\varrho(t,x), \rho_f(t,x)) \frac{1}{1+ \vert k \vert^2} a_{f}(t,x,\gamma, \tau,k),$$
where $a_f$ has been defined in \eqref{def-af} and where
$$\mathfrak{m}(\varrho(t,x), \rho_f(t,x)):=\frac{p'(\varrho(t,x))\varrho(t,x)}{1-\rho_f(t,x)}.$$
We have
\begin{align*}
\omega[  \mathscr{P}_{f, \varrho} ]  & \lesssim \underset{\substack{\alpha \in \times \N^d \\ \alpha_i \in \lbrace 0,1 \rbrace}}{\sup} \left\Vert (1+t) \partial_x^{\alpha} \left( \mathfrak{m}(\varrho, \rho_f) a_f \right) \right\Vert_{\Ld^{\infty}(\R \times \T^d; \Ld^{\infty}_{\eta})} +\underset{\substack{\alpha \in \times \N^d \\ \alpha_i \in \lbrace 0,1 \rbrace}}{\sup} \left\Vert (1+t) \partial_x^{\alpha}\left( \mathfrak{m}(\varrho, \rho_f) \partial_t a_f \right) \right\Vert_{\Ld^{\infty}(\R \times \T^d; \Ld^{\infty}_{\eta})} \\
& \quad + \underset{\substack{\alpha \in \times \N^d \\ \alpha_i \in \lbrace 0,1 \rbrace}}{\sup} \left\Vert (1+t) \partial_x^{\alpha}\left(\partial_1\mathfrak{m}(\varrho, \rho_f)\partial_t \varrho+\partial_2 \mathfrak{m}(\varrho, \rho_f) \partial_t \rho_f) a_f \right) \right\Vert_{\Ld^{\infty}(\R \times \T^d; \Ld^{\infty}_{\eta})}.
\end{align*}
Using Leibniz rule, we get for all $\alpha \in  \N^d$ such that $\alpha_i \in \lbrace 0,1 \rbrace$
\begin{align*}
(1+t) \left\vert \partial_x^{\alpha}\left( \mathfrak{m}(\varrho, \rho_f)a_f \right) \right\vert &\lesssim \sum_{\beta \leq \alpha} \vert  \partial_x^{\beta}\left( \mathfrak{m}(\varrho, \rho_f) \right) \vert \vert  \partial_x^{\alpha-\beta}a_f \vert \leq \omega[a_f]  \sum_{\beta \leq \alpha} \vert  \partial_x^{\beta}\left( \mathfrak{m}(\varrho, \rho_f) \right) \vert,
\end{align*}
and one can observe that for all $k>\frac{d}{2}$
\begin{align*}
 \sum_{\beta \leq \alpha} \vert  \partial_x^{\beta}\left( \mathfrak{m}(\varrho, \rho_f) \right) \vert & \lesssim \LRVert{\mathfrak{m}(\varrho, \rho_f)}_{\H^{d+k}} \leq  \Lambda (\LRVert{\varrho}_{\Ld^{\infty}}, \LRVert{\rho_f}_{\Ld^{\infty}}) \LRVert{\varrho}_{\H^{d+k}}   \LRVert{\rho_f}_{\H^{d+k}},
\end{align*}
where we have used Proposition \ref{Bony-Meyer} in the Appendix and the fact that $\H^{k+d}$ is an algebra. Doing the same with the term involving $\partial_t a_f$, we obtain
\begin{multline*}
\underset{\substack{\alpha \in \times \N^d \\ \alpha_i \in \lbrace 0,1 \rbrace}}{\sup} \left\Vert (1+t)\partial_x^{\alpha} \left( \mathfrak{m}(\varrho, \rho_f) a_f \right) \right\Vert_{\Ld^{\infty}(\R \times \T^d; \Ld^{\infty}_{\eta})}+\underset{\substack{\alpha \in \times \N^d \\ \alpha_i \in \lbrace 0,1 \rbrace}}{\sup} \left\Vert (1+t) \partial_x^{\alpha}\left( \mathfrak{m}(\varrho, \rho_f) \partial_t a_f \right) \right\Vert_{\Ld^{\infty}(\R \times \T^d; \Ld^{\infty}_{\eta})} \\ 
\leq  \Lambda\left(\LRVert{\varrho}_{\Ld^{\infty}(\R;\Ld^{\infty})}, \LRVert{\rho_f}_{\Ld^{\infty}(\R;\Ld^{\infty})} \right) \LRVert{\varrho}_{\Ld^{\infty}(\R;\H^{d+k})}   \LRVert{\rho_f}_{\Ld^{\infty}(\R;\H^{d+k})} \omega[a_f].
\end{multline*}
Likewise, the same kind of computations, that we do not detail, show that for all $k>\frac{d}{2}$
\begin{align*}
&\underset{\substack{\alpha \in \times \N^d \\ \alpha_i \in \lbrace 0,1 \rbrace}}{\sup} \left\Vert \partial_x^{\alpha}\left( [\partial_1G(\varrho, \rho_f)\partial_t \varrho+\partial_2 G(\varrho, \rho_f)] \partial_t \rho_f) a_f \right) \right\Vert_{\Ld^{\infty}(\R \times \T^d; \Ld^{\infty}_{\eta})} \\
&\leq \Lambda\left(\LRVert{\varrho}_{\Ld^{\infty}(\R;\Ld^{\infty})}, \LRVert{\rho_f}_{\Ld^{\infty}(\R;\Ld^{\infty})} \right) \\
& \qquad \qquad \qquad \qquad \times \Lambda\left(\LRVert{\varrho}_{\Ld^{\infty}(\R;\H^{d+k})} ,\LRVert{\partial_t \varrho}_{\Ld^{\infty}(\R;\H^{d+k})},\LRVert{\rho_f}_{\Ld^{\infty}(\R;\H^{d+k})} ,\LRVert{\partial_t \rho_f}_{\Ld^{\infty}(\R;\H^{d+k})}   \right)\omega[a_f].
\end{align*}
This concludes the proof of the estimate \eqref{estim-techn:P-0} of the lemma, thanks to Sobolev embedding. 
\end{proof}

We finally obtain the following result, yielding some control on the seminorms of the Penrose function in terms of the initial data only.

\begin{coro}\label{coro:Estim-symbols}
There hold
\begin{align}
\label{estim:P-0}\omega[  \mathscr{P}_{f, \varrho} ] &\lesssim \Lambda(1+M_{\mathrm{in}}), \\[2mm]
\label{estim:P-4A}\Xi[ \mathscr{P}_{f, \varrho}]_{\mathrm{M}} &\lesssim \Lambda(1+M_{\mathrm{in}}), \ \ 2\mathrm{M}<m-11-d/2, \\[2mm]
\label{estim:P-N} \Omega\left[\mathscr{P}_{f, \varrho} \right] & \lesssim \Lambda(1+M_{\mathrm{in}}).
\end{align}
\end{coro}
\begin{proof}
We combine the estimates \eqref{estim:af-0}--\eqref{estim:af-4A}--\eqref{estim:af-N} of Lemma \ref{LM:estim-af} with \eqref{estim-techn:P-0}--\eqref{estim-techn:P-A4}--\eqref{estim-techn:P-N} of Lemma \ref{LM:estim-techn}. We first get for $3+3d/2 \leq \ell <m-3$
\begin{align*}
\omega[  \mathscr{P}_{f, \varrho} ] &\leq \Lambda\left(\underset{i=0,1}{\sup} \LRVert{\partial_t^i \varrho}_{\Ld^{\infty}(\R;\H^{d+\ell})},\underset{i=0,1}{\sup} \LRVert{\partial_t^i\rho_f}_{\Ld^{\infty}(\R;\H^{d+\ell})}  \right) \underset{i=0,1}{\sup} \Vert (1+t) \partial_t^i f \Vert_{\Ld^{\infty}(\R;\mathcal{H}^{\ell}_r)} \\
&\leq \Lambda\left(\underset{i=0,1}{\sup} \LRVert{\partial_t^i \varrho}_{\Ld^{\infty}(\R;\H^{d+\ell})},\underset{i=0,1}{\sup} \LRVert{\partial_t^i f}_{\Ld^{\infty}(\R;\mathcal{H}^{d+\ell}_{\sigma})}  \right)\Lambda(1+M_{\mathrm{in}}),
\end{align*}
for $\sigma>d/2$. Hence, by using the equation on $\varrho$ with Remarks \ref{rem:estim-rho-lowderivative} and Lemma \eqref{energy:f-unif}, and by taking $\ell<m-3-d$, we can use \eqref{estim:Partial_t01f} from Lemma \ref{LM:highDtf} to obtain \eqref{estim:P-0}. Next, we have for $4+\mathrm{M}+d/2 < \ell <m-6$
\begin{align*}
\Xi[ \mathscr{P}_{f, \varrho}] ]_{\mathrm{M}} &\leq \Lambda\left(\underset{i=0,1,2,3,4}{\sup} \LRVert{\partial_t^i \varrho}_{\Ld^{\infty}(\R;\H^{1+\mathrm{M}+\ell})},\underset{i=0,1,2,3,4}{\sup} \LRVert{\partial_t^i\rho_f}_{\Ld^{\infty}(\R;\H^{1+\mathrm{M}+\ell})}  \right)\underset{i=0,1,2,3,4}{\sup} \,  \Vert \partial_t^{i} f \Vert_{\Ld^{\infty}(\R;\mathcal{H}^{\ell}_r)} \\
& \leq \Lambda\left(\underset{i=0,1,2,3,4}{\sup} \LRVert{\partial_t^i \varrho}_{\Ld^{\infty}(\R;\H^{1+\mathrm{M}+\ell})},\underset{i=0,1,2,3,4}{\sup} \LRVert{\partial_t^i f}_{\Ld^{\infty}(\R;\mathcal{H}^{1+\mathrm{M}+\ell}_{\sigma})}  \right)\Lambda(1+M_{\mathrm{in}}),
\end{align*}
for $\sigma>d/2$. Using \eqref{estim:Partial_t01234f} from Lemma \ref{LM:highDtf}, the equation on $\varrho$,  Remark \ref{rem:estim-rho-lowderivative} and Lemma \eqref{energy:f-unif} with $\ell+1+\mathrm{M}<m-6$, we deduce \eqref{estim:P-4A}. Finally, we also have for $7+3d/2<\ell<m-3-d$
\begin{align*}
\Omega[\mathscr{P}_{f, \varrho}]
& \leq  \Lambda\left(\underset{i=0,1}{\sup} \LRVert{\partial_t^i \varrho}_{\Ld^{\infty}(\R;\H^{d+\ell})},\underset{i=0,1}{\sup} \LRVert{\partial_t^i f}_{\Ld^{\infty}(\R;\mathcal{H}^{d+\ell}_{\sigma})}  \right)\Lambda(1+M_{\mathrm{in}}),
\end{align*}
for $\sigma>d/2$ and as before, we obtain \eqref{estim:P-N}.
\end{proof}

\subsection{Elliptic estimates through pseudodifferential analysis}\label{Subsection:EllipticEstimate} 
Let $T \in (0,T^{\star}_\eps)$. In view of the equation \eqref{finaleq:h} on $h=\partial_x ^{\alpha} \varrho$ obtained in Proposition \ref{coro:Facto}, we initiate the study of the equation
\begin{align}\label{eq:Htilde}
\left( \mathrm{Id}-\frac{\varrho}{1-\rho_f}\mathrm{K}_G^{\mathrm{free}}\circ\mathrm{J}_{\eps} \right)\left[\widetilde{H} \right]=\widetilde{\mathcal{R}}, \ \ 0 \leq t \leq T,
\end{align}
where $\widetilde{\mathcal{R}}$ is a given source term defined on $(0,T)$. Given a solution $\widetilde{H}$ to this equation, we want to derive an $\Ld^2(0,T; \Ld^2)$ estimate of $\widetilde{H}$ in terms of $\widetilde{R}$. This will be possible thanks to the Penrose stability condition satisfied by $(f(t), \varrho(t))$ (see the forthcoming Proposition \ref{Prop:PenroseForalltimes}). 

Note that the operator involved in the equation \eqref{eq:Htilde} depends on $\varrho$ and $f$, which are defined for all times.

\medskip

Following \cite{HKR}, we would like to link the operator $\mathrm{Id}-\frac{\varrho}{1-\rho_f}\mathrm{K}_G^{\mathrm{free}}\circ\mathrm{J}_{\eps}$ which appears in \eqref{eq:Htilde} to a pseudodifferential operator of order $0$ in time-space.
For $\gamma > 0$ (which will be chosen large enough in the end, but always independent of $\eps$), we set
\begin{align*}
\widetilde{H}(t,x):=e^{\gamma t} H(t,x), \ \  \widetilde{\mathcal{R}}(t,x):=e^{\gamma t} \mathcal{R}(t,x).
\end{align*}
We can rewrite \eqref{eq:Htilde} as
\begin{align}\label{eq:H}
H(t,x)-\frac{\varrho}{1-\rho_f}e^{-\gamma t} \mathrm{K}_G^{\mathrm{free}}\Big[e^{\gamma \bullet} \mathrm{J}_{\eps}H \Big](t,x)=\mathcal{R}(t,x), \ \ 0 \leq t \leq T.
\end{align}

To study the solution $H$ on $[0,T]$, we will extend the equation \eqref{eq:H} on the whole line $\R$. This is possible thanks to the following lemma, which can be seen as a causality principle for the equation \eqref{eq:H}. We postpone its proof to the end of Section \ref{Subsection:EllipticEstimate}.

\begin{lem}\label{LM:causality}
Let $T>0$. Assume that $\widetilde{H}_1, \widetilde{H}_2, \widetilde{S}_1, \widetilde{S}_2  \in \Ld^2(\R \times \T^d)$ satisfy the equation \eqref{eq:Htilde} with
$$
 \quad \widetilde{H}_i|_{(-\infty;0]}=0, \quad i =1,2,
$$
and
$$
\widetilde{S}_1|_{[0,T]} = \widetilde{S}_2|_{[0,T]}.
$$
Then 
$$
\widetilde{H}_1|_{[0,T]} = \widetilde{H}_2|_{[0,T]}.
$$
\end{lem}

We thus consider the same equation as \eqref{eq:H} satisfied on $\R \times \T^d$ by a solution $H$ with source $\mathcal{R}$ where
\begin{itemize}
\item the source $\mathcal{R}(t,x)$ is equal to the original source on $[0,T]$ and  is equal to zero for $t<0$ and for $t>T$,
\item the solution $H(t,x)$ is equal to zero for $t<0$.
\end{itemize}

In what follows, we use the following notation for the Fourier transform in time-space, and we also refer to Section \ref{Section:Pseudo} in the Appendix:
\begin{align*}
\forall (\tau, k) \in \R \times \Z^d, \ \ \mathcal{F}_{t,x} g(\tau,k) &=\int_{\R \times \T^d} e^{-i(\tau t+ k  \cdot x)} g(t,x) \, \mathrm{d}t \, \mathrm{d}x.
\end{align*}
 For symbols of the form $a(t,x,\gamma, \tau,k)$ on $[0,T] \times  \T^d \times \R^+ \times \R \times \R^d {\setminus} \lbrace 0 \rbrace$ and in the Schwartz class, we rely on the quantification
\begin{align*}
\mathrm{Op}^{\gamma}(a)(h)(t,x)=\frac{1}{(2\pi)^{d+1}}\int_{\R \times \Z^d}e^{i(\tau t+ k  \cdot x)}a(t,x,\gamma,\tau,k) \mathcal{F}_{t,x}h(\tau,k) \, \mathrm{d}\tau \, \mathrm{d}k.
\end{align*}
The measure on $\Z^d$ is the discrete measure. Note that the symbols we shall use are defined on $\R \times \T^d$ in the physical space, thanks to the extension procedure from Section \ref{Subsection-Extension}. Note also that we shall handle symbols defined in the whole space $\R^d$ for the $k$ variable, even if we only use them for $k \in \Z^d$ in the formula.

The following lemma now provides the link between the intego-differential operator of \eqref{eq:H} and pseudodifferential operators.

\begin{lem}\label{LM:rewriteEqPseudo}
We have
\begin{align*}
e^{-\gamma t} \mathrm{K}_G^{\mathrm{free}}\Big[e^{\gamma \bullet} H \Big](t,x)=p'(\varrho(t,x))\mathrm{Op}^{\gamma}(a_{f})(H)(t,x), \ \ \text{on} \ \ \R \times \T^d,
\end{align*}
where $a_f$ is defined in \ref{def-af}.
In particular, the equation \eqref{eq:H} on $H$ reads
\begin{align}\label{eq:pseudo:H}
H-\frac{p'(\varrho)\varrho}{1-\rho_f}\mathrm{Op}^{\gamma}(a_{f})(\mathrm{J}_{\eps} H)=\mathcal{R}, \ \ \text{on} \ \ \R \times \T^d.
\end{align}
\end{lem}

\begin{proof}
We write
\begin{align*}
H(t,x)&=\frac{1}{(2\pi)^{d+1}}\int_{\R \times \Z^d} e^{i(\tau t+ k  \cdot x)} \mathcal{F}_{t,x} H(\tau,k) \, \mathrm{d}\tau \, \mathrm{d}k,
\end{align*}
therefore
\begin{align*}
&e^{-\gamma t}\mathrm{K}_G^{\mathrm{free}}(e^{\gamma \bullet}H)(t,x)\\
&=\frac{p'(\varrho(t,x))}{(2\pi)^{d+1}}  \int_{-\infty}^t \int_{\R^d}\int_{\R \times \Z^d} e^{-\gamma (t-s)} e^{i(\tau s+ k  \cdot (x-(t-s)v)} ik \cdot \nabla_v f(t,x,v) \mathcal{F}_{t,x} H(\tau,k) \, \mathrm{d}v \, \mathrm{d}s  \, \mathrm{d}\tau \, \mathrm{d}k, \\
\end{align*}
because $H$ is $0$ on negative times. We use the Fubini theorem (which holds since $\gamma >0$) and get
\begin{multline*}
e^{-\gamma t}\mathrm{K}_G^{\mathrm{free}}(e^{\gamma \bullet}h)(t,x)=\frac{p'(\varrho(t,x))}{(2\pi)^{d+1}} \int_{\R \times \Z^d} e^{i(\tau t+ k  \cdot x)}  \\
\left( \int_{-\infty}^t e^{-(\gamma +i \tau) (t-s)}  ik \cdot \int_{\R^d} e^{-i k  \cdot v (t-s)} \nabla_v f(t,x,v)  \, \mathrm{d}v \, \mathrm{d}s \right)  \mathcal{F}_{t,x} H(\tau,k) \, \mathrm{d}\tau \, \mathrm{d}k,
\end{multline*}
therefore, setting $s'=t-s$
\begin{align*}
&e^{-\gamma t}\mathrm{K}_G^{\mathrm{free}}(e^{\gamma \bullet}h)(t,x)\\
&=\frac{p'(\varrho(t,x))}{(2\pi)^{d+1}} \int_{\R \times \Z^d} e^{i(\tau t+ k  \cdot x)}   \left( \int_{0}^{+\infty} e^{-(\gamma +i \tau) s}  ik \cdot \int_{\R^d} e^{-i k  \cdot v s} \nabla_v f(t,x,v)  \, \mathrm{d}v \, \mathrm{d}s \right)  \mathcal{F}_{t,x} H(\tau,k) \, \mathrm{d}\tau \, \mathrm{d}k \\
&=\frac{p'(\varrho(t,x))}{(2\pi)^{d+1}} \int_{\R \times \Z^d} e^{i(\tau t+ k  \cdot x)}  \left( \int_{0}^{+\infty} e^{-(\gamma +i \tau) s}  ik \cdot \left( \mathcal{F}_v \nabla_v f \right)(t,x,ks)\, \mathrm{d}s \right)  \mathcal{F}_{t,x} H(\tau,k) \, \mathrm{d}\tau \, \mathrm{d}k w\\
&=p'(\varrho(t,x,))\mathrm{Op}(a_{f})(h)(t,x).
\end{align*}
This concludes the proof.
\end{proof}

Having in mind a semiclassical approach (see also Apppendix \ref{Section:Pseudo}), we introduce the following quantization.
\begin{defi}
For any symbol $b(t,x,\gamma, \tau,k)$ on $R \times  \T^d \times \R^+ \times \R \times \R^d {\setminus} \lbrace 0 \rbrace$, we set for $\eps \in (0,1)$
\begin{align*}
b^{\eps}(t,x,\gamma, \tau,k)&:=b(t,x,\eps \gamma,\eps \tau, \eps k), \\
\mathrm{Op}^{\gamma, \eps}(b)&:=\mathrm{Op}^{\gamma}(b^{\eps}).
\end{align*}
\end{defi}

\begin{lem}\label{LM:homogeneous}
We have
\begin{align}\label{new-symb}
p'(\varrho)\mathrm{Op}^{\gamma}(a_{f, \varrho})(\mathrm{J}_\eps H) =\mathrm{Op}^{\gamma, \eps}(\mathfrak{a}_{f, \varrho})(H),
\end{align}
where
\begin{align*}
\mathfrak{a}_{f, \varrho}(t,x,\eta):=\frac{1}{1+ \vert k \vert^2}p'(\varrho(t,x))a_{f}(t,x,\eta).
\end{align*}
\end{lem}
\begin{proof}
We have the exact composition formula
\begin{align*}
p'(\varrho)\mathrm{Op}^{\gamma}(a_{f})(\mathrm{J}_\eps H)&=\mathrm{Op}^{\gamma}(\widetilde{a}_{\eps, f, \varrho})( H), \\ 
\widetilde{a}_{\eps, f, \varrho}(t,x,\eta)&:=p'(\varrho(t,x))a_{f}(t,x,\eta)\frac{1}{1+ \vert \eps k \vert^2},
\end{align*}
since we are composing on the right by a Fourier multiplier. Since $a_{f, \varrho}$ is homogeneous of degree $0$ in the variable $\eta$ (see Proposition \ref{LM:estim-af}), we have
$$\widetilde{a}_{\eps, f, \varrho}(t,x,\eta)=p'(\varrho(t,x))a_{f, \varrho}(t,x,\eps \eta)\frac{1}{1+ \vert \eps k \vert^2}=\mathfrak{a}_{f, \varrho}(t,x,\eps\eta),$$
and the conclusion follows.
\end{proof}

In the following proposition, we show that we can choose an extension of $(f, \varrho, u)$ as in Subsection \ref{Subsection-Extension} and such that it satisfies a Penrose condition for all times. By definition of $T^{\star}_\eps$, we know that 
\begin{align*}
\forall t \in [0,T^{\star}_\eps], \ \ \underset{(x,\gamma, \tau,k)}{\inf} \, \vert 1- \mathscr{P}_{f(t), \varrho(t)}(x,\gamma, \tau, k) \vert \geq c_0/2,
\end{align*}
where we recall the expression of the Penrose symbol:
\begin{align*}
\mathscr{P}_{f(t), \varrho(t)}(x,\gamma, \tau,k)=\frac{p'(\varrho(t,x) )\varrho(t,x)}{1-\rho_f(t,x)}\int_0^{+\infty} e^{-(\gamma+i \tau) s}  \frac{ik}{1+\vert k \vert^2} \cdot \left( \mathcal{F}_v \nabla_v f \right)(t,x,ks)  \, \mathrm{d}s.
\end{align*}
We want this condition to be true for the extension of $(f, \varrho)$.
%
%
%
We have the following result, which requires the technical assumption \eqref{Assumption-pressure} on the pressure.

\begin{propo}\label{Prop:PenroseForalltimes}
There exists $\delta^{\star }=\delta(c_0, M_{\mathrm{in}})>0$ small enough such that, considering the extension of $f$ and $\varrho$ with respect to $T^{\star}_\eps$ and $\delta^{\star}$, we have
\begin{align*}
\forall t \in \R, \ \ \underset{(x,\gamma, \tau,k)}{\inf} \, \vert 1- \mathscr{P}_{f(t), \varrho(t)}(x,\gamma, \tau, k) \vert \geq c_0/4.
\end{align*}
\end{propo}
\begin{proof}
We only treat the case of $ t \in (T^{\star}_\eps, +\infty)$, the case of negative times being identical. Let us set
\begin{align*}
\chi^{\star}_{\delta}(t):=\chi_{\delta}(t-T^{\star}_\eps), \  \ \underline{\chi}^{\star}(t):=\underline{\chi}(t-T^{\star}_\eps),
\end{align*}
where we use the notations of Section \ref{Subsection-Extension}. By definition of the extension for $f$, we have
\begin{align}\label{decompo:1-P}
\begin{split}
1-\mathscr{P}_{f(t), \varrho(t)}&=1-\frac{\chi^{\star}_{\delta}(t)p'(\varrho(t,x) )\varrho(t,x)}{1-\rho_f(t,x)}\int_0^{+\infty} e^{-(\gamma+i \tau) s}  \frac{ik}{1+\vert k \vert^2} \cdot \left( \mathcal{F}_v \nabla_v f \right)(T^{\star}_\eps,x,ks)  \, \mathrm{d}s \\
&\quad -\frac{\chi^{\star}_{\delta}(t)p'(\varrho(t,x) )\varrho(t,x)}{1-\rho_f(t,x)} \sum_{k=1}^{N_f}\frac{(t-T^{\star}_\eps)^k}{k!} \int_0^{+\infty} e^{-(\gamma+i \tau) s}  \frac{ik}{1+\vert k \vert^2} \cdot \left( \mathcal{F}_v \nabla_v \partial_t^{k} f \right)(T^{\star}_\eps,x,ks)  \, \mathrm{d}s. 
\end{split}
\end{align}
We next proceed to the following decompositions: we have
\begin{align*}
\frac{1}{1-\rho_f(t)}&=\frac{1}{1-\chi^{\star}_{\delta}(t) \rho_{f}(T^{\star}_\eps) - \chi^{\star}(t) \sum_{k=1}^{N_f} \partial_t^{k} \rho_f(T^{\star}_\eps) \frac{(t-T^{\star}_\eps)^k}{k!}} \\[1mm]
&=\frac{1}{1-\chi^{\star}_{\delta}(t)\rho_{f}(T^{\star}_\eps)  } +\frac{1}{1-\chi^{\star}_{\delta}(t)\rho_{f}(T^{\star}_\eps)  } \frac{Q^{\star}(t)}{1-Q^{\star}(t)},
\end{align*}
where
\begin{align*}
Q^{\star}(t):=\frac{ \chi^{\star}_{\delta}(t) \sum_{k=1}^{N_f} \partial_t^{k} \rho_f(T^{\star}_\eps) \frac{(t-T^{\star}_\eps)^k}{k!}}{1-\chi^{\star}_{\delta}(t)\rho_{f}(T^{\star}_\eps)},
\end{align*}
and by writing
\begin{align*}
\varrho(t)=\underline{\chi}^{\star}(t)\varrho(T^{\star}_\eps)+R^{\star}(t) , \ \ R^{\star}(t):=\chi_{1}(t-T^{\star}_\eps) \sum_{k=1}^{N_{\varrho}} \varrho^{(k)}(T^{\star}_\eps) \frac{(t-T^{\star}_\eps)^k}{k!},
\end{align*}
we also have
\begin{align*}
p'(\varrho(t))\varrho(t)&=p'\left(\underline{\chi}^{\star}(t)\varrho(T^{\star}_\eps) \right)\underline{\chi}^{\star}(t)\varrho(T^{\star}_\eps) + S^{\star}(t),
\end{align*}
where 
\begin{align*} 
S^{\star}(t):=\Big[p'\left(\underline{\chi}^{\star}(t)\varrho(T^{\star}_\eps)+R^{\star}(t)\right)-p'\left(\underline{\chi}^{\star}(t)\varrho(T^{\star}_\eps)\right) \Big] \underline{\chi}^{\star}(t)\varrho(T^{\star}_\eps) + p'\left(\underline{\chi}^{\star}(t)\varrho(T^{\star}_\eps)+R^{\star}(t)\right)R^{\star}(t).
\end{align*}
Note that
\begin{align*}
S^{\star}(t)=\underset{t \rightarrow T^{\star}_\eps}{\mathcal{O}}(t-T^{\star}_\eps).
\end{align*}
Now, the equality \eqref{decompo:1-P} turns into
\begin{align*}
1-\mathscr{P}_{f(t), \varrho(t)}&=1-\frac{\chi^{\star}_{\delta}(t)p'\left(\underline{\chi}^{\star}(t)\varrho(T^{\star}_\eps) \right)\underline{\chi}^{\star}(t)\varrho(T^{\star}_\eps)}{1-\chi^{\star}_{\delta}(t)\rho_{f}(T^{\star}_\eps) }\frac{1}{1+ \vert k \vert^2 }\int_0^{+\infty} e^{-(\gamma+i \tau) s}  ik \cdot \left( \mathcal{F}_v \nabla_v f \right)(T^{\star}_\eps,x,ks)  \, \mathrm{d}s \\
& \quad +\mathfrak{E}^{\star}(t),
\end{align*}
with a remainder 
\begin{align*}
\mathfrak{E}^{\star}(t)&:= -\chi^{\star}_{\delta}(t)\Bigg(\frac{p'\left(\underline{\chi}^{\star}(t)\varrho(T^{\star}_\eps) \right)\underline{\chi}^{\star}(t)\varrho(T^{\star}_\eps)}{1-\chi^{\star}_{\delta}(t)\rho_{f}(T^{\star}_\eps)  } \frac{Q^{\star}(t)}{1-Q^{\star}(t)}\int_0^{+\infty} e^{-(\gamma+i \tau) s}  \frac{ik}{1+\vert k \vert^2} \cdot \left( \mathcal{F}_v \nabla_v f \right)(T^{\star}_\eps,x,ks)  \, \mathrm{d}s\\
&\quad +\frac{S^{\star}(t)}{1-\chi^{\star}_{\delta}(t)\rho_{f}(T^{\star}_\eps)  } \int_0^{+\infty} e^{-(\gamma+i \tau) s}  \frac{ik}{1+\vert k \vert^2} \cdot \left( \mathcal{F}_v \nabla_v f \right)(T^{\star}_\eps,x,ks)  \, \mathrm{d}s \\
&\quad +\frac{S^{\star}(t)}{1-\chi^{\star}_{\delta}(t)\rho_{f}(T^{\star}_\eps)  } \frac{Q^{\star}(t)}{1-Q^{\star}(t)}\int_0^{+\infty} e^{-(\gamma+i \tau) s}  \frac{ik}{1+\vert k \vert^2} \cdot \left( \mathcal{F}_v \nabla_v f \right)(T^{\star}_\eps,x,ks)  \, \mathrm{d}s \\
&\quad +\frac{p'(\varrho(t,x) )\varrho(t,x)}{1-\rho_f(t,x)} \sum_{k=1}^{N_f}\frac{(t-T^{\star}_\eps)^k}{k!} \int_0^{+\infty} e^{-(\gamma+i \tau) s}  \frac{ik}{1+\vert k \vert^2} \cdot \left( \mathcal{F}_v \nabla_v \partial_t^{k} f \right)(T^{\star}_\eps,x,ks)  \, \mathrm{d}s \Bigg) .
\end{align*}
We claim that an homogeneity argument shows that
\begin{align*}
&\underset{(x,\gamma, \tau,k)}{\inf} \, \left\vert 1-\frac{\chi^{\star}_{\delta}(t)p'\left(\underline{\chi}^{\star}(t)\varrho(T^{\star}_\eps,x) \right)\underline{\chi}^{\star}(t)\varrho(T^{\star}_\eps,x)}{1-\chi^{\star}_{\delta}(t)\rho_{f}(T^{\star}_\eps,x) }\frac{1}{1+ \vert k \vert^2 }\int_0^{+\infty} e^{-(\gamma+i \tau) s}  ik \cdot \left( \mathcal{F}_v \nabla_v f \right)(T^{\star}_\eps,x,ks)  \, \mathrm{d}s \right\vert \\
&=\underset{(x,\gamma, \tau,k)}{\inf} \, \left\vert 1-\frac{p'\left(\varrho(T^{\star}_\eps,x) \right)\varrho(T^{\star}_\eps,x)}{1-\rho_{f}(T^{\star}_\eps,x) }\frac{1}{1+ \vert k \vert^2 }\int_0^{+\infty} e^{-(\gamma+i \tau) s}  ik \cdot \left( \mathcal{F}_v \nabla_v f \right)(T^{\star}_\eps,x,ks)  \, \mathrm{d}s \right\vert \geq c_0/2.
\end{align*}
Indeed, we know from Lemma \ref{LM:estim-af} that for all $x \in \T^d$, the function 
$$(\gamma, \tau,k) \mapsto \int_0^{+\infty} e^{-(\gamma+i \tau) s}  ik \cdot \left( \mathcal{F}_v \nabla_v f \right)(T^{\star}_\eps,x,ks)  \, \mathrm{d}s$$
is homogeneous of degree $0$, and since by Assumption \eqref{Assumption-pressure} and by construction we have
\begin{align*}
0 \leq \frac{\chi^{\star}_{\delta}(t)p'\left(\underline{\chi}^{\star}(t)\varrho(T^{\star}_\eps,x) \right)\underline{\chi}^{\star}(t)\varrho(T^{\star}_\eps,x)}{1-\chi^{\star}_{\delta}(t)\rho_{f}(T^{\star}_\eps,x) }  \leq \frac{p'\left(\varrho(T^{\star}_\eps,x) \right)\varrho(T^{\star}_\eps,x)}{1-\rho_{f}(T^{\star}_\eps,x) },
\end{align*}
we can rely on the factor $(1+ \vert k \vert^2)^{-1}$ to show that the two infima are equal. By writing
\begin{align*}
\underset{(x,\gamma, \tau,k)}{\inf} \, \vert 1- \mathscr{P}_{f(t), \varrho(t)}(x,\gamma, \tau, k) \vert \geq \frac{c_0}{2}-\underset{(x,\gamma, \tau,k)}{\sup}  \vert \mathfrak{E}^{\star}(t,x, \gamma, \tau, k) \vert
\end{align*}
thanks to the triangular inequality, it remains to prove that a suitable choice of $\delta$ can lead to
\begin{align*}
\forall t \in [T^{\star}_\eps, + \infty), \ \ \underset{(x,\gamma, \tau,k)}{\sup} \, \vert \mathfrak{E}^{\star}(t,x, \gamma, \tau, k) \vert \leq \frac{c_0}{4}.
\end{align*}
First note that the remainder $\mathfrak{E}^{\star}$ has a factor $\chi^{\star}_{\delta}(t)$ in factor of all the terms in its expression, and is therefore compactly supported in time, with support in $(T^{\star}_\eps, T^{\star}_\eps+ \delta)$. Relying on the bounds for $f$ and $\varrho$ depending only on $M_{\mathrm{in}}$ (see Remarks \ref{rem:estim-rho-lowderivative} and Lemma \ref{energy:f-unif}), we can proceed as in the proof of the estimates \eqref{estim:P-0}--\eqref{estim:P-4A}--\eqref{estim:P-N} and show that
\begin{align*}
\vert \mathfrak{E}^{\star}(t) \vert \leq \Lambda(M_{\mathrm{in}}) \chi^{\star}_{\delta}(t) \vert t-T^{\star}_\eps \vert \leq \Lambda(M_{\mathrm{in}}) \delta.
\end{align*}
This procedure is allowed since the extension and the remainder $\mathfrak{E}^{\star}$ only involve a finite number of derivatives in time of $\varrho$ and $f$ at $t=T^{\star}_\eps$. Choosing $\delta$ small enough is now sufficient to conclude.
\end{proof}

We are now in position to provide the key $\Ld^2(0,T, \Ld^2)$ estimates for a solution to the equation \eqref{eq:pseudo:H}.

\begin{propo}\label{Prop:keyL2}
Assume that $H$ is a solution of the equation \eqref{eq:pseudo:H} on $\R \times \T^d$. There exists $\Lambda(M_{\mathrm{in}})>0$ such that for any $\gamma \geq \Lambda(M_{\mathrm{in}})$, we have
\begin{align*}
\Vert H \Vert_{\Ld^2(\R \times \T^d)}  \leq \Lambda(M_{\mathrm{in}})\Vert \mathcal{R} \Vert_{\Ld^2(\R \times \T^d)}.
\end{align*}
\end{propo}
\begin{proof}
Thanks to \eqref{new-symb} and Lemma \ref{LM:rewriteEqPseudo}, the equation \eqref{eq:pseudo:H} can be rewritten as
\begin{align*}
\left( \mathrm{Id}-\frac{\varrho}{1-\rho_f}\mathrm{Op}^{\gamma, \eps}(\mathfrak{a}_{f, \varrho}) \right)(H)=\mathcal{R},
\end{align*}
where $\mathfrak{a}_{f, \varrho}$ has been defined in \eqref{new-symb}. Now observe that, recalling the definition \eqref{def:PenroseSymbol} of the Penrose symbol $ \mathscr{P}_{f, \varrho}$, we have
\begin{align*}
\frac{\varrho}{1-\rho_f}\mathfrak{a}_{f, \varrho}= \mathscr{P}_{f, \varrho},
\end{align*}
therefore $H$ satisfies
\begin{align}\label{eq:Pseudo-1-P}
\mathrm{Op}^{\gamma, \eps}(1-\mathscr{P}_{f, \varrho})(H)=\mathcal{R}.
\end{align}
Relying on Proposition \ref{Prop:PenroseForalltimes} on the Penrose condition satisfied by the (extension) of $(f(t), \varrho(t))$ on $\R$, we can consider
\begin{align*}
c_{f, \varrho}:=\frac{1}{1-\mathscr{P}_{f, \varrho}}.
\end{align*}
Note that the symbol $c_{f, \varrho}-1$ vanishes outside a compact set in time and hence, in view of the estimates \eqref{estim:P-0}--\eqref{estim:P-4A}--\eqref{estim:P-N} of Corollary \ref{coro:Estim-symbols} on the symbol $\mathscr{P}_{f, \varrho}$ and the Faà di Bruno's formula, we get
\begin{align}\label{estim-c}
\omega[ c_{f, \varrho}^{\eps}-1 ]+\Omega\left[ c_{f, \varrho}^{\eps}-1 \right] \leq \Lambda(1+M_{\mathrm{in}}).
\end{align}
Applying $\mathrm{Op}^{\gamma, \eps}\left(c_{f, \varrho}-1 \right)$ to the previous equation \eqref{eq:Pseudo-1-P} yields 
\begin{align*}
H&= \mathrm{Op}^{\gamma, \eps}\left( c_{f, \varrho}-1 \right) (\mathcal{R}) + \Big[\mathrm{Op}^{\gamma, \eps}\big((c_{f, \varrho}-1)(1-\mathscr{P}_{f, \varrho}) \big)-\mathrm{Op}^{\gamma, \eps}\left( c_{f, \varrho}-1\right)\mathrm{Op}^{\gamma, \eps}(1-\mathscr{P}_{f, \varrho}) \Big](H) \\
&=\mathrm{Op}^{\gamma, \eps}\left( c_{f, \varrho}-1 \right) (\mathcal{R}) - \Big[\mathrm{Op}^{\gamma, \eps}\big((c_{f, \varrho}-1)\mathscr{P}_{f, \varrho} \big)-\mathrm{Op}^{\gamma, \eps}\left( c_{f, \varrho}-1\right)\mathrm{Op}^{\gamma, \eps}(\mathscr{P}_{f, \varrho}) \Big](H).
\end{align*}
Using the $\Ld^2$ continuity property from Theorem \ref{prop:conti-pseudoCV} and the commutation estimates from Proposition \ref{prop:compo-pseudoCV} in the Appendix, we get for all $\gamma >0$ and $\mathrm{M}>1+2d$
\begin{align*}
\LRVert{H}_{\Ld^2(\R \times \T^d)} &\leq  C\omega[  c_{f, \varrho}^{\eps}-1]\LRVert{\mathcal{R}}_{\Ld^2(\R \times \T^d)} + \frac{C}{\gamma} \Omega[c_{f, \varrho}^{\eps}-1] \Xi \big[\mathscr{P}_{f, \varrho}^{\eps} \big]_{\mathrm{M}} \LRVert{H}_{\Ld^2(\R \times \T^d)},
\end{align*}
for some constant $C>0$ depending only on the dimension. By homogeneity of the previous seminorms with respect to the semiclassical quantification in $\eps$ (in particular the fact that $\Omega[c_{f, \varrho}^{\eps}-1] \leq \Omega[c_{f, \varrho}-1]$ for $\eps \leq 1$) , we infer
\begin{align*}
\LRVert{H}_{\Ld^2(\R \times \T^d)} &\leq  C\omega[  c_{f, \varrho}-1 ]\LRVert{\mathcal{R}}_{\Ld^2(\R \times \T^d)}+ \frac{C}{\gamma} \Omega[c_{f, \varrho}-1] \big[\mathscr{P}_{f, \varrho} \big]_{\mathrm{M}}  \LRVert{H}_{\Ld^2(\R \times \T^d)}.
\end{align*}
Thanks to \eqref{estim-c} and the estimate \eqref{estim:P-4A} of Corollary \ref{coro:Estim-symbols}, we get
\begin{align*}
\LRVert{H}_{\Ld^2(\R \times \T^d)} \leq C\Lambda(1+M_{\mathrm{in}})\LRVert{\mathcal{R}}_{\Ld^2(\R \times \T^d)} +\frac{C}{\gamma} \Lambda(1+M_{\mathrm{in}}) \LRVert{H}_{\Ld^2(\R \times \T^d)}.
\end{align*}
Taking $\gamma$ large enough with respect to $M_{\mathrm{in}}$ allows  to perform an absorption argument: we get the existence of some $\Lambda(M_{\mathrm{in}})>0$ such that for any $\gamma \geq \Lambda(M_{\mathrm{in}})$, we have
\begin{align*}
\Vert H \Vert_{\Ld^2(R \times \T^d)}  \lesssim \Lambda(M_{\mathrm{in}})\Vert \mathcal{R} \Vert_{\Ld^2(\R \times \T^d)},
\end{align*}
which was the desired conclusion.
\end{proof}

As a consequence, we can eventually obtain an estimate for the equation \eqref{eq:Htilde} on $(0,T) \times \T^d$.

\begin{coro}\label{Coro:keyL2}
Consider $\widetilde{H}$ the solution to \eqref{eq:Htilde} on $(0,T) \times \T^d$. We have
\begin{align*}
\LRVert{\widetilde{H}}_{\Ld^2(0,T; \Ld^2(\T^d))} \leq \Lambda(T,M_{\mathrm{in}})\LRVert{\widetilde{\mathcal{R}}}_{\Ld^2(0,T; \Ld^2(\T^d))}.
\end{align*}
\end{coro}
\begin{proof}
First recall that 
\begin{align*}
\widetilde{H}(t,x):=e^{\gamma t} H(t,x), \ \  \widetilde{\mathcal{R}}(t,x)=e^{\gamma t} \mathcal{R}(t,x),
\end{align*}
and that $\mathcal{R}$ has been set to $0$ outside $[0,T]$. By Lemma \ref{LM:causality} and Proposition \ref{Prop:keyL2}, we thus get for all $\gamma \geq \Lambda(M_{\mathrm{in}})$
\begin{align*}
\int_0^T e^{-2\gamma t} \| \widetilde{H}(t) \|_{\Ld^2(\T^d)}^2 \, \dd t \leq \|  H\|_{\Ld^2(\R \times \T^d)}^2 &\leq \Lambda(M_{\mathrm{in}})^2\Vert \mathcal{R} \Vert_{\Ld^2(\R \times \T^d)}^2 \leq \Lambda(M_{\mathrm{in}})^2 \int_0^{T}  e^{-2 \gamma t}\|  \widetilde{\mathcal{R}}(t)  \|_{\Ld^2(\T^d)}^2  \, \dd t,
\end{align*}
therefore taking $\gamma=\Lambda(M_{\mathrm{in}})$ provides
\begin{align*}
\LRVert{\widetilde{H}}_{\Ld^2(0,T; \Ld^2(\T^d))} \leq e^{\Lambda(M_{\mathrm{in}}) T}\Lambda(M_{\mathrm{in}})\LRVert{\widetilde{\mathcal{R}}}_{\Ld^2(0,T; \Ld^2(\T^d))},
\end{align*}
which concludes the proof.
\end{proof}

 Let us conclude this section by giving a proof of the causality principle which was stated in Lemma \ref{LM:causality}.

\begin{proof}[Proof of Lemma \ref{LM:causality}]
Observe that thanks to the equation satisfied by $\widetilde{H}_1$ and $\widetilde{H}_2$, having $\widetilde{H}_1|_{(-\infty;0]}=\widetilde{H}_2|_{(-\infty;0]}=0$ implies $\widetilde{S}_1|_{(-\infty;0]}=\widetilde{S}_2|_{(-\infty;0]} =0$.

Let $H :=e^{-\gamma t} \left( \widetilde{H}_1-\widetilde{H}_2 \right)$ and $S:=e^{-\gamma t} \left(\widetilde{S}_1-\widetilde{S}_2 \right)$ for some $\gamma>0$, which satisfy
$$
H(t,x)-\frac{\varrho}{1-\rho_f}e^{-\gamma t} \mathrm{K}_G^{\mathrm{free}}\Big[e^{\gamma \bullet} \mathrm{J}_{\eps}H \Big](t,x) = S, \ \ \quad H|_{(-\infty;0]} =0, \ \ \,  S|_{(-\infty;T]} =0.
$$
By 
Proposition \ref{Prop:keyL2}, there exists $\Lambda(M_{\mathrm{in}})$ such that for all $\gamma \geq \Lambda(M_{\mathrm{in}})$
$$
\int_0^T e^{-2\gamma t} \LRVert{ \widetilde{H}_1(t)-\widetilde{H}_2 (t)}_{\Ld^2(\T^d)}^2 \, \dd t \leq \| H \|_{\Ld^2(\R \times \T^d)}^2 \leq \Lambda(M_{\mathrm{in}})^2  \| S \|_{\Ld^2(\R \times \T^d)}^2,
$$
for some $\gamma_0>0$ and $C_0 \geq 0$, therefore 
$$ \int_0^T e^{-2\gamma t} \LRVert{ \widetilde{H}_1(t)-\widetilde{H}_2 (t)}_{\Ld^2(\T^d)}^2 \, \dd t  \leq  \Lambda(M_{\mathrm{in}})^2 \int_T^{+\infty}  e^{-2 \gamma t} \LRVert{  \widetilde{S}_1(t)-\widetilde{S}_2(t) }_{\Ld^2(\T^d)}^2  \, \dd t,$$
since $\widetilde{S}_1|_{(-\infty;T]}=\widetilde{S}_2|_{(-\infty;T]}$. We thus infer
$$
\int_0^T \| H(t) \|_{\Ld^2(\T^d)}^2 \, \dd t \leq  \Lambda(M_{\mathrm{in}})^2 \int_T^{+\infty}  e^{2 \gamma (T- t)}\|  S(t)  \|_{\Ld^2(\T^d)}^2  \, \dd t.
$$
By letting $\gamma \to +\infty$, we deduce that $H|_{[0,T]} =0$, hence the result.
\end{proof}

\subsection{Final hyperbolic estimates}\label{Subsec:HyperbEstim}
To conclude this section, it remains to perform an energy estimate on the hyperbolic part $\left(\partial_t + u \cdot \nabla_x \right)(h)$.

Let us observe that by Remark \ref{rem:estim-u-lowderivative}, we have
\begin{align*}
\Vert \mathrm{div}_x \, u  \Vert_{\Ld^{\infty}(0,T;\Ld^{\infty})} \lesssim \left(1+T^{1/2}\Lambda(T,R) \right)M_{\mathrm{in}}+  T^{1/2}\Lambda(T,R),
\end{align*}
 by Sobolev embedding (since $m>1+d/2$), therefore for all $t \in (0,T^{\star}_{\eps})$, there holds
\begin{align}
\label{eq:estimdivu}
\Vert \mathrm{div}_x \, u  \Vert_{\Ld^{\infty}(0,t;\Ld^{\infty})} \leq 1+2M_{\mathrm{in}}.
\end{align}

We can therefore state the following lemma. 

\begin{lem}\label{LM:keyHyperb}
Let $T \in \left(0 , \min \left(T_\eps(R),\overline{T}(R),\widetilde{T}_0(R),\widehat{T}(R) \right)\right)$.
Assume that $h$ is a solution of the equation
\begin{align*}
\left(\partial_t + u \cdot \nabla_x \right)(h)=\widetilde{H}, \ \ t \in [0,T],
\end{align*}
There holds
\begin{align*}
\LRVert{h}_{\Ld^2(0,T; \Ld^2(\T^d))} \leq e^{(1+M_{\mathrm{in}})T} T^{\frac{1}{2}} \left(\Vert h(0) \Vert_{\Ld^2}  + T^{\frac{1}{2}}\LRVert{\widetilde{H}}_{\Ld^2(0,T; \Ld^2(\T^d))}\right).
\end{align*}
\end{lem}
\begin{proof}
The proof is standard; the hyperbolic nature of the equation provides pointwise in time $\Ld^2$ bounds thanks to an energy estimate  (in the same spirit as the proof of Proposition \ref{prop:energyRHO:eps}).  For all $t \in [0,T]$, we have
\begin{align*}
 \Vert h(t) \Vert_{\Ld^2} \leq  e^{\Vert \mathrm{div}_x \, u  \Vert_{\Ld^{\infty}(0,t;\Ld^{\infty})} t} \Vert h(0) \Vert_{\Ld^2} +\int_0^t e^{\Vert \mathrm{div}_x \, u  \Vert_{\Ld^{\infty}(0,t;\Ld^{\infty})} (t-\tau)} \Vert \widetilde{H}(\tau) \Vert_{\Ld^2} \, \mathrm{d}\tau.
\end{align*}
By \eqref{eq:estimdivu}, we get by the Cauchy-Schwarz inequality in time that for all $t \in [0,T]$
\begin{align*}
\Vert h(t) \Vert_{\Ld^2} \leq e^{(1+M_{\mathrm{in}})T} \left(\Vert h(0) \Vert_{\Ld^2}  + T^{\frac{1}{2}} \int_0^T \Vert \widetilde{H}(\tau) \Vert_{\Ld^2}^2 \, \mathrm{d}\tau\right),
\end{align*}
hence the conclusion, eventually taking the $\Ld^2$ norm in time on $(0,T)$ in the previous inequality .
\end{proof}

Gathering the results of Lemma \ref{LMestim:ResteRalpha}, Proposition \ref{coro:Facto}, Corollary \ref{Coro:keyL2} and Lemma \ref{LM:keyHyperb}, we directly infer the following statement.

\begin{coro}\label{Coro:keyHyperb}
For all $\vert \alpha \vert \leq m$,  for all $T \in  \left(0 , \min \left(T_\eps(R),\overline{T}(R),\widetilde{T}_0(R),\widehat{T}(R) \right)\right)$, we have the estimate 
\begin{align*}
\LRVert{\partial_x ^{\alpha} \varrho}_{\Ld^2(0,T; \Ld^2(\T^d))} \leq T^{\frac{1}{2}} \Lambda(M_{\mathrm{in}},T,R).
\end{align*}
\end{coro}

\section{End of the proof}\label{Section:END}
\subsection{Conclusion of the bootstrap}\label{Section:cclBootstrap}
Let us conclude the bootstrap argument. 
We choose
\begin{align*}
T \in \left( 0,\min \left(T_\eps(R),\overline{T}(R),\widetilde{T}_0(R),\widehat{T}(R) \right)\right).
\end{align*}
We can consider the following explicit quantity, which appears in all the estimates from Section \ref{Section:FluidDensity-estimate}:
 \begin{align*}
M_{\mathrm{in}}:=\LRVert{f^{\mathrm{in}}}_{\mathcal{H}^m_r}+\LRVert{\varrho^{\mathrm{in}}}_{\H^{m+1}}+\LRVert{u^{\mathrm{in}}}_{\H^m}.
\end{align*}
We now use Corollary \ref{Coro:keyHyperb} to get
\begin{align*}
\LRVert{ \varrho}_{\Ld^2(0,T; \H^m)} \leq T^{1/2} \Lambda(M_{\mathrm{in}},T,R).
\end{align*}
We also invoke the energy estimate of Lemma \ref{energy:f-unif} on $f$ that yields
\begin{align*}
 \Vert f\Vert_{\Ld^{\infty}\left(0,T;\mathcal{H}^{m-1}_{r}\right)} \leq M_{\mathrm{in}}+T^{\frac{1}{4}}\Lambda\left(T,R \right),
\end{align*}
and the energy estimate of Proposition \ref{prop:energy-u_eps} on $u$ giving
\begin{align*}
\Vert u \Vert_{\Ld^{\infty}(0,T;\H^m)\cap \Ld^{2}(0,T;\H^{m+1})}  
& \leq M_{\mathrm{in}}+ T^{1/2}\Lambda(M_{\mathrm{in}},T,R) + \Lambda(T,R)\Vert \varrho \Vert_{\Ld^{2}(0,T;\H^{m})},
\end{align*}
and therefore, using the previous estimate on $\varrho$, there holds
\begin{align*}
\Vert u \Vert_{\Ld^{\infty}(0,T;\H^m)\cap \Ld^{2}(0,T;\H^{m+1})}
& \leq M_{\mathrm{in}}+ T^{1/2}\Lambda(M_{\mathrm{in}},T,R).
\end{align*}
Combining all these estimates together, we finally obtain the key estimate
\begin{align*}
\mathcal{N}_{m,r}(f,\varrho,u,T) \leq C \left(M_{\mathrm{in}}+T^{1/4}\Lambda(T,R)+T^{1/2}\Lambda(M_{\mathrm{in}},T,R) \right),
\end{align*}
for some universal constant $C>0$. Note that the previous r.h.s is independent of $\eps$. Next, we choose $R$ large enough so that \begin{align*}
C M_{\mathrm{in}}<\frac{1}{2} R.
\end{align*}
Now, with $R$ being fixed, we use the continuity at $0$ of the function
\begin{align*}
s \mapsto s^{1/4}\Lambda(s,R)+s^{1/2}\Lambda(M_{\mathrm{in}},s,R)
\end{align*}
to find some time $T^{\sharp} >0$ independent of $\eps$ with
\begin{align*}
T^{\sharp} \in \left( 0,\min \left(\overline{T}(R),\widetilde{T}_0(R),\widehat{T}(R) \right)\right),
\end{align*}
and such that for every $T \in [0,T^{\sharp}]$
\begin{align*}
C \left(T^{1/4}\Lambda(M_{\mathrm{in}},T,R)+T^{1/2}\Lambda(M_{\mathrm{in}},T,R) \right)<\frac{1}{2}R.
\end{align*}
We deduce that for all $T \in [0,\min(T^{\sharp}, T_\eps(R))]$, we have $\mathcal{N}_{m,r}(f,\varrho,u,T)<R$. In addition, thanks to Lemmas \ref{LM:bound:Haut-rho_falpha}--\ref{LM:bound:HautBas-alpharho}, we easily get the fact that the condition \eqref{def:Bound-HAUTBAS} is satisfied, up to reducing $T^{\sharp}$ so that $\Theta +T^{\sharp}R<1$ and 
\begin{align*}
T^{\sharp} \in \left( 0,\frac{1}{R}\min \left( \frac{1-\Theta}{2}, \ln(2), \ln\left( \frac{2\underline{\theta}}{\mu} \right)\right) \right).
\end{align*}
Since we were assuming that $\mathcal{N}_{m,r}(f,\varrho,u,T_\eps(R))=R$, this shows that we must have $T_\eps >T^{\sharp}$. 

In conclusion, we have found $R>0$ and $T>0$ such that for all $\eps >0$
\begin{align*}
\mathcal{N}_{m,r}(f,\varrho,u,T) \leq R.
\end{align*}
\subsection{Existence of a solution}\label{Subsection:Existence}
Let us first focus on the existence part of Theorem \ref{THM:main}. We will rely on a standard compactness argument, that we briefly detail. In view of Section \ref{Section:cclBootstrap}, there exist $T>0$, $R>$ and 
$$(f_\eps, \varrho_\eps, u_\eps) \in \mathscr{C}([0,T]; \mathcal{H}^m_r) \times \mathscr{C}([0,T]; \H^m) \times \mathscr{C}([0,T]; \H^{m}) \cap \Ld^2(0,T; \H^{m+1})$$ 
a solution to \eqref{S_eps} with initial data $(f^{\mathrm{in}}, \varrho^{\mathrm{in}}, u^{\mathrm{in}})$  such that 
\begin{align*}
\underset{\eps \in (0, 1]}{\sup}\mathcal{N}_{m,r}(f_\eps,\varrho_\eps,u_\eps,T) \leq R.
\end{align*}
Hence, we deduce that $(f_\eps)_{\eps }$ is bounded in $\Ld^{\infty}(0,T;\mathcal{H}^{m-1}_r)$, $(\varrho_\eps)_{\eps }$ is bounded in $\Ld^{2}(0,T;\H^{m})$ and $(u_\eps)_{\eps }$ is bounded in $\Ld^{\infty}(0,T; \H^m)$ and in $ \Ld^{2}(0,T;\H^{m+1})$. From Lemma \ref{LM:rho-pointwise-Hm-2}, we also know that $(\varrho_\eps)_{\eps }$ is bounded in $\Ld^{\infty}(0,T;\H^{m-2})$. We deduce that $(f_\eps, \varrho_\eps, u_\eps)$ has a weak(-$\star$) limit $(f, \varrho, u)$ (up to some extraction) in the previous spaces.

Furthermore, using the equations on $f_\eps$, $\varrho_\eps$ and $u_\eps$, we know that $(\partial_t f_\eps)_{\eps \leq \eps_0}$ is bounded in $\Ld^{\infty}(0,T;\mathcal{H}^{m-3}_{r-1})$, $(\partial_t \varrho_\eps)_{\eps \leq \eps_0}$ is bounded in $\Ld^{\infty}(0,T;\H^{m-3})$ and $(\partial_t u_\eps)_{\eps \leq \eps_0}$ is bounded in $\Ld^{\infty}(0,T;\H^{m-3})$. Invoking the Aubin-Lions-Simon lemma (see e.g. \cite[Theorem II.5.16]{BF}), we deduce that, up to some extraction, we have
\begin{align*}
f_\eps \underset{\eps \rightarrow 0}{\longrightarrow} f \ \ \text{in} \  \mathscr{C}([0,T]; \mathcal{H}^{m-2}_{r-1}), \ \ \ \varrho_\eps \underset{\eps \rightarrow 0}{\longrightarrow} \varrho \ \ \text{in} \ \mathscr{C}([0,T]; \H^{m-3}), \ \ \ u_\eps \underset{\eps \rightarrow 0}{\longrightarrow} u  \ \ \text{in} \ \mathscr{C}([0,T]; \H^{m-1}).
\end{align*}
These strong convergences allow us to pass to the limit in \eqref{S_eps} and to obtain the fact that $(f, \varrho,u)$ is a solution to the thick spray equations \eqref{eq:TSgenBaro} on $[0,T]$, and that $(f, \varrho)$ satisfies a Penrose stability condition \eqref{cond:Penrose} on the same interval of time.

It remains to prove the fact that $f \in \mathscr{C}([0,T]; \mathcal{H}^{m-1}_{r})$ and $u \in \mathscr{C}([0,T]; \H^m)$, as we only have
$f \in \Ld^{\infty}(0,T; \mathcal{H}^{m-1}_{r})$ and $u \in \Ld^{\infty}(0,T;\H^m)$
for the moment. 
Since $f \in \Ld^{\infty}(0,T;\mathcal{H}^{m-1}_r) \cap \mathscr{C}_{w}([0,T]; \mathcal{H}^{m-2}_{r-1})$ and $u \in \Ld^{\infty}(0,T; \H^m) \cap \mathscr{C}_w([0,T]; \H^{m-1})$ , we know (see e.g. \cite[Lemma II.5.9]{BF}) that $f \in \mathscr{C}_{w}([0,T]; \mathcal{H}^{m-1}_{r})$ and $u \in \mathscr{C}_w([0,T]; \H^{m})$. It is now sufficient to prove that $t \mapsto \LRVert{f(t)}_{\mathcal{H}^{m-1}_{r}}$ and $t \mapsto \LRVert{u(t)}_{\H^{m}}$ are continuous functions on $[0,T]$ to conclude. Coming back to the energy estimates of Subsection \ref{Subsection:EstimateREG}, we have
\begin{align*}
\frac{\mathrm{d}}{\mathrm{d}t}\LRVert{f(t)}_{\mathcal{H}^{m-1}_{r}}^2 <+\infty, \ \ \frac{\mathrm{d}}{\mathrm{d}t}\LRVert{u(t)}_{\H^{m}}^2 <+\infty.
\end{align*}
Since $t \mapsto \LRVert{f(t)}_{\mathcal{H}^{m-1}_{r}}^2$ and $t \mapsto \LRVert{u(t)}_{\H^{m}}^2$ are integrable (because $f \in \Ld^{\infty}(0,T; \mathcal{H}^{m-1}_{r})$ and $u \in \Ld^{\infty}(0,T;\H^m)$), we obtain the fact that $t \mapsto \LRVert{f(t)}_{\mathcal{H}^{m-1}_{r}}^2$ and $t \mapsto \LRVert{u(t)}_{\H^{m}}^2$ belongs to $\mathrm{W}^{1,1}(0,T)$, hence the continuity in time of these quantities.
This finally yields the desired continuity for $f$ and $u$ and concludes the proof.

\subsection{Uniqueness of the solution}\label{Subsection:Uniqueness}
Let us turn to the uniqueness part of  Theorem \ref{THM:main}. Let $(f_1, \varrho_1, u_1)$ and $(f_2, \varrho_2, u_2)$ be two solutions of \eqref{eq:TSgenBaro} belonging to
\begin{align*}
\Ld^{\infty}(0,\bm{T}; \mathcal{H}^{m-1}_r) \times \Ld^{2}(0,\bm{T}; \H^m) \times \Ld^{\infty}(0,\bm{T}; \H^{m}) \cap \Ld^2(0,\bm{T}; \H^{m+1}),
\end{align*}
for some $\bm{T} >0$, with the same initial condition $(f^{\mathrm{in}}, \varrho^{\mathrm{in}}, u^{\mathrm{in}})$ and such that $t \mapsto(f_1(t), \varrho_1(t))$ satisfies the Penrose stability condition \eqref{cond:Penrose} on $[0,\bm{T}]$. Let us set
\begin{align*}
f:=f_1-f_2, \ \ \alpha_i:=1-\rho_{f_i}, \ \ \alpha:=\alpha_1-\alpha_2, \ \ \varrho:=\varrho_1-\varrho_2, \ \ u:=u_1-u_2,
\end{align*}
and 
\begin{align*}
R:= \underset{i=1,2}{\max} \left(\LRVert{f_i}_{\Ld^{\infty}(0,\bm{T}; \mathcal{H}^{m-1}_r)}+\LRVert{\varrho_i}_{\Ld^{2}(0,\bm{T}; \H^m)}+\LRVert{u_i}_{\Ld^{\infty}(0,\bm{T}; \H^{m}) \cap \Ld^2(0,\bm{T}; \H^{m+1}),} \right).
\end{align*}
Note that $R$ depends on $\bm{T}$.

\medskip

$\bullet$ \underline{\textbf{Step 1}}: let us show that $\varrho_1=\varrho_2$, at least on a small interval of time.  The key is to obtain $\Ld^2(0,T; \Ld^2)$ estimates for $\varrho$ for some $T=T(R)\leq \bm{T}$. To this end, we write the equation satisfied by $\varrho$ as 
\begin{align*}
&\alpha_1 \left( \pa_t \varrho + u_1 \cdot \na_x \varrho \right) + \varrho_1\left(\pa_t \alpha +  u_1 \cdot \na_x \alpha \right) \\
&=- \Bigg( \pa_t \varrho_2 \alpha + \pa_t \alpha_2 \varrho + [\varrho u_1 + \varrho_2 u] \cdot \na_x \alpha_2 + [\alpha u_1 + \alpha_2 u ] \cdot \na_x \varrho_2 + \alpha \rho_1  \mathrm{div}_x u_1 + \alpha_2 \varrho  \mathrm{div}_x u_1 + \alpha_2 \varrho_2  \mathrm{div}_x u
\Bigg),
\end{align*}
and since $\partial_t \alpha_i=\mathrm{div}_x \, j_{f_i}$, we get 
\begin{align}\label{Eq1:UniqRho}
\pa_t \varrho + u_1 \cdot \na_x \varrho + \frac{\varrho_1}{1-\rho_{f_1}} \mathrm{div}_x (j_f-\rho_f u_1)=S_{1,2}(f, \varrho,u),
\end{align}
where
\begin{multline}\label{source-S12}
S_{1,2}(f, \varrho,u):=- \frac{1}{1-\rho_{f_1}}\Bigg( -\pa_t \varrho_2 \rho_f+ \pa_t \alpha_2 \varrho + [\varrho u_1 + \varrho_2 u] \cdot \na_x \alpha_2 + [-\rho_f u_1 + \alpha_2 u ] \cdot \na_x \varrho_2 \\
 -\rho_f \rho_1  \mathrm{div}_x u_1 + \alpha_2 \varrho  \mathrm{div}_x u_1 + \alpha_2 \varrho_2  \mathrm{div}_x u-\varrho_2 \rho_f\mathrm{div}_x \, u_1
\Bigg).
\end{multline}
The r.h.s $S_{1,2}(f, \varrho,u)$ can be seen as a source term whose $\Ld^2(0,T; \Ld^2)$ norm will be estimated by that of $\varrho$, without any loss of derivative in $(f, \varrho, u)$. Note that we have rewritten the equation as above in order to perform the right pseudodifferential factorization. It is of course reminiscent of what we have done in Section \ref{Section:FluidDensity-estimate}. Next, the equations on the differences $f$ and $u$ read
\begin{equation}\label{eq:f-u}
\left\{
      \begin{aligned}
&\partial_t f +v \cdot \nabla_x f +\mathrm{div}_v \big[ (u_2-v)f-\nabla_x p(\varrho_2)f \big]  +\Big(u-\nabla_x p(\varrho_1) +\nabla_x p(\varrho_2) \Big) \cdot \nabla_v f_1=0, \\[2mm]
&\partial_t u +(u_2 \cdot \nabla_x)u+(u \cdot \nabla_x )u_1+ \frac{1}{\varrho_1} \nabla_x p(\varrho_1)-\frac{1}{\varrho_2} \nabla_x p(\varrho_2) \\
& \qquad \qquad \qquad \qquad  -\frac{1}{\alpha_1 \varrho_1} \Big( \Delta_x  + \nabla_x \mathrm{div}_x \Big)u-\left(\frac{1}{\alpha_1 \varrho_1}-\frac{1}{\alpha_2 \varrho_2}  \right)\Big( \Delta_x  + \nabla_x \mathrm{div}_x \Big)u_2=j_{f}-\rho_f u_1 - \rho_{f_2}u,
\end{aligned}
    \right.
\end{equation}
and in particular
\begin{multline*}
\partial_t f +v \cdot \nabla_x f +\mathrm{div}_v \big[ (u_2-v)f-\nabla_x p(\varrho_2)f \big] \\[1mm] =p'(\varrho_1) \nabla_x \varrho \cdot \nabla_v f_1+ (p'(\varrho_1)-p'(\varrho_2))\nabla_x \varrho_2 \cdot \nabla_v f_1-u \cdot \nabla_x f_1.
\end{multline*}
Again, the last two terms in the r.h.s should be seen as some source terms, without any loss of derivative in $(f, \varrho, u)$. One can show that
\begin{align*}
\LRVert{(p'(\varrho_1)-p'\varrho_2))\nabla_x \varrho_2 \cdot \nabla_v f_1}_{\Ld^2(0,T; \Ld^2)} +\LRVert{u \cdot \nabla_x f_1}_{\Ld^2(0,T; \Ld^2)} \leq \Lambda(T,R) \LRVert{\varrho}_{\Ld^2(0,T; \Ld^2)}.
\end{align*}
The proof of the estimate for the second term will be similar to the one for the term $S[\varrho]$ we will treat below.

Arguing as in Sections \ref{Section:Lagrangian}--\ref{Section:Averag-lemma}--\ref{Section:Kinetic-moments}, but for a force field
 $$\mathfrak{F}(t,x):=u_2(t,x)-\nabla_x p(\varrho_2(t,x)),$$
 instead of $E^{\varrho,u}_{\mathrm{reg}, \eps}$, one can also prove that for $T(R)$ small enough, we have for all $t \in [0, T(R)]$
\begin{align*}
\rho_f(t,x)&= p'(\varrho_1(t,x))\int_0^t \int_{\R^d} \nabla_x \varrho(s, x-(t-s)v) \cdot \nabla_v f_1(t,x,v) \, \mathrm{d}s \, \mathrm{d}v+\mathcal{R}[\rho_f](t,x),\\
j_f(t,x)&= p'(\varrho_1(t,x))\int_0^t \int_{\R^d} v \nabla_x \varrho(s, x-(t-s)v) \cdot \nabla_v f_1(t,x,v) \, \mathrm{d}s \, \mathrm{d}v+\mathcal{R}[j_f](t,x),
\end{align*}
where
\begin{align*}
\LRVert{\mathcal{R}[\rho_f]}_{\Ld^2(0,T; \H^1)}+\LRVert{\mathcal{R}[j_f]}_{\Ld^2(0,T; \H^1)} \leq \Lambda(T,R)\LRVert{\varrho}_{\Ld^2(0,T; \Ld^2)}.
\end{align*}
Here, we have also used the fact that $f_{\mid t=0}=0$. Injecting these expressions in equation \eqref{Eq1:UniqRho} on $\varrho$, we get as in Section \ref{Section:FluidDensity-estimate}
\begin{align*}
\partial_t \varrho + u_1 \cdot \na_x \varrho + \frac{\varrho_1}{1-\rho_{f_1}} \mathrm{div}_x \Big[ \mathrm{K}_{1,G_1}^{\mathrm{free}}(\varrho)-\mathrm{K}_{G_1}^{\mathrm{free}}(\varrho)u_1\Big]=S[\varrho],
\end{align*}
with
$$G_1(t,x,v):=  p'( \varrho_1(t,x)) \nabla_v f_1(t,x,v),$$
and
$$S[\varrho]= -\frac{\varrho_1}{1-\rho_{f_1}} \mathrm{div}_x \left(\mathcal{R}[j_f]-\mathcal{R}[\rho_f] u_1 \right)+ S_{1,2}(f, \varrho,u),$$
(the operator $\mathrm{K}_{1,G_1}^{\mathrm{free}}$ being defined in \eqref{eqdef:K_1G}) and then
\begin{align}\label{Eq2:Uniq-rho}
\left( \mathrm{Id}-\frac{\varrho_1}{1-\rho_{f_1}} \mathrm{K}^{\mathrm{free}}_{G_1} \right) \Big[\partial_t \varrho + u_1 \cdot \nabla_x \varrho \Big]=S[\varrho], \ \ t \in (0,T).
\end{align}
It is straightforward to prove that the source $S[\varrho]$ satisfies
\begin{align*}
\LRVert{S[\varrho]}_{\Ld^2(0,T; \Ld^2)} \leq \Lambda(T,R)\LRVert{\varrho}_{\Ld^2(0,T; \Ld^2)}.
\end{align*}
The main issue comes from obtaining the same estimate for the term $S_{1,2}(u,f)$ defined in \eqref{source-S12}. We observe that it is made of a sum of three types of terms, of the following form:
\begin{enumerate}
\item[$-$] some terms where $\varrho$ is in factor, so that we directly obtain the estimate;
\item[$-$]  some terms where $\rho_f$ is in factor: we can rely on the previous decomposition of $\rho_f$ as $\rho_f=\mathrm{K}^{\mathrm{free}}_{G_1}[\varrho]+\mathcal{R}[\rho_f]$  and an $\Ld^2(0,T; \Ld^2)$ estimate is then provided by the smoothing estimate of Proposition \ref{propo:AveragStandard};
\item[$-$] some terms where $u$ or $\mathrm{div}_x \, u$ is in factor: to estimate them, we perfom the same kind of energy estimate as in Subsection \ref{Appendix:LWPeps} of the Appendix, following the steps leading to the equality \eqref{ineq:dtw_n-uniquenessLWP}, which gives after integration in time (since $u_{\mid t=0}=0$)
\begin{align*}
&\LRVert{u(t)}_{\Ld^2}^2+\int_0^t \LRVert{\mathrm{div}_x \, u(s)}_{\Ld^2}^2 \, \mathrm{d}s \\
&\lesssim \int_0^t \LRVert{u(s)}_{\Ld^2}^2 \, \mathrm{d}s+  \int_0^t \left( \LRVert{\rho_f(s)}_{\Ld^2}^2+\LRVert{j_f(s)}_{\Ld^2}^2+\LRVert{(\alpha_1 \varrho_1(s) - \alpha_2 \varrho_2(s)}_{\Ld^2}^2 \right) \, \mathrm{d}s \\
& \lesssim \int_0^t \LRVert{u(s)}_{\Ld^2}^2 \, \mathrm{d}s+  \Lambda(T,R)\LRVert{\varrho}_{\Ld^2(0,T; \Ld^2)}.
\end{align*}
Here, we have also used the smoothing estimate from Proposition \ref{propo:AveragStandard} combined with the previous decomposition of $\rho_f$ and $j_f$. We obtain the desired control on $u$ after an application of Grönwall's lemma.
\end{enumerate}

\bigskip

As in Subsection \ref{Subsection:EllipticEstimate}, we then study the equation 
\begin{align*}
\left( \mathrm{Id}-\frac{\varrho_1}{1-\rho_{f_1}}\mathrm{K}^{\mathrm{free}}_{G_1} \right)\left[\widetilde{H} \right]=\widetilde{\mathcal{R}}, \ \ 0 \leq t \leq T,
\end{align*}
where $\widetilde{\mathcal{R}}$ is a given source term and we want to derive an $\Ld^2(0,T; \Ld^2)$ estimate on the solution $\widetilde{H}$. After applying the same extension procedure for the coefficients (depending on $(f_1, \varrho_1,u_1)$) in the equation as in Section \ref{Subsection-Extension}, and by setting $\widetilde{H}(t,x):=e^{\gamma t} H(t,x)$ and $\widetilde{\mathcal{R}}(t,x)=e^{\gamma t} \mathcal{R}(t,x)$ for $\gamma > 0$ (with the same continuation by zero outside $[0,T]$ as in Subsection \ref{Subsection:EllipticEstimate}), we are led to the study of the pseudodifferential equation
\begin{align*}
\mathrm{Op}^{\gamma}(1-\mathcal{P}_{f_1, \varrho_1})(H)=\mathcal{R},
\end{align*}
where
\begin{align*}
\mathcal{P}_{f_1(t), \varrho_1(t)}(x,\gamma, \tau,k):=\frac{p'(\varrho_1(x))\varrho_1(x)}{1-\rho_{f_1}(x)}\int_0^{+\infty} e^{-(\gamma+i \tau) s} i k \cdot \left( \mathcal{F}_v \nabla_v f_1 \right)(t,x,ks)  \, \mathrm{d}s.
\end{align*}
Observe that there is no factor $(1+\vert k \vert^2)^{-1}$ in the definition of this symbol, because there is no regularization operator in the equation. Note also that the estimates \eqref{estim-techn:P-0}--\eqref{estim-techn:P-A4}--\eqref{estim-techn:P-N} holds true for $\mathcal{P}_{f_1, \varrho_1}$, in view of the regularity of $(f_1, \varrho_1)$. Likewise, as in Corollary \ref{Coro:keyL2}, we have 
\begin{align*}
\LRVert{\widetilde{H}}_{\Ld^2(0,T; \Ld^2(\T^d))} \leq \Lambda(T,R)\LRVert{\widetilde{\mathcal{R}}}_{\Ld^2(0,T; \Ld^2(\T^d))},
\end{align*}
provided that one can apply $\mathrm{Op}^{\gamma}\left( \frac{1}{1-\mathcal{P}_{f_1, \varrho_1}}\right)$ to the previous pseudodifferential equation and take $\gamma$ large enough in order to invert it up to a small remainder. To do so, we prove that if the condition \eqref{cond:Penrose}
\begin{align}
\tag{\textbf{P}}
\forall t \in \R, \ \ \forall x \in \T^d, \ \ \underset{(\gamma, \tau,k)\in (0,+\infty) \times \R \times \R^d {\setminus} \lbrace 0 \rbrace}{\inf} \, \vert 1- \mathscr{P}_{f_1(t), \varrho_1(t)}(x,\gamma, \tau, k) \vert>c,
\end{align}
holds for some $c>0$, then the condition
\begin{align}\label{cond:PenroseOpti}
\tag{\textbf{Opt-P}}
\forall t \in \R, \ \ \forall x \in \T^d, \ \ \underset{(\gamma, \tau,k)\in (0,+\infty) \times \R \times \R^d {\setminus} \lbrace 0 \rbrace}{\inf} \, \vert 1- \mathcal{P}_{f_1(t), \varrho_1(t)}(x,\gamma, \tau, k) \vert>c,
\end{align}
holds as well. Here, we implicitly consider the extension in time of $f_1$ and $\varrho_1$, as it was done in Section \ref{Subsection-Extension}. To this end, we rely on a homogeneity argument, as in \cite{HKR}: we define the function
\begin{align*}
\widetilde{\mathscr{P}}_{f_1, \varrho_1}(x,\widetilde{\gamma}, \widetilde{\tau}, \widetilde{k}, \sigma):=\mathscr{P}_{f_1, \varrho_1}(x,\sigma\widetilde{\gamma}, \sigma\widetilde{\tau}, \sigma\widetilde{k}), \ \ \ x \in \T^d, \  (\widetilde{\gamma}, \widetilde{\tau}, \widetilde{k}) \in S^+, \  \sigma>0,
\end{align*}
where
\begin{align*}
S^+:= \left\lbrace (\widetilde{\gamma}, \widetilde{\tau}, \widetilde{k}) \in (0,+\infty) \times \R \times \R^d {\setminus} \lbrace 0 \rbrace \mid \widetilde{\gamma}^2+ \widetilde{\tau}^2+ \widetilde{k}^2=1 \right\rbrace.
\end{align*}
Since $f_1$ is regular enough, one can prove that $\widetilde{\mathscr{P}}_{f_1, \varrho_1}$ can be extended as a continuous function on $\T^d \times S^+ \times [0, + \infty)$ and we obtain
\begin{align*}
\forall t \in \R, \ \ \forall x \in \T^d, \ \ \underset{(\widetilde{\gamma}, \widetilde{\tau}, \widetilde{k}, \sigma )\in S^+ \times [0, + \infty)}{\inf} \, \vert 1- \widetilde{\mathscr{P}}_{f_1(t), \varrho_1(t)}(x, \widetilde{\gamma}, \widetilde{\tau}, \widetilde{k}, \sigma) \vert>c.
\end{align*}
In view of the homogeneity of degree $0$ of the symbol $a_{f,\varrho}$ with respect to the variable $(\gamma, \tau, \eta)$ (see Lemma \ref{LM:homogeneous}), we also  have
\begin{align*}
\widetilde{\mathscr{P}}_{f_1, \varrho_1}(x,\widetilde{\gamma}, \widetilde{\tau}, \widetilde{k}, \sigma)=\frac{1}{1+\sigma^2 \vert \widetilde{k} \vert^2} \frac{p'(\varrho_1(x))\varrho_1(x)}{1-\rho_{f_1}(x)}\int_0^{+\infty} e^{-(\widetilde{\gamma}+i \widetilde{\tau}) s} i \widetilde{k} \cdot \left( \mathcal{F}_v \nabla_v f_1 \right)(t,x,\widetilde{k}s)  \, \mathrm{d}s,
\end{align*}
hence
\begin{align*}
\widetilde{\mathscr{P}}_{f_1(t), \varrho_1(t)}(x,\widetilde{\gamma}, \widetilde{\tau}, \widetilde{k}, 0)=\mathcal{P}_{f_1(t), \varrho_1(t)}(x,\widetilde{\gamma}, \widetilde{\tau}, \widetilde{k}),
\end{align*}
from which we infer that \eqref{cond:PenroseOpti} holds on $S^+$. Again, the homogeneity of degree $0$ of $\mathscr{P}_{f_1(t), \varrho_1(t)}(x,\gamma, \tau, k)$ with respect to the variable $(\gamma, \tau, k)$ implies that $\eqref{cond:PenroseOpti}$ holds.

\medskip

Next, we come up with the transport equation on $\varrho$
\begin{align*}
\partial_t \varrho + u_1 \cdot \nabla_x \varrho=\widetilde{H},
\end{align*}
where $\varrho_{\mid t=0}=0$ and where the source term $\widetilde{H}$ has been shown to satisfied
\begin{align*}
\LRVert{\widetilde{H}}_{\Ld^2(0,T; \Ld^2(\T^d))} \leq \Lambda(T,R)\LRVert{S[\varrho]}_{\Ld^2(0,T; \Ld^2(\T^d))} \leq \Lambda(T,R)\LRVert{\varrho}_{\Ld^2(0,T; \Ld^2(\T^d))}.
\end{align*}
Performing an $\Ld^2$ energy estimate gives for all $t \in [0,T]$ (since $\varrho_{\mid t=0}=0$)
\begin{align*}
 \Vert \varrho(t) \Vert_{\Ld^2} \leq  \Lambda(T,R)\int_0^t  \Vert \widetilde{H}(\tau) \Vert_{\Ld^2} \, \mathrm{d}\tau \leq T^{\frac{1}{2}} \Lambda(T,R)\LRVert{\varrho}_{\Ld^2(0,T; \Ld^2(\T^d))}, 
\end{align*}
by the Cauchy-Schwarz inequality. We end up with
\begin{align*}
\LRVert{\varrho}_{\Ld^2(0,T; \Ld^2(\T^d))}\leq T \Lambda(T,R)\LRVert{\varrho}_{\Ld^2(0,T; \Ld^2(\T^d))}.
\end{align*}
We deduce the fact that there exists a small enough $T=T(R)>0$ which depends only on $R$ such that 
\begin{align*}
\forall t \in [0, T(R)], \ \ \varrho(t)=0.
\end{align*}

\medskip

$\bullet$ \underline{\textbf{Step 2}}: let us now prove that $f_1=f_2$ and $u_1=u_2$ on $[0,T(R)]$. It is in fact a direct consequence of $\varrho=\varrho_1-\varrho_2=0$ on $[0,T(R)]$. Indeed, the previous step has shown that for all $t \in [0,T(R)]$
\begin{align*}
&\LRVert{u(t)}_{\Ld^2}^2 \lesssim \int_0^t \LRVert{u(s)}_{\Ld^2}^2 \, \mathrm{d}s+  \Lambda(t,R)\LRVert{\varrho}_{\Ld^2(0,t; \Ld^2)} \lesssim \int_0^t \LRVert{u(s)}_{\Ld^2}^2 \, \mathrm{d}s,
\end{align*}
therefore we directly have $u=0$ on $[0,T(R)]$. The equations on \eqref{eq:f-u} on $f$ now turn into
\begin{equation*}
\left\{
      \begin{aligned}
\partial_t f +v \cdot \nabla_x f -v \cdot \nabla_v f  +\big( u_2-\nabla_x p(\varrho_2) \big)  \cdot \nabla_v f=df, \\[2mm]
f_{\mid t=0}=0.
\end{aligned}
    \right.
\end{equation*}
The method of characteristics thus shows that $f=0$ on $[0,T(R)]$.

\medskip

In conclusion, we have obtained $(f, \varrho, u)=(0,0,0)$ on $[0,T(R)]$. We eventually observe that we can repeat this procedure starting from $t=T(R)$ instead of $t=0$. As a matter of fact, $f_1(T(R),\cdot)$ still satisfies the Penrose stability condition. Since the time $T(R)$ only depends on $R$, we obtain $\varrho=0$ on $[0,2T(R)]$ and then $(f, u)=(0,0)$ on $[0,2T(R)]$. After a finite number of steps, this yields $(f, \varrho, u)=(0,0,0)$ on $[0,\bm{T}]$. This finally concludes the proof of the uniqueness part of the statement.

\section{Generalization to the non-barotropic case}\label{Section:AppendixInternalEnergy}
In this section, we show how to handle the case of the full Navier-Stokes system for non-barotropic fluids, where we consider the additional internal energy $\mathfrak{e} \in \R^+$ for the fluid and where the the pressure depends on $\varrho$ and $\mathfrak{e}$ . As explained in the introduction, the system which is at stake is the following:
\begin{equation}
\left\{
      \begin{aligned}
\partial_t f +v \cdot \nabla_x f +{\rm div}_v \Big[ f (u-v)-f \nabla_x p(\varrho, \mathfrak{e}) \Big]&=0,  \\
\partial_t (\alpha \varrho ) + \mathrm{div}_x (\alpha \varrho u)&=0, \\
\partial_t (\alpha \varrho u) + \mathrm{div}_x(\alpha \varrho u \otimes u)+\alpha \nabla_x p(\varrho, \mathfrak{e}) -\Delta_x u - \nabla_x \mathrm{div}_x u&=j_f-\rho_f u, \\
\partial_t (\alpha \varrho \mathfrak{e}) + \mathrm{div}_x(\alpha \varrho \mathfrak{e}u) +p(\varrho, \mathfrak{e})\left(\partial_t \alpha+ \mathrm{div}_x(\alpha u) \right)&=\int_{\R^d}\vert u-v \vert^2 f \, \mathrm{d}v, \\
\alpha&=1-\rho_f.
\end{aligned}
    \right.
\end{equation}
Here, we will assume\footnote{It is likely that more general pressure $p(\varrho, \mathfrak{e})$ could be treated by our method.} that the pressure law is given as
\begin{align*}
p(\varrho, \mathfrak{e})=\pi(\varrho \mathfrak{e}),
\end{align*}
for some given function $\pi : \R^+ \rightarrow \R$ such that $\pi \in \mathscr{C}( \R^+ ) \cap \mathscr{C}^{\infty}(\R^+ {\setminus} \lbrace 0 \rbrace)$. For instance, the relation $p(\varrho,\mathfrak{e})=b \varrho \mathfrak{e}$ (with $b>0$) is a perfect gas pressure law. 
Similarly to the technical hypothesis \eqref{Assumption-pressure}, we shall assume that
\begin{align*}
y \mapsto \pi'(y) (y+\pi(y)) \ \ \text{is nondecreasing on} \ \  \R.
\end{align*}

\bigskip

Setting
\begin{align*}
\vartheta:=\varrho  \mathfrak{e},
\end{align*}
the system in $(f,\varrho, u, \vartheta)$ can be rewritten as
\begin{equation}\label{eq:TSe}
\tag{TS$_e$}
\left\{
      \begin{aligned}
\partial_t f +v \cdot \nabla_x f   +\mathrm{div}_v \big[f(u-v-\nabla_x \pi(\vartheta)) \big] &=0,   \\[2mm]
\partial_t (\alpha \varrho ) + \mathrm{div}_x (\alpha \varrho u)&=0, \\[2mm]
\partial_t (\alpha \varrho u) + \mathrm{div}_x(\alpha \varrho u \otimes u) +\alpha\nabla_x \pi(\vartheta)-\Delta_x u - \nabla_x \mathrm{div}_x u &=j_f-\rho_f u,  \\[2mm]
\partial_t (\alpha \vartheta) + \mathrm{div}_x(\alpha \vartheta u) +\pi(\vartheta) \left( \partial_t \alpha+ \mathrm{div}_x(\alpha u) \right)&=\int_{\R^d} \vert v-u \vert^2 f \, \mathrm{d}v,  \\
 \alpha&=1-\rho_f.
\end{aligned}
    \right.
\end{equation}
As we shall see below, it is significant to define the following Penrose symbol of a function $(f(x,v), \vartheta(x))$ as
\begin{align}\label{def:PenroseSymbol-energyintern}
\mathbf{P}_{f, \vartheta}^{\mathrm{energy}}(x,\gamma, \tau, \eta):=\frac{\pi'(\vartheta(x))\left[\vartheta(x)+\pi(\vartheta(x))\right]}{1-\rho_f(x)}\int_0^{+\infty} e^{-(\gamma+i \tau) s} \frac{i k}{1+ \vert k \vert^2} \cdot \left( \mathcal{F}_v \nabla_v f \right)(x,s\eta)  \, \mathrm{d}s.
\end{align}
We can now introduce the following Penrose condition, adapted to \eqref{eq:TSe}, which is that there exists $c>0$ such that
\begin{align}\label{cond:Penrose-energyintern}
\tag{\textbf{P$^{\mathrm{energy}}$}}
\forall x \in \T^d, \ \ \underset{(\gamma, \tau,\eta)\in (0,+\infty) \times \R \times \R^d {\setminus} \lbrace 0 \rbrace}{\inf} \, \vert 1- \mathbf{P}_{f, \vartheta}^{\mathrm{energy}}(x,\gamma, \tau, \eta) \vert>c,
\end{align}
Our main result for the system \eqref{eq:TSe} reads as follows.
\begin{thm}\label{THM:Energyintern}
There exist $m_0>0$ and $r_0>0$, depending only on the dimension, such that the following holds for all $m \geq m_0$ and $r \geq r_0$. Let 
\begin{align*}
f^{\mathrm{in}} \in  \mathcal{H}^{m}_{r}, \ \  \varrho^{\mathrm{in}} \in \mathrm{H}^{m-1}, \ \ u^{\mathrm{in}} \in \mathrm{H}^{m}, \ \ \vartheta^{\mathrm{in}} \in \mathrm{H}^{m+1},
\end{align*}
such that $(f^{\mathrm{in}}, \vartheta^{\mathrm{in}})$ satisfies the $c$-Penrose stability condition $\eqref{cond:Penrose-energyintern}_c$ (with $c>0$) and 
\begin{align*}
0 &\leq f^{\mathrm{in}}, \ \ \ \rho_{f^{\mathrm{in}}}<\Upsilon<1, \ \ \ 0<\mu \leq \varrho^{\mathrm{in}}, \vartheta^{\mathrm{in}} \ \ \ 0<\underline{\nu} \leq (1-\rho_{f^{\mathrm{in}}})\varrho^{\mathrm{in}}, (1-\rho_{f^{\mathrm{in}}})\vartheta^{\mathrm{in}} \leq \overline{\nu},
\end{align*}
for some fixed constants $\Upsilon, \mu, \underline{\nu}, \overline{\nu}$.
Then there exist $T>0$ and a solution $(f,  \varrho, u, \vartheta)$ to \eqref{eq:TSe}  with initial condition $(f^{\mathrm{in}}, \varrho^{\mathrm{in}}, u^{\mathrm{in}}, \vartheta^{\mathrm{in}})$ such that
\begin{align*}
f \in  \mathscr{C}\left([0,T];\mathcal{H}^{m-1}_{r} \right), \ \ \varrho \in \mathscr{C}([0,T];\mathrm{H}^{m-1}), \ \  \ u \in \mathscr{C}\left([0,T]; \mathrm{H}^{m} \right)\cap \Ld^2(0,T;\mathrm{H}^{m+1}), \ \ \vartheta \in \Ld^2(0,T;\mathrm{H}^{m}),
\end{align*}
and with $(f(t), \vartheta(t))$ satisfying the $c/2$-Penrose stability condition $\eqref{cond:Penrose-energyintern}_{c/2}$ for all $t \in [0,T]$. In addition, this solution is unique in this class.
\end{thm}
Let us explain the main strategy for the proof of Theorem \ref{THM:Energyintern}. First, we observe that the equation on $\vartheta$ can be rewritten as
\begin{align}\label{eq:theta}
\partial_t \vartheta+ u \cdot \nabla_x \vartheta + \frac{\vartheta + \pi(\vartheta)}{1-\rho_f} \mathrm{div}(j_f-\rho_f u)=-\frac{\vartheta + \pi(\vartheta)}{1-\rho_f} \mathrm{div}_x \, u+ \frac{1}{1-\rho_f}\int_{\R^d} \vert v-u \vert^2 f \, \mathrm{d}v.
\end{align}
Appart from the last term, this equation has exactly the same structure as the one for $\varrho$ in Lemma \ref{LM:rewriteEqrho}. Hence, the following estimate holds
\begin{align*}
\Vert \vartheta(t) \Vert_{\H^{m}} \leq   \Vert \varrho^{\mathrm{in}} \Vert_{\H^{m}} \Phi \Big( T, \cdots, \Vert u \Vert_{\Ld^{\infty}(0,T;\H^{m+1})},  \Vert  f \Vert_{\Ld^{\infty}(0,T;\mathcal{H}^{m+1}_{r})} \Big),
\end{align*}
and features the same loss of derivative as for $\varrho$ in \eqref{eq:TSgenBaro}. Note also that the equation on $\varrho$ can be directly solved once $f$ and $u$ are given. As in Section \ref{Subsection:EstimateREG}, we consider the regularization $-\pi'(\vartheta ) \mathrm{J}_\eps \nabla_x \vartheta$ of the term $-\nabla_x \pi(\vartheta)$ in the kinetic equation of \eqref{eq:TSe} and introduce the quantity
\begin{multline*}
\textbf{N}_{m,r}(f_\eps,\varrho_\eps,u_\eps, \vartheta_\eps,T):= \Vert f_\eps \Vert_{\Ld^{\infty}\left(0,T;\mathcal{H}^{m-1}_{r}\right)}+\Vert \varrho_\eps \Vert_{\Ld^{\infty}(0,T;\H^{m-1})}\\
+\Vert u_\eps \Vert_{\Ld^{\infty}(0,T; \H^m) \cap \Ld^{2}(0,T;\H^{m+1})}+\Vert \vartheta_\eps \Vert_{\Ld^2(0,T;\H^{m})}.
\end{multline*}
Following the bootstrap procedure we have set for the case of \eqref{eq:TSgenBaro}, we mainly want to control the quantity $\Vert \vartheta_\eps \Vert_{\Ld^2(0,T; \H^m)}$.

Using \eqref{eq:theta} (and dropping the dependency in $\eps$) it is possible to obtain the following equation on $h=\partial_x^{\alpha} \vartheta$ with $\vert \alpha \vert \leq m$ (see Proposition \ref{coro:Facto}): one has 
\begin{align*}
\left( \mathrm{Id}-\frac{\vartheta+ \pi(\vartheta)}{1-\rho_f}\mathrm{K}_G^{\mathrm{free}}\circ\mathrm{J}_{\eps} \right)\Big[\partial_t h + u \cdot \nabla_x h\Big]=\mathcal{R}, \ \ t \in (0,T),
\end{align*}
with $G(t,x,v)=  \pi'( \vartheta(t,x)) \nabla_v f(t,x,v)$ and 
$$
\| \mathcal{R} \|_{\Ld^2(0,T; \Ld^2(\T^d))} \leq \Lambda (T,R, \|h(0)\|_{\H^1(\T^d)}).
$$
Following the arguments of Section \ref{Section:FluidDensity-estimate}, we are thus led to the study of the following pseudodifferential equation 
\begin{align*}
H-\frac{\pi'(\vartheta)(\vartheta+ \pi(\vartheta))}{1-\rho_f}\mathrm{Op}^{\gamma}(a_{f})(\mathrm{J}_{\eps} H)=\mathcal{R}, \ \ \text{on} \ \ (0,T) \times \T^d,
\end{align*}
where $a_f$ is defined in \eqref{def-af}, that is
\begin{align*}
\mathrm{Op}^{\gamma, \eps}(1-\mathbf{P}_{f, \vartheta}^{\mathrm{energy}})(H)=\mathcal{R}.
\end{align*}
In particular, it explains the introduction of the Penrose symbol \eqref{def:PenroseSymbol-energyintern} above, which allows to invert the previous equation on $H$, up to a small remainder.

Some additional arguments have also to be given to treat the last term in \eqref{eq:theta}. Since 
\begin{align*}
\int_{\R^d} \vert v-u \vert^2 f \, \mathrm{d}v = \int_{\R^d} \vert  v \vert^2 f \, \mathrm{d}v + \vert u \vert^2 \rho_f- 2 u \cdot j_f,
\end{align*}
we have to include the treatment of the second order moment in velocity $m_2f(t,x):=\int_{\R^d} \vert  v \vert^2 f(t,x,v) \, \mathrm{d}v$ in the analysis of Section \ref{Section:Kinetic-moments}. In addition to Proposition \ref{coroFInal:D^I:rho-j}, we have
\begin{propo}
For all $|I|\leq m$, we have for any $t \in (0,T)$
\begin{align*}
\partial_x^I  m_2f(t,x)&=p'(\varrho(t,x))\int_0^t  \int_{\R^d}  \vert v \vert^2 \nabla_x  [\mathrm{J}_\eps \partial_x^I \varrho](s, x-(t-s)v) \cdot \nabla_v f(t,x,v) \, \mathrm{d}v \, \mathrm{d}s+R^I[m_2f](t,x),
\end{align*}
where the remainder $R^I[m_2f]$ satisfies
\begin{align*}
 \left\Vert R^I[M_2f] \right\Vert_{\Ld^2(0,T, \H^1_x)} \leq \Lambda(T,R).
\end{align*}
In particular, we have
\begin{align*}
\LRVert{\partial_x^I \int_{\R^d} \vert v-u \vert^2 f \, \mathrm{d}v}_{\Ld^2(0,T; \Ld^2)} \leq \Lambda(T,R).
\end{align*}
\end{propo}
The fairly straightforward adaptation of the analysis of Section \ref{Section:Kinetic-moments} is left to the reader.

\section{Generalization to the inelastic Boltzmann case}\label{Section:AppendixInelasticBoltzmann}
In this section, we consider the case where one takes into account inelastic collisions between particles. This corresponds to the following Vlasov-Boltzmann equation in the coupling (the other equations on $(\varrho,u)$ being unchanged):
\begin{equation}\label{eq:TS-Boltz}
\tag{TS-Coll}
\left\{
      \begin{aligned}
\partial_t f +v \cdot \nabla_x f +{\rm div}_v \big[ f (u-v)-f \nabla_x p(\varrho) \big]&=\mathcal{Q}_{\lambda}(f,f),  \\
\partial_t (\alpha \varrho ) + \mathrm{div}_x (\alpha \varrho u)&=0, \\
\partial_t (\alpha \varrho u) + \mathrm{div}_x(\alpha \varrho u \otimes u)+\alpha\nabla_x p -\Delta_x u - \nabla_x \mathrm{div}_x \, u&=j_f-\rho_f u,
\end{aligned}
    \right.
\end{equation}
where $\mathcal{Q}_{\lambda}(f,f)$ stands for a quadratic collision operator of Boltzmann-type, in a inelastic hard-spheres regime. Here, the fixed parameter $\lambda \in (0,1)$ corresponds to the so-called \textit{restitution coefficient}: if $'v$ and $'v_{\star}$ denote the velocities of two particles before collision, their respective velocities  $v$ and $v_{\star}$ after collision are given by
\begin{equation*}
\left\{
\begin{aligned}
v&= {'v}-\frac{1+\lambda}{2}({'u}\cdot n)n, \\
v_{\star}&= {'v_{\star}}+\frac{1+\lambda}{2}({'u}\cdot n)n,
\end{aligned}
 \right.
\end{equation*}
where ${'u}:={'v}-{'v_{\star}}$ is the relative pre-collision velocity and $n \in \mathbb{S}^{d-1}$ is a unit vector that points from the particle center with velocity $v$ to the particle center with velocity $v_{\star}$ at the impact. Note that $\lambda=1$ corresponds to the standard elastic case.

In this representation, given two distribution functions $f=f(v)$ and $g=g(v)$, we can consider the following expression for Boltzmann collision operator, as a difference of a gain and loss term
\begin{align*}
\mathcal{Q}_{\lambda}(f,g)=\mathcal{Q}_{\lambda}^+(f,g)-\mathcal{Q}_{\lambda}^-(f,g),
\end{align*}
where, setting $u=v-v_{\star}$ and $\widehat{u}=u/ \vert u \vert$, we define
\begin{align}\label{eq:collisionOp}
\begin{split}
\mathcal{Q}_{\lambda}^+(f,g)(v)&=\frac{1}{\lambda^2}\int_{ \R^d \times  \mathbb{S}^{d-1} }  \vert u \cdot n \vert  b(\widehat{u}\cdot n) f({'v})g({'v_{\star}}) \, \mathrm{d}v_{\star} \mathrm{d}n, \\[1mm]
\mathcal{Q}_{\lambda}^-(f,g)(v)&=f(v)\int_{ \R^d \times  \mathbb{S}^{d-1} }   \vert u \cdot n \vert b(\widehat{u}\cdot n)g(v_{\star})  \, \mathrm{d}v_{\star}  \mathrm{d}n.
\end{split}
\end{align}
Here, the function $b \in \Ld^1([-1,1])$ is a given angular collision kernel. For the sake of simplicity, we consider $b\equiv 1$.  Note that in the (truly) inelastic case $\lambda \in (0,1)$, we have
\begin{align*}
\vert v \vert^2+ \vert v_{\star} \vert^2=\vert {'v} \vert^2+ \vert {'v_{\star}} \vert^2-\frac{1-\lambda^2}{2}({'u}\cdot n)^2,
\end{align*}
thus inducing a loss of kinetic energy at each collision, while mass and momentum are conserved. We refer to \cite{Villani-inelastic} (see also the introduction of \cite{ALT}) for more details on this model coming from the theory of \textit{granular media} and which describes a cloud of macroscopic particles of size larger than that usually described by the usual Boltzmann equation with elastic collisions (for so-called molecular gases). To include dissipative effects, inelastic collisions are thus considered. Note that the presence of such a collision operator (with a large parameter $\eps^{-1}$ in front of it) formally leads to a biphasic fluid model when starting from \eqref{eq:TS-Boltz}) (see \cite{DesMathiaud}). We also refer to \cite{oro, Barran, de2009modeles}
for its use in the study of sprays.

\medskip

Our main result (which also includes the elastic case $\lambda=1$)  reads as follows.

\begin{thm}\label{THM:Boltzmann}
Let $\lambda \in (0,1]$. There exist $m_0>0$, depending only on the dimension, such that the following holds for all $m \geq m_0$. Let 
\begin{align*}
e^{\vert v \vert^2}  f^{\mathrm{in}} \in  \mathcal{H}_{0}^{m}, \ \ \varrho^{\mathrm{in}} \in \mathrm{H}^{m+1}, \ \ u^{\mathrm{in}} \in \mathrm{H}^{m},
\end{align*}
such that $(f^{\mathrm{in}}, \varrho^{\mathrm{in}})$ satisfies the $c$-Penrose stability condition $\eqref{cond:Penrose}_c$ (with $c>0$) and 
\begin{align*}
0 &\leq f^{\mathrm{in}} \ \ \ \rho_{f^{\mathrm{in}}}<\Theta<1, \ \ \ 0<\mu \leq \varrho^{\mathrm{in}}, \ \ \ 0<\underline{\theta} \leq (1-\rho_{f^{\mathrm{in}}})\varrho^{\mathrm{in}} \leq \overline{\theta},
\end{align*}
for some fixed constants $\Theta, \mu, \underline{\theta}, \overline{\theta}$. 
Then there exist $T>0$ and a solution $(f, \varrho,u)$ to \eqref{eq:TS-Boltz} with initial condition $(f^{\mathrm{in}}, \varrho^{\mathrm{in}}, u^{\mathrm{in}})$ such that
\begin{align*}
e^{\vert v \vert^2} f \in  \mathscr{C}\left([0,T];\mathcal{H}^{m-1}_{0} \right), \ \ \varrho \in \Ld^2(0,T;\mathrm{H}^m), \ \ u \in \mathscr{C}\left([0,T]; \mathrm{H}^{m} \right)\cap \Ld^2(0,T;\mathrm{H}^{m+1}),
\end{align*}
and with $(f(t), \varrho(t))$ satisfying the $c/2$-Penrose stability condition $\eqref{cond:Penrose}_{c/2}$ for all $t \in [0,T]$. In addition, this solution is unique in this class.
\end{thm}
As seen below, the friction term in the kinetic equation comes in handy in order to treat some of the new terms due to the collision operator. This was remarked already in \cite{Mathiaud}. 

Let us present the main changes that must be considered in our strategy of proof and that are due to the collision operator. It mainly concerns:
\begin{itemize}
\item[$-$] the energy estimates for $f$ from Section \ref{Section:prelimBootstrap};
\item[$-$] the integro-differential system for the derivatives of $f$ from Section \ref{Subsection-Integrodiff}.
\end{itemize}
The rest of Section \ref{Section:prelimBootstrap} and of Section \ref{Section:Kinetic-moments} then remains unchanged, as well as Sections \ref{Section:FluidDensity-estimate}--\ref{Section:END}.

\subsection{New energy estimates for the kinetic part }
Let us focus on the energy estimates for the new kinetic equation
\begin{align}\label{eq:VlasovB-f}
\partial_t f +v \cdot \nabla_x f + {\rm div}_v [f (u-v)-f \nabla_x p(\varrho)f]&=\mathcal{Q}_{\lambda}(f,f).
\end{align}
Following \cite{Mathiaud}, we first define $g(t,x,v)=e^{\vert v \vert^2} f(t,x,v)$ with $f$ solving \eqref{eq:VlasovB-f}.
Setting
$$E^{u, \varrho}=u-\nabla_x p(\varrho),$$
it implies that $g$ satisfies the following modified Vlasov-Boltzmann equation
\begin{align}\label{eq:modifiedVlasovB-g}
\partial_t g +v \cdot \nabla_x g + {\rm div}_v [(E^{u, \varrho}-v)g]-2v \cdot (E^{u, \varrho}-v)g =\Gamma_{\lambda}[g,g],
\end{align}
where for all functions $h_1(v), h_2(v)$, the operator  $\Gamma_{\lambda}[g,g]=\Gamma_{\lambda}^+[g,g]-\Gamma_{\lambda}^-[g,g]$ is defined via
\begin{align*}
\Gamma_{\lambda}^+[h_1,h_2](v)&:=\frac{1}{\lambda^2}\int_{ \R^d \times  \mathbb{S}^{d-1} } e^{-\vert v_{\star} \vert^2-\frac{1-\lambda^2}{2}(({'v}-{'v_{\star}})\cdot n)^2}  \vert u \cdot n \vert   h_1({'v})h_2({'v_{\star}}) \, \mathrm{d}v_{\star} \mathrm{d}n, \\
\Gamma_{\lambda}^-[h_1,h_2](v)&:=h_2(v)\int_{ \R^d \times  \mathbb{S}^{d-1} }  e^{-\vert v_{\star} \vert^2} \vert u \cdot n \vert h_1(v_{\star})  \, \mathrm{d}v_{\star}  \mathrm{d}n.
\end{align*}
Note that the additional term $2v \cdot (E^{u, \varrho}-v)g$ comes from the friction term in \eqref{eq:VlasovB-f}. Using
\begin{align*}
{'u}={'v_{\star}}-{'v}=v_{\star}-v-(1+\lambda)({'u} \cdot n) n = v_{\star}-v+(1+\lambda)\lambda(u \cdot n) n,
\end{align*}
we observe that for all $\lambda \in (0,1)$, there exists a constant $c(\lambda)>0$ (and $c(1)=0)$ such that
\begin{align*}
\Gamma_{\lambda}^+[h_1,h_2]&:=\frac{1}{\lambda^2}\int_{ \R^d \times  \mathbb{S}^{d-1} } e^{-\vert v_{\star} \vert^2-c(\lambda)(({v}-{v_{\star}})\cdot n)^2}  \vert (v-v_{\star})\cdot n \vert   h_1({'v})h_2({'v_{\star}}) \, \mathrm{d}v_{\star} \mathrm{d}n.
\end{align*}
The exponential inside the integral is roughly behaving as $e^{-\vert v_{\star} \vert^2-c(\lambda)\vert v-v_{\star} \vert^2}$. We have the following bilinear estimates on the previous collision operators, where some loss of weights in velocity classically shows up.
\begin{lem}\label{LM:BilinQBoltz}
There exists $s=s(d)>0$ large enough such that for all $\sigma \geq 0$, we have for any smooth nonnegative function $g=g(x,v)$ 
\begin{align*}
\sum_{\vert \alpha \vert+ \vert \beta \vert\leq s }\int_{\T^d} \int_{\R^d} \langle v \rangle^{2\sigma} \partial_x^{\alpha} \partial_v^{\beta}[\Gamma_{\lambda}(g,g)] \partial_x^{\alpha} \partial_v^{\beta} g \, \mathrm{d}x \, \mathrm{d}v \lesssim \LRVert{g}_{\mathcal{H}^{s}_{\sigma}}^2 \LRVert{g}_{\mathcal{H}^{s}_{\sigma+1}}.
\end{align*}
\end{lem}
\begin{proof}
We refer to \cite[Lemma 2.3]{Mathiaud} combined to the (proof of) \cite[Appendix A]{ALT}.
\end{proof}
The key estimate allowing to recover the previous loss of weight then comes from the following lemma bearing on the extra term $2v \cdot (E^{u, \varrho}-v)g$.
\begin{lem}\label{LM:DragBoltz}
Let $s \geq 0$ and $\sigma >0$. For any smooth nonnegative function $g=g(x,v)$ and $\delta \in (0,1)$, we have
\begin{multline*}
\sum_{\vert \alpha \vert+ \vert \beta \vert\leq s }\int_{\T^d} \int_{\R^d} \langle v \rangle^{2\sigma} \partial_x^{\alpha} \partial_v^{\beta} \left[2v \cdot (E^{u, \varrho}-v)g \right] \partial_x^{\alpha} \partial_v^{\beta} g \, \mathrm{d}x \, \mathrm{d}v \\
\lesssim -(1-\delta) \LRVert{g}_{\mathcal{H}^{s}_{\sigma+1}}^2 + \left( 1+ \frac{1}{\delta} \LRVert{E^{u, \varrho}}_{\H^s}^2\right)\LRVert{g}_{\mathcal{H}^{s}_{\sigma}}^2.
\end{multline*}
\end{lem}
\begin{proof}
We refer to \cite[Lemma 2.7]{Mathiaud}.
\end{proof}
Let us show how one can now obtain an \textit{a priori} energy estimate for a solution $g$ to \eqref{eq:modifiedVlasovB-g}, that reads as
\begin{align*}
\mathcal{T}^{u,\varrho}(g)-2v\cdot (E^{u, \varrho}-v)g=\Gamma_{\lambda}(g,g),
\end{align*}
where
$$\mathcal{T}^{u,\varrho}=\partial_t +v \cdot\nabla_x -v \cdot \nabla_v + E^{u,\varrho}(t,x)\cdot \nabla_v-d \mathrm{Id}.
$$
The result is the following.
\begin{lem}
For all $r \geq 0$, $m >3+d/2$, $c>0$ and all $T>0$, for all smooth functions $(f,\varrho,u)$  satisfying
$$
\partial_tg +v \cdot\nabla_x g -v \cdot \nabla_v g + E^{u,\varrho}(t,x)\cdot \nabla_v-2v\cdot (E^{u, \varrho}-v)g-d g=\Gamma_{\lambda}(g,g) \quad \text{on} \quad [0,T],
$$
and  $\varrho \geq c$ on $[0,T]$, there holds, for all $t \in [0,T]$
\begin{align}\label{ineqEnergyE(g)}
\begin{split}
\mathcal{E}_{m, \sigma}[g(t)] &\leq  \Vert g(0) \Vert_{\mathcal{H}^{m}_{\sigma}}^2 \exp \Big[C\Big( (1+ \Vert u \Vert_{\Ld^{\infty}(0,T;\H^{m})}+ \LRVert{\mathcal{E}_{m, \sigma}[g]}_{\Ld^{\infty}(0,T)})T \\
& \qquad \qquad \qquad \qquad \qquad \qquad \qquad+ {\sqrt{T}}  \Lambda\left( \Vert \varrho \Vert_{\Ld^{\infty}(0,T;\H^{m-2})} \Big) \Vert \varrho \Vert_{\Ld^2(0,T;\H^{m+1})} \right) \Big],
\end{split}
\end{align}
for some universal constant $C>0$ and where
\begin{align*}
\mathcal{E}_{m, \sigma}[g(t)]:= \Vert g(t) \Vert_{\mathcal{H}^{m}_{\sigma}}^2+ \frac{1}{4}\int_0^t \Vert g(\tau) \Vert_{\mathcal{H}^{m}_{\sigma+1}}^2 \, \mathrm{d}\tau.
\end{align*}
\end{lem}
\begin{proof}
By Lemma \ref{LM:applyD-kin}, we first have
\begin{multline*}
\mathcal{T}^{u,\varrho}(\partial_x^{\alpha} \partial_v^{\beta} g)=- \sum_{\substack{i=1\\\beta_i \neq 0}}^d  \left(  \partial_x^{\widehat{\alpha}^i} \partial_v^{\overline{\beta}^i} g-\partial_x^{\alpha} \partial_v^{\beta} g\right)  - \left[\partial_x^{\alpha} \partial_v^{\beta}, E^{u,\varrho}(t,x)\cdot \nabla_v\right]g \\
+\partial_x^{\alpha} \partial_v^{\beta} \left[2v \cdot (E^{u, \varrho}-v)g \right]+\partial_x^{\alpha} \partial_v^{\beta} \left[\Gamma_{\lambda}(f,f) \right].
\end{multline*}
The analogue of inequality \eqref{ineq:d/dt:fweight} from Proposition \ref{prop:energy:f} is, thanks to Lemmas \ref{LM:BilinQBoltz}--\ref{LM:DragBoltz}, for all $\delta \in (0,1)$
\begin{align*}
\frac{\mathrm{d}}{\mathrm{d}t}\Vert g(t) \Vert_{\mathcal{H}^{m}_{\sigma}}^2+ (1-\delta) \LRVert{g(t)}_{\mathcal{H}^{m}_{\sigma+1}}^2  & \lesssim \left(1+\Vert E^{u,\varrho}(t) \Vert_{\H^{m}}+\frac{1}{\delta} \LRVert{E^{u, \varrho}(t)}_{\H^m}^2 \right) \Vert g(t)\Vert_{\mathcal{H}^{m}_{\sigma}}^2+\LRVert{g(t)}_{\mathcal{H}^{m}_{\sigma}}^2 \LRVert{g(t)}_{\mathcal{H}^{m}_{\sigma+1}} \\
& \lesssim \left(1+\Vert E^{u,\varrho}(t) \Vert_{\H^{m}}+\frac{1}{\delta} \LRVert{E^{u, \varrho}(t)}_{\H^m}^2 \right) \Vert g(t)\Vert_{\mathcal{H}^{m}_{\sigma}}^2 \\
& \quad +\frac{2}{1-\delta}\LRVert{g(t)}_{\mathcal{H}^{m}_{\sigma}}^4 + \frac{1-\delta}{2}\LRVert{g(t)}_{\mathcal{H}^{m}_{\sigma+1}}^2,
\end{align*}
therefore after absorbing the last term with $\delta=1/2$, we get
\begin{align*}
\frac{\mathrm{d}}{\mathrm{d}t} \mathcal{E}_{m, \sigma}[g(t)]
\lesssim \left(1+ \LRVert{E^{u, \varrho}(t)}_{\H^m}^2 +\mathcal{E}_{m, \sigma}[g(t)] \right)\mathcal{E}_{m, \sigma}[g(t)].
\end{align*}
Concluding as in the proof of Proposition \ref{prop:energy:f}, we finally obtain the result.
\end{proof}
The bootstrap argument from Section \ref{Subsection:EstimateREG} is then performed with the quantity
\begin{align*}
\mathscr{N}_{m,r}(g,\varrho,u,T):=  \LRVert{\mathcal{E}_{m-1, \sigma}[g]}_{\Ld^{\infty}(0,T)} +\Vert \varrho \Vert_{\Ld^2(0,T;\H^{m})}+\Vert u \Vert_{\Ld^{\infty}(0,T; \H^m) \cap \Ld^{2}(0,T;\H^{m+1})},
\end{align*}
instead of $\mathcal{N}_{m,r}(f,\varrho,u,T)$, by considering the same regularization $\mathrm{J}_\eps \nabla_x \varrho$ in the equation \eqref{eq:modifiedVlasovB-g}. Taking into account \eqref{ineqEnergyE(g)} and the quantity $\mathscr{N}_{m,r}(g,\varrho,u,T)$, the content of Section \ref{Section:prelimBootstrap} can be modified accordingly. Concerning the local in time existence for the regularized system from Section \ref{Appendix:LWPeps},  we just modify the kinetic part of the scheme of approximation:
\begin{equation*}
\left\{
      \begin{aligned}
      &\partial_t g^{n+1} +v \cdot \nabla_x g^{n+1}  +\mathrm{div}_v\left[(E^n(t,x)-v) g^{n+1} \right]-2v \cdot (E^{u, \varrho}_{\mathrm{reg}, \eps}-v)g^{n+1} \\
      & \qquad \qquad \qquad \qquad  \qquad \qquad \qquad \qquad  \qquad \qquad  \qquad \qquad   =\Gamma_{\lambda}^+[g^n,g^n]-\Gamma_{\lambda}^-[g^n,g^{n+1}], \\
      &{f^{n+1}}_{\mid t=0}=f^{in}.
\end{aligned}
    \right.
\end{equation*}
where $\eps>0$ is given. We refer to \cite{Mathiaud} for more details. 

\subsection{Equation on the augmented variable $\mathcal{F}$}
Let us highlight the main changes that arise in the end of Subsection \ref{Subsection-Integrodiff}, more precisely in Definition \ref{def:AugmentedVar+Coupling}. Recall that the goal is to consider a new unknown $\mathcal{F}=\left(\partial_x^{K} \partial_v^{K}f \right)_{\vert K \vert+\vert L \vert \in \lbrace m-1,m \rbrace }$. Here, one has to consider a new coupling matrix $\mathbb{M}=\left( \mathbb{M}_{(I,J),(K,L)} \right)$ which takes into account the collision operator $\mathcal{Q}$, defined by
\begin{align*}
\mathbb{M}_{(I,J),(K,L)}:=\mathcal{M}_{(I,J),(K,L)}+ \mathcal{M}_{(I,J),(K,L)}^{\mathcal{Q}},
\end{align*}
with $\vert I \vert+\vert J \vert, \vert K \vert+\vert L \vert \in \lbrace m-1,m \rbrace$ and where
\begin{itemize}
\item[$-$] $\mathcal{M}_{(I,J),(K,L)} \in \R$ stands for the former terms of the coupling matrix already appearing in Definition \ref{def:AugmentedVar+Coupling};
\item[$-$] $\mathcal{M}_{(I,J),(K,L)}^{\mathcal{Q}}$ is an new operator term coming from the collision operator and defined by the relation
\begin{align*}
\partial_x^I \partial_v^J \mathcal{Q}_{\lambda}(f,f)
&= \sum_{\substack{0 \leq \alpha \leq  I \\ 0 \leq \beta\leq  J}} \binom{I}{\alpha} \binom{J}{\beta}  \mathcal{Q}_{\lambda}(\partial_x^{\alpha} \partial_v^{\beta}f,\partial_x^{I-\alpha} \partial_v^{J-\beta} f)\\
&=\sum_{\vert K \vert+\vert L \vert \in \lbrace m-1,m \rbrace} \mathcal{M}_{(I,J),(K,L)}^{\mathcal{Q}}\left[\partial_x^K \partial_v^L f \right],
\end{align*}
that is 
\begin{align*}
\mathcal{M}_{(I,J),(K,L)}^{\mathcal{Q}}(\bullet):=\sum_{\substack{0 \leq \alpha \leq I \\ 0 \leq \beta \leq J \\ \vert \alpha \vert+ \vert \beta \vert \leq 1}}\binom{I}{\alpha} \binom{J}{\beta} \mathbf{1}_{\substack{K=I-\alpha \\ L=J-\beta}} \mathcal{Q}_{\lambda} \Big(\partial_x^{\alpha} \partial_v^{\beta}f, \bullet \Big).
\end{align*}
Hence, $\mathcal{M}^{\mathcal{Q}}$ is a matrix with operator coefficients, acting on $\mathcal{F}=\left(\partial_x^{K} \partial_v^{K}f \right)_{\vert K \vert+\vert L \vert \in \lbrace m-1,m \rbrace }$.
\end{itemize} 
Using the same notations as in Section \ref{Subsection-Integrodiff}, we obtain the following equation on $\mathcal{F}=\left(\partial_x^{K} \partial_v^{K}f \right)_{\vert K \vert+\vert L \vert \in \lbrace m-1,m \rbrace }$ (see \eqref{eq:augmentedVarF})
\begin{align*}
\mathcal{T}_{\reg,\eps}^{u,\varrho} \mathcal{F}+\mathbb{M} \mathcal{F}+\mathcal{L}=-\mathcal{R}_0-\mathcal{R}_1.
\end{align*}
After the composition by $ (t,x,v) \mapsto (t,\mathrm{X}^{t;0}(x,v),\mathrm{V}^{t;0}(x,v))$ where $\mathrm{Z}=(\X, \V)$ is the solution to \eqref{EDO-charac}, there holds the equation
\begin{align*}
\partial_t \widetilde{\mathcal{F}}+\widetilde{\mathbb{M}} \widetilde{\mathcal{F}}+\widetilde{\mathcal{L}}=d\widetilde{\mathcal{F}} -\widetilde{\mathcal{R}}_0-\widetilde{\mathcal{R}}_1,
\end{align*}
with $\widetilde{g}(t,x,v)=g(t,\X^{t;0}(x,v),\V^{t;0}(x,v))$. We can still consider the resolvant associated to the previous operator $\mathbb{M}-d \mathrm{Id}$, that is the solution $s \mapsto \mathfrak{N}^{s,t}(x,v)$ of 
\begin{equation*}
\left\{
      \begin{aligned}
        \partial_s \mathfrak{N}^{s;t}+\left[\mathbb{M} \circ \mathrm{Z}^{s;0}-d \mathrm{Id} \right] \mathfrak{N}^{s;t}&=0,\\
	\mathfrak{N}^{t;t}&= \mathrm{Id},
      \end{aligned}
    \right.
\end{equation*}
whose existence and uniqueness is still provided by the Cauchy-Lipschitz theorem,
Hence, this shows that the contribution of the collision operator $\mathcal{Q}$ can be handled by the modified operator $\mathbb{M}$. The strategy of proof which is performed in the rest of Section \ref{Section:Kinetic-moments} applies \emph{mutatis mutandis}.

\section{Generalization to the density-dependent drag case}\label{Section:AppendixDragTerm}
In this section, we show how one can deal with the case of density-dependent drag in the force acting on the particles, that is, with the notation of the introduction, when the force in the kinetic equation is
\begin{align*}
\Gamma(t,x,v)= \varrho(t,x)(u(t,x)-v)-\nabla_x [p(\varrho)](t,x).
\end{align*}
This additional factor also appears in the feedback of particles on the fluid, that is in the ource term in the Navier-Stokes equations. We are  thus led to consider the following system:
\begin{equation}\label{eq:DragCase}
\left\{
      \begin{aligned}
\partial_t f +v \cdot \nabla_x f   -\nabla_x p \cdot \nabla_v f + {\rm div}_v [f \varrho(u-v)-f\nabla_x p(\varrho)]&=0, \\[2mm]
\partial_t (\alpha \varrho ) + \mathrm{div}_x (\alpha \varrho u)&=0,  \\[2mm]
\partial_t (\alpha \varrho u) + \mathrm{div}_x(\alpha \varrho u \otimes u) +\alpha\nabla_x p -\Delta_x u - \nabla_x \mathrm{div}_x \, u &=\varrho(j_f-\rho_f u), \\[2mm]
 \alpha&=1-\rho_f.
\end{aligned}
    \right.
\end{equation}
Our main result of local well-posedness is the following.
\begin{thm}\label{THM:Drag}
Consider the same assumptions of Theorem \ref{THM:main}. Assume also that $f^{\mathrm{in}}$ is compactly supported in velocity. Then the conclusion of Theorem \ref{THM:main} holds for the density-dependent drag case of System \eqref{eq:DragCase}.
\end{thm}
To fix notations, let us assume that
\begin{align}\label{assump-fcpt-Min}
\mathrm{supp}\, f^{\mathrm{in}} \subset \T^d \times \B(0,M^{\mathrm{in}}), \ \ M^{\mathrm{in}}>0.
\end{align}
In what follows, we shall only present the main modifications which have to be added to the strategy used in this article. 
\subsection{Modification of the energy estimates and of the bootstrap argument} Our first goal is to adapt the energy estimates of Section \ref{Section:prelimBootstrap} and the bootstrap procedure to the density-dependent drag case. The main change comes from the estimate for $f$ from Proposition \ref{prop:energy:f}.

Indeed, let us set
$$\mathcal{T}^{u,\varrho}_{\mathrm{drag}}=\partial_t +v \cdot\nabla_x -\varrho v \cdot \nabla_v + E^{u,\varrho}(t,x)\cdot \nabla_v-d\varrho \mathrm{Id},$$ 
where $E^{u,\varrho}_{\mathrm{drag}}:=\varrho u - \nabla_x p(\varrho) $. Observe that because of the term $\varrho v \cdot \nabla_v$ (coming from friction), a growth in velocity can occur in the analysis, that is if $f$ is controlled in $\mathcal{H}^m_r$, then this term would \textit{a priori} require a control in  $\mathcal{H}^m_{r+1}$. This explains the additional assumption of compact support in velocity in the next proposition. We mention though that the use of exponential-weighted norms in velocity could relax this assumption (see the work \cite{Asano} on the Vlasov-Maxwell equations, and also \cite{CJdragBGK} in the context of fluid-kinetic equations).

\begin{propo}\label{prop:energy:f-DRAGCASE}
For all $r \geq 0$, $m >3+d/2$, $c>0$ and all $T>0$, for all smooth functions $(f,\varrho,u)$ with $f$ having a compact support in velocity: 
\begin{align*}
\forall t \in [0,T], \ \ \mathrm{supp}\, f(t) \subset \T^d \times \B(0,M(t)),
\end{align*}
for some $M \in \Ld^{\infty}(0,T)$, satisfying
$$
\mathcal{T}_{\mathrm{drag}}^{u,\varrho} (f) =0 \quad \text{on  } [0,T],
$$
and $\varrho \geq c$ on $[0,T]$, the following holds. For all $t \in [0,T]$, we have
\begin{align*}
\Vert f(t) \Vert_{\mathcal{H}^{m}_{r}}^2& \leq \Vert f(0) \Vert_{\mathcal{H}^{m}_{r}}^2 \exp \Big[C(1+\LRVert{M}_{\Ld^{\infty}(0,T)})\Big( T+\sqrt{T}\Vert u \Vert_{\Ld^{\infty}(0,T;\H^{m})}\Vert \varrho \Vert_{\Ld^{2}(0,T;\H^{m})}\\
& \qquad \qquad \qquad \qquad \qquad \qquad \qquad \qquad \qquad  + \sqrt{T} \Lambda\left( \Vert \varrho \Vert_{\Ld^{\infty}(0,T;\H^{m-2})} \right) \Vert \varrho \Vert_{\Ld^2(0,T;\H^{m+1})} \Big) \Big],
\end{align*}
\end{propo}

\begin{proof}
Let us suppose that there exists $M \in \Ld^{\infty}(0,T)$ sucht that
\begin{align*}
\forall t \in [0,T], \ \  \mathrm{supp}\, f(t) \subset \T^d \times \B(0,M(t)).
\end{align*}
Since $\mathcal{T}_{\mathrm{drag}}^{u,\varrho} (f) =0$, we have by Lemma \ref{LM:applyD-kin}
\begin{align*}
\mathcal{T}_{\mathrm{drag}}^{u,\varrho}(\partial_x^{\alpha} \partial_v^{\beta} f)&=-\sum_{\substack{i=1\\\beta_i \neq 0}}^d  \partial_x^{\widehat{\alpha}^i} \partial_v^{\widehat{\beta}^i} f +\left[\partial_x^{\alpha} \partial_v^{\beta}, \varrho(t,x) v \cdot \nabla_v\right]-\left[\partial_x^{\alpha} \partial_v^{\beta},E_{\mathrm{drag}}^{u,\varrho}(t,x)\cdot \nabla_v\right]f +d\left[\partial_x^{\alpha} \partial_v^{\beta}, \varrho \mathrm{Id} \right] f,
\end{align*}
for all $\alpha, \beta \in \N^d$. We now take the scalar product of this equality with $(1+\vert v \vert^2)^{r}\partial_x^{\alpha} \partial_v^{\beta} f$, sum for all $\vert \alpha \vert + \vert \beta \vert \leq m$ and then integrate on $\T^d \times \R^d$. For the left-hand side, we have as in Proposition \ref{prop:energy:f}
\begin{multline*}
\sum_{\vert \alpha \vert + \vert \beta \vert \leq m}\int_{\T^d \times \R^d}  (1+\vert v \vert^2)^{r}\mathcal{T}_{\mathrm{drag}}^{u,\varrho}(\partial_x^{\alpha} \partial_v^{\beta} f)\partial_x^{\alpha} \partial_v^{\beta} f  \\
= \frac{1}{2}\frac{\mathrm{d}}{\mathrm{d}t}\Vert f(t) \Vert_{\mathcal{H}^{m}_{r}}^2-\sum_{\vert \alpha \vert + \vert \beta \vert \leq m}\int_{\T^d \times \R^d}  \nabla_v(1+\vert v \vert^2)^{r} \cdot (E_{\mathrm{drag}}^{u,\varrho}-\varrho v) \frac{\vert \partial_x^{\alpha} \partial_v^{\beta} f \vert^2}{2}, 
\end{multline*}
the last term satisfying
\begin{align*}
\sum_{\vert \alpha \vert + \vert \beta \vert \leq m}\int_{\T^d \times \R^d}  \nabla_v(1+\vert v \vert^2)^{r} \cdot (E_{\mathrm{drag}}^{u,\varrho}-\varrho v) \frac{\vert \partial_x^{\alpha} \partial_v^{\beta} f \vert^2}{2} \leq (\LRVert{\varrho(t)}_{\Ld^{\infty}}+\Vert E_{\mathrm{drag}}^{u,\varrho}(t) \Vert_{\Ld^{\infty}})\Vert f(t)\Vert_{\mathcal{H}^{m}_{r}}^2.
\end{align*}
We now look at the four terms of the right-hand side. For the first, third and fourth ones, we proceed as in the proof of Proposition \ref{prop:energy:f} (with a variant of \eqref{Sob:estim3}) and get
\begin{multline*}
\sum_{\vert \alpha \vert + \vert \beta \vert \leq m}\int_{\T^d \times \R^d}  (1+\vert v \vert^2)^{r} \Bigg(-\sum_{\substack{i=1\\\beta_i \neq 0}}^d   \partial_x^{\widehat{\alpha}^i} \partial_v^{\widehat{\beta}^i} f+\left[\partial_x^{\alpha} \partial_v^{\beta},E_{\mathrm{drag}}^{u,\varrho}(t,x)\cdot \nabla_v+ d\varrho \mathrm{Id} \right]f \Bigg) \partial_x^{\alpha} \partial_v^{\beta} f \\
\lesssim (1+\Vert E_{\mathrm{drag}}^{u,\varrho}(t) \Vert_{\H^{m}}+ \Vert \varrho(t) \Vert_{\H^{m}})\Vert f(t)\Vert_{\mathcal{H}^{m}_{r}}^2.
\end{multline*}
The treatment of the third term requires the use of the compact support in velocity of $f$. Invoking the inequality \eqref{Sob:estim4}, we have
\begin{align*}
\sum_{\vert \alpha \vert + \vert \beta \vert \leq m}\int_{\T^d \times \R^d}  (1+\vert v \vert^2)^{r}\left[\partial_x^{\alpha} \partial_v^{\beta}, \varrho(t,x) v \cdot \nabla_v\right]\partial_x^{\alpha} \partial_v^{\beta} f    \lesssim  (1+M(t)) \Vert \varrho(t) \Vert_{\H^{m}} \Vert f(t)\Vert_{\mathcal{H}^{m}_{r}}^2.
\end{align*}
All in all, we obtain
\begin{align}\label{ineq:d/dt:fweight-DRAGCASE}
\frac{\mathrm{d}}{\mathrm{d}t}\Vert f(t) \Vert_{\mathcal{H}^{m}_{r}}^2  \lesssim (1+M(t))(1+\Vert \varrho(t) \Vert_{\H^{m}}+\Vert E^{u,\varrho}(t) \Vert_{\H^{m}} ) \Vert f(t)\Vert_{\mathcal{H}^{m}_{r}}^2,
\end{align}
if $m > d/2$. As in the proof of Proposition \ref{prop:energy:f}, an by Sobolev embedding, we have 
\begin{align*}
\Vert  E^{u,\varrho}(t) \Vert_{\H^{m}}  \lesssim \Vert \varrho (t) \Vert_{\H^{m}}\Vert u(t) \Vert_{\H^{m}} 
+ \Lambda\left(\Vert \varrho(t) \Vert_{\H^{m-2}}  \right)   \Vert \varrho (t) \Vert_{\H^{m+1}}.
\end{align*}
By integrating in time the inequality \eqref{ineq:d/dt:fweight-DRAGCASE}, we get
\begin{multline*}
\Vert f(t) \Vert_{\mathcal{H}^{m}_{r}}^2 \leq \Vert f(0) \Vert_{\mathcal{H}^{m}_{r}}^2 \\
+ C\int_0^t (1+M(s))\left(1+\Vert \varrho(s) \Vert_{\H^{m}}  \Vert u(s) \Vert_{\H^{m}} +  \Lambda\left(\Vert \varrho \Vert_{\Ld^{\infty}(0,T; \H^{m-2})} \right)\Vert \varrho(s) \Vert_{\H^{m+1}} \right) \Vert f(s)\Vert_{\mathcal{H}^{m}_{r}}^2 \, \mathrm{d}s,
\end{multline*}
for all $t \in [0,T)$ and for some constant $C>0$ independent of $\eps$. Using the Cauchy-Schwarz inequality and the Grönwall's inequality, this implies for all $t \in  [0,T)$
\begin{align*}
\Vert f(t) \Vert_{\mathcal{H}^{m}_{r}}^2& \leq \Vert f(0) \Vert_{\mathcal{H}^{m}_{r}}^2 \exp \Big[C(1+\LRVert{M}_{\Ld^{\infty}(0,T)})\Big( T+\sqrt{T}\Vert u \Vert_{\Ld^{\infty}(0,T;\H^{m})})\Vert \varrho \Vert_{\Ld^{2}(0,T;\H^{m})})\\
& \qquad \qquad \qquad \qquad \qquad \qquad \qquad \qquad \qquad  + \sqrt{T} \Lambda\left( \Vert \varrho \Vert_{\Ld^{\infty}(0,T;\H^{m-2})} \right) \Vert \varrho \Vert_{\Ld^2(0,T;\H^{m+1})} \Big) \Big],
\end{align*}
and this concludes the proof.
\end{proof}
The proof of the other estimates from Section \ref{Section:prelimBootstrap} is mainly unchanged and details are left to the reader. We need to adapt the bootstrap procedure, by taking into account the need of a compact support in velocity. 
We thus define the following modified condition. 
\begin{defi}
Let $T>0$. For any nonnegative functions $f(t,x,v)$ and $\varrho(t,x)$ on $[0,T]$, we define the property 
\begin{align}\label{def:Bound-HAUTBAS-Compactsupp}
\tag{B$^{\mu, \theta}_{\Theta, M^{\mathrm{in}}}(T)$}
\forall t \in [0,T], \
\left\{
\begin{aligned}
\rho_{f}(t)\leq \frac{\Theta+1}{2}, \ \ \ \frac{\mu}{2} \leq \varrho(t), \ \ \ \frac{\underline{\theta}}{2} \leq (1-\rho_{f}(t))\varrho(t) \leq 2\overline{\theta}, \\[2mm]
\mathrm{supp}\, f(t) \subset \T^d \times \B(0,1+M^{\mathrm{in}}),
\end{aligned}
\right.
\end{align}
where $\Theta, \mu, \underline{\theta}, \overline{\theta}$ are given in the statement of Theorem \ref{THM:main} and where $M^{\mathrm{in}}$  has been introduced in \eqref{assump-fcpt-Min}.
\end{defi}
If $T^{\ast}_\eps>0$ is the maximal time of existence to the system \eqref{S_eps} (with density-dependent drag term), we introduce the following time for all $\eps>0$:
\begin{align}
T_\eps=T_\eps(R):=\sup \big\lbrace T \in [0,T^{\ast}_\eps[, \ \  \mathcal{N}_{m,r}(f_\eps,\varrho_\eps,u_\eps,T) \leq R \ \ \text{and} \ \ \eqref{def:Bound-HAUTBAS-Compactsupp} \ \ \text{holds}  \big\rbrace,
\end{align}
where $R>0$ has to be chosen large enough and independent of $\eps$. Here, the quantity $\mathcal{N}_{m,r}(f_\eps,\varrho_\eps,u_\eps,T)$ is exactly the same as in Definition \ref{def:N}.

\subsection{Modification in the straightening change of variable} As a matter of fact, the main difference in the analysis appears in the part related to characteristics, namely Section \ref{Section:Lagrangian}. Our purpose here is to explain how to modify the arguments of Section \ref{Section:Lagrangian} about the  straightening change of variable in velocity, that is Lemma \ref{straight:velocity}. It turns out that in the density-dependent drag case, obtaining a suitable diffeomorphism $\psi^{s;t}_x$ is not as straightforward as in Lemma \ref{straight:velocity}. Indeed, again because of the term $-\varrho v \cdot \nabla_v f$, there could be a growth in velocity in the dynamics which prevents our proof of Lemma \ref{straight:velocity}, which is based on a perturbative approach, to directly hold. This is where the assumption compact support in velocity appears to be crucial; we do not know whether it is possible to replace it here by an exponential moment assumption.


With the notations of Section \ref{Section:Lagrangian}, we will actually directly straighten the total kinetic operator
$$\mathcal{T}_{\mathrm{drag},\mathrm{F}}=\partial_t +v \cdot\nabla_x -\varrho (t,x)v \cdot \nabla_v + \mathrm{F}(t,x)\cdot \nabla_v-d \mathrm{Id},$$ 
into the free-transport operator
$$\mathcal{T}^{\mathrm{free}}=\partial_t +v \cdot\nabla_x.$$
For $(x,v) \in \T^d \times \B(0,1+\M^{\mathrm{in}})$ and $t \in [0,T]$, we consider the solution $s \mapsto (\X_{\mathrm{drag}}^{s;t},\V_{\mathrm{drag}}^{s;t})(x,v)$ to the following system of differential equations
\begin{equation*}
\left\{
\begin{aligned}
\frac{\mathrm{d}}{\mathrm{d}s} \X_{\mathrm{drag}}^{s;t} &= \V_{\mathrm{drag}}^{s;t}, \ \ \X_{\mathrm{drag}}^{t;t}(x,v)=x, \\
\frac{\mathrm{d}}{\mathrm{d}s} \V_{\mathrm{drag}}^{s;t}  &= -\varrho(s,\X_{\mathrm{drag}}^{s;t})  \V_{\mathrm{drag}}^{s;t} + \mathrm{F}(s,\X_{\mathrm{drag}}^{s;t}), \ \  \V_{\mathrm{drag}}^{t;t}(x,v)=v.\\
\end{aligned}
\right.
\end{equation*}

\begin{lem}\label{LM:straightDragDep}
Let $T>0$ and $k \geq 1$. Let $\mathrm{F} \in  \Ld^2(0,T; \W^{k, \infty}(\T^d)) $ be a vector field such that 
\begin{align*}
\Vert \mathrm{F} \Vert_{\Ld^2(0,T; \W^{k, \infty}(\T^d))} \leq \Lambda(T,\mathrm{R}),
\end{align*}
for some $\mathrm{R}>0$.
There exists $\overline{T}(\mathrm{R})>0$ such that for all $x \in \T^d$ and $s,t \in [0, \min(\overline{T}(\mathrm{R}),T)]$, there exists a diffeomorphism $\psi_{s,t}(x, \cdot) : \B(0,1+\M^{\mathrm{in}}) \rightarrow \B(0,1+\M^{\mathrm{in}})$ satisfying for all $v \in \B(0,1+\M^{\mathrm{in}})$
\begin{align*}
\mathrm{X}^{s;t}_{\mathrm{drag}}\left(x,\psi_{s,t}(x, v) \right)=x-(t-s)v,
\end{align*}
which furthermore verifies the estimates 
\begin{align*}
 \frac{1}{C} \leq \det \left(\mathrm{D}_v  \psi_{s,t}(x,v) \right) &\leq C, \\
\underset{s,t \in [0,T]}{\sup} \, \left\Vert \partial_{x,v}^{\alpha} \left(\psi_{s,t}(x,v)-v \right) \right\Vert_{\Ld^{\infty}(\T^d \times \R^d)} &\leq \varphi(T) \Lambda(T,\mathrm{R}), \ \ \vert \alpha \vert \leq k, \\
\underset{s,t \in [0,T]}{\sup} \, \left\Vert \partial_{x,v}^{\beta} \partial_s \psi_{s,t}(x,v) \right\Vert_{\Ld^{\infty}(\T^d \times \R^d)} &\leq \varphi(T) \Lambda(T,\mathrm{R}), \ \ \vert \beta \vert \leq k-1,
\end{align*}
for some $C>0$ and some nondecreasing continuous function $\varphi : \R^+ \rightarrow \R^+$ vanishing at $0$.
\end{lem}
\begin{proof}
Let us sketch the proof. Dropping the $(x,v)$ dependency in the trajectories, we have
\begin{align*}
\V_{\mathrm{drag}}^{s;t} =  \exp \left(\int_s^t  \varrho(\tau,\X_{\mathrm{drag}}^{\tau;t}) \, \mathrm{d} \tau  \right) v -  \int_s^t \exp \left( \int_s^\tau  \varrho(\tau',\X_{\mathrm{drag}}^{\tau';t}) \, \mathrm{d} \tau'  \right) \mathrm{F} (\tau,\X_{\mathrm{drag}}^{\tau;t}) \, \mathrm{d} \tau,
\end{align*}
from which we deduce
\begin{align*}
 \X_{\mathrm{drag}}^{s,t} &=x - \left(\int_s^t \exp \left(\int_{s'}^t  \varrho(\tau,\X_{\mathrm{drag}}^{\tau;t}) \, \mathrm{d} \tau  \right)  \, \mathrm{d} s'\right) v  +  \int_s^t \int_{s'}^t \exp \left( \int_{s'}^\tau  \varrho(\tau',\X_{\mathrm{drag}}^{\tau';t}) \, \mathrm{d} \tau'  \right) \mathrm{F}(\tau,\X_{\mathrm{drag}}^{\tau;t}) \, \mathrm{d} \tau  \, \mathrm{d} s' \\
 &=x - (t-s)\Bigg[ v  + \left(\frac{1}{t-s}\int_s^t \exp \left(\int_{s'}^t  \varrho(\tau,\X_{\mathrm{drag}}^{\tau;t}) \, \mathrm{d} \tau  \right)  \, \mathrm{d} s'- 1\right)  v  \\
 &\qquad \qquad \qquad \qquad \qquad \qquad - \frac{1}{t-s}  \int_s^t \int_{s'}^t \exp \left( \int_{s'}^\tau  \varrho(\tau',\X_{\mathrm{drag}}^{\tau';t}) \, \mathrm{d} \tau'  \right)   \mathrm{F} (\tau,\X_{\mathrm{drag}}^{\tau;t}) \, \mathrm{d} \tau  \, \mathrm{d} s'  \Bigg].
\end{align*}
Taking one derivative in velocity, we observe that because of the term 
$$\left(\frac{1}{t-s}\int_s^t \exp \left(\int_{s'}^t  \varrho(\tau,\X_{\mathrm{drag}}^{\tau;t}) \, \mathrm{d} \tau  \right)  \, \mathrm{d} s'- 1\right)  v,$$
we obtain a term where the derivative is not falling on the $v$ factor. Hence, the latter is \textit{a priori} unbounded, inducing a potential linear growth in velocity of the derivative. But since we restrict to  $v \in \B(0,1+\M^{\mathrm{in}})$, we can however deduce a rough bound on $\na_v  \X_{\mathrm{drag}}^{s;t} $ by Grönwall's inequality, which takes the form
\begin{equation*}
|\na_v  \X_{\mathrm{drag}}^{s;t} | \lesssim  \exp(T \Lambda(T,R) (|v|+1)) T \Lambda(T,R) \lesssim \exp(T \Lambda(T,R) (2+M^{\mathrm{in}})) T \Lambda(T,R).
\end{equation*}
We can now check that for $T $ small enough, for all $0\leq s,t \leq T$, the map 
\begin{multline*}
v \mapsto v  + \left(\frac{1}{t-s}\int_s^t \exp \left(\int_{s'}^t  \varrho(\tau,\X_{\mathrm{drag}}^{\tau;t}) \, \mathrm{d} \tau  \right)  \, \mathrm{d} s'- 1\right)  v  \\
 - \frac{1}{t-s}  \int_s^t \int_{s'}^t \exp \left( \int_{s'}^\tau  \varrho(\tau',\X_{\mathrm{drag}}^{\tau';t}) \, \mathrm{d} \tau'  \right)   E^{u,\varrho}_{\mathrm{drag}}(\tau,\X_{\mathrm{drag}}^{\tau;t}) \, \mathrm{d} \tau  \, \mathrm{d} s' 
\end{multline*}
is a $\mathscr{C}^1$ diffeomorphism from $\B(0,1+\M^{\mathrm{in}})$ onto its image, since it is a small Lipschitz perturbation of the identity map. Details are left to the reader. 
\end{proof}

As a result, for small times, we can directly come down to the case of free-transport case.

\subsection{Modifications in remainder terms and conclusion of the bootstrap}
To conclude, let us point out the last main modifications which have to be made to conclude the bootstrap argument. 

In Section \ref{Section:Kinetic-moments}, one shall be careful when handling the remainder terms $\mathscr{R}_{I}^{\mathrm{Diff}}, \mathscr{R}_{I,1}^{\mathrm{Duha}}, \mathscr{R}_{I,2}^{\mathrm{Duha}}, \mathscr{R}_I$ in $\Ld^2(0,T;\H^1)$ because they now involve some terms with at most $m+1$ derivatives on $\varrho u$ stemming from the force field  $E^{u,\varrho}_{\mathrm{drag}}$. It was somehow harmless in the linear-drag case because the corresponding terms involved only $u$, which has an additional regularity provided by the Navier-Stokes equation, namely a control in $\Ld^{\infty}(0,T; \H^m) \cap \Ld^2(0,T; \H^{m+1})$ (see the terms $\mathbf{S_2}$, $\mathbf{S_5}$, $\mathbf{S_{12}}$, $\mathbf{S_{16}}$). Since we only have a control of $\varrho$ in $\Ld^2(0,T; \H^{m})$, this requires a additional argument. 

The main idea is to rely on the same type of decomposition as in Lemma \ref{LM:decompoForce-gradient} combined with the use of the smoothing estimates from Section \ref{Section:Averag-lemma}. The expression $\partial_x^K(\varrho u)$ ($K \in \N^d$) will involve terms or sum of terms 
\begin{itemize}
\item of the form
$$\varrho\partial_x^K u \ \ \text{or} \  \ \partial_x^{\overline{K}^i-\beta} u \nabla_x (\partial_x^{\beta} \varrho ) \cdot e_i,$$
with $i=1, \cdots, d$ and $\vert \beta \vert =0, \cdots, \vert \lfloor \frac{\vert K \vert-1}{2} \rfloor$. They are treated using $\Ld^{\infty}$ bounds on $\varrho$ (with Sobolev embeddings) and $\Ld^2$ bounds on $u$;
\item of the form 
$$\nabla_x (\partial_x^{\beta} \varrho ) \cdot e_i \partial_x^{\overline{K}^i-\beta}u,$$
with $i=1, \cdots, d$ and $\vert \beta \vert = \lfloor \frac{\vert K \vert-1}{2} \rfloor+1, \cdots, \vert K \vert-1$ 
These terms are adressed thanks to Proposition \ref{propo:AveragStandard}. We refer to Section \ref{Section:Kinetic-moments} for the same kind of procedure for the estimates on remainders via the smoothing estimates of Section \ref{Section:Averag-lemma}.
\end{itemize}
Let us eventually briefly conclude by showing how the bootstrap procedure ends, when one considers the condition \eqref{def:Bound-HAUTBAS-Compactsupp}. Namely, we have to show how to control the compact support in velocity for $f$, for short times. We rely on the Lagrangian structure of the equation which implies a finite propagation in time of the support of $f$ in velocity, namely 
\begin{align*}
\forall t>0, \ \ \mathrm{supp} \, f(t) \subset \T^d \times \V^{t;0}_{\mathrm{drag}}( \mathrm{supp} \, f^{\mathrm{in}}).
\end{align*}
Since 
\begin{align*}
\V_{\mathrm{drag}}^{t;0} =  \exp \left(-\int_0^t  \varrho(\tau,\X_{\mathrm{drag}}^{\tau;0}) \, \mathrm{d} \tau  \right) v +  \int_0^t \exp \left(- \int_{\tau}^t  \varrho(\tau',\X_{\mathrm{drag}}^{\tau';0}) \, \mathrm{d} \tau'  \right) E^{u,\varrho}_{\mathrm{drag}}(\tau,\X_{\mathrm{drag}}^{\tau;0}) \, \mathrm{d} \tau,
\end{align*}
the same kind of estimates on the trajectories as those performed in the proof of Lemma \ref{LM:straightDragDep} give for all $(x,v) \in  \T^d \times \B_v(0,M^{\mathrm{in}})$
\begin{align*}
\vert \V_{\mathrm{drag}}^{t;0} \vert \leq  \vert v \vert + T \Lambda(T,R)  \leq  M^{\mathrm{in}} + T \Lambda(T,R).
\end{align*}
As a consequence, choosing $ T$ small enough (once $R$ is given) is sufficient to control the size of the support of $f(t)$ for short times.

\appendix

\section{Some classical (para-)differential inequalities on $\T^d$ and $\T^d \times \R^d$}\label{Section:AppendixDIFF}
We recall and state several classic inequalities of (para-)differential type. First, we have the following tame estimate for commutators (see \cite[Lemma 3.4]{MB}).
\begin{propo}\label{CommutSOB-KlainBerto}
Let $s \geq 1$. There exists $C_s>0$ such that for any functions $g,E \in \H^s \cap \Ld^{\infty}$, we have
\begin{align*}
\forall \vert \alpha \vert \leq s, \ \ \left\Vert [\partial^{\alpha}_x, E]g \right\Vert_{\Ld^2} \leq C_s \left( \Vert \nabla_x E \Vert_{\Ld^{\infty}} \Vert  g \Vert_{\H^{s-1}} + \Vert E \Vert_{\H^s} \Vert g \Vert_{\Ld^{\infty}} \right).
\end{align*}
\end{propo}

The following result is about tame estimates in Sobolev spaces (see e.g. \cite[Corollary 2.86]{BCD}).
\begin{propo}\label{tame:estimate}
Let $s >0$. There exists $C_s>0$ such that for all $w_1,w_2 \in \H^s \cap \Ld^{\infty}$, we have
\begin{align*}
\LRVert{w_1 w_2}_{\H^s_x} \leq C_s \left(\LRVert{w_1}_{\Ld^{\infty}}\LRVert{w_2}_{\H^s}+\LRVert{w_1}_{\H^s}\LRVert{w_2}_{\Ld^{\infty}} \right).
\end{align*}
\end{propo}

We also recall the following result of Bony about Sobolev continuity of the composition by a smooth function  (see e.g. \cite[Corollary 2.87]{BCD} or \cite[Proposition 1.4.8]{Danchin-cours} for a more precise version)
\begin{propo}\label{Bony-Meyer}
Let $I$ be an open interval of $\R$ containing $0$. Let $s>0$ and $\sigma$ be the smallest integer such that $\sigma>s$. Let $F \in \W^{\sigma+1, \infty}(I; \R)$ such that $F(0)=0$. If $w \in \H^s$ has value in $J \Subset I$, then there exists $C_s>0$ such that 
\begin{align*}
\Vert F(w) \Vert_{\H^s} \leq C_s(1+ \Vert w \Vert_{\Ld^{\infty}} )^{\sigma} \Vert F' \Vert_{\W^{\sigma, \infty}(I)}  \Vert w \Vert_{\H^s} .
\end{align*}
\end{propo}

\begin{rem}\label{Rem:Bony-Meyer}
Note that if $0 \notin I$, and $w \in \H^s$ with value in $J \Subset I$, we can extend $F$ outside $I$ by a smooth extension such that $F(0)=0$, and the previous proposition remains valid. Indeed, by Faà di Bruno's formula, we observe that $\Vert F(w) \Vert_{\H^s}$ only involves $F$ through its derivatives evaluated at $w$.
\end{rem}

\begin{lem}\label{LM:(1-g)-1}
For all $k \in \N$ and $g: \T^d \rightarrow \R^+$ such that $0<c \leq g \leq C <1$, we have
\begin{align*}
\left\Vert\frac{1}{1-g} \right\Vert_{\H^{k}} \leq 1+\mathrm{C}_{k}\left( \LRVert{g}_{\Ld^{\infty}} \right) \LRVert{g}_{\H^k},
\end{align*}
for some non-decreasing continuous functions $\mathrm{C}_{k} :\R^+ \longrightarrow \R^+$.
\end{lem}
\begin{proof}
We rely on Proposition \ref{Bony-Meyer} by writing
\begin{align*}
\left\Vert\frac{1}{1-g} \right\Vert_{\H^{k}}=\left\Vert F(g)-F(0) \right\Vert_{\H^{k}}+1,
\end{align*}
and this directly concludes the proof.
\end{proof}

Let us finally state several product and commutator laws using weighted-Sobolev norms.
\begin{lem}\label{LM:ProducLawWeight}
Let $s \geq 0$. Consider a smooth nonnegative function $\chi=\chi(v)$ such that $\vert \partial^{\gamma} \chi \vert \lesssim_{\gamma} \chi$ for all $\gamma  \in \N^d$ such that $\vert \gamma \vert \leq s$.
\begin{itemize}
\item For any functions $f=f(x,v)$, $g=g(x,v)$ and $k \geq s/2$, we have
\begin{align}\label{Sob:estim1}
\Vert \chi f g \Vert_{\H^k_{x,v}} \lesssim \Vert f \Vert_{W^{k,\infty}_{x,v}}\Vert \chi g \Vert_{\H^s_{x,v}} + \Vert g \Vert_{W^{k,\infty}_{x,v}} \Vert \chi f \Vert_{\H^s_{x,v}}.
\end{align}
\item For any functions $a=a(x)$, $F=F(x,v)$ and $s_0>d$, we have
\begin{align}\label{Sob:estim2}
\Vert \chi  aF \Vert_{\H^s_{x,v}} \lesssim \Vert a \Vert_{\H^{s_0}_{x}} \Vert \chi F \Vert_{\H^s_{x,v}} + \Vert a \Vert_{\H^s_{x}} \Vert \chi F \Vert_{\H^s_{x,v}}.
\end{align}
\item For any vector field $E=E(x)$ and any function $f=f(x,v)$, there holds for all $s_0>1+d$ and all $\alpha, \beta \in \N^d$ satisfying $\vert \alpha \vert + \vert \beta \vert =s \geq 1$
\begin{align}\label{Sob:estim3}
\left\Vert \chi \left[ \partial_x^{\alpha} \partial_v^{\beta}, E(x) \cdot \nabla_v\right]f \right\Vert_{\Ld^2_{x,v}} \lesssim \Vert E \Vert_{\H^{s_0}_{x}} \Vert \chi f \Vert_{\H^s_{x,v}} + \Vert E \Vert_{\H^s_{x}} \Vert \chi f \Vert_{\H^s_{x,v}}.
\end{align}
\item For any functions $a=a(x)$, $f=f(x,v)$ such that $f$ has a compact suport in velocity, there holds for all $s_0>1+d$ and all $\alpha, \beta \in \N^d$ satisfying $\vert \alpha \vert + \vert \beta \vert =s \geq 1$
\begin{align}\label{Sob:estim4}
\left\Vert \chi \left[ \partial_x^{\alpha} \partial_v^{\beta}, a(x)v \cdot \nabla_v\right]f \right\Vert_{\Ld^2_{x,v}} \lesssim (1+M_f) \Vert a \Vert_{\H^{s_0}_{x}} \Vert \chi f \Vert_{\H^s_{x,v}} + \Vert a \Vert_{\H^s_{x}} \Vert \chi f \Vert_{\H^s_{x,v}},
\end{align}
where 
\begin{align*}
\mathrm{supp} \, f \, \subset \T^d \times \B(0,M_f), \ \ \M_f \in (0, + \infty).
\end{align*}
\end{itemize}
\end{lem}
\begin{proof}
We refer to \cite[Lemma 1, Section 3]{HKR} for the proof the \eqref{Sob:estim1}--\eqref{Sob:estim2}--\eqref{Sob:estim3}. Let us briefly sketch the proof of \eqref{Sob:estim4}. By expanding the commutator, we have to estimate terms of the form
\begin{align*}
I_{\gamma, \mu}=\LRVert{\chi \partial_x^{\gamma} a \partial_v^{\mu}(v_i) \partial_{v_i} \partial_x^{\alpha-\gamma} \partial_v^{\beta-\mu} f}_{\Ld^2_{x,v}}
\end{align*}
for some $1 \leq i \leq d$ and with $(\gamma, \mu) \neq (0,0)$, $\gamma \leq \alpha$, $\mu \leq \beta$ and $\vert \mu \vert<2$.  If $\gamma \neq 0$ and $\mu \neq 0$, then $I_{\gamma, \mu}=0$ or $\partial_v^{\mu}(v_i)=1$ and we can conclude as for \eqref{Sob:estim3}. If $\gamma=0$ and $\mu>0$ then $\partial_v^{\mu}(v_i)=0$ or $1$ and we conclude by Sobolev embedding in $x$ since
$I_{\gamma, \mu} \leq \LRVert{a}_{\Ld^{\infty}_x} \Vert \chi f \Vert_{\H^s_{x,v}}$. Lastly, if $\gamma>0$ and $\mu=0$, we rely on the compact support in velocity for $f$ to get
\begin{align*}
I_{\gamma, \mu} \leq M_f\LRVert{\chi \partial_x^{\gamma} a\partial_{v_i} \partial_x^{\alpha-\gamma} \partial_v^{\beta} f}_{\Ld^2_{x,v}},
\end{align*}
and we end the proof as for \eqref{Sob:estim3}.
\end{proof}

\section{Local well-posedness for \eqref{S_eps}: proof of Proposition \ref{prop:existenceEPS} }\label{Appendix:LWPeps}
Recall that we want to find a solution $(f_\eps, \varrho_\eps, u_\eps)$ (with $\eps>0$) defined on some interval $[0,T_\eps^*)$ to the system
\begin{equation*}\label{S_eps-primitif}
\tag{$\widetilde{S}_{\eps}$}
\left\{
      \begin{aligned}
\partial_t f_\eps +v \cdot \nabla_x f_\eps   -p'(\varrho_\eps)\nabla_x \left[(I-\eps^2 \Delta_x)^{-1}\varrho_\eps \right] \cdot \nabla_v f_\eps + {\rm div}_v [f_\eps (u_\eps-v)]=0,&   \\[2mm]
\partial_t (\alpha_\eps\varrho_{\eps})+\mathrm{div}_x( \alpha_\eps \varrho_\eps u_\eps)=0,&\\[2mm]
\partial_t u_\eps +(u_\eps \cdot \nabla_x) u_\eps +\frac{1}{\varrho_\eps}\nabla_x p(\varrho_\eps) -\frac{1}{\varrho_\eps(1-\rho_{f_\eps})}\Big( \Delta_x  + \nabla_x \mathrm{div}_x \Big)u_\eps=\frac{1}{\varrho_\eps(1-\rho_{f_\eps})}(j_{f_\eps}-\rho_{f_\eps} u_\eps),& \\[2mm]
{f_\eps}_{\mid t=0}=f^{\mathrm{in}}, \ \ {\varrho_\eps}_{\mid t=0}=\varrho^{\mathrm{in}}, \ \ {u_\eps}_{\mid t=0}=u^{\mathrm{in}},
\end{aligned}
    \right.
\end{equation*}
where
\begin{align*}
\rho_{f_\eps}(t,x):= \int_{\R^d} f_\eps(t,x,v) \, \mathrm{d}v, \ \ j_{f_\eps}(t,x):=\int_{\R^d} f_\eps(t,x,v) v \, \mathrm{d}v, \ \  \alpha_\eps(t,x)&=1-\rho_{f_\eps}(t,x).
\end{align*}
To do so, we rely on a standard iterative scheme. We will derive two types of estimates on the inductive solutions to that scheme: 
\begin{itemize}
\item[$-$] first, a \textit{uniform bound in high regularity} which allows to obtain, through a weak compactness argument, a weak converging (sub)sequence in spaces with high regularity. This will be possible if we consider a small enough time of existence.
\item[$-$] next, \textit{contraction estimates in low regularity} which aim at proving that this sequence of solutions is a Cauchy sequence in spaces with lower regularity, thus strongly converging.
\end{itemize}
This will be enough in order the pass to the limit in the iterative scheme, obtaining a solution with the aforementioned high order regularity. Note that the first type of estimates is actually required to prove the second one. Uniqueness will classically follow from the same computations leading to the contraction estimates.

We fix $\eps>0$ and we now drop the dependency in $\eps$ (we can take $\eps=1$ for instance): we set
\begin{align*}
f^0=f^{in}, \ \ \varrho^{0}=\varrho^{in}, \ \ u^{0}=u^{in},
\end{align*} 
and for all $n \in \N$, a triplet $(f^n,u^n, \varrho^n)$ and a time $T_n>0$ being given with
\begin{align*}
(f^n, \varrho^n,u^n) \in \mathscr{C}(0,T_n; \mathcal{H}^m_r) \times \mathscr{C}(0,T_n; \H^m) \times \mathscr{C}(0,T_n; \H^{m}) \cap \Ld^2(0,T_n; \H^{m+1}) ,
\end{align*}
we consider the system
\begin{equation}\label{S_{n+1}}
\tag{$\widetilde{S}^{n+1}$}
\left\{
      \begin{aligned}
\partial_t f^{n+1} +v \cdot \nabla_x f^{n+1} + \mathrm{div}_v\left[f^{n+1}(E^n(t,x)-v) \right]=0,   \\[2mm]
\partial_t \mathfrak{m}^{n+1}+\mathrm{div}_x( \mathfrak{m}^{n+1}u^n)=0,\\[2mm]
\partial_t u^{n+1} -\frac{1}{\mathfrak{m}^n}\Big( \Delta_x  + \nabla_x \mathrm{div}_x \Big)u^{n+1} =-(u^{n}\cdot \nabla_x) u^{n} -\nabla_x \pi(\varrho^n)+ \frac{1}{\mathfrak{m}^n}(j_{f^n}-\rho_{f^n} u^n), \\[2mm]
\varrho^{n+1}= \frac{1}{1-\rho_{f^n}}\mathfrak{m}^{n+1}, \\[2mm] 
E^n:=u^n-p'(\varrho^n)\nabla_x \left[(I-\eps^2 \Delta_x)^{-1}\varrho^n \right], \\
{f^{n+1}}_{\mid t=0}=f^{\mathrm{in}}, \ \ {\mathfrak{m}^{n+1}}_{\mid t=0}=\alpha^{\mathrm{in}}\varrho^{\mathrm{in}}, \ \ {u^{n+1}}_{\mid t=0}=u^{\mathrm{in}},
\end{aligned}
    \right.
\end{equation}
with 
\begin{align*}
\pi'(x):=\frac{p'(x)}{x}, \ \ \pi(0)=0,
\end{align*}
and where
\begin{align*}
\rho_{f^n}(t,x):= \int_{\R^d} f^n(t,x,v) \, \mathrm{d}v, \ \ j_{f^n}(t,x):=\int_{\R^d} f^n(t,x,v) v \, \mathrm{d}v.
\end{align*}

We also set
\begin{align*}
R_0:=100 \left( \Vert f^{\mathrm{in}} \Vert_{\mathcal{H}^{m}_r}+\Vert \mathfrak{m}^{\mathrm{in}} \Vert_{\H^{m}} + \Vert u^{\mathrm{in}} \Vert_{\H^{m}} \right).
\end{align*}

\medskip

\textbf{\underline{Step 1: construction of $(f^{n+1}, \varrho^{n+1}, u^{n+1})$ and uniform estimates}}. In what follows, we construct  the next iteration of the scheme which is a solution to \eqref{S_{n+1}}, thus proving that our inductive scheme is well-defined.  We derive at the same time some uniform estimates in high regularity on the sequence of solutions.

Let $n \in \N$. By induction, we assume that there exists $T>0$ (depending on $\eps$) such that for all $k =0, \cdots ,n$
\begin{align*}
\Vert f^k \Vert_{\Ld^{\infty}(0,T;\mathcal{H}^{m}_r)} +\Vert \mathfrak{m}^k \Vert_{\Ld^{\infty}(0,T;\H^{m})} + \Vert u^{k} \Vert_{\Ld^{\infty}(0,T; \H^{m}) \cap \Ld^2(0,T; \H^{m+1})} < R_0,
\end{align*}
the functions $f_k$ and $\varrho_k$ are nonnegative, and for $t \in [0,T]$
\begin{align}\label{cond:NonAnnuleLWPn}
\forall t \in [0,T], \ \ \rho_{f_k}(t)\leq \frac{\Theta+1}{2}, \ \ \ \frac{\mu}{2} \leq \varrho_k(t), \ \ \ \frac{\underline{\theta}}{2} \leq \mathfrak{m}_k(t) \leq 2\overline{\theta},
\end{align}
where $\Theta, \mu, \underline{\theta}, \overline{\theta}$ are the constants given in Proposition \ref{prop:existenceEPS}. In particular, we have
\begin{align*}
1-\rho_{f_k}(t) \geq \frac{1-\Theta}{2}>0.
\end{align*}
Note also that via Sobolev embedding, Lemma \ref{LM:(1-g)-1} and Lemma \ref{LM:weightedSob}, we have
\begin{align*}
 \Vert \varrho^n \Vert_{\Ld^{\infty}(0,T;\H^{m})} &\leq \left\Vert \frac{1}{1-\rho_{f^{n-1}}} \right\Vert_{\Ld^{\infty}(0,T;\H^{m})} \Vert \mathfrak{m}^n \Vert_{\Ld^{\infty}(0,T;\H^{m})} \\
 & \leq \left( 1+ \Lambda(\Vert \rho_{f^{n-1}} \Vert_{\Ld^{\infty}(0,T;\H^{m})} \right) \Vert \rho_{f^{n-1}} \Vert_{\Ld^{\infty}(0,T;\H^{m})} \Vert \mathfrak{m}^n \Vert_{\Ld^{\infty}(0,t;\H^{m})} \\
  & \leq \left( 1+ \Lambda(\Vert f^{n-1} \Vert_{\Ld^{\infty}(0,T;\mathcal{H}^{m}_r)} \right) \Vert f^{n-1} \Vert_{\Ld^{\infty}(0,T;\mathcal{H}^{m}_r)} \Vert \mathfrak{m}^n \Vert_{\Ld^{\infty}(0,T;\H^{m})} \\
  &\leq \Lambda(R_0).
\end{align*}

In what follows, we rely on the a priori estimates of Section \ref{Section:prelimBootstrap}.

\medskip

$\bullet$ We can obtain a unique nonnegative solution $f^{n+1} \in \mathscr{C}(0,T; \mathcal{H}^m_r)$ to the Vlasov equation by the method of characteristics. Since there is the regularization due to $\mathrm{J}_\eps$, we can invoke Proposition \ref{prop:energy:f:eps} so that we have for all $t \in [0,T]$ 
\begin{align*}
\Vert f^{n+1}(t) \Vert_{\mathcal{H}^{m}_{r}} &\leq \Vert f^{\mathrm{in}} \Vert_{\mathcal{H}^{m}_{r}} \exp \left[C\left( (1+ \Vert u^n \Vert_{\Ld^{\infty}(0,t;\H^{m})})t+ \frac{t}{\eps} \Lambda\left( \Vert \varrho^n \Vert_{\Ld^{\infty}(0,t;\H^{m-2})} \right) \Vert \varrho^n \Vert_{\Ld^{\infty}(0,t;\H^{m})} \right) \right] \\
& \leq \frac{R_0}{100}\exp \left[C\left( (1+ R_0)t+ \frac{t}{\eps} \Lambda(R_0) \right) \right].
\end{align*}
We then define $T_{[1]}=T_{[1]}(R_0)>0$ with $T_{[1]}(R_0)<T$ 
(depending on $\eps$) such that 
\begin{align*}
\frac{R_0}{100}\exp \left[C\left( (1+ R_0)T_{[1]}+ \frac{\sqrt{T_{[1]}}}{\eps}  \Lambda(R_0) \right) \right]<\frac{R_0}{3}.
\end{align*}

\medskip

$\bullet$ Since $u^n \in \mathscr{C}(0,T; \H^{m+1})$ is given, we can construct a nonnegative solution $\mathfrak{m}^{n+1}$ to the second equation in \eqref{S_{n+1}}, relying on standard arguments for continuity equations. 
We obtain a unique solution $\mathfrak{m}^{n+1} \in \mathscr{C}^1(0,T; \H^m)$, with for all $t \in [0,T]$
\begin{align*}
\Vert \mathfrak{m}^{n+1}(t) \Vert_{\H^{m}} &\leq e^{\Vert \mathrm{div}_x \, u^n \Vert_{\Ld^{\infty}((0,t) \times \T^d)}t/2} \left(\Vert \mathfrak{m}^{in} \Vert_{\H^{m}} +\int_0^t  \Vert  u^n(\tau) \Vert_{\H^{m+1}} \Vert \mathfrak{m}^{n+1}(\tau) \Vert_{\H^{m}}  \, \mathrm{d}\tau  \right),
\end{align*}
thanks to Proposition \ref{prop:energyRHO:eps}. Assuming that $e^{R_0t/2} \leq 2$, we infer by Grönwall's lemma
\begin{align*}
 \Vert \mathfrak{m}^{n+1}(t) \Vert_{\H^{m}} &\leq 2 \Vert \mathfrak{m}^{in} \Vert_{\H^{m}} \exp\left(2\int_0^t  \Vert  u^n(\tau) \Vert_{\H^{m+1}}   \, \mathrm{d}\tau \right) \\
 & \leq  2  \Vert \mathfrak{m}^{in} \Vert_{\H^{m}} \exp\left(2 \sqrt{t} \Vert  u^n \Vert_{\Ld^2(0,t;\H^{m+1})} \right),
\end{align*}
therefore for such times $t$, we have
\begin{align*}
\forall t \in [0,T], \ \ \ 
 \Vert \mathfrak{m}^{n+1}(t) \Vert_{\H^{m}} \leq 2 \frac{R_0}{100} e^{2\sqrt{t} R_0}.
\end{align*}
We then define $T_{[2]}=T_{[2]}(R_0)>0$ such that $e^{R_0T_{[2]}/2} \leq 2$ and such that 
\begin{align*}
2\frac{R_0}{100} e^{2\sqrt{T_{[2]}} R_0} <\frac{R_0}{3}.
\end{align*}

\medskip

$\bullet$  We define $u^{n+1}$ as the unique solution of the parabolic equation 
\begin{align*}
\partial_t w -\frac{1}{\mathfrak{m}^n}\Big( \Delta_x  + \nabla_x \mathrm{div}_x \Big)w =-(u^{n}\cdot \nabla_x) u^{n} -\nabla_x \pi(\varrho^n)+ \frac{1}{\mathfrak{m}^n}(j_{f^n}-\rho_{f^n} u^n),
\end{align*}
starting from $u^{in}$, which satisfies
\begin{align*}
 &\Vert u^{n+1} \Vert_{\Ld^{\infty}(0,t; \H^{m})}+ \Vert u^{n+1} \Vert_{\Ld^2(0,t; \H^{m+1})} \\
 &\leq (1+\sqrt{t}\Lambda(t,R_0))\Bigg(\Vert u^{in} \Vert_{\H^m} + \sqrt{t}\Lambda(\Vert \varrho^n \Vert_{\Ld^{\infty}(0,t;\H^{m})}) \Vert \varrho^n \Vert_{\Ld^{\infty}(0,t;\H^{m})}+ \sqrt{t}\left\Vert u^n\right\Vert_{\Ld^{\infty}(0,t;\H^{m})}^2 \\
&  \quad \qquad \qquad \quad \qquad \qquad \quad \qquad \qquad \qquad \qquad\qquad \qquad    +  \sqrt{t} \left\Vert \frac{1}{\mathfrak{m}^n}(j_{f^n}-\rho_{f^n} u^n)\right\Vert_{\Ld^{\infty}(0,t;\H^{m-1})} \Bigg).
\end{align*}
Here, we have applied the same argument as for the proof of Proposition \ref{prop:energy-u_eps}, making $\Vert \varrho^n \Vert_{\Ld^{\infty}(0,t;\H^{m})}$ appear to get a factor $\sqrt{t}$. The second term in the parenthesis is controlled by $\Lambda(R_0)$ while for the third term, we can use the same of kind of arguments (with Remark \ref{Rem:Bony-Meyer}) to get 
\begin{align*}
\Big\Vert \frac{1}{\mathfrak{m}^n}(j_{f^n}-\rho_{f^n} u^n & )\Big\Vert_{\Ld^{\infty}(0,t;\H^{m-1})} \\ 
&\leq \left\Vert \frac{1}{\mathfrak{m}^n}\right\Vert_{\Ld^{\infty}(0,t;\H^{m-1})} \left(\left\Vert j_{f^n}\right\Vert_{\Ld^{\infty}(0,t;\H^{m-1})} + \left\Vert \rho_{f^n}\right\Vert_{\Ld^{\infty}(0,t;\H^{m-1})}\left\Vert u^n\right\Vert_{\Ld^{\infty}(0,t;\H^{m-1})}\right) \\
& \leq \Lambda(R_0).
\end{align*}
We eventually obtain 
\begin{align*}
 \Vert u^{n+1} \Vert_{\Ld^{\infty}(0,t; \H^{m})}+\Vert u^{n+1} \Vert_{\Ld^2(0,t; \H^{m+1})} &\leq (1+\sqrt{t}\Lambda(t,R_0)) \frac{R_0}{100} + \sqrt{t}\Lambda(R_0), 
\end{align*}
We then define $T_{[3]}=T_{[3]}(R_0)>0$ with $T_{[3]}(R_0)<T$
 such that 
\begin{align*}
\left(1+\sqrt{T_{[3]}}\Lambda(T_{[3]},R_0) \right) \frac{R_0}{100} + \sqrt{T_{[3]}} \Lambda(R_0)<\frac{R_0}{3}.
\end{align*}

\medskip

$\bullet$ Lastly, we can rely on Lemmas \ref{LM:bound:Haut-rho_falpha}--\ref{LM:bound:HautBas-alpharho} to find a time $T_{[4]}=T_{[4]}(R_0)>0$ such that the condition \eqref{cond:NonAnnuleLWPn} is satisfied for $(f^{n+1}, \mathfrak{m}^{n+1}, u^{n+1})$ on the interval $[0, \min(T_{[4]}, T))$.

\medskip

All in all, we define 
\begin{align*}
T(R_0):=\min \left(T_{[1]}(R_0),T_{[2]}(R_0),T_{[3]}(R_0), T_{[4]}(R_0)\right)<T,
\end{align*}
which may depend on $\eps$ but which is independent of $n$. 
An induction procedure based on the three previous estimate shows that one can obtain a time $T^{[\eps]}>0$ and a sequence $(f^n, \mathfrak{m}^n,u^n)_{n \in \N}$ satisfying for all $n \in \N$
\begin{align*}
(f^n, \mathfrak{m}^n,u^n) \in \mathscr{C}(0,T^{[\eps]}; \mathcal{H}^m_r) \times \mathscr{C}(0,T^{[\eps]}; \H^m) \times \mathscr{C}(0,T^{[\eps]}; \H^{m}) \cap \Ld^2(0,T; \H^{m+1}) ,
\end{align*}
with \eqref{cond:NonAnnuleLWPn} and the uniform estimate
\begin{align}\label{ineq:BoundUnif-n}
\Vert f^n \Vert_{\Ld^{\infty}(0,T^{[\eps]};\mathcal{H}^{m}_r)} +\Vert \mathfrak{m}^n \Vert_{\Ld^{\infty}(0,T^{[\eps]};\H^{m})} + \Vert u^{n} \Vert_{\Ld^{\infty}(0,T^{[\eps]}; \H^{m}) \cap \Ld^2(0,T^{[\eps]}; \H^{m+1})} < R_0.
\end{align}

\medskip

\textbf{\underline{Step 2: contraction estimates in $\Ld^{\infty}_T \mathcal{H}^0_r \times \Ld^{\infty}_T \Ld^2  \times \Ld^{\infty}_T \Ld^2 \cap \Ld^{2}_T \H^1 $}}. For $n \in  \N {\setminus} \lbrace 0 \rbrace$, we set
\begin{align*}
g^n:=f^{n+1}-f^n, \ \ \mathfrak{M}^n=\mathfrak{m}^{n+1}-\mathfrak{m}^n, \ \ w^n:=u^{n+1}-u^n,
\end{align*}
which satisfy the system of equations:
\begin{equation}\label{Cauchy-S_{n+1}}
\left\{
      \begin{aligned}
\partial_t g^{n} +v \cdot \nabla_x g^{n} +\mathrm{div}_v\left[g^{n}(E^n(t,x)-v) \right]+\left(E^{n}-E^{n-1} \right) \cdot \nabla_v f^n=0,   \\[2mm]
\partial_t \mathfrak{M}^{n}+\mathrm{div}_x( \mathfrak{M}^{n}u^n)=-\mathrm{div}_x( \mathfrak{m}^{n}w^{n-1}),\\[2mm]
\partial_t w^{n} -\frac{1}{\mathfrak{m}^{n+1}}\Big( \Delta_x  + \nabla_x \mathrm{div}_x \Big)w^{n} =\mathfrak{S}^n, \\[2mm]
\varrho^{n+1}:= \frac{1}{1-\rho_{f^n}}\mathfrak{m}^{n+1}, \\[2mm] 
E^n:=u^n-p'(\varrho^n)\nabla_x \left[(I-\eps^2 \Delta_x)^{-1}\varrho^n \right], \\
{g^{n}}_{\mid t=0}=0, \ \ {\mathfrak{M}^{n}}_{\mid t=0}=0, \ \ {w^{n}}_{\mid t=0}=0,
\end{aligned}
    \right.
\end{equation}
and where
\begin{align*}
\mathfrak{S}^n&:=\left (\frac{1}{\mathfrak{m}^{n+1}}-\frac{1}{\mathfrak{m}^{n}} \right) \Big( \Delta_x  + \nabla_x \mathrm{div}_x \Big)u^n-(w^{n-1}\cdot \nabla_x) u^{n}-(u^{n-1}\cdot \nabla_x) w^{n-1} -\nabla_x \left[\pi(\varrho^{n-1})- \pi(\varrho^{n}) \right]\\
 &\quad + \frac{1}{\mathfrak{m}^{n}}(j_{f^{n}}-\rho_{f^{n}} u^{n})-\frac{1}{\mathfrak{m}^{n-1}}(j_{f^{n-1}}-\rho_{f^{n-1}} u^{n-1}).
\end{align*}

Let us derive some $\Ld^2$ estimates on $(g^n, \mathfrak{M}^n, w^n)$. They will be satisfied on $[0,\widetilde{T}]$ for some $\widetilde{T}^{[\eps]}<T^{[\eps]}$.

\medskip

$\bullet$ We perform $\Ld^2_{x,v}$-weighted estimates in the first equation on $g^n$ and as before (see \eqref{ineq:d/dt:fweight}), we get
\begin{align*}
\frac{\mathrm{d}}{\mathrm{d}t}\Vert g^n(t) \Vert_{\mathcal{H}^{0}_{r}}^2  \lesssim (1+\Vert E^n(t) \Vert_{\H^{m}} ) \Vert g^n(t)\Vert_{\mathcal{H}^{0}_{r}}^2+  \LRVert{E^{n}(t)-E^{n-1}(t)}_{\Ld^2} \LRVert{f^n(t)}_{\mathcal{H}^m_r} \LRVert{g^n(t)}_{\mathcal{H}^0_r}.
\end{align*}
For $t \in (0,T)$, we can infer
\begin{align*}
\Vert g^n(t) \Vert_{\mathcal{H}^{0}_{r}} &\lesssim \int_0^t (1+\Vert E^n(s) \Vert_{\H^{m}} ) \Vert g^n(s)\Vert_{\mathcal{H}^{0}_{r}} \, \mathrm{d}s + \int_0^t\LRVert{E^{n}(s)-E^{n-1}(s)}_{\Ld^2} \LRVert{f^n(s)}_{\mathcal{H}^m_r} \, \mathrm{d}s \\
& \leq \sqrt{t} \Bigg( \left(\sqrt{t}+\LRVert{E^n}_{\Ld^2(0,t; \mathrm{H}^m)} \right)\Vert g^n\Vert_{\Ld^{\infty}(0,t;\mathcal{H}^{0}_{r})}+R_0\LRVert{E^{n}-E^{n-1}}_{\Ld^2(0,t;\Ld^2)} \Bigg) \\
& \lesssim \sqrt{t} \Bigg( \left(\sqrt{t}+\frac{\sqrt{t}}{\eps} R_0\right)\Vert g^n\Vert_{\Ld^{\infty}(0,t;\mathcal{H}^{0}_{r})}+R_0\LRVert{E^{n}-E^{n-1}}_{\Ld^2(0,t;\Ld^2)} \Bigg).
\end{align*}
Choosing $\widetilde{T}^{[\eps]}<T^{[\eps]}$ small enough independent of $n$, we can absorb the first term in the parenthesis in the left-hand side so that for all $t \in (0,\widetilde{T}^{[\eps]})$
\begin{align*}
\Vert g^n(t) \Vert_{\mathcal{H}^{0}_{r}} &\lesssim \sqrt{t} R_0\LRVert{E^{n}-E^{n-1}}_{\Ld^2(0,t;\Ld^2)}.
\end{align*}
Then observe that by Sobolev embedding and Proposition \ref{Bony-Meyer}, we have
\begin{align*}
\Vert E^{n}-E^{n-1} \Vert_{\Ld^2(0,t;\Ld^2)} &\leq \Vert u^n-u^{n-1} \Vert_{\Ld^2(0,t;\Ld^2)}+ \left\Vert p'(\varrho^n)\nabla_x \left[\mathrm{J}_{\eps}\varrho^n \right]-p'(\varrho^{n-1})\nabla_x \left[\mathrm{J}_{\eps}\varrho^{n-1} \right] \right\Vert_{\Ld^2(0,t;\Ld^2)} \\
& \leq \Vert u^n-u^{n-1} \Vert_{\Ld^2(0,t;\Ld^2)}+ \LRVert{p'(\varrho^n)}_{\Ld^{\infty}(0,t; \Ld^{\infty})} \left\Vert \nabla_x \left[\mathrm{J}_{\eps}(\varrho^n -\varrho^{n-1} \right] \right\Vert_{\Ld^2(0,t \Ld^2)} \\
& \quad + \Vert p'(\varrho^{n})-p'(\varrho^{n-1}) \Vert_{\Ld^{2}(0,t; \Ld^{2})} \left\Vert \nabla_x \left[\mathrm{J}_{\eps}\varrho^{n-1} \right] \right\Vert_{\Ld^{\infty}(0,t;\Ld^{\infty})} \\
& \lesssim \Vert u^n-u^{n-1} \Vert_{\Ld^2(0,t;\Ld^2)}+ \Lambda(R_0) \frac{1}{\eps} \left\Vert \varrho^n -\varrho^{n-1} \right\Vert_{\Ld^2(0,t \Ld^2)} \\
& \quad + \Lambda(t,R_0) \frac{1}{\eps} \Vert p'(\varrho^{n})-p'(\varrho^{n-1}) \Vert_{\Ld^{2}(0,t; \Ld^{2})} \\
& \lesssim \Vert u^n-u^{n-1} \Vert_{\Ld^2(0,t;\Ld^2)}+ \Lambda(t,R_0) \frac{1}{\eps} \Vert \mathfrak{m}^{n}-\mathfrak{m}^{n-1} \Vert_{\Ld^{2}(0,t; \Ld^{2})}.
\end{align*}
This yields for all $t \in (0,\widetilde{T}^{[\eps]})$
\begin{align*}
\Vert g^n \Vert_{\Ld^{\infty}(0,t;\mathcal{H}^{0}_{r})}  \lesssim \sqrt{t} \left( R_0\Vert w^{n-1} \Vert_{\Ld^2(0,t;\Ld^2)}+ \Lambda(t,R_0) \frac{1}{\eps} \Vert \mathfrak{M}^{n-1} \Vert_{\Ld^{2}(0,t; \Ld^{2})} \right).
\end{align*}

\medskip

$\bullet$  An $\Ld^2$ energy estimate on the second equation on $\mathfrak{M}^{n}$ also leads to 
\begin{align*}
\LRVert{\mathfrak{M}^{n}(t)}_{\Ld^{\infty}(0,t;\Ld^2)} &\leq e^{\LRVert{\mathrm{div}_x \, u_n}_{\Ld^{\infty}(0,t; \Ld^{\infty})}t}\int_0^t \LRVert{\mathrm{div}_x \, \mathfrak{m}^n w^{n-1}(\tau)}_{\Ld^2} \, \mathrm{d}\tau \\
& \leq e^{R_0t} \left( \int_0^t \LRVert{w^{n-1}\cdot \nabla_x \mathfrak{m}^n (\tau)}_{\Ld^2} \, \mathrm{d}\tau  + \int_0^t \LRVert{\mathfrak{m}^n\mathrm{div}_x \,  w^{n-1}(\tau)}_{\Ld^2} \, \mathrm{d}\tau  \right) \\
& \leq \sqrt{t} e^{R_0t}  \left( \LRVert{\nabla_x \mathfrak{m}^n}_{\Ld^{\infty}(0,t; \Ld^{\infty})} \LRVert{w^{n-1}}_{\Ld^2(0,t; \Ld^2)} + \LRVert{ \mathfrak{m}^n}_{\Ld^{\infty}(0,t; \Ld^{\infty})} \LRVert{\mathrm{div}_x w^{n-1}}_{\Ld^2(0,t; \Ld^2)} \right) \\
& \lesssim \sqrt{t}\Lambda(t,R_0) \LRVert{w^{n-1}}_{\Ld^2(0,t; \H^1)},
\end{align*}
for all $t \in (0,\widetilde{T}^{[\eps]})$.

\medskip

$\bullet$ Performing an $\Ld^{\infty}_T \Ld^2 \cap \Ld^{2}_T \H^1 $ estimate by multiplying by $w_n$ in the parabolic equation satisfied by $w^n$ yields 
\begin{multline*}
\frac{1}{2} \frac{\mathrm{d}}{\mathrm{d}t}\Vert w^n \Vert_{\Ld^2}^2+\int_{\T^d } \frac{1}{\mathfrak{m}^{n+1}} \left( \vert \nabla_x w^{n} \vert^2 + \vert \mathrm{div}_x w^n \vert^2\right) \\
=\sum_{k=1}^5\int_{\T^d} \mathfrak{S}^n_k \cdot w^n - \sum_{i=1}^d \left\lbrace\int_{\T^d } \nabla_x w^n_i \cdot \nabla_x \left(\frac{1}{\mathfrak{m}^{n+1}} \right) w^n_i -\int_{\T^d } \mathrm{div}_x \, w^{n} \partial_i \left(\frac{1}{\mathfrak{m}^{n+1}} \right) w^n_i \right\rbrace ,
\end{multline*}
with
\begin{align*}
\mathfrak{S}^n_1:=\left (\frac{1}{\mathfrak{m}^{n+1}}-\frac{1}{\mathfrak{m}^{n}} \right)\Big( \Delta_x  + \nabla_x \mathrm{div}_x \Big)u^n, \ \ \mathfrak{S}^n_2:=-(w^{n-1}\cdot \nabla_x) u^{n}, \ \ \mathfrak{S}^n_3:=-(u^{n-1}\cdot \nabla_x) w^{n-1}, \\
\mathfrak{S}^n_4:=-\nabla_x \left[\pi(\varrho^{n-1})- \pi(\varrho^{n}) \right], \ \ \mathfrak{S}^n_5:=\frac{1}{\mathfrak{m}^{n}}(j_{f^{n}}-\rho_{f^{n}} u^{n})-\frac{1}{\mathfrak{m}^{n-1}}(j_{f^{n-1}}-\rho_{f^{n-1}} u^{n-1}).
\end{align*}
For all $\eta >0$, we obtain by Young's inequality
\begin{multline*}
\frac{1}{2} \frac{\mathrm{d}}{\mathrm{d}t}\Vert w^n \Vert_{\Ld^2}^2+\int_{\T^d } \left(\frac{1}{\mathfrak{m}^{n+1}}-\eta\LRVert{\nabla_x\left(\frac{1}{\mathfrak{m}^{n+1}} \right) }_{\Ld^{\infty}}^2 \right) \left( \vert \nabla_x w^{n} \vert^2 + \vert \mathrm{div}_x w^n \vert^2\right) \\
\leq \sum_{k=1}^5\int_{\T^d} \mathfrak{S}^n_k \cdot w^n +\frac{1}{\eta} \int_{\T^d}  \vert  w^{n} \vert^2.
\end{multline*}
Let us focus on the source terms $\mathfrak{S}^n_k$. We have by Sobolev embedding
\begin{align*}
\int_{\T^d} \mathfrak{S}^n_1 \cdot w^n &\lesssim \LRVert{( \Delta_x  + \nabla_x \mathrm{div}_x )u^n}_{\Ld^{\infty}}^2 \LRVert{\frac{1}{\mathfrak{m}^{n+1}}-\frac{1}{\mathfrak{m}^{n}}}_{\Ld^{2}}^2 + \int_{\T^d}  \vert  w^{n} \vert^2 \\
&\lesssim R_0^2 \LRVert{\frac{1}{\mathfrak{m}^{n+1} \mathfrak{m}^n}}_{\Ld^{\infty}}^2  \LRVert{\mathfrak{m}^{n+1}-\mathfrak{m}^{n}}_{\Ld^{2}}^2 + \int_{\T^d}  \vert  w^{n} \vert^2 \\
&\lesssim \Lambda(R_0)  \LRVert{\mathfrak{M}^{n}}_{\Ld^{2}}^2 + \LRVert{w^{n}}_{\Ld^{2}}^2.
\end{align*}
We next have 
\begin{align*}
\int_{\T^d} \mathfrak{S}^n_2 \cdot w^n  \lesssim  \int_{\T^d}  \vert  \mathfrak{S}^n_2 \vert^2  + \int_{\T^d}  \vert  w^{n} \vert^2 & \leq \LRVert{\nabla_x u^{n-1}}_{\Ld^{\infty}}^2 \LRVert{w^{n-1}}_{\Ld^{2}}^2 + \int_{\T^d}  \vert  w^{n} \vert^2 \\
& \leq R_0^2\LRVert{w^{n-1}}_{\Ld^{2}}^2 + \LRVert{w^{n}}_{\Ld^{2}}^2,
\end{align*}
again by Sobolev embedding while for all $\delta_1>0$, we have by integration by parts and Sobolev embedding
\begin{align*}
\int_{\T^d} \mathfrak{S}^n_3 \cdot w^n &\lesssim \LRVert{\nabla_x u^{n-1}}_{\Ld^{\infty}}^2 \LRVert{w^{n-1}}_{\Ld^{2}}^2 + \LRVert{w^{n}}_{\Ld^{2}}^2 +\frac{1}{\delta_1}\LRVert{ u^{n-1}}_{\Ld^{\infty}}^2 \LRVert{ w^{n-1}}_{\Ld^{2}}^2 + \delta_1 \int_{\T^d}  \vert  \nabla_x w^{n} \vert^2 \\
& \lesssim \Lambda_{\delta_1}(R_0)\LRVert{ w^{n-1}}_{\Ld^{2}}+\LRVert{w^{n}}_{\Ld^{2}}^2+ \delta_1 \int_{\T^d}  \vert  \nabla_x w^{n} \vert^2.
\end{align*}
We also have for all $\delta_2>0$
\begin{align*}
\int_{\T^d} \mathfrak{S}^n_4 \cdot w^n &\lesssim \frac{1}{\delta_2} \int_{\T^d}  \vert  \left[\pi(\varrho^{n-1})- \pi(\varrho^{n}) \right] \vert^2  + \delta_2 \int_{\T^d}  \vert  \mathrm{div}_x \, w^{n} \vert^2  \\
& \lesssim \Lambda_{\delta_2}(R_0) \LRVert{\varrho^{n-1}-\varrho^{n}}_{\Ld^2}^2  + \delta_2 \int_{\T^d}  \vert  \mathrm{div}_x \, w^{n} \vert^2.
\end{align*}
For $\mathfrak{S}^n_5$, we write
\begin{multline*}
\mathfrak{S}^n_5
=\left(\frac{1}{\mathfrak{m}^{n}}-\frac{1}{\mathfrak{m}^{n-1}} \right)(j_{f^{n}}-\rho_{f^{n}} u^{n})+\frac{1}{\mathfrak{m}^{n-1}}(j_{f^n}-j_{f^{n-1}})\\
+\frac{1}{\mathfrak{m}^{n}}\left(\rho_{f^{n-1}}-\rho_{f^n} \right)u^{n-1}+\frac{1}{\mathfrak{m}^{n-1}}\rho_{f^n}(u^{n-1}-u^n),
\end{multline*} 
therefore by Sobolev embedding and Lemma \ref{LM:weightedSob}
\begin{align*}
\int_{\T^d} \mathfrak{S}^n_5 \cdot w^n & \lesssim   \LRVert{j_{f^{n}}-\rho_{f^{n}} u^{n}}_{\Ld^{\infty}}^2\LRVert{\frac{1}{\mathfrak{m}^{n}}-\frac{1}{\mathfrak{m}^{n-1}}}_{\Ld^{2}}^2+\LRVert{\frac{1}{\mathfrak{m}^{n}}}_{\Ld^{\infty}}^2 \ \LRVert{j_{f^{n}}-j_{f^{n-1}}}_{\Ld^2}^2\\
&\quad +\LRVert{\frac{1}{\mathfrak{m}^{n}}}_{\Ld^{\infty}}^2\LRVert{u^{n-1}}_{\Ld^{\infty}}^2\LRVert{\rho_{f^{n}}-\rho_{f^{n-1}}}_{\Ld^2}^2+ \LRVert{\frac{1}{\mathfrak{m}^{n-1}}}_{\Ld^{\infty}}^2 \LRVert{\rho_{f^n}}_{\Ld^{\infty}}^2 \LRVert{u^{n-1}-u^n}_{\Ld^2}^2 \\
& \quad +\int_{\T^d}  \vert   w^{n} \vert^2 \\
& \lesssim \Lambda(R_0) \left(\LRVert{\mathfrak{m}^n-\mathfrak{m}^{n-1}}_{\Ld^2}^2+ \LRVert{f^n-f^{n-1}}_{\mathcal{H}^0_r}^2+ \LRVert{u^{n-1}-u^n}_{\Ld^2}^2 \right)+ \LRVert{w^{n}}_{\Ld^{2}}^2.
\end{align*}
We finally obtain
\begin{align}\label{ineq:dtw_n-uniquenessLWP}
\begin{split}
&\frac{\mathrm{d}}{\mathrm{d}t} \Vert w^n(t) \Vert_{\Ld^2}^2+\int_{\T^d } \left(\frac{1}{\mathfrak{m}^{n+1}}-\eta R_0^2-\delta_1 -\delta_2 \right) \left( \vert \nabla_x w^{n} \vert^2 + \vert \mathrm{div}_x w^n \vert^2\right) \\
&\leq \Lambda_{\delta_{1,2}, \eta}(R_0) \left( \LRVert{\mathfrak{M}^{n-1}(t)}_{\Ld^2}^2+ \LRVert{g^{n-1}(t)}_{\mathcal{H}^0_r}^2+ \LRVert{w^{n-1}(t)}_{\Ld^2}^2\right)+ \Lambda(R_0)\LRVert{\mathfrak{M}^{n}(t)}_{\Ld^{2}}^2 \\
& \quad + \Lambda_{\eta} \Vert w^n(t) \Vert_{\Ld^2}^2 .
\end{split}
\end{align}
for some constant $\Lambda_{\eta} >0$. A good choice of $\eta, \delta_1$ and $\delta_2$ with respect to $\inf \frac{1}{\mathfrak{m}^{n+1}}$ combined with Grönwall's lemma shows that 
\begin{align*}
\Vert w^n(t) \Vert_{\Ld^2}^2 \lesssim  e^{\Lambda_{\eta} t}\int_0^t \left( \LRVert{\mathfrak{M}^{n-1}(\tau)}_{\Ld^2}^2+ \LRVert{g^{n-1}(\tau)}_{\mathcal{H}^0_r}^2+ \LRVert{w^{n-1}(\tau)}_{\Ld^2}^2+\LRVert{\mathfrak{M}^{n}(\tau)}_{\Ld^{2}}^2 \right) \, \mathrm{d}\tau.
\end{align*}
Reducing $\widetilde{T}^{[\eps]}$ such that $\widetilde{T}^{[\eps]} \leq 1/\Lambda_{\eta}$, we can integrate in time the previous differential inequality \eqref{ineq:dtw_n-uniquenessLWP} to obtain for all $t \in (0,\widetilde{T}^{[\eps]})$
\begin{multline*}
\Vert w^n(t) \Vert_{\Ld^{\infty}(0,t; \Ld^2)}^2+\Vert w^n(t) \Vert_{\Ld^{2}(0,t; \H^1)}^2 \\ \lesssim \Lambda(R_0)t\left( \LRVert{\mathfrak{M}^{n-1}}_{\Ld^{\infty}(0,t;\Ld^2)}^2+ \LRVert{g^{n-1}}_{\Ld^{\infty}(0,t;\mathcal{H}^0_r)}^2+ \LRVert{w^{n-1}}_{\Ld^{\infty}(0,t;\Ld^2)}^2\right)+ \Lambda(R_0)t\LRVert{\mathfrak{M}^{n}}_{\Ld^{\infty}(0,t;\Ld^{2})}^2.
\end{multline*}
Combining the three previous points, we reach the following inequality valid for all $t \in (0,\widetilde{T}^{[\eps]})$ 
\begin{align*}
&\Vert g^n(t) \Vert_{\Ld^{\infty}(0,t;\mathcal{H}^{0}_{r})}+\LRVert{\mathfrak{M}^{n}(t)}_{\Ld^{\infty}(0,t;\Ld^2)}+\Vert w^n(t) \Vert_{\Ld^{\infty}(0,t; \Ld^2)\cap \Ld^{2}(0,t; \H^1)} \\
& \lesssim \sqrt{t} \left( R_0 \sqrt{t}\Vert w^{n-1} \Vert_{\Ld^{\infty}(0,t;\Ld^2)}+ \Lambda(t,R_0) \frac{1}{\eps} \Vert \mathfrak{M}^{n-1} \Vert_{\Ld^{2}(0,t; \Ld^{2})} \right)+\sqrt{t}\Lambda(t,R_0) \LRVert{w^{n-1}}_{\Ld^2(0,t; \H^1)}\\
& \quad + \Lambda(R_0)t\left( \LRVert{\mathfrak{M}^{n-1}}_{\Ld^{\infty}(0,t;\Ld^2)}+ \LRVert{g^{n-1}}_{\Ld^{\infty}(0,t;\mathcal{H}^0_r)}+ \LRVert{w^{n-1}}_{\Ld^{\infty}(0,t;\Ld^2)}\right)+ \Lambda(R_0)t\LRVert{\mathfrak{M}^{n}}_{\Ld^{\infty}(0,t;\Ld^{2})}.
\end{align*}
Reducing again $\widetilde{T}^{[\eps]}$, it means that that for all $n \in  \N {\setminus} \lbrace 0 \rbrace$
\begin{multline}\label{ineq:Contract-n-1/2}
\Vert g^n \Vert_{\Ld^{\infty}(0,\widetilde{T}^{[\eps]};\mathcal{H}^{0}_{r})}+\LRVert{\mathfrak{M}^{n}}_{\Ld^{\infty}(0,\widetilde{T}^{[\eps]};\Ld^2)}+\Vert w^n \Vert_{\Ld^{\infty}(0,\widetilde{T}^{[\eps]}; \Ld^2)\cap \Ld^{2}(0,\widetilde{T}^{[\eps]}; \H^1)} \\  \leq \frac{1}{2} \left( \Vert g^{n-1} \Vert_{\Ld^{\infty}(0,\widetilde{T}^{[\eps]};\mathcal{H}^{0}_{r})}+\LRVert{\mathfrak{M}^{n-1}}_{\Ld^{\infty}(0,\widetilde{T}^{[\eps]};\Ld^2)}+\Vert w^{n-1} \Vert_{\Ld^{\infty}(0,\widetilde{T}^{[\eps]}; \Ld^2)\cap \Ld^{2}(0,\widetilde{T}^{[\eps]}; \H^1)}  \right).
\end{multline}
Note that $\widetilde{T}^{[\eps]}$ has been chosen independent of $n \in  \N {\setminus} \lbrace 0 \rbrace$.

\medskip

\textbf{\underline{Step 3: obtaining a unique solution to $\eqref{S_eps}$}}. Combining the uniform bound \eqref{ineq:BoundUnif-n} with the contraction estimate \eqref{ineq:Contract-n-1/2}, we deduce that the sequence $(f_n, \mathfrak{m}_n, u_n)$ is weakly$(-\star)$ compact in the space $$\Ld^{\infty}(0,\widetilde{T}^{[\eps]}; \mathcal{H}^m_r) \times \Ld^{\infty}(0,\widetilde{T}^{[\eps]}; \H^m) \times \Ld^{\infty}(0,\widetilde{T}^{[\eps]}; \H^{m}) \cap \Ld^2(0,\widetilde{T}^{[\eps]}; \H^{m+1}), $$ and is a Cauchy sequence in the space $$\Ld^{\infty}(0,\widetilde{T}^{[\eps]}; \mathcal{H}^0_r) \times \Ld^{\infty}(0,\widetilde{T}^{[\eps]}; \Ld^2)  \times \Ld^{\infty}(0,\widetilde{T}^{[\eps]};\Ld^2) \cap \Ld^{2}(0,\widetilde{T}^{[\eps]};\H^1) .$$
This shows that the whole sequence converges weakly$(-\star)$ in the first space. The limit $(f, m,u)$ belongs in particular to the first space with high regularity. Using weak-strong convergence principles then allows to prove that the limit is a solution to the system $\eqref{S_eps}$ on $(0,\widetilde{T}^{[\eps]})$ in the sense of distribution. We don't detail this part of the proof. Using the equation and the time derivative of the solution, one can show that the solution actually belongs to
$$\mathrm{X}_{\widetilde{T}^{[\eps]}}^m:=\mathscr{C}(0,\widetilde{T}^{[\eps]}; \mathcal{H}^m_r) \times \mathscr{C}(0,\widetilde{T}^{[\eps]}; \H^m) \times \mathscr{C}(0,\widetilde{T}^{[\eps]}; \H^{m}) \cap \Ld^2(0,\widetilde{T}^{[\eps]}; \H^{m+1}).$$
The uniqueness of the Cauchy problem for $\eqref{S_eps}$ in the former space is eventually obtained by mimicking the contraction estimates of the Step 2. In fact, if $(f_1, \mathfrak{m}_1, u_1)$ and $(f_2, \mathfrak{m}_2, u_2)$ are two solutions in $\mathrm{X}_{\widetilde{T}^{[\eps]}}^m$ starting at the same initial condition, performing the same computations prove that there exists $\overline{T}^{[\eps]}$ (depending on $\Vert f_{1,2}, \mathfrak{m}_{1,2}, u_{1,2} \Vert_{\mathrm{X}_{\widetilde{T}^{[\eps]}}^m}$) and with $\overline{T}^{[\eps]}<\widetilde{T}^{[\eps]}$ such that for all $t \in [0,\overline{T}^{[\eps]}]$, we have
\begin{multline*}
\Vert (f_1-f_2)(t) \Vert_{\mathcal{H}^{0}_{r}}^2+\LRVert{(\mathfrak{m}_1-\mathfrak{m}_2)(t)}_{\Ld^2}^2+\Vert (u_1-u_2)(t) \Vert_{\Ld^2}^2 \\
\leq \frac{1}{2}\left( \Vert (f_1-f_2)(t) \Vert_{\mathcal{H}^{0}_{r}}^2+\LRVert{(\mathfrak{m}_1-\mathfrak{m}_2)(t)}_{\Ld^2}^2+\Vert (u_1-u_2)(t) \Vert_{\Ld^2}^2 \right),
\end{multline*}
from which we directly infer $(f_1, \mathfrak{m}_1, u_1)=(f_2, \mathfrak{m}_2, u_2)$ on $[0, \overline{T}^{[\eps]}]$. Repeating this procedure starting from $\overline{T}^{[\eps]}<T_\eps'$, we obtain $(f_1, \mathfrak{m}_1, u_1)=(f_2, \mathfrak{m}_2, u_2)$ on $[0, 2\overline{T}^{[\eps]}]$. After a finite number of steps, we eventually obtain uniqueness on $[0, \widetilde{T}^{[\eps]}]$.
%
\section{Pseudodifferential calculus with a large parameter on $\R \times \T^d$}\label{Section:Pseudo}
In this section, we collect several results on pseudodifferential calculus that we shall need in this article. We refer to \cite{AG,Met} for a more general and complete approach. Here, our framework is adapted to the physical space $\R \times \T^d$.

For any symbol of the form $a(t,x,\gamma, \tau,k)$ defined on $\R \times  \T^d \times (0, + \infty) \times \R \times \R^d {\setminus} \lbrace 0 \rbrace$, we use the quantization
\begin{align*}
\mathrm{Op}^{\gamma}(a)(h)(t,x)=\frac{1}{(2\pi)^{d+1}}\int_{\R \times \Z^d}e^{i(\tau t+ k  \cdot x)}a(t,x,\gamma,\tau,k) \mathcal{F}_{t,x}h(\tau,k) \, \mathrm{d}\tau \, \mathrm{d}k.
\end{align*}
Here, we use the discrete measure on $\Z^d$. 
Above, the variable $\gamma>0$ should be seen as a parameter. The Fourier transform of any function $h(t,x)$ defined on $\R \times \T^d$ is denoted by
\begin{align*}
\mathcal{F}_{t,x}h(\tau,k)&=\int_{\R \times \Z^d} e^{-i (\tau t+k  \cdot x)} h(t,x)   \, \mathrm{d}t \, \mathrm{d}x,
\end{align*}
while the Fourier transform of any symbol $a(t,x,\gamma, \tau,k)$
defined on $\R \times  \T^d \times (0, +\infty) \times \R \times \R^d {\setminus} \lbrace 0 \rbrace$ 
is written as
\begin{align*}
\mathcal{F}_{t,x} a(\theta, \ell, \gamma, \tau, k)&=\int_{\R \times \Z^d} e^{-i (\theta t+\ell  \cdot x)} a(t,x,\gamma, \tau,k)   \, \mathrm{d}t \, \mathrm{d}x.
\end{align*}
For the sake of readibility, we introduce the notation
\begin{align*}
\eta&:=(\gamma, \tau,k) \in (0, + \infty) \times \R \times \R^d {\setminus} {\lbrace 0 \rbrace}, \\ 
\vert \eta \vert&:=\left(\gamma^2+\tau^2 + \vert k \vert^2 \right)^{\frac{1}{2}}, 
\end{align*}and we will use the notation $\Ld^{\infty}_{t,x, \eta}$ to denote $\Ld^{\infty}(\R \times \T^d \times  (0, + \infty) \times \R \times \R^d {\setminus} {\lbrace 0 \rbrace})$.

\medskip

We state $\Ld^2$ continuity results of Calderón-Vaillancourt-type: namely, we shall ask for $\Ld^{\infty}$ bounds in all the variables for the symbols of the operators (see \cite{CaldVail, Hpseudo}). We first introduce the following family of seminorms for our symbols.
\begin{nota}
For any $\mathrm{M} \geq 0$ and for any symbol $a(t,x,\eta)$ with $\eta=(\gamma, \tau, k)$, we set
\begin{align}
\label{seminorm0}
\omega[a]&:= \underset{\substack{\alpha \in \N^d \\ \alpha_i \in \lbrace 0,1 \rbrace}}{\sup} \,  \left\Vert (1+t)\partial_{x}^{\alpha}   a \right\Vert_{\Ld^{\infty}_{t,x,\eta}}+\underset{\substack{\alpha \in \N^d \\ \alpha_i \in \lbrace 0,1 \rbrace}}{\sup} \,  \left\Vert (1+t)\partial_{t}\partial_x^{\alpha}   a \right\Vert_{\Ld^{\infty}_{t,x,\eta}}, \\
\label{seminormN}\Omega[a]&:=\underset{\substack{\alpha \in \N^d \\ \alpha_i \in \lbrace 0,1 \rbrace}}{\sup} \,   \left\lbrace \left\Vert  \vert \eta \vert \partial_{x}^{\alpha}   \nabla_{\tau,k} a \right\Vert_{\Ld^{\infty}_{t,x,\eta}}+\left\Vert  \vert \eta \vert \partial_{x}^{\alpha}   \nabla_{\tau,k} \partial_t a \right\Vert_{\Ld^{\infty}_{t,x,\eta}} \right \rbrace \\
& \notag \quad +\underset{\substack{\alpha \in \N^d \\ \alpha_i \in \lbrace 0,1 \rbrace}}{\sup} \,   \left\lbrace \left\Vert  \vert \eta \vert \partial_{x}^{\alpha}   \partial_{\tau}\nabla_{\tau,k} a \right\Vert_{\Ld^{\infty}_{t,x,\eta}}+\left\Vert  \vert \eta \vert \partial_{x}^{\alpha}   \partial_{\tau}\nabla_{\tau,k} \partial_t a \right\Vert_{\Ld^{\infty}_{t,x,\eta}} \right \rbrace, \\
\label{seminormA4}\Xi [a]_{\mathrm{M}}&:= \underset{\substack{1 \leq \vert \alpha \vert \leq 1+\mathrm{M}\\ \beta =0,1,2,3,4 } }{\sup} \, \left\Vert \partial_{x}^{\alpha} \partial_t^{\beta} a  \right\Vert_{\Ld^{\infty}_{t,x,\eta}}.
\end{align}
\end{nota}
The following result states the $\Ld^2$ continuity of pseudodifferential operator with symbol having a finite seminorm $\omega[\cdot]$. We refer to \cite[Theorem 1]{Hpseudo} for a proof for symbols with compact support in $x$, whose adaptation to the physical space $\R \times \T^d$ is fairly straightforward\footnote{More precisely, one can introduce a weight $(1+t)^{-2}$ in \cite[proof of Theorem 1, eq $\mathrm{(2.5)}$]{Hpseudo} to get some integrability in time. 
This turns out to be sufficient to consider the seminorm $\omega[\cdot]$ in our statement.}.

\begin{thm}\label{prop:conti-pseudoCV}
There exists $C_d>0$ such that if $a$ is a symbol satisfying $\omega[a]<\infty$ 
then the following holds: for every $\gamma >0$, we have
\begin{align*}
\forall  h \in \Ld^2(\R \times \T^d), \ \ \Vert \mathrm{Op}^{\gamma}(a)(h) \Vert_{\Ld^2(\R \times \T^d)} \leq C_d \omega[a] \Vert h \Vert_{\Ld^2(\R \times \T^d)},
\end{align*}
\end{thm}
Next, we have the following symbolic calculus result, with the use of the parameter $\gamma >0$. Note that taking $\gamma$ large can obviously be useful in view of an absorption argument. Here, we need to assume that one symbol has a compact support in time because the time variable $t \in \R$ in unbounded.

\begin{propo}\label{prop:compo-pseudoCV}
There exists $C_d>0$ and continuous nonnegative and nondecreasing function $\Lambda$ such that for any symbols $a,b$ satisfying
\begin{align*}
\Omega[a] , \Xi[b]_{\mathrm{M}} < \infty, \ \ \mathrm{M}>1+2d, \\
\nabla_x b  \text{ has compact support in time},
\end{align*}
the following holds: for all $\gamma>0$ and all $ h \in \Ld^2(\R \times \T^d)$, we have
\begin{align*}
\Vert \mathrm{Op}^{\gamma}(a)\mathrm{Op}^{\gamma}(b)(h)-\mathrm{Op}^{\gamma}(ab)(h) \Vert_{\Ld^2(\R \times \T^d)} \leq \frac{C_d}{\gamma}  \Lambda(\vert \mathrm{supp}_t \, \nabla_x b \vert)  \Omega[a]\Xi[b]_{\mathrm{M}} \Vert h \Vert_{\Ld^2(\R \times \T^d)}.
\end{align*}
\end{propo}
\begin{proof}
A standard formula about the composition of pseudodifferential operators first shows that
\begin{align*}
\mathrm{Op}^{\gamma}(a)\mathrm{Op}^{\gamma}(b)=\mathrm{Op}^{\gamma}(c),
\end{align*}
with
\begin{align*}
c(t,x, \gamma, \tau, k)&=\frac{1}{(2\pi)^{d+1}}\int_{\R \times \T^d}\int_{\R \times \Z^d}e^{i(\tau'-\tau) (t-t')}e^{i(k'-k)\cdot (x-x')}a(t,x,\gamma,\tau', k')b(t',x',\gamma, \tau, k) \, \mathrm{d} t'  \mathrm{d}x'  \mathrm{d}\tau'  \mathrm{d}k' \\
& =\frac{1}{(2\pi)^{d+1}}\int_{\R \times \Z^d}e^{i(\tau't+k'\cdot x) }a(t,x,\gamma,\tau+\tau', k+k')\mathcal{F}_{t,x}b(\tau',k',\gamma, \tau, k) \, \mathrm{d}\tau'  \mathrm{d}k'.
\end{align*}
Therefore for $\eta=(\gamma, \tau, k) \in (0, + \infty) \times \R \times \R^d {\setminus} {\lbrace 0 \rbrace}$, we have
\begin{align*}
c(t,x, \eta)-a(t,x, \eta)b(t,x, \eta)
&=\frac{1}{(2\pi)^{d+1}}\int_{\R \times \Z^d}e^{i(\tau't+k'\cdot x) } \mathcal{F}_{t,x}b(\tau',k',\eta) \\
&  \qquad \qquad\qquad\left\lbrace \int_0^1 \nabla_{\tau,k}a(t,x,\gamma,\tau+s\tau', k+sk') \cdot (\tau', k') \, \mathrm{d}s \right\rbrace \, \mathrm{d}\tau'  \mathrm{d}k' \\
&=: \frac{1}{\gamma }\mathrm{m}(t,x, \eta),
\end{align*}
where
\begin{align*}
\mathrm{m}(t,x, \eta):=\frac{1}{(2\pi)^{d+1}}\int_{\R \times \Z^d}e^{i(\tau't+k'\cdot x) } \mathcal{F}_{t,x}b(\tau',k',\eta) \mathcal{J}(a)(t,x, \eta, \tau', k') \, \mathrm{d}\tau'  \mathrm{d}k',
\end{align*}
with
\begin{align*}
\mathcal{J}(a)(t,x, \eta, \tau', k'):=\int_0^1 \gamma \nabla_{\tau,k}a(t,x,\gamma,\tau+s\tau', k+sk') \cdot (\tau', k') \, \mathrm{d}s.
\end{align*}
Since $\mathrm{Op}^{\gamma}(a)\mathrm{Op}^{\gamma}(b)-\mathrm{Op}^{\gamma}(ab)=\frac{1}{\gamma}\mathrm{Op}^{\gamma}( \mathrm{m})$, and in view of the continuity property stated in Theorem \ref{prop:conti-pseudoCV}, it remains to estimate the seminorm $\omega[\mathrm{m}]$. For all $\alpha \in \N^d$ such that $\alpha_i  \in \lbrace 0, 1 \rbrace$, we now have
\begin{align*}
(1+t)\mathrm{m}(t,x, \eta)&=\frac{1}{(2\pi)^{d+1}}\int_{\R \times \Z^d}(1-i\partial_{\tau'})\big(e^{i\tau't}\big) e^{ik'\cdot x } \mathcal{F}_{t,x}b(\tau',k',\eta) \mathcal{J}(a)(t,x, \eta, \tau', k') \mathrm{d}\tau'  \mathrm{d}k',
\end{align*}
and
\begin{align*}
(1+t)\partial_t\mathrm{m}(t,x, \eta)&=\frac{1}{(2\pi)^{d+1}}\int_{\R \times \Z^d}i \tau'(1-i\partial_{\tau'})\big(e^{i\tau't}\big) e^{ik'\cdot x } \mathcal{F}_{t,x}b(\tau',k',\eta) \mathcal{J}(a)(t,x, \eta, \tau', k') \mathrm{d}\tau'  \mathrm{d}k' \\
& +\quad \frac{1}{(2\pi)^{d+1}}\int_{\R \times \Z^d}(1-i\partial_{\tau'})\big(e^{i\tau't}\big) e^{ik'\cdot x } \mathcal{F}_{t,x}b(\tau',k',\eta) \partial_t\left(\mathcal{J}(a) \right)(t,x, \eta, \tau', k') \mathrm{d}\tau'  \mathrm{d}k'.
\end{align*}
Tedious but standard computations then show that for $\mathrm{M}>d$, we have
\begin{align*}
\omega[\mathrm{m}] &\lesssim \Omega[a]\underset{\substack{\alpha \in \N^d \\ \alpha_i \in \lbrace 0,1 \rbrace}}{\sup} \,  \left\Vert  (1+\vert \tau' \vert^2) \vert k'\vert^{1+\vert \alpha \vert }  \mathcal{F}_{t,x}\big[(1+t)b \big] \right\Vert_{\Ld^1(\R_{\tau'} \times \Z^d_{k'};\Ld^{\infty}_{\eta})} \\
& \lesssim \Omega[a] \left( \underset{1 \leq \vert \alpha \vert \leq 1+\mathrm{M}  }{\sup} \, \left\Vert  \mathcal{F}_{t,x}\big[(1+t)\partial_{x}^{\alpha} b \big]\right\Vert_{\Ld^1(\R \times \Z^d;\Ld^{\infty}_{\eta})} +\underset{\substack{1\leq \vert \alpha \vert \leq 1+\mathrm{M} \\ \beta =1,2  }}{\sup} \, \left\Vert  \mathcal{F}_{t,x}\big[(1+t)\partial_{x}^{\alpha} \partial_t^{\beta} b \big] \right\Vert_{\Ld^1(\R \times \Z^d;\Ld^{\infty}_{\eta})}  \right).
\end{align*}
As a consequence, we obtain for $\mathrm{M}>1+2d$
\begin{align*}
&\omega[\mathrm{m}]\\
 & \lesssim   \Omega[a] \left( \underset{\substack{1 \leq \vert \alpha \vert \leq 1+\mathrm{M} \\ \beta=0,1,2}}{\sup} \, \left\Vert  \mathcal{F}_{t,x}\big[(1+t)\partial_{x}^{\alpha} \partial_t^{\beta} b \big]\right\Vert_{\Ld^{\infty}(\R \times \Z^d;\Ld^{\infty}_{\eta})} +\underset{\substack{1 \leq \vert \alpha \vert \leq 1+\mathrm{M} \\ \beta =0,1,2,3,4  }}{\sup} \, \left\Vert  \mathcal{F}_{t,x}\big[(1+t)\partial_{x}^{\alpha} \partial_t^{\beta} b \big] \right\Vert_{\Ld^{\infty}(\R \times \Z^d;\Ld^{\infty}_{\eta})}  \right) \\
& \lesssim \Omega[a] \underset{\substack{1 \leq \vert \alpha \vert \leq 1+ \mathrm{M}\\ \beta =0,1,2,3,4 } }{\sup} \, \left\Vert  (1+t)\partial_{x}^{\alpha} \partial_t^{\beta} b  \right\Vert_{\Ld^{1}(\R \times \T^d;\Ld^{\infty}_{\eta})},
\end{align*}
and therefore
\begin{align*}
\omega[\mathrm{m}]  & \lesssim  \Lambda(\vert \mathrm{supp}_t \, \nabla_x b \vert)\Omega[a]\Xi[b]_{\mathrm{M}}.
\end{align*}
From Theorem \ref{prop:conti-pseudoCV}, we have for all $ h \in \Ld^2(\R \times \T^d)$
\begin{align*}
\Vert \mathrm{Op}^{\gamma}(\mathrm{m})h \Vert_{\Ld^2(\R \times \T^d)} \leq C_d \omega[\mathrm{m}]  \Vert h \Vert_{\Ld^2(\R \times \T^d)} \lesssim C_d \Lambda(\vert \mathrm{supp}_t \, \nabla_x b \vert)  \Omega[a]  \Xi[b]_{\mathrm{M}} \Vert h \Vert_{\Ld^2(\R \times \T^d)}.
\end{align*}
We obtain the conclusion since $\mathrm{Op}^{\gamma}(a)\mathrm{Op}^{\gamma}(b)-\mathrm{Op}^{\gamma}(ab)=\frac{1}{\gamma}\mathrm{Op}^{\gamma}( \mathrm{m})$.

\end{proof}

\section*{Acknowledgements}
We are grateful to Laurent Boudin, Laurent Desvillettes, Julien Mathiaud and Ayman Moussa for inspiring discussions concerning spray equations during the preparation of this article. This work originated from our collaboration with Aymeric Baradat, whom we warmly thank.

Partial support by the grant ANR-19-CE40-0004 is acknowledged.

\bibliographystyle{alpha}
\bibliography{biblio}
\end{document}